\documentclass[10pt,a4paper]{amsart}
\usepackage[lmargin=0.9in,rmargin=0.9in]{geometry}
\usepackage[T1]{fontenc}
\usepackage[utf8]{inputenc}
\usepackage[british]{babel}
\usepackage{mathtools}
\usepackage{amsthm}
\usepackage{lmodern}
\usepackage{calrsfs}
\usepackage{amssymb}
\usepackage{amsmath, calligra, mathrsfs}
\usepackage[mathscr]{euscript}
\usepackage{enumitem}
\usepackage{tikz-cd}
\usetikzlibrary{decorations.markings}
\usepackage{hyperref}
\usepackage[noabbrev,nameinlink]{cleveref}
\theoremstyle{plain}
\newtheorem{thm}{Theorem}[section]
\newtheorem*{thm*}{Theorem}
\newtheorem{lem}[thm]{Lemma}
\newtheorem{prop}[thm]{Proposition}
\newtheorem{cor}[thm]{Corollary}

\theoremstyle{definition}
\newtheorem{defn}[thm]{Definition}
\newtheorem{nota}[thm]{Notation}

\newtheorem{constr}[thm]{Construction}

\theoremstyle{remark}
\newtheorem{rem}[thm]{Remark}
\Crefname{thm}{Theorem}{Theorems}
\Crefname{lm}{Lemma}{Lemmata}
\Crefname{prop}{Proposition}{Propositions}
\Crefname{cor}{Corollary}{Corollaries}
\Crefname{hyp}{Hypothesis}{Hypotheses}
\Crefname{q}{Question}{Questions}
\Crefname{defn}{Definition}{Definitions}
\Crefname{nota}{Notation}{Notations}
\Crefname{ex}{Example}{Examples}
\Crefname{xca}{Exercise}{Exercises}
\Crefname{rem}{Remark}{Remarks}
\Crefname{constr}{Construction}{Constructions}
\Crefname{conj}{Conjecture}{Conjecture}

\newcommand{\Z}{\mathbb{Z}}
\newcommand{\Q}{\mathbb{Q}}

\newcommand{\C}{\mathbb{C}}
\newcommand{\Acal}{\mathcal{A}}
\newcommand{\Bcal}{\mathcal{B}}
\newcommand{\Ccal}{\mathcal{C}}
\newcommand{\Dcal}{\mathcal{D}}
\newcommand{\Ecal}{\mathcal{E}}
\newcommand{\Fcal}{\mathcal{F}}
\newcommand{\Gcal}{\mathcal{G}}

\newcommand{\Mcal}{\mathcal{M}}

\newcommand{\Ocal}{\mathcal{O}}

\newcommand{\Scal}{\mathcal{S}}
\newcommand{\Ucal}{\mathcal{U}}

\newcommand{\tfr}{\mathfrak{t}}
\newcommand{\id}{\textup{id}}
\newcommand{\Hom}{\textup{Hom}}
\newcommand{\unit}{\mathbf{1}}
\newcommand{\Sets}{\mathbf{Sets}}
\newcommand{\Ab}{\mathbf{Ab}}

\newcommand{\Hbb}{\mathbb{H}} 
\newcommand{\coker}{\textup{coker}} 
\newcommand{\im}{\textup{im}} 
\newcommand{\Fun}{\mathbf{Fun}} 
\newcommand{\Op}{\mathsf{Op}}

\newcommand{\vect}{\textup{vect}}
\newcommand{\Thom}{\mathsf{Th}}
\newcommand*{\isoarrow}[1]{\arrow[#1,"\rotatebox{90}{\(\sim\)}"]}
\newcommand{\Dbb}{\mathbb{D}}

\newcommand{\Aff}{\mathsf{Aff}}

\newcommand{\Sm}{\mathsf{Sm}}

\newcommand{\Spec}{\textup{Spec}}

\newcommand{\et}{\textup{ét}}
\newcommand{\Var}{\mathsf{Var}}
\newcommand{\A}{\mathbf{A}} 
\newcommand{\Rbf}{\mathbf{R}}
\newcommand{\Lbf}{\mathbf{L}}
\newcommand{\DA}{\mathbf{DA}} 
\newcommand{\DM}{\mathbf{DM}}
\newcommand{\Bti}{\mathsf{Bti}}
\newcommand{\Nri}{\mathsf{Nri}}

\newcommand{\can}{\textup{can}}
\newcommand{\var}{\textup{var}}
\newcommand{\Pairs}{\mathsf{Pairs}}

\newcommand{\conn}{\textup{conn}}
\newcommand{\Perv}{\mathsf{Perv}}
\newcommand{\Cof}{\textup{Cof}}

\newcommand{\Tot}{\textup{Tot}}
\newcommand{\Homcal}{\underline{\textup{Hom}}} 
\newcommand{\Rdf}{\mathbf{R}} 
\newcommand{\Ind}{\mathrm{Ind}} 

\newcommand{\Ct}{\mathcal{C}\textit{t}} 
\newcommand{\Loc}{\mathsf{Loc}} 

\newcommand{\twocolim}{\textup{2-colim}}

\newcommand\blfootnote[1]{%
	\begingroup
	\renewcommand\thefootnote{}\footnote{#1}%
	\addtocounter{footnote}{-1}%
	\endgroup
}

\newlength{\offsetpage}
\newenvironment{widepage}{\begin{adjustwidth}{-\offsetpage}{-\offsetpage}%
		\addtolength{\textwidth}{2\offsetpage}}%
	{\end{adjustwidth}}

\usepackage{changepage} 

\title{Tensor structure on perverse Nori motives}
\author{Luca Terenzi}
\address{Luca Terenzi \newline
	\indent UMPA, ENS de Lyon \newline
	\indent 46 Allée d'Italie, 69007 Lyon (France)}
\email{\normalfont\href{mailto:luca.terenzi@ens-lyon.fr}{luca.terenzi@ens-lyon.fr}}

\makeatletter
\hypersetup{
	pdfauthor={\author},
	pdftitle={\@title},
	colorlinks,
	linkcolor=[rgb]{0.2,0.2,0.6},
	citecolor=[rgb]{0.2,0.6,0.2},
	urlcolor=[rgb]{0.6,0.2,0.2}}
\makeatother

\calclayout
\begin{document}

\maketitle

\begin{abstract}
	Let $k$ be a field of characteristic $0$ endowed with a complex embedding $\sigma: k \hookrightarrow \C$. 
	In this paper we complete the construction of the six functor formalism on perverse Nori motives over quasi-projective $k$-varieties, initiated by Ivorra--Morel. 
	Our main contribution is the construction of a closed monoidal structure on the derived categories of perverse Nori motives, compatibly with the analogous structure on the underlying constructible derived categories.
	This is based on an alternative presentation of perverse Nori motives, related to the conjectural motivic perverse $t$-structure on Voevodsky motivic sheaves.
	As a consequence, we obtain well-behaved Tannakian categories of motivic local systems over smooth, geometrically connected $k$-varieties.
	Exploiting the relation with Voevodsky motivic sheaves in its full strength, we are able to define Chern classes in the setting of perverse Nori motives, which leads to a motivic version of the relative Hard Lefschetz Theorem.
\end{abstract}

\tableofcontents

\blfootnote{\textit{\subjclassname}. 14F42, 14F43, 14C15.}
\blfootnote{\textit{\keywordsname}. Motives, perverse sheaves, sheaf operations.}

\blfootnote{The author acknowledges support by the GK 1821 "Cohomological Methods in Geometry" at the University of Freiburg and by the Labex Milyon at the ENS Lyon.}

\section*{Introduction}

\subsection*{Motivation and goal of the paper}

Let $k$ be a field of characteristic $0$ endowed with a complex embedding $\sigma: k \hookrightarrow \C$. 
In the late 1990s, Nori constructed an unconditional candidate for the conjectural $\Q$-linear abelian category of \textit{mixed motives} over $k$. 
By design, Nori's category $\Mcal(k)$ is the target of the universal cohomology theory on $k$-varieties comparable to $\Q$-linear singular cohomology of the underlying complex-analytic spaces. 
It comes equipped with a faithful exact functor into finite-dimensional $\Q$-vector spaces
\begin{equation}\label{iota_k:intro}
	\iota_k: \Mcal(k) \rightarrow \vect_{\Q},
\end{equation}  
which recovers singular cohomology groups when evaluated on the cohomological motives of $k$-varieties. 
Through the classical comparison isomorphisms relating singular cohomology with algebraic de Rham cohomology and with the various $\ell$-adic cohomologies, Nori's category acquires similar forgetful functors into the associated categories of coefficients.

According to the conjectures summarized by Beilinson in \cite[\S~5.10]{BeiHeight}, the theory of mixed motives should extend to a coefficient system: 
there should exist $\Q$-linear abelian categories of \textit{mixed motivic sheaves} over all (quasi-projective) $k$-varieties $S$, providing a motivic enhancement of the classical categories of perverse sheaves; 
as $S$ varies, their bounded derived categories should admit a complete six functor formalism, analogously to the classical constructible derived categories, as well as a compatible theory of weights.
The theory of \textit{perverse Nori motives}, introduced by Ivorra--Morel in \cite{IM19}, provides unconditional candidates for the conjectural abelian categories of mixed motivic sheaves: 
these are abelian categories $\Mcal(S)$ equipped with faithful exact functors into algebraically constructible $\Q$-linear perverse sheaves
\begin{equation}\label{iota_S:intro}
	\iota_S: \Mcal(S) \rightarrow \Perv(S),
\end{equation} 
recovering Nori's category $\Mcal(k)$ and its forgetful functor \eqref{iota_k:intro} when $S = \Spec(k)$. 
The main result \cite[Thm.~5.1]{IM19} asserts that, as $S$ varies, the derived categories $D^b(\Mcal(S))$ are endowed with Grothendieck's operations $f^*$, $f_*$, $f_!$ and $f^!$, compatibly with the analogous operations on the underlying constructible derived categories. 
The main goal of the present paper is to complete the construction of this motivic six functor formalism:
we define a canonical closed monoidal structure on the derived categories $D^b(\Mcal(S))$, again compatibly with the one on the underlying constructible derived categories.

\subsection*{History and previous work}

Before explaining our results in more detail, let us review the construction of perverse Nori motives. 
In order to put it into perspective, it is convenient to start from the very beginning of the theory. 
Nori's original definition of $\Mcal(k)$ is based on an abstract result about representations of quivers into vector spaces. 
The idea is to introduce a quiver $\Pairs_k$ encoding the basic features of singular cohomology of $k$-varieties: functoriality along morphisms of pairs, long exact cohomology sequences of triples, and Tate twists. 
Vertices of $\Pairs_k$ are labels of the form $(X,Y;n,w)$ consisting of a $k$-variety $X$, a closed subvariety $Y$, a cohomological degree $n$ and a Tate degree $w$. 
Singular cohomology can be regarded as a quiver representation
\begin{equation}\label{b_k:intro}
	b_k: \Pairs_k \rightarrow \vect_{\Q}, \quad (X,Y;n,w) \mapsto H^n(X,Y;\Q)(w),
\end{equation}
and the conjectural category of mixed motives over $k$ should be the smallest non-full abelian subcategory of $\vect_{\Q}$ containing the image of \eqref{b_k:intro}. 
Nori showed how to construct such an abelian category starting from an abstract quiver representation: 
in the case of the singular representation \eqref{b_k:intro}, this defines the abelian category $\Mcal(k)$.
While Nori's construction was tailored to quiver representations into finite-dimensional vector spaces (or, slightly more generally, finitely generated modules over a Noetherian ring), it was later extended to representations into arbitrary abelian categories in \cite{BVCL18}, using the language of categorical logic. 
The latter result was then translated into the familiar language of categories in \cite{BV-P}, based on Freyd's classical construction of the abelian hull of an additive category from \cite{Freyd}. 
The final version asserts that, for every quiver representation into an abelian category
\begin{equation*}
	\beta: \Dcal \rightarrow \Acal,
\end{equation*}
there exists a smallest non-full abelian subcategory of $\Acal$ containing the image of $\beta$, called the associated \textit{universal abelian factorization}.
However, the simplicity of Nori's construction and of its later variants hides a serious drawback: it does not allow one to compute morphisms (let alone extension groups) in $\Mcal(k)$ explicitly, except for a few trivial cases. 

One of the most surprising achievements of Nori's work was the construction of a tensor product on $\Mcal(k)$ making it neutral Tannakian over $\Q$, with canonical fibre functor the forgetful functor \eqref{iota_k:intro}. 
The Tannaka dual $\Gcal_{mot}(k)$ of $\Mcal(k)$ is called the \textit{motivic Galois group} of $k$. 
This mysterious group allows for an elegant motivic interpretation of the Kontsevich--Zagier conjecture about periods of algebraic varieties - this was arguably Nori's original motivation for introducing $\Mcal(k)$ (see \cite[\S\S~II, III]{HMS17} for more details).

In contrast with the existence of the abelian category $\Mcal(k)$, the existence of the tensor structure is a highly non-formal result. 
In order to obtain a tensor product on $\Mcal(k)$, one needs to define a suitable pairing on the quiver $\Pairs_k$, compatibly with tensor product of singular cohomology groups. 
However, the only reasonable guess for such a pairing on $\Pairs_k$, given by the rule
\begin{equation*}
	(X_1,Y_1;n_1,w_1) \times (X_2,Y_2;n_2,w_2) := (X_1 \times X_2, Y_1 \times X_2 \cup X_1 \times Y_2; n_1 + n_2, w_1 + w_2),
\end{equation*}
is not always compatible with tensor product in singular cohomology.
By K\"{u}nneth's formula, this naive pairing is well-behaved only when restricted to the subquiver $\Pairs_k^0 \subset \Pairs_k$ collecting all those vertices $(X,Y;n,w)$ such that the singular cohomology of the pair $(X,Y)$ (twisted by $w$) is concentrated precisely in degree $n$. 
Applying Nori's categorical construction to the restricted representation
\begin{equation*}
	b_k^0 := b_k|_{\Pairs_k^0}: \Pairs_k^0 \rightarrow \vect_{\Q},
\end{equation*}
one obtains an abelian $\Mcal^0(k)$ endowed with a well-defined tensor product
\begin{equation*}
	- \otimes -: \Mcal^0(k) \times \Mcal^0(k) \rightarrow \Mcal^0(k).
\end{equation*}
Now the problem becomes to show that the natural faithful exact functor $\Mcal^0(k) \to \Mcal(k)$ is in fact an equivalence. 
This was proved by Nori using Beilinson's result \cite[Lemma~3.3]{Bei87}, quoted in Nori's paper \cite{NoriConstr} as the \textit{Basic Lemma}. 
The Basic Lemma is interpreted by Nori as a result about the existence of algebraic skeletal filtrations on affine $k$-varieties.
It is also the fundamental ingredient in the geometric interpretation of the perverse filtration in cohomology obtained by de Cataldo--Migliorini in \cite{DeC-Mig}.

The Basic Lemma turns out to be a crucial input into another major result of Nori's theory: 
namely, the existence of a canonical triangulated realization 
\begin{equation}\label{R_k:intro}
	\Nri^*_k: \DM_{gm}(k,\Q) \rightarrow D^b(\Mcal^0(k)).
\end{equation}
Here, $\DM_{gm}(k,\Q)$ denotes Voevodsky's triangulated category of $\Q$-linear geometric motives over $k$, which so far represents the best known candidate for the bounded derived category of the conjectural abelian category of mixed motives over $k$; 
it is endowed with a triangulated functor
\begin{equation}\label{Bti_k:intro}
	\Bti_k^*: \DM_{gm}(k,\Q) \rightarrow D^b(\vect_{\Q})
\end{equation} 
called the Betti realization, which should be regarded as a triangulated analogue of \eqref{iota_k:intro}. 
Nori's arguments show that the composite functor
\begin{equation*}
	\DM_{gm}(k,\Q) \xrightarrow{\Nri^*_k} D^b(\Mcal^0(k)) \xrightarrow{\iota_k} D^b(\vect_{\Q})
\end{equation*}
recovers the Betti realization \eqref{Bti_k:intro} up to canonical natural isomorphism. 
Since, in addition, the singular representation \eqref{b_k:intro} factors as the composite
\begin{equation*}
	\Pairs_k \xrightarrow{M_k} \DM_{gm}(k,\Q) \xrightarrow{\Bti_k^*} D^b(\vect_{\Q}) \xrightarrow{H^0} \vect_{\Q}
\end{equation*}
for a suitable map of quivers $M_k: \Pairs_k \rightarrow \DM_{gm}(k,\Q)$, the equivalence $\Mcal^0(k) = \Mcal(k)$ follows formally. 
Furthermore, it follows that Nori's category $\Mcal(k)$ can be reconstructed directly from Voevodsky's category $\DM_{gm}(k,\Q)$: 
namely, $\Mcal(k)$ coincides with the universal abelian factorization of the composite functor
\begin{equation}\label{eq:beta_k-intro}
	\beta_k: \DM_{gm}(k,\Q) \xrightarrow{\Bti^*_k} D^b(\vect_{\Q}) \xrightarrow{H^0} \vect_{\Q},
\end{equation}
regarded as a quiver representation. 
As pointed out in \cite[Thm.~4.10]{BVHP20}, the tensor product on $\Mcal(k)$ can be also constructed starting from the tensor product of $\DM_{gm}(k,\Q)$, and Nori's functor \eqref{R_k:intro} is tensor-triangulated.
The existence of a natural bridge between Voevodsky's and Nori's theories is a strong indication that Nori's category should gives the correct definition of mixed motives over $k$. 
In fact, the triangulated category $\DM_{gm}(k,\Q)$ is expected to carry a \textit{motivic $t$-structure} making the Betti realization \eqref{Bti_k:intro} $t$-exact; 
its heart would be the conjectural abelian category of mixed motives (see \cite[\S~3.3]{AyoConj}). 
A formal consequence of the construction of $\Mcal(k)$ out of \eqref{eq:beta_k-intro} is that, if the motivic $t$-structure on $\DM_{gm}(k,\Q)$ exists, its heart is necessarily equivalent to $\Mcal(k)$. 
Let us mention that an alternative construction of the motivic Galois group $\Gcal_{mot}(k)$ avoiding Nori motives was provided by Ayoub in \cite{Ayo14H1, Ayo14H2}, based on his weak Tannakian formalism for tensor-triangulated categories. 
The fact that Ayoub's motivic Galois group coincides with Nori's $\Gcal_{mot}(k)$ was proved by Chudhurry--Gallauer in \cite{ChudGal}, using Nori's realization \eqref{R_k:intro} in a crucial way.

The theory of Voevodsky motives has been already expanded to a theory of motivic sheaves satisfying a large part of Beilinson's predictions:
for every (quasi-projective) $k$-variety $S$ there exists a triangulated category $\DM_{gm}(S,\Q)$ of $\Q$-linear geometric motives over $S$; 
as $S$ varies, they enjoy a complete six functor formalism, and they admit a compatible weight structure in the sense of Bondarko's work \cite{Bon14}. 
The construction of the six operations in the motivic world was sketched by Voevodsky in his lecture \cite{DelVoe}, studied systematically in Ayoub's thesis \cite{Ayo07a,Ayo07b}, and further refined in Cisinski--Déglise's book \cite{CisDeg}. 
By \cite[Thm.~B.1]{AyoEt}, Voevodsky's category $\DM_{gm}(S,\Q)$ is canonically equivalent to Ayoub's category $\DA_{ct}^{\et}(S,\Q)$ of constructible étale motives over $S$ without transfers, defined in \cite[\S~3]{AyoEt}; 
by \cite[Thm.~16.2.18]{CisDeg}, the latter is also equivalent to the category of constructible Beilinson motives, defined in \cite[\S~IV]{CisDeg}.
The theory of \cite{IM19} is developed using mostly Ayoub's categories $\DA_{ct}^{\et}(S,\Q)$ and adopting the language of \textit{stable homotopy $2$-functors} developed in \cite{Ayo07a,Ayo07b}, complemented by \cite{Ayo10} and \cite{AyoEt}. 
For every $k$-variety $S$, Ayoub constructed in \cite{Ayo10} a triangulated Betti realization
\begin{equation*}
	\Bti_{S}^*: \DA_{ct}^{\et}(S,\Q) \rightarrow D^b_c(S,\Q)
\end{equation*} 
into the usual algebraically constructible derived category $D^b_c(S,\Q)$ associated to the complex-analytic space $S^{\sigma}$. 
By \cite[Thm.~3.19]{Ayo10}, the Betti realization functors are compatible with the six operations, thereby defining a morphism of stable homotopy $2$-functors
\begin{equation}\label{Bti:intro}
	\Bti^* \colon \DA_{ct}^{\et}(\cdot,\Q) \rightarrow D^b_c(\cdot,\Q).
\end{equation} 
By \cite[Prop.~4.7]{Ayo10}, they are also compatible with nearby and vanishing cycles.

The category of perverse Nori motives $\Mcal(S)$ is defined in \cite[\S~2.1]{IM19} as the universal abelian factorization of the homological functor
\begin{equation}\label{intro:beta_S}
	\beta_S: \DA_{ct}^{\et}(S,\Q) \xrightarrow{\Bti_{S}^*} D^b_c(S,\Q) \xrightarrow{{^p H^0}} \Perv(S),
\end{equation}
which may be also regarded as a quiver representation. 
The associated forgetful functor \eqref{iota_S:intro} extends to a conservative triangulated functor
\begin{equation}\label{iota_S:Db-intro}
	\iota_S \colon D^b(\Mcal(S)) \to D^b(\Perv(S)) \xrightarrow{\sim} D^b_c(S,\Q),
\end{equation}
where the second passage witnesses Beilinson's equivalence \cite[Thm.~1.3]{Bei87}.
By \cite[Thm.~5.1]{IM19}, the triangulated categories $D^b(\Mcal(S))$ assemble into a stable homotopy $2$-functor, and the triangulated functors \eqref{iota_S:Db-intro} assemble into a morphism of stable homotopy $2$-functors
\begin{equation}\label{iota-mot}
	\iota: D^b(\Mcal(\cdot)) \rightarrow D^b(\Perv(\cdot)).
\end{equation}
Furthermore, as proved in \cite[\S~3]{IM19}, the theory of perverse Nori motives is endowed with the analogue of Beilinson's gluing functors from \cite{BeiGlue}, including unipotent nearby and vanishing cycles. 
These functors, whose definition builds upon the theory of specialization systems of \cite{Ayo07b}, are used crucially in the construction of the motivic inverse and direct image functors in \cite[\S\S~4, 5]{IM19}.
Indeed, the very definition of perverse Nori motives as universal abelian factorizations allows one to lift functors from the underlying constructible derived categories, provided they can be expressed in terms of the six operations and they are $t$-exact for the perverse $t$-structures.
This lifting principle allows one to construct inverse images under smooth morphisms and direct images under closed immersions directly, but cannot be applied to construct general inverse or direct image functors.
The arguments of \cite{IM19} show that one can recover the four operations by combining $t$-exact functors and natural transformations thereof in a clever way.

\subsection*{Main results}

Our main result asserts that the four functor formalism on perverse Nori motives constructed in \cite{IM19} can be completed to a six functor formalism. In the language of stable homotopy $2$-functors, it can be conveniently summarized as follows:

\begin{thm*}[\Cref{thm-boxtimesM}, \Cref{thm_homcal}]
	The stable homotopy $2$-functor $D^b(\Mcal(\cdot))$ is canonically closed unitary symmetric monoidal. 
	Moreover, the morphism of stable homotopy $2$-functors \eqref{iota-mot} is canonically unitary symmetric monoidal, and it respects internal homomorphisms.
\end{thm*}
\noindent
The first step is the construction of the monoidal structure.
One basic technical issue is that the internal tensor product
\begin{equation*}
	- \otimes -: D^b(\Mcal(S)) \times D^b(\Mcal(S)) \rightarrow D^b(\Mcal(S))
\end{equation*} 
is not $t$-exact for the perverse $t$-structures unless $\dim(S) = 0$. 
The abstract theory developed in \cite{Ter23Fib} allows us to overcome this issue by working with the external tensor product
\begin{equation*}
	- \boxtimes -: D^b(\Mcal(S_1)) \times D^b(\Mcal(S_2)) \rightarrow D^b(\Mcal(S_1 \times S_2)),
\end{equation*}
which is $t$-exact in each variable by \cite[\S~4.2.7]{BBD82}. 
But this is still not enough, even when $\dim(S) = 0$:
indeed, in order to apply the general lifting principles of universal abelian factorizations, one would need to have the equality
\begin{equation*}
	{^p H^0}(K_1^{\bullet}) \boxtimes {^p H^0}(K_2^{\bullet}) = {^p H^0}(K_1^{\bullet} \boxtimes K_2^{\bullet})
\end{equation*} 
for all choices of $K_i^{\bullet} \in D^b(\Perv(S_i))$.
This naive equality does not hold unless, for example, one among $K_1^{\bullet}$ and $K_2^{\bullet}$ is concentrated in perverse degree $0$. 
In fact, this is just a different incarnation of the K\"{u}nneth formula issue for the tensor product on $\Mcal(k)$, which we commented on above.
Inspired by Nori's strategy, for every $k$-variety $S$ we introduce a convenient variant $\Mcal^0(S)$ of $\Mcal(S)$ on which the external tensor product can be defined directly, and we then show that it coincides with $\Mcal(S)$. 
Here is the precise statement:
\begin{thm*}[\Cref{thm_Mp0=Mp}]
	For every (quasi-projective) $k$-variety $S$, consider the full additive subcategory
	\begin{equation*}
		\DA_{ct}^{\et,0}(S,\Q) := \left\{A \in \DA_{ct}^{\et}(S,\Q) \; | \; \Bti_S^*(A) \in \Perv(S) \right\} \subset \DA_{ct}^{\et}(S,\Q),
	\end{equation*}  
    and let $\Mcal^0(S)$ denote the universal abelian factorization of the additive functor
	\begin{equation*}
		\beta_S^0 := \beta_S|_{\DA_{ct}^{\et,0}(S,\Q)}: \DA_{ct}^{\et,0}(S,\Q) \rightarrow \Perv(S).
	\end{equation*}
    Then the canonical faithful exact functor $\Mcal^0(S) \hookrightarrow \Mcal(S)$ is in fact an equivalence.
\end{thm*}
\noindent
Note that $\DA_{ct}^{\et,0}(S,\Q)$ is the candidate for the heart of the conjectural motivic perverse $t$-structure on $\DA_{ct}(S,\Q)$.
Our result asserts that, as far as the Betti realization can measure, the putative perverse heart really contains the perverse cohomology objects of all motivic complexes;
some results in the same vein have been obtained by Cisinski in his recent work \cite{Cis24}.
Unfortunately, we are not able prove the equality $\Mcal^0(S) = \Mcal(S)$ by generalizing Nori's geometric ideas over arbitrary base varieties - although we believe that this should be possible. Instead, we use Nori's equivalence over fields as the starting step of a Noetherian induction argument, the inductive step of which exploits Beilinson's gluing functors.

Once the single external tensor product functors are constructed, turning their collection into the sought-after monoidal structure amounts to constructing monoidality isomorphisms along morphisms of (quasi-projective) $k$-varieties as well as compatible associativity, commutativity, and unit constraints. 
Using the abstract theory developed in \cite{Ter23Fact,Ter23Emb}, most of this construction can be reduced to the level of abelian categories, natural transformations between exact functors can be lifted formally.
The most delicate point is constructing the unit constraint and checking its compatibility with the rest of the structure. 
This requires a closer analysis of the tensor product with certain distinguished motivic local systems.
 
The consequences of this analysis play a crucial role also in the construction of internal homomorphisms, which completes the six functor formalism. 
Our strategy is inspired by a construction of Arapura from \cite[\S~5]{Ara-mot}: 
in a first step, we extend part of the functoriality of perverse Nori motives to the derived ind-categories $D(\Ind \Mcal(S))$, which allows us to define internal homomorphisms abstractly as a formal right-adjoint of the tensor product; 
in a second step, we show that these formal internal homomorphisms respect the bounded derived categories $D^b(\Mcal(S))$ and are compatible with the usual internal homomorphisms on the underlying constructible derived categories.
The proof is by Noetherian induction, using localization triangles in order to reduce to the case of distinguished motivic local systems.

As a consequence of our construction, we obtain a well-behaved Tannakian theory of \textit{motivic local systems} over general $k$-varieties:
\begin{thm*}[\Cref{thm_Mloc}]
	For every smooth, geometrically connected $k$-variety $S$, the full subcategory
	\begin{equation*}
		\Mcal\Loc_p(S) := \left\{M \in \Mcal(S) \; | \; \iota_S(M) \in \Loc_p(S) \right\} \subset \Mcal(S)
	\end{equation*}
    is canonically neutral Tannakian over $\Q$.
\end{thm*} 
\noindent
Here, $\Loc_p(S)$ denotes the image of the $\dim(S)$-shifted embedding of the abelian category $\Loc(S)$ of ordinary local systems into $D^b_c(S,\Q)$; 
the smoothness assumption ensures that shifted local systems are perverse sheaves.
The relation between the motivic Galois group of $\Mcal\Loc_p(S)$ and Nori's motivic Galois group $\Gcal_{mot}(k)$ has been studied by Jacobsen in \cite{JacMalcev}:
he shows that the difference between motivic local systems over $S$ and motives over $k$ is measured by the theory of local systems of geometric origin on $S^{\sigma}$.

Lastly, we equip the derived categories of perverse Nori motives with a canonical \textit{orientation}, in the sense of \cite[\S~2.4.c]{CisDeg}:
this allows us to construct Chern classes of vector bundles in terms of the six operations.
As an application, we promote the classical relative Hard Lefschetz Theorem \cite[Thm.~6.2.10]{BBD82} to the motivic level:
\begin{thm*}[\Cref{cor:chern}, \Cref{thm:Lefs}]
	Let $f \colon X \to S$ be a projective morphism of $k$-varieties, and let $L$ be a relatively ample line bundle over $X$.
	Then, for every $r \geq 0$, there is a canonical natural transformation between functors $\Mcal(X) \to \Mcal(S)$
	\begin{equation*}
		c_1^{\Mcal}(L)^r \colon {^p H^{-r}} (f_* M) \to {^p H^n} (f_* M)(n),
	\end{equation*}
	which is invertible when evaluated on semi-simple objects.
\end{thm*}
\noindent
This uses the theory of weights for perverse motives, developed in \cite[\S~6]{IM19}, which characterizes the semi-simple objects of $\Mcal(S)$ as those objects with split weight filtration.

\subsection*{Related work}

In the same period when the results of the present paper were completed, significant progress in the theory of perverse Nori was achieved in a complementary direction by Tubach. 
Building upon Nori's classical argument from \cite{NoriConstr}, he shows in \cite{Tub23} that the triangulated category $D^b(\Mcal(S))$ is canonically equivalent to the bounded derived category of its so-called constructible heart (as defined in \cite[\S~5.4]{IM19}). 
The main consequence is that the six operations on perverse Nori motives can be canonically enhanced to $\infty$-functors. 
In addition to extending the six operations to not necessarily quasi-projective $k$-varieties, this puts Nori's realization functor \eqref{R_k:intro} into a morphism of stable homotopy $2$-functors
\begin{equation}\label{Nri:intro}
	\Nri^*: \DA_{ct}^{\et}(\cdot,\Q) \rightarrow D^b(\Mcal(\cdot))
\end{equation} 
factoring the Betti realization \eqref{Bti:intro}:
both the existence and the essential uniqueness of \eqref{Nri:intro} follow immediately from Drew--Gallauer's work \cite{DrewGal}, where it is shown that Voevodsky motivic sheaves afford the universal six functor formalism in the setting of stable $\infty$-categories.
Let us stress that even constructing the single realization functors $\Nri_S^*$ over general $S$ directly seems a hard task. To the author's knowledge, the only plausible candidate construction available to date is the one proposed by Ivorra in \cite{IvReal}, which only works over smooth bases.
The main problem with Ivorra's realization is that its compatibility with the Betti realization and with the six operations is not fully understood yet.

The results of \cite{Tub23} are related to our results in two ways. 
In one direction, the closed monoidal structure of perverse Nori motives is used crucially in Tubach's variation of Nori's argument, as described in \cite[\S~2]{Tub23}. 
In the other direction, we use the Nori realization in the proof of some results of the present paper:
\Cref{prop:MVerdier} about the relation between Verdier duality and internal homomorphisms, \Cref{prop:orient_Nori} about the orientation, and \Cref{lem:MLoc_p-semisimple}, which is used in the proof of the motivic relative Hard Lefschetz Theorem.
Let us stress that these results are not involved in the construction of the six functor formalism, so there are no circular arguments here.

In our recent work \cite{JacTer25}, Jacobsen and I provide an alternative description of perverse Nori motives as a six functor formalism, based on some constructions from Ayoub's paper \cite{AyoAnab}:
namely, we show that $D^b(\Mcal(S))$ is equivalent to the category of $\Gcal_{mot}(k)$-equivariant constructible complexes of geometric origin on $S^{\sigma}$.
Intuitively, this means that Nori motivic sheaves can be thought of as sheaves of Nori motives.
The proof is based on Jacobsen's study of motivic Galois groups from \cite{JacMalcev}, which uses the full six functor formalism on perverse Nori motives as well as Tubach's results on mixed Hodge modules from \cite{Tub23} (but not the ones on Nori motives).
As a consequence, we get an independent construction of the morphism \eqref{Nri:intro}.

\subsection*{Structure of the paper}

Most constructions of the present paper are based on the abstract properties of perverse motives as universal abelian factorizations.
The basic functoriality results for universal abelian factorizations used in the setting of perverse motives are quickly discussed in \cite[\S~2.1]{IM19}, based on the detailed results of \cite{IvPerv}. A more systematic treatment, which also includes the extension of the above discussion to multi-linear functors, is given in our note \cite{Ter23UAF}, which we will often refer to. 

The goal of \Cref{sect:glue-Mp} is to describe an alternative presentation of perverse Nori motives (\Cref{thm_Mp0=Mp}), which is more adapted to constructing the external tensor product functors.
In the first place, we recall the construction of Beilinson's gluing functors in the setting of perverse motives from \cite[\S\S~3, 4]{IM19}.
We check that this allows one to glue perverse motives exactly as with perverse sheaves (\Cref{prop-Bei-glue}), and we deduce an abstract comparison criterion for subcategories of perverse motives stable under the gluing formalism (\Cref{thm-Beilinson-simpler}). 
Finally, we obtained the sought-after alternative presentation:
we introduce a suitable system of subcategories (\Cref{nota:Mp0}), and we show that they satisfy the hypotheses of the abstract comparison criterion (\Cref{thm_Mp0=Mp}).
This requires some preliminary observations about Nori motives over function fields of $k$-varieties (\Cref{prop:M(L)-indep} and \Cref{cor:M(L)}). 

In \Cref{sect_ETS} we state our first main result about the monoidal structure on perverse motives (\Cref{thm-boxtimesM}). 
After some general categorical preliminaries (\Cref{prop:VoloDG}), we construct the whole monoidal structure except for the unit constraint:
this includes the single external tensor product functors (\Cref{prop-extcore-Mp}), the external monoidality isomorphisms (\Cref{prop-extcore-Mp}), and the associativity and commutativity constraints (\Cref{lem-asso-Mp}, \Cref{lem-comm-Mp} and \Cref{lem-ac-Mp}). 

\Cref{sect_otimes-dist} is a crucial technical interlude, needed both for the unit constraint and for internal homomorphisms.
We introduce abelian subcategories of distinguished motivic local systems over smooth varieties (\Cref{nota:Mp0}), and we study the restriction of the internal tensor product to these. 
The main technical result is that the restricted internal tensor product, which is obtained indirectly from the external tensor product, admits a more direct definition (\Cref{prop-distinguished-otimes}). 
We readily deduce the existence of internal homomorphisms attached to distinguished motivic local systems (\Cref{cor:adj-Mloc}).

In \Cref{sect_unit-Mp}, we finally complete the proof of \Cref{thm-boxtimesM} by constructing the unit constraint for perverse motives (\Cref{thm:unit-constr-Mp}). 
The first step is to construct the motivic unit section (\Cref{prop-unitMcal}). 
The second step is to construct the motivic unit constraint and to check that its compatibility with the associativity and commutativity constraints defined previously (\Cref{lem:au-Mp} and \Cref{lem:cu-Mp}).

In \Cref{sect_Homcal} we prove the existence of internal homomorphisms (\Cref{thm_homcal}). 
In the first place, we construct an abstract ind-adjoint to the tensor product functors (\Cref{constr_homcal}).
Then, using the partial results obtained at the end of \Cref{sect_otimes-dist}, we check that these functors respects the bounded derived categories and that they are compatible with the usual internal homomorphisms on the underlying constructible derived categories by applying an abstract criterion (\Cref{prop-crit-homcal} and \Cref{prop-crit-homcal-R}). 
As a final comment, we check that Verdier duality and internal homomorphisms have the expected relation (\Cref{prop:MVerdier}).

In the final \Cref{sect:appli-compl} we describe some important consequences of our construction.
We introduce abelian categories of motivic local systems over smooth, geometrically connected $k$-varieties (\Cref{defn:MLoc_p}) and we show that they are neutral Tannakian over $\Q$ (\Cref{thm_Mloc}). 
This implies a topological reconstruction result for perverse motives (\Cref{thm:reconstr}).
Lastly, we construct a canonical orientation on the derived categories of perverse motives (\Cref{prop:orient_Nori}), and we lift the relative Hard Lefschetz Theorem to the motivic level (\Cref{thm:Lefs}).

\subsection*{Acknowledgments}

The contents of this paper correspond to the fourth and central chapter of my Ph.D. thesis, written at the University of Freiburg under the supervision of Annette Huber-Klawitter. It is a pleasure to thank her for introducing me to the subject of Nori motives and suggesting to me the problem of constructing the tensor structure on perverse motives, for following my progresses on it closely, as well as for her constant support and encouragement. I would also like to thank Joseph Ayoub, Denis-Charles Cisinski, Martin Gallauer, Fritz H\"{o}rmann, Florian Ivorra, Emil Jacobsen, Sophie Morel, Wolfgang Soergel, and Swann Tubach for useful discussions around the subject of this article.

\section*{Notation and conventions}

\subsection*{Fibered categories}

\begin{itemize}
	\item We follow the notation and conventions on fibered categories used in \cite{Ayo07a} and in \cite{CisDeg}, as displayed explicitly in \cite[\S~1]{Ter23Fib}.
\end{itemize}

\subsection*{Algebraic Geometry}

\begin{itemize}
	\item We fix a field $k$ of characteristic $0$ endowed with a fixed field embedding $\sigma: k \hookrightarrow \C$. 
	\item $\Var_k$ denotes the category of algebraic varieties over $k$: objects are reduced, separated, finite type $k$-schemes; morphisms are $k$-morphisms.
	\item By the closed immersion \textit{complementary} to a given open immersion $j: U \hookrightarrow S$ in $\Var_k$ we mean the closed immersion $i: Z \hookrightarrow S$ where $Z$ denotes the closed complement of $U$ in $S$ endowed with its reduced scheme structure.
	\item We consider the following full subcategories of $\Var_k$:
	\begin{itemize}
		\item $\Sm_k$: the subcategory of smooth $k$-varieties;
		\item $\Aff_k$: the subcategory of affine $k$-varieties.
	\end{itemize}
	\item For every $S \in \Var_k$ we write:
	\begin{itemize}
		\item $a_S: S \rightarrow \Spec(k)$ for the structural morphism;
		\item $\Ocal(S)$ for the $k$-algebra of global regular functions on $S$;
		\item $\Op(S)$ for the full subcategory of the slice category $\Sm_k/S$ whose objects are the Zariski-dense open immersions into $S$;
		\item $\Aff\Op(S)$ for the full subcategory $\Aff_k \subset \Op(S)$;
		\item $\Sm\Op(S)$ for the full subcategory $\Sm_k \subset \Op(S)$.
	\end{itemize}
\end{itemize}

\subsection*{Analytic sheaves}

\begin{itemize}
	\item For every $S \in \Var_k$, we use the following notation:
	\begin{itemize}
		\item $S^{\sigma}$ denotes the complex analytic space whose points are the $\C$-points of $S$ over $k$ via $\sigma$.
		\item $\Ct(S)$ denotes the abelian category of sheaves of $\Q$-vector spaces on $S^{\sigma}$ which are algebraically constructible.
		\item $\Loc(S)$ denotes the full abelian subcategory of $\Ct(S)$ consisting of all locally constant sheaves.
		\item $D^b_c(S,\Q)$ denotes the derived category of complexes of sheaves of $\Q$-vector spaces on $S^{\sigma}$ with bounded algebraically constructible cohomology.
		\item $\Perv(S)$ denotes the heart of the self-dual perverse $t$-structure on $D^b_c(S,\Q)$ as defined in \cite{BBD82}, and ${^p H^0}: D^b_c(S,\Q) \rightarrow \Perv(S)$ denotes the corresponding homological functor.
	\end{itemize}
\end{itemize}

\subsection*{$\ell$-adic sheaves}

\begin{itemize}
	\item Let $\ell$ be a fixed prime number. For every $S \in \Var_k$, we use the following notation:
	\begin{itemize}
		\item $\Ct(S)_{\ell}$ denotes the abelian category of $\ell$-adic algebraically constructible sheaves on the small étale site $S_{\et}$.
		\item $\Loc(S)_{\ell}$ denotes the full abelian subcategory of $\Ct(S)_{\ell}$ consisting of all lisse sheaves.
		\item $D^b_{c,\ell}(S,\Q_{\ell})$ denotes the constructible $\ell$-adic derived category as defined by Ekedahl in \cite{Eke} or equivalently by Bhatt--Scholze in \cite{BS}.
		\item $\Perv(S)_{\ell}$ denotes the heart of the self-dual perverse $t$-structure on $D^b_{c,\ell}(S,\Q_{\ell})$ as defined in \cite[\S~2.2]{MorelEll}, and ${^p H^0}: D^b_{c,\ell}(S,\Q_{\ell}) \rightarrow \Perv(S)_{\ell}$ denotes the corresponding homological functor.
	\end{itemize}
\end{itemize}

\subsection*{Triangulated motives and realizations}

\begin{itemize}
	\item For every $S \in \Var_k$, we use the following notation:
	\begin{itemize}
		\item $\DA_{ct}^{\et}(S,\Q)$ denotes the triangulated category of constructible étale motives with $\Q$-coefficients constructed in \cite[\S~4.5.2]{Ayo07b} (see also \cite[\S~3]{AyoEt}), implicitly identified with the category of constructible Beilinson motives defined in \cite[\S~IV]{CisDeg} via the equivalence of \cite[Thm.~16.2.18]{CisDeg}.
		\item We let 
		\begin{equation*}
			\Bti_S^*: \DA_{ct}^{\et}(S,\Q) \rightarrow D^b_{c,\ell}(S,\Q)
		\end{equation*}
		denote the \textit{Betti realization} over $S$ associated to $\sigma$ constructed in \cite{Ayo10}.
		\item Let $\ell$ be a fixed prime number. We let
		\begin{equation*}
			R_{\ell,S}: \DA_{ct}^{\et}(S,\Q) \rightarrow D^b_{c,\ell}(S,\Q_{\ell})
		\end{equation*}
		denote the \textit{$\ell$-adic realization} over $S$ constructed in \cite{AyoEt}.
	\end{itemize}
\end{itemize}

\subsection*{Universal abelian factorizations}

\begin{itemize}
	\item For every additive category $\Dcal$, we write:
	\begin{itemize}
		\item $\Rbf(\Dcal)$ for the full subcategory of $\Fun(\Dcal^{op},\Ab)$ consisting of all presheaves of finite presentation (i.e. cokernels of representable presheaves).
		\item $\Lbf(\Dcal) := \Rbf(\Dcal^{op})^{op}$.
		\item $\A(\Dcal) := \Lbf(\Rbf(\Dcal))$ for Freyd's abelian hull of $\Dcal$ introduced in \cite{Freyd} (see also \cite[\S~1]{Ter23UAF}). 
	\end{itemize}
    \item For every additive functor $\beta: \Dcal \rightarrow \Acal$ from an additive category $\Dcal$ to an abelian category $\Acal$, we write:
    \begin{itemize}
    	\item $\beta^+: \A(\Dcal) \rightarrow \Acal$ for the unique exact functor extending $\beta$;
    	\item $\A(\beta): \A(\Dcal)/\ker(\beta^+)$ for the associated universal abelian factorization.
    \end{itemize}
\end{itemize}

\subsection*{Perverse Nori motives}

\begin{itemize}
	\item For every $S \in \Var_k$, we write:
	\begin{itemize}
		\item $\Dcal(S)$ for the underlying additive category of $\DA_{ct}^{\et}(S,\Q)$;
		\item $\beta_S: \Dcal(S) \rightarrow \Perv(S)$ for the homological functor 
		\begin{equation*}
			\DA_{ct}^{\et}(S,\Q) \xrightarrow{\Bti_{S}^*} D^b_c(S,\Q) \xrightarrow{{^p H^0}} \Perv(S)
		\end{equation*}
		regarded as an additive functor.
		\item $\Mcal(S)$ for the category of \textit{perverse Nori motives} over $S$, defined as the universal abelian factorization $\A(\beta_S)$of the additive functor $\beta_S$.
		\item $\pi_S$ for the quotient functor $\A(\Dcal(S)) \rightarrow \Mcal(S)$, and also for the composite functor $\Dcal(S) \rightarrow \A(\Dcal(S)) \rightarrow \Mcal(S)$,
		\item $\iota_S$ for the induced faithful exact functor $\Mcal(S) \hookrightarrow \Perv(S)$.
	\end{itemize}
\end{itemize}

\section{Alternative presentation of perverse motives}\label{sect:glue-Mp}

The goal of the present section is to obtain an alternative presentation of perverse motives, which is more adapted to defining the external tensor product.
The precise result is \Cref{thm_Mp0=Mp} below, which can be regarded as an extension of Nori's yoga of \textit{good pairs} (see \cite[\S~II.9]{HMS17}) to arbitrary base varieties. 
The proof relies on an abstract comparison criterion, formalized as \Cref{thm-Beilinson-simpler}, which is based on the gluing formalism for perverse motives.

First of all, we review the construction of Beilinson's gluing functors in the setting of perverse motives, obtained in \cite[\S~3.5]{IM19}.
We then lift Beilinson's gluing formalism for perverse sheaves, as stated in \cite{BeiGlue}, to the motivic level; our presentation is modeled on Morel's note \cite{MorelGlue}. 
In the final part, we describe the promised alternative presentation of perverse motives.
The proof of the main result is preceded by a discussion about Nori motives over function fields, which leads to some auxiliary results complementing those of \cite[\S~6.5]{IM19}.

While \Cref{thm_Mp0=Mp} is the starting point for the construction of \Cref{sect_ETS}, the motivic gluing formalism is not used directly in the rest of the paper.
The reader is therefore invited to skip most of the technical discussion about the gluing functors.

\subsection{Review of the motivic gluing functors}

We start by recalling the relevant results from \cite[\S~3.5]{IM19}. To begin with, fix a $k$-variety $S$ and a global regular function $f \in \Ocal(S)$ with vanishing locus $Z$; let $i: Z \hookrightarrow S$ denote the corresponding closed immersion, and let $j: U \hookrightarrow S$ denote the complementary open immersion. 
Since $j$ is affine by construction, by \cite[\S~2.4]{IM19} there exist canonical exact functors
\begin{equation*}
	j_!, \; j_*: \Mcal(U) \rightarrow \Mcal(S)
\end{equation*}
which mimic the formal properties of the underlying functors on perverse sheaves, namely: they are left and right adjoint to the inverse image functor $j^*: \Mcal(S) \rightarrow \Mcal(U)$, respectively; they are fully faithful; they are intertwined under the motivic Verdier duality functors of \cite[\S~2.8]{IM19}.

As shown in \cite[\S~3.5]{IM19}, the gluing functors for perverse sheaves considered in \cite{BeiGlue} can be canonically lifted to exact functors
\begin{equation*}
	\begin{aligned}
		{^p \Psi_f}: \Mcal(U) \rightarrow \Mcal(Z), & \quad {^p \Phi_f}: \Mcal(S) \rightarrow \Mcal(Z), \\ \quad
		{^p \Xi_f}: \Mcal(S) \rightarrow \Mcal(S), & \quad {^p \Omega_f}: \Mcal(S) \rightarrow \Mcal(S).
	\end{aligned}
\end{equation*}
In the setting of perverse sheaves, the functors ${^p \Psi_f}$ and ${^p \Phi_f}$ are the classical unipotent nearby cycles and vanishing cycles functors (shifted by $-1$, so that they become $t$-exact for the perverse $t$-structures), while the functor ${^p \Xi_f}$ is Beilinson's maximal extension functor (or rather, the composite of the latter with $j^*$). The functor $\Omega_f$ first appeared in Saito's work on mixed Hodge modules \cite{Sai90}.

\begin{rem}
	In \cite{IM19} the functor ${^p \Psi_f}: \Mcal(U) \rightarrow \Mcal(Z)$ is denoted ${^p \mathsf{Log}_f}$, since it comes from the logarithmic specialization system of \cite[\S~3.6]{Ayo07b}.
	We prefer to keep our notation close to that of \cite{BeiGlue} and \cite{MorelGlue}.
\end{rem}

The functors ${^p \Psi_f}$ and ${^p \Xi_f}$ fit into the two short exact sequences of functors $\Mcal(S) \rightarrow \Mcal(S)$
\begin{equation}\label{ses_glue1}
	0 \rightarrow j_! j^* M \rightarrow {^p \Xi_f}(M) \rightarrow i_* {^p \Psi_f}(j^* M)(-1) \rightarrow 0
\end{equation}
and
\begin{equation}\label{ses_glue2}
	0 \rightarrow i_* {^p \Psi_f}(j^* M) \rightarrow {^p \Xi_f}(M) \rightarrow j_* j^* M \rightarrow 0,
\end{equation}
which are intertwined under motivic Verdier duality. 
Thus the composite $j_! j^* M \hookrightarrow {^p \Xi_f}(M) \twoheadrightarrow j_* j^* M$ coincides with the composite $j_! j^* M \xrightarrow{\epsilon} M \xrightarrow{\eta} j_* j^* M$.
There are two further relevant short exact sequences of functors $\Mcal(S) \rightarrow \Mcal(S)$
\begin{equation}\label{ses-deduced-gluey1}
	0 \rightarrow i_* {^p \Psi_f}(j^* M) \rightarrow {^p \Omega_f}(M) \rightarrow M \rightarrow 0
\end{equation}
and
\begin{equation}\label{ses-deduced-gluey2}
	0 \rightarrow j_! j^* M \rightarrow {^p \Omega_f}(M) \rightarrow i_* {^p \Phi_f}(M) \rightarrow 0.
\end{equation}
The functors ${^p \Omega_f}$ and ${^p \Xi_f}$ are also related by a canonical natural transformation of functors $\Mcal(S) \rightarrow \Mcal(S)$
\begin{equation*}
	{^p \Omega_f}(M) \rightarrow {^p \Xi_f}(M),
\end{equation*}
which allows one to combine the exact sequences \eqref{ses_glue2} and \eqref{ses-deduced-gluey1} into a commutative diagram of the form
\begin{equation}\label{dia:Omega_f-Xi_f}
	\begin{tikzcd}
		0 \arrow{r} & i_* {^p \Psi_f}(j^* M) \arrow{r} & {^p \Xi_f}(M) \arrow{r} & j_* j^* M \arrow{r} & 0 \\
		0 \arrow{r} & i_* {^p \Psi_f}(j^* M) \arrow{r} \arrow[equal]{u} & {^p \Omega_f}(M) \arrow{r} \arrow{u} & M \arrow{u}{\eta} \arrow{r} & 0.
	\end{tikzcd}
\end{equation}
All these fact were noted in \cite[\S~3.5]{IM19}. Our first task is to deduce further information about the mutual relations among the gluing functors.

In the proof of the following results, we often need to use the fully faithfulness of the direct image functor $i_*: \Mcal(Z) \rightarrow \Mcal(S)$, proved in \cite[Thm.~4.1]{IM19}.

\begin{lem}\label{lem:OmegaPhi-mot}
	With the above notation, there exists a canonical natural isomorphism between functors $\Mcal(S) \rightarrow \Mcal(S)$
	\begin{equation}\label{Omega_f:ker}
		{^p \Omega_f}(M) = \ker\left\{{^p \Xi_f}(M) \oplus M \xrightarrow{-} j_* j^* M \right\}
	\end{equation}
	and a canonical natural isomorphism between functors $\Mcal(S) \rightarrow \Mcal(Z)$
	\begin{equation}\label{Phi_f:H^0}
		{^p \Phi_f}(M) = i^* H^0[j_! j^* M \xrightarrow{+} {^p \Xi_f}(M) \oplus M \xrightarrow{-} j_* j^* M].
	\end{equation}
\end{lem}
\begin{proof}
	Regard the commutative diagram \eqref{dia:Omega_f-Xi_f} as a double complex. Since both rows are exact, we deduce that its total complex is exact as well. Hence the total complex of the truncated double complex 
	\begin{equation*}
		\begin{tikzcd}
			0 \arrow{r} & {^p \Xi_f}(M) \arrow{r} & j_* j^* M \arrow{r} & 0 \\
			0 \arrow{r} \arrow[equal]{u} & {^p \Omega_f}(M) \arrow{r} \arrow{u} & M \arrow{u}{\eta} \arrow{r} & 0
		\end{tikzcd}
	\end{equation*}
	is exact as well. This yields the natural isomorphism \eqref{Omega_f:ker}. Moreover, since the arrow $j_! j^* M \rightarrow {^p \Xi_f}(M)$ in the short exact sequence \eqref{ses_glue1} is a monomorphism, the canonical arrow
	\begin{equation*}
		\im\left\{j_! j^* M \xrightarrow{+} {^p \Xi_f}(M) \oplus M \right\} \rightarrow \im\left\{j_! j^* M \rightarrow {^p \Xi_f}(M) \right\}
	\end{equation*}
	is an isomorphism. This yields a natural isomorphism
	\begin{align*}
		i_* {^p \Phi_f}(M) &= \coker\left\{j_! j^* M \rightarrow {^p \Omega_f}(M) \right\} && \textup{by \eqref{ses-deduced-gluey2}}\\
		&= \coker\left\{j_! j^* M \xrightarrow{+} \ker\left\{{^p \Xi_f}(M) \oplus M \xrightarrow{-} j_* j^* M \right\}\right\} && \textup{by \eqref{Omega_f:ker}} \\
		&= H^0[j_! j^* M \xrightarrow{+} {^p \Xi_f}(M) \oplus M \xrightarrow{-} j_* j^* M]. &&
	\end{align*}
	Thus we obtain the natural isomorphism \eqref{Phi_f:H^0} as the composite
	\begin{equation*}
		{^p \Phi_f}(M) = i^* i_* {^p \Phi_f}(M) = i^* H^0[j_! j^* M \xrightarrow{+} {^p \Xi_f}(M) \oplus M \xrightarrow{-} j_* j^* M],
	\end{equation*}
	where the first passage follows from the fully faithfulness of $i_*$.
\end{proof}

The following consequence plays a fundamental role in the construction of motivic inverse images under closed immersions, described in \cite[\S~4.1]{IM19}. We spell it out since it is relevant for our constructions as well.

\begin{cor}
	With the above notation, there exists a canonical isomorphism of functors $\Mcal(Z) \rightarrow \Mcal(Z)$
	\begin{equation}\label{iso:Phi_f-i_*=id}
		{^p \Phi_f}(i_* M) = M.
	\end{equation} 
	In other words, the functor ${^p \Phi_f}$ is canonically left-inverse to $i_*$.
\end{cor}
\begin{proof}
	Combining any of the two short exact sequences \eqref{ses_glue1} and \eqref{ses_glue2} with the vanishing $j^* \circ i_* = 0$, we deduce the vanishing ${^p \Xi_f} \circ i_* = 0$. Hence we obtain the sought-after isomorphism \eqref{iso:Phi_f-i_*=id} as the composite
	\begin{align*}
		{^p \Phi_f}(i_* M) &= i^* H^0[j_! j^* i_* M \xrightarrow{+} {^p \Xi_f}(i_* M) \oplus i_* M \xrightarrow{-} j_* j^* i_* M] && \textup{by \eqref{Phi_f:H^0}} \\
		&= i^* H^0[0 \rightarrow 0 \oplus i_* M \oplus 0] && \textup{as $j^* \circ i_* = 0$} \\
		&= i^* i_* M = M, 
	\end{align*}
	where the last passage follows again from the fully faithfulness of $i_*$.
\end{proof}

With these results at our disposal, we can lift the classical nilpotent monodromy operator $N_f$, together with its canonical factorization (see \cite[\S\S~2, 6]{MorelGlue}), to the setting of perverse motives.

\begin{constr}\label{constr_canvar}
	Keep the above notation. Consider the composite natural transformation of functors $\Mcal(S) \rightarrow \Mcal(S)$
	\begin{equation}\label{first:N_f}
		i_* {^p \Psi_f}(j^* M) \hookrightarrow {^p \Xi_f}(M) \twoheadrightarrow i_* {^p \Psi_f}(j^* M)(-1),
	\end{equation}
	where the two arrows come from the short exact sequences \eqref{ses_glue1} and \eqref{ses_glue2}, respectively. Since the functor $i_*: \Mcal(Z) \rightarrow \Mcal(S)$ is fully faithful, there is a unique natural transformation of functors $\Mcal(S) \rightarrow \Mcal(Z)$
	\begin{equation}\label{N_f-mot}
		N_f: {^p \Psi_f}(j^* M) \rightarrow {^p \Psi_f}(j^* M)(-1)
	\end{equation}
	whose image under $i_*$ coincides with \eqref{first:N_f}. We look for a factorization of $N_f$ through the functor ${^p \Phi_f}$.
	To this end, note that the arrow $i_* {^p \Psi_f}(j^* M) \hookrightarrow {^p \Xi_f}(M)$ factors through ${^p \Omega_f}(M) = \ker\left\{{^p \Xi_f}(M) \oplus M \xrightarrow{-} j_* j^* M \right\}$, because we have
	\begin{equation*}
		\Hom_{\Mcal(S)}(i_* {^p \Psi_f}(j^* M),j_* j^* M) = \Hom_{\Mcal(U)}(j^* i_* {^p \Psi_f}(j^* M),j^* M) = 0
	\end{equation*}
	as $j^* \circ i_* = 0$. Hence we obtain a canonical natural transformation of functors $\Mcal(S) \rightarrow \Mcal(Z)$
	\begin{equation}\label{can_f-mot}
		\can_f: {^p \Psi_f}(j^* M) \rightarrow {^p \Phi_f}(M)
	\end{equation}
	by declaring that its image under the fully faithful functor $i_*: \Mcal(Z) \rightarrow \Mcal(S)$ be the induced arrow
	\begin{equation*}
		\begin{tikzcd}
			i_* {^p \Psi_f}(j^* M) \arrow{d} &  i_* {^p \Phi_f}(M). \arrow[equal]{d} \\
			\ker\left\{{^p \Xi_f}(M) \oplus M \xrightarrow{-} j_* j^* M \right\} \arrow[two heads]{r} & H^0[j_! j^* M \xrightarrow{+} {^p \Xi_f}(M) \oplus M \xrightarrow{-} j_* j^* M] 
		\end{tikzcd}
	\end{equation*}
	Dually, note that the arrow ${^p \Xi_f}(M) \twoheadrightarrow i_* {^p \Psi_f}(M)(-1)$ is trivial on $\im\left\{j_! j^* M \xrightarrow{+} {^p \Xi_f}(M) \oplus M \right\}$, because we have
	\begin{equation*}
		\Hom_{\Mcal(S)}(j_! j^* M,i_* {^p \Psi_f}(M)) = \Hom_{\Mcal(U)}(j^* M, j^* i_* {^p \Psi_f}(M)) = 0
	\end{equation*}
	as $j^* \circ i_* = 0$. Hence we obtain a canonical natural transformation of functors $\Mcal(S) \rightarrow \Mcal(Z)$
	\begin{equation}\label{var_f-mot}
		\var_f: {^p \Phi_f}(M) \rightarrow i_* {^p \Psi_f}(j^* M)(-1)
	\end{equation}
	by declaring that its image under the fully faithful functor $i_*: \Mcal(Z) \rightarrow \Mcal(S)$ be the induced arrow
	\begin{equation*}
		\begin{tikzcd}
			i_* {^p \Phi_f}(M) \arrow[equal]{d} & i_* {^p \Psi_f}(j^* M). \\
			H^0[j_! j^* M \xrightarrow{+} {^p \Xi_f}(M) \oplus M \xrightarrow{-} j_* j^* M] \arrow[hook]{r} & \coker\left\{j_! j^* M \xrightarrow{+} {^p \Xi_f}(M) \oplus M \right\} \arrow{u}
		\end{tikzcd}
	\end{equation*}
	By construction, we have the equality of natural transformations
	\begin{equation}\label{varcanN}
		\var_f \circ \can_f = N_f: {^p \Psi_f}(j^* M) \rightarrow {^p \Psi_f}(j^* M)(-1).
	\end{equation}
\end{constr}

Before stating the main gluing results, it is convenient to define one last short exact sequence of gluing functors.

\begin{constr}
	Keep the above notation. We have a canonical natural isomorphism between functors $\Mcal(U) \rightarrow \Mcal(Z)$
	\begin{equation}\label{iso_Phi_j=Psi}
		{^p \Phi_f}(j_! M) = {^p \Psi_f}(M)
	\end{equation}
	given by the composite
	\begin{align*}
		{^p \Phi_f}(j_! M) &= i^* H^0[j_! j^* (j_! M) \xrightarrow{+} {^p \Xi_f}(j_! M) \oplus j_! M \xrightarrow{-} j_* j^* (j_! M)] && \textup{by \eqref{Phi_f:H^0}} \\
		&= i^* H^0[j_! M \xrightarrow{+} {^p \Xi_f}(j_! M) \oplus j_! M \xrightarrow{-} j_* j^* (j_! M)] \\
		&= i^* \ker\left\{{^p \Xi_f}(j_! M) \rightarrow j_* j^* (j_! M) \right\} \\
		&= i^* i_* {^p \Psi_f}(j^* j_! M) && \textup{by \eqref{ses_glue2}} \\
		&= {^p \Psi_f}(M),
	\end{align*}
	where the last passage follows from the fully faithfulness of $i_*$ and of $j_!$.
	Now, composing the short exact sequence \eqref{ses_glue1} with the exact functor $j_!$ on the right and with the exact functor ${^p \Phi_f}$ on the left, and using the isomorphisms \eqref{iso:Phi_f-i_*=id} and \eqref{iso_Phi_j=Psi} together with the fully faithfulness of $j_!$, we obtain a canonical short exact sequence of functors $\Mcal(U) \rightarrow \Mcal(Z)$
	\begin{equation}\label{ses_glue3}
		0 \rightarrow {^p \Psi_f}(M) \rightarrow {^p \Phi_f} {^p \Xi_f}(j_! M) \rightarrow {^p \Psi_f}(M)(-1) \rightarrow 0.
	\end{equation}
\end{constr}

\subsection{Motivic gluing formalism and applications}

We are ready to lift Beilinson's gluing construction to the motivic setting. We keep the notation introduced in the previous subsection: we work over a fixed $k$-variety $S$ and consider a global regular function $f \in \Ocal(S)$, with vanishing locus $i: Z \hookrightarrow S$ and non-vanishing locus $j: U \hookrightarrow S$.

\begin{nota}
	We introduce the category of \textit{motivic gluing data} $\Mcal(S,f)$ defined as follows:
	\begin{itemize}
		\item objects are quadruples $(M_U, M_Z; u,v)$ consisting of
		\begin{itemize}
			\item two objects $M_U \in \Mcal(U)$, $M_Z \in \Mcal(Z)$,
			\item two morphisms $u: {^p \Psi_f}(M_U) \rightarrow M_Z$ and $v: M_Z \rightarrow {^p \Psi_f}(M_U)(-1)$ satisfying the relation
			\begin{equation*}
				v \circ u = N_f(M_U): {^p \Psi_f}(M_U) \rightarrow {^p \Psi_f}(M_U)(-1);
			\end{equation*}
		\end{itemize}
		\item a morphism $\alpha: (M_U, M_Z; u,v) \rightarrow (M'_U, M'_Z; u',v')$ is the datum of two morphisms $\alpha_U: M_U \rightarrow M'_U$ and $\alpha_Z: M_Z \rightarrow M'_Z$ such that the diagram in $\Mcal(Z)$
		\begin{equation*}
			\begin{tikzcd}
				{^p \Psi_f}(M_U) \arrow{r}{u} \arrow{d}{{^p \Psi_f}(\alpha_U)} & M_Z \arrow{r}{v} \arrow{d}{\alpha_Z} & {^p \Psi_f}(M_U)(-1) \arrow{d}{{^p \Psi_f}(\alpha_U)(-1)} \\
				{^p \Psi_f}(M'_U) \arrow{r}{u'} & M'_Z \arrow{r}{v'} & {^p \Psi_f}(M'_U)(-1)
			\end{tikzcd}
		\end{equation*}
		is commutative.
	\end{itemize}
\end{nota}

The following result is the motivic lifting of \cite[Prop.~3.1]{BeiGlue}. The proof that we present is merely an adaptation of the detailed argument given in \cite[Thm.~8.1]{MorelGlue}.

\begin{prop}\label{prop-Bei-glue}
	Keep the above notation. Then the functors
	\begin{equation*}
		\delta_{S,f}: \Mcal(S) \rightarrow \Mcal(S,f), \qquad \gamma_{S,f}: \Mcal(S,f) \rightarrow \Mcal(S)
	\end{equation*}
	defined by the formulae
	\begin{equation*}
		\delta_{S,f}(M) := (j^* M, {^p \Phi_f}(M); \can_f, \var_f)
	\end{equation*}
	and
	\begin{equation*}
		\gamma_{S,f}(M_U, M_Z; u,v) := H^0[i_* {^p \Psi_f}(M_U) \xrightarrow{+} {^p \Xi_f}(j_! M_U) \oplus i_* M_Z \xrightarrow{-} i_* {^p \Psi_f}(M_U)(-1)]
	\end{equation*}
	are canonically mutually quasi-inverse equivalences.
\end{prop}
\begin{proof}
	The fact that $\delta_{S,f}$ is well-defined follows from the relation \eqref{varcanN}.
	Let us construct the natural isomorphism $\delta_{S,f} \circ \gamma_{S,f} \simeq \id_{\Mcal(S,f)}$. Given an object $(M_U, M_Z; u,v) \in \Mcal(S,f)$, write the object $\delta_{S,f} \gamma_{S,f}(M_U, M_Z; u,v) \in \Mcal(S,f)$ as a quadruple $(M'_U, M'_Z; u',v')$. We have a canonical isomorphism
	\begin{align*}
		M'_U &:= j^* \gamma_{S,f}(M_U, M_Z; u,v) \\
		&:= j^* H^0[i_* {^p \Psi_f}(M_U) \xrightarrow{+} {^p \Xi_f}(j_! M_U) \oplus i_* M_Z \xrightarrow{-} i_* {^p \Psi_f}(M_U)(-1)] \\
		&= H^0[j^* i_* {^p \Psi_f}(M_U) \xrightarrow{+} j^* {^p \Xi_f}(j_! M_U) \oplus j^* i_* M_Z \xrightarrow{-} j^* i_* {^p \Psi_f}(M_U)(-1)] \\
		&= H^0[0 \rightarrow j^* {^p \Xi_f}(j_! M_U) \oplus 0 \rightarrow 0] && \textup{as $j^* \circ i_* = 0$} \\
		&= j^* {^p \Xi_f}(j_! M_U) \\
		&= j^* j_! M_U && \textup{by \eqref{ses_glue1}} \\
		&= M_U, \tag*{}
	\end{align*}
	where the last passage follows from the fully faithfulness of $j_!$.
	Similarly, we have a canonical isomorphism
	\begin{align*}
		M'_Z &:= {^p \Phi_f} \gamma_{S,f}(M_U, M_Z; u,v) \\
		&:= {^p \Phi_f} H^0[i_* {^p \Psi_f}(M_U) \xrightarrow{+} {^p \Xi_f}(j_! M_U) \oplus M_Z \xrightarrow{-} i_* {^p \Psi_f}(M_U)(-1)] \\
		&= H^0[{^p \Phi_f} i_* {^p \Psi_f}(M_U) \xrightarrow{+} {^p \Phi_f} {^p \Xi_f}(j_! M_U) \oplus {^p \Phi_f} i_* M_Z \xrightarrow{-} {^p \Phi_f} i_* {^p \Psi_f}(M_U)(-1)] \\
		&= H^0[{^p \Psi_f}(M_U) \xrightarrow{+} {^p \Phi_f} {^p \Xi_f}(j_! M_U) \oplus M_Z \xrightarrow{-} {^p \Psi_f}(M_U)(-1)] && \textup{by \eqref{iso:Phi_f-i_*=id}} \\
		&= M_Z. && \textup{by \eqref{ses_glue3}}
	\end{align*}
	The fact that these two isomorphisms glue to an isomorphism in $\Mcal(S,f)$ follows formally from the construction; we omit the details.
	
	In the other direction, let us construct the natural isomorphism $\gamma_{S,f} \circ \delta_{S,f} \simeq \id_{\Mcal(S)}$. Given an object $M \in \Mcal(S)$, we have
	\begin{equation*}
		\begin{aligned}
			\delta_{S,f} \gamma_{S,f}(M) &:= H^0[i_* {^p \Psi_f}(j^* M) \xrightarrow{+} {^p \Xi_f}(M) \oplus i_* {^p \Phi_f}(M) \xrightarrow{-} i_* {^p \Psi_f}(j^* M)(-1)] = \\
			&= \coker\left\{i_* {^p \Psi_f}(j^* M) \rightarrow O_f(M)\right\},
		\end{aligned}
	\end{equation*}
	where, for notational convenience, we have set
	\begin{equation*}
		O_f(M) := \ker\left\{{^p \Xi_f}(M) \oplus i_* {^p \Phi_f}(M) \xrightarrow{-} i_* {^p \Psi_f}(j^* M)(-1)\right\}.
	\end{equation*}
	Applying the Snake Lemma to the commutative diagram with exact rows
	\begin{equation*}
		\begin{tikzcd}
			0 \arrow{r} & {^p \Omega_f}(M) \arrow{r}{+} \arrow{d} & {^p \Xi_f}(M) \oplus {^p \Omega_f}(M) \arrow{r}{-} \arrow{d} & {^p \Xi_f}(M) \arrow{r} \arrow{d} & 0 \\
			0 \arrow{r} & O_f(M) \arrow{r}{+}  & {^p \Xi_f}(M) \oplus i_* {^p \Phi_f}(M) \arrow{r}{-}  & i_* {^p \Psi_f}(j^* M)(-1) \arrow{r}  & 0,
		\end{tikzcd}
	\end{equation*}
	and using the short exact sequences \eqref{ses_glue1} and \eqref{ses-deduced-gluey2}, we obtain the vanishing
	\begin{equation*}
		\ker\left\{{^p \Omega_f}(M) \rightarrow O_f(M)\right\} = 0 = \coker\left\{{^p \Omega_f}(M) \rightarrow O_f(M)\right\},
	\end{equation*}
	from which it follows that the canonical arrow ${^p \Omega_f}(M) \rightarrow O_f(M)$ is an isomorphism. Hence we obtain a canonical isomorphism
	\begin{equation*}
		\delta_{S,f} \gamma_{S,f}(M) = \coker\left\{i_* {^p \Psi_f}(j^* M) \rightarrow O_f(M)\right\}
		= \coker\left\{i_* {^p \Psi_f}(j^* M) \rightarrow {^p \Omega_f}(M)\right\} = M,
	\end{equation*}
	where the last passage follows from the short exact sequence \eqref{ses-deduced-gluey1}.
\end{proof}

As a consequence of the motivic gluing formalism, we can now state an abstract comparison criterion for systems of subcategories stable under the gluing functors.

For sake of clarity, it is convenient to introduce a suitable notion of cofinality for inclusions of fibered categories over a cofiltered small category:
\begin{defn}\label{defn:w-local}
	Let $\Ucal$ be a cofiltered small category. Given a $\Ucal$-fibered category $\Ccal$ and a $\Ucal$-fibered subcategory $\Ccal^0$, we say that the inclusion of $\Ucal$-fibered categories $\Ccal^0 \subset \Ccal$ is \textit{cofinal} if the canonical functor
	\begin{equation*}
		\twocolim_{U \in \Ucal^{op}} \Ccal^0(U) \rightarrow \twocolim_{U \in \Ucal^{op}} \Ccal(U)
	\end{equation*}
	is an equivalence.
\end{defn}

For every $k$-variety $S$, we let $\Sm\Op(S)$ denote the poset of smooth Zariski-dense open subsets of $S$. 
We can state our abstract comparison criterion as follows:

\begin{prop}\label{thm-Beilinson-simpler}
	Suppose that we are given, for every $k$-variety $S$, a (not necessarily full) abelian subcategory $\Mcal^0(S)$ of $\Mcal(S)$, in such a way that the following conditions are satisfied:
	\begin{enumerate}
		\item[(i)] As $S$ varies, the categories $\Mcal^0(S)$ are stable under the following exact functors:
		\begin{enumerate}
			\item for every affine open immersion $j: U \hookrightarrow S$, the inverse image functor $j^*$ as well as the direct image functors $j_!$ and $j_*$;
			\item for every closed immersion $i: Z \hookrightarrow S$, the direct image functor $i_*$;
			\item for every $k$-variety $S$, the Tate twist functor $(-)(-1)$;
			\item for every $k$-variety $S$ and every $f \in \Ocal(S)$, the functors ${^p \Psi_f}$, ${^p \Phi_f}$, ${^p \Xi_f}$ and ${^p \Omega_f}$.
		\end{enumerate}
		\item[(ii)] For every affine $k$-variety $S$ and every $f \in \Ocal(S)$, the natural transformations in the exact sequences \eqref{ses_glue1}, \eqref{ses_glue2}, \eqref{ses-deduced-gluey1} and \eqref{ses-deduced-gluey2} restrict to natural transformations of functors $\Mcal^0(S) \rightarrow \Mcal^0(S)$.
		\item[(iii)] For every affine $k$-variety $S$, the natural inclusion of $\Sm\Op(S)$-fibered categories $\Mcal^0 \subset \Mcal$ is cofinal in the sense of \Cref{defn:w-local}.
	\end{enumerate}
	Then, for every $k$-variety $S$, the inclusion $\Mcal^0(S) \subset \Mcal(S)$ is in fact an equivalence.
\end{prop}
\begin{proof}
	First of all, we reduce the proof to the affine case:
	assume that the inclusion is known to be an equivalence for every affine $k$-variety, and let us show that the same conclusion holds for arbitrary $k$-varieties. 
	This is just a standard Zariski-descent argument, but we spell out the details for the reader's convenience. 
	So fix a $k$-variety $S$, and choose a finite affine open covering $\left\{u_r: U_r \hookrightarrow S \right\}_{r \in I}$; 
	for every $r,s \in I$, form the Cartesian diagram
	\begin{equation*}
		\begin{tikzcd}
			U_r \cap U_s \arrow{r}{u_s^{(r)}} \arrow{d}{u_r^{(s)}} & U_r \arrow{d}{u_r} \\
			U_s \arrow{r}{u_s} & S,
		\end{tikzcd}
	\end{equation*}
	and let $u_{r,s}$ denote the composite open immersion $U_r \cap U_s \hookrightarrow S$. 
	Note that each $u_r$, each $u_{r}^{(s)}$ and each $u_{r,s}$ is affine, because $S$ is separated over $k$ by definition. Consider the usual exact sequence of exact functors $\Mcal(S) \rightarrow \Mcal(S)$
	\begin{equation*}
		\bigoplus_{r,s \in I} u_{r,s,!} u_{r,s}^* M \rightarrow \bigoplus_{r \in I} u_{r,!} u_r^* M \rightarrow M \rightarrow 0.
	\end{equation*}
	Note that the functors $\bigoplus_{r,s \in I} u_{r,s,!} u_{r,s}^*$ and $\bigoplus_{r \in I} u_{r,!} u_r^*$ take values in the essential image of $\Mcal^0(S)$ by condition (i) in the statement, since, by our assumption, the functors $u_{r,s}^*$ and $u_r^*$ take values in the essential image of $\Mcal^0(U_r \cap U_s)$ and $\Mcal^0(U_r)$, respectively. We claim that the natural transformation of functors $\Mcal(S) \rightarrow \Mcal(S)$ 
	\begin{equation*}
		\bigoplus_{r,s \in I} u_{r,s,!} u_{r,s}^* M \rightarrow \bigoplus_{r \in I} u_{r,!} u_r^* M
	\end{equation*} 
	takes values in $\Mcal^0(S)$ as well; this will immediately imply that the inclusion of abelian categories $\Mcal^0(S) \subset \Mcal(S)$ is both essentially surjective and fully faithful, thus an equivalence.
	To prove the claim, it suffices to recall that, for every $r,s \in I$, the arrow $u_{r,s,!} u_{r,s}^* M \rightarrow u_{r,!} u_r^* M$ is nothing but the composite
	\begin{equation*}
		u_{r,s,!} u_{r,s}^* M = u_{r,!} u_{s,!}^{(r)} u_s^{(r),*} u_r^* M \xrightarrow{\epsilon} u_{r,!} u_r^* M,
	\end{equation*}
	and to note the second arrow belongs to $\Mcal^0(S)$ by condition (i), since the arrow $u_{s,!}^{(r)} u_s^{(r),*} u_r^* M \xrightarrow{\epsilon} u_r^* M$ belongs to $\Mcal^0(U_r)$ by our assumption. 
	
	This reduces the proof to the affine case.
	Now we show that the inclusion $\Mcal^0(S) \subset \Mcal(S)$ is an equivalence if $S$ is affine. 
	We argue by Noetherian induction on $S$, starting from the case where $\dim(S) = 0$. 
	The base step follows directly from hypothesis (iii), since in this case the only object of $\Sm\Op(S)$ is $S$ itself. 
	For the inductive step, assume that $\dim(S) > 0$ and that the result is known to hold for all proper closed subvarieties $Z$ of $S$ (this is legitimate, since every such $Z$ is itself affine).
	
	Let us first show that the inclusion $\Mcal^0(S) \subset \Mcal(S)$ is essentially surjective. 
	Given an object $M \in \Mcal(S)$, choose any regular function $f \in \Ocal(S)$ with smooth Zariski-dense non-vanishing locus $D(f)$ such that the object $M|_{D(f)} \in \Mcal(D(f))$ belongs to $\Mcal^0(D(f))$: 
	the existence of such a function follows from condition (iii) in the statement (using the fact that, as $f$ varies in $\Ocal(S)$, the distinguished open subsets $D(f)$ form a basis for the Zariski topology of the affine variety $S$). 
	Let $j: D(f) \hookrightarrow S$ and $i: V(f) \hookrightarrow S$ denote the corresponding open and closed immersions. 
	By \Cref{prop-Bei-glue}, we have a canonical isomorphism in $\Mcal(S)$
	\begin{equation*}
		M = \gamma_{S,f} \delta_{S,f}(M) := H^0[i_* {^p \Psi_f}(j^* M) \xrightarrow{+} {^p \Xi_f}(M) \oplus i_* {^p \Phi_f}(M) \xrightarrow{-} i_* {^p \Psi_f}(j^* M)(-1)].
	\end{equation*}
	We claim that the complex on the right is isomorphic to a complex in $\Mcal^0(S)$; 
	this will imply that $M$ belongs to the essential image of the abelian subcategory $\Mcal^0(S)$ inside $\Mcal(S)$. 
	In order to prove the claim, it suffices to note that:
	\begin{itemize}
		\item[–] The objects $i_* {^p \Psi_f}(j^* M)$, ${^p \Xi_f}(M)$, $i_* {^p \Psi_f}(j^* M)(-1)$ belong to the essential image of $\Mcal^0(S)$ by condition (i), since $j^* M$ lies in the essential image of $\Mcal^0(D(f))$ by construction. 
		\item[–] The object $i_* {^p \Phi_f}(M)$ belongs to $\Mcal^0(S)$ by condition (i), since the inclusion $\Mcal^0(V(f)) \subset \Mcal(V(f))$ is essentially surjective by inductive hypothesis.
		\item[–] The arrows $i_* {^p \Psi_f}(j^* M) \rightarrow {^p \Xi_f}(M)$ and ${^p \Xi_f}(M) \rightarrow i_* {^p \Psi_f}(j^* M)(-1)$ belong to $\Mcal^0(S)$ by condition (ii).
		\item[–] The arrows $i_* {^p \Psi_f}(j^* M) \rightarrow i_* {^p \Phi_f}(M)$ and $i_* {^p \Phi_f}(M) \rightarrow i_* {^p \Psi_f}(j^* M)(-1)$, being in the image of $i_*$, belong to $\Mcal^0(S)$ by condition (i), since the inclusion $\Mcal^0(V(f)) \subset \Mcal(V(f))$ is fully faithful by inductive hypothesis.
	\end{itemize}
	This shows the essential surjectivity.
	
	Let us now show that the inclusion $\Mcal^0(S) \subset \Mcal(S)$ is full. 
	Given a morphism $\alpha: M_1 \rightarrow M_2$ in $\Mcal(S)$, choose any regular function $f \in \Ocal(S)$ with smooth, Zariski-dense non-vanishing locus $D(f)$ such that the objects $M_1|_{D(f)}, M_2|_{D(f)} \in \Mcal(D(f))$ belong to the essential image of $\Mcal^0(D(f))$ and the morphism $\alpha|_{D(f)}: M_1|_{D(f)} \rightarrow M_2|_{D(f)}$ also lies in $\Mcal^0(D(f))$: 
	as in the previous case, the existence of such a function follows from condition (iii) in the statement. 
	As before, let $j: D(f) \hookrightarrow S$ and $i: V(f) \hookrightarrow S$ denote the corresponding open and closed immersions. Under the canonical isomorphisms $M_r = \gamma_{S,f} \delta_{S,f}(M_r)$, $r = 1,2$, constructed in the proof of \Cref{prop-Bei-glue}, the morphism $\alpha$ is induced by the morphism of complexes in $\Mcal(S)$
	\begin{equation*}
		\begin{tikzcd}
			i_* {^p \Psi_f}(j^* M_1) \arrow{r}{+} \arrow{d}{i_* {^p \Psi_f}(j^* \alpha)} & {^p \Xi_f}(M_1) \oplus i_* {^p \Phi_f}(M_1) \arrow{r}{-} \arrow{d}{{^p \Xi_f}(\alpha) \oplus i_* {^p \Phi_f}(\alpha)} & i_* {^p \Psi_f}(j^* M_1)(-1) \arrow{d}{i_* {^p \Psi_f}(\alpha)(-1)} \\
			i_* {^p \Psi_f}(j^* M_2) \arrow{r}{+} & {^p \Xi_f}(M_2) \oplus i_* {^p \Phi_f}(M_2) \arrow{r}{-} & i_* {^p \Psi_f}(j^* M_2)(-1).
		\end{tikzcd}
	\end{equation*}
	By the same argument as in the previous paragraph, both complexes are isomorphic to complexes in $\Mcal^0(S)$. 
	We claim that the three vertical arrows lie in $\Mcal(S)$ as well; this will imply that the morphism $\alpha$ belongs to the abelian subcategory $\Mcal^0(S)$ of $\Mcal(S)$. 
	In order to prove the claim, it suffices to note that:
	\begin{itemize}
		\item[–] The arrows $i_* {^p \Psi_f}(j^* \alpha)$, ${^p \Xi_f}(\alpha)$ and $i_* {^p \Psi_f}(j^* \alpha)(-1)$ lie in $\Mcal^0(S)$ by condition (i), since the arrow $j^* \alpha$ lies in $\Mcal^0(D(f))$ by construction.
		\item[–] The arrow $i_* {^p \Phi_f}(\alpha)$, being in the essential image of the fully faithful functor $i_*$, lies in $\Mcal^0(S)$ by condition (i), since the inclusion $\Mcal^0(V(f)) \subset \Mcal(V(f))$ is full by inductive hypothesis.
	\end{itemize}  
	This shows the fully faithfulness, thereby concluding the proof.
\end{proof}

\subsection{Construction of the alternative presentation}

In the final part of this section, we apply the criterion of \Cref{thm-Beilinson-simpler} to obtain the promised alternative presentation of perverse motives. 
Let us first fix the relevant notation:

\begin{nota}\label{nota:Mp^00}
	Let $S$ be a $k$-variety.
	\begin{itemize}
		\item We introduce the full additive subcategory
		\begin{equation*}
			\Dcal^0(S) := \left\{A \in \DA_{ct}^{\et}(S,\Q) \; | \; \Bti_S^*(A) \in \Perv(S) \right\} \subset \Dcal(S).
		\end{equation*}
		\item We let $\Mcal^0(S)$ denote the universal abelian factorization of the additive functor
		\begin{equation*}
			\beta_S^{0} := \beta_S|_{\Dcal^0(S)}: \Dcal^0(S) \rightarrow \Perv(S).
		\end{equation*}
		We implicitly regard it as a (possibly non-full) abelian subcategory of $\Mcal(S)$ via the canonical faithful exact functor 
		\begin{equation*}
			\Mcal^0(S) \hookrightarrow \Mcal(S)
		\end{equation*}
		obtained by applying \cite[Prop.~2.5]{Ter23UAF} to the diagram
		\begin{equation*}
			\begin{tikzcd}
				\Dcal^0(S) \arrow[hook]{r} \arrow{d}{\beta_S^0} & \Dcal(S) \arrow{d}{\beta_S} \\
				\Perv(S)\arrow{r}{\id} & \Perv(S).
			\end{tikzcd}
		\end{equation*}
		\item For simplicity, we still write
		\begin{itemize}
			\item $\pi_S$ for the quotient functor $\A(\Dcal^0(S)) \rightarrow \Mcal^0(S)$, and also for the composite functor $\Dcal^0(S) \rightarrow \A(\Dcal^0(S)) \rightarrow \Mcal^0(S)$,
			\item $\iota_S$ for the induced faithful exact functor $\Mcal^0(S) \hookrightarrow \Perv(S)$.
		\end{itemize}
		This convention will not create any ambiguity.
	\end{itemize}
\end{nota}

Even if the definition of the additive categories $\Dcal^0(S)$ (and, therefore, also that of the abelian categories $\Mcal^0(S)$) depends a priori on the complex embedding $\sigma: k \hookrightarrow \C$, we purposely omit any mention of it from the notation. This choice is justified by the following result:

\begin{lem}\label{lem:D^0}
	Choose two distinct complex embeddings $\sigma_1, \sigma_2: k \hookrightarrow \C$. 
	For every $k$-variety $S$, write $\Bti_{S,\sigma_i}^*: \DA_{ct}^{\et}(S,\Q) \rightarrow D^b_c(S^{\sigma_i},\Q)$ for the Betti realization functor associated to $\sigma_i$, and set
	\begin{equation*}
		\Dcal^0_{\sigma_i}(S) := \left\{A \in \DA_{ct}^{\et}(S,\Q) \; | \; \Bti_{S,\sigma_i}^*(A) \in \Perv(S^{\sigma_i})\right\} \subset \Dcal(S).
	\end{equation*}
	Then we have $\Dcal^0_{\sigma_1}(S) = \Dcal^0_{\sigma_2}(S)$.
\end{lem}
\begin{proof}
	We have to show that, for every object $A \in \DA_{ct}^{\et}(S,\Q)$, the equivalence
    \begin{equation*}
		\Bti_{S,\sigma_1}^*(A) \in \Perv(S) \iff \Bti_{S,\sigma_2}(A) \in \Perv(S)_{\ell}
	\end{equation*}
	holds. 
	We claim that, in fact, for every fixed $n \in \Z$, the equivalences
	\begin{equation*}
		{^p H^n}(\Bti_{S,\sigma_i}^*(A)) = 0 \iff {^p H^n}(R_{\ell,S}(A)) = 0 \quad (i = 1,2)
	\end{equation*}
	hold; 
	the claim clearly implies the thesis. 
	The claim can be proved by induction on $\dim(S)$ using the method of the proof of \cite[Prop.~6.11]{IM19};
	we leave the details to the interested reader.
\end{proof}

The rest of this section is devoted to the proof of the following result:

\begin{thm}\label{thm_Mp0=Mp}
	For every $k$-variety $S$, the canonical inclusion $\Mcal^0(S) \subset \Mcal(S)$ is in fact an equivalence.
\end{thm}

The proof of this result is based on \Cref{thm-Beilinson-simpler}.
At first sight, the amount of conditions to check looks massive. Actually, all the actual work goes into checking the weak cofinality property at point (iii). 
This requires some careful preparation about Nori motives over function fields:
the goal is to understand categories of Nori motives over the function field $L$ of a $k$-variety $S$ (defined by the natural colimit construction over dense open subvarieties of $S$) as actual categories of Nori motives (defined starting from a complex embedding of $L$).
The principle is that a category of motives over a scheme should only depend on the scheme itself.
Some results is this direction are collected in \cite[\S~6.5]{IM19}.

\begin{constr}\label{constr:M(L)}
	Let $L$ be a finitely generated extension field of $k$, and fix a $k$-variety $S$ with function field $L := k(S)$: this means that $S$ is irreducible and the sheaf of $\Ocal_S$-modules
	\begin{equation*}
		U \mapsto \textup{Frac}(\Ocal_S(U))
	\end{equation*} 
    is identified with the constant sheaf with value $L$ over the Zariski site of $S$.
	We can write 
	\begin{equation*}
		\Spec(L) = \varprojlim_{U \in \Aff\Op(S)} U
	\end{equation*}
    as the inverse limit of all affine open subsets of $S$ in the category of $k$-schemes. 
    By the continuity property of triangulated étale motives proved in \cite{Ayo07a}, the canonical functor
    \begin{equation*}
    	\twocolim_{U \in \Aff\Op(S)^{op}} \DA_{ct}^{\et}(U,\Q) \rightarrow \DA_{ct}^{\et}(\Spec(L),\Q)
    \end{equation*}
    induced by the inverse image functors $\DA_{ct}^{\et}(U,\Q) \rightarrow \DA_{ct}^{\et}(\Spec(L),\Q)$ is an equivalence. 
    Therefore, by \cite[Lemma~5.12(2)]{Ter23UAF}, the filtered $2$-colimit of abelian categories
    \begin{equation}\label{formula:M(L)}
    	\twocolim_{U \in \Aff\Op(S)} \Mcal(U)
    \end{equation}
    can be canonically identified with the universal abelian factorization of the additive functor
    \begin{equation*}
    	\beta_{L,\sigma}: \Dcal(\Spec(L)) = \twocolim_{U \subset S} \Dcal(U) \xrightarrow{\beta_U} \twocolim_{U \subset S} \Perv(U),
    \end{equation*}
    where, in analogy to the notation $\Dcal(U)$ for $\DA_{ct}^{\et}(U,\Q)$ employed for $U \in \Var_k$, we write $\Dcal(\Spec(L))$ for the underlying additive category of $\DA_{ct}^{\et}(\Spec(L),\Q)$.
    Let us suggestively write $\Mcal(\Spec(L))$ for the abelian category \eqref{formula:M(L)}; note that, in the case where the extension $L/k$ is finite (so that the $k$-scheme $\Spec(L)$ belongs to $\Var_k$), this agrees with the definition of perverse motives over $\Spec(L)$ given in \cite{IM19}. 
    Clearly, in the definition of the $2$-colimit \eqref{formula:M(L)}, one can replace the poset $\Aff\Op(S)$ by any cofinal subposet: for instance, we can use the poset $\Sm\Op(S) \cap \Aff\Op(S)$ consisting of all smooth affine open subsets of $S$.
    It is easy to see that the definition of $\Mcal(\Spec(L))$ is independent of the chosen model $S$ of $L$.
    Keeping the above notation, let us also introduce the additive subcategory 
    \begin{equation*}
    	\Dcal^0(\Spec(L)) := \twocolim_{U \in \Aff\Op(S)^{op}} \Dcal^0(U) \subset \Dcal(\Spec(L)).
    \end{equation*}
    Again, this definition does not depend on the chosen model $S$ of $L$ and does not change if one replaces the poset $\Aff\Op(S)$ by any cofinal subposet such as $\Sm\Op(S) \cap \Aff\Op(S)$. 
    Moreover, let $\Mcal^0(\Spec(L))$ denote the universal abelian factorization of the additive functor
    \begin{equation*}
    	\beta_{L,\sigma}^0 := \beta_{L,\Sigma}|_{\Dcal^0(\Spec(L))}: \Dcal^0(\Spec(L)) = \twocolim_{U \subset S} \Dcal^0(U) \xrightarrow{\beta_U} \twocolim_{U \subset S} \Perv(U),
    \end{equation*}
    which is independent of the chosen presentation as well; if the extension $L/k$ is finite, this is compatible with the original definition given in \Cref{nota:Mp0}. 
    Applying \cite[Lemma~5.12]{Ter23UAF} as above, we see that $\Mcal^0(\Spec(L))$ can be canonically identified with the filtered $2$-colimit of abelian categories
    \begin{equation*}
    	\twocolim_{U \in \Aff\Op(S)^{op}} \Mcal^0(U),
    \end{equation*}
    or with the analogous $2$-colimit over the cofinal subposet $\Sm\Op(S) \cap \Aff\Op(S)$.
\end{constr}

A priori, the categories of perverse Nori motives over function fields just introduced are not categories of motives over a field in Nori's sense; this might give rise to some confusion.
For instance, suppose that the field extension $L/k$ is finite, so that $\Spec(L)$ is a $k$-variety; 
fix a complex embedding $\Sigma: L \hookrightarrow \C$ extending $\sigma: k \hookrightarrow \C$.
Then we have two distinct ways to define an abelian category of perverse Nori motives out of $\DA_{ct}^{\et}(\Spec(L),\Q)$: on the one side, as the category $\Mcal(\Spec(L))$ obtained via $\sigma$ (by considering $\Spec(L)$ as a $k$-variety); on the other side, as the category $\Mcal(L)$ obtained via $\Sigma$ (by considering $L$ as the base field).
The following result shows that this does not make any difference; 
the same holds if the field extension $L/k$ is finitely generated.

\begin{prop}\label{prop:M(L)-indep}
	Let $L/k$ be a finitely generated field extension, and assume that there exists a field embedding $\Sigma: L \hookrightarrow \C$ extending the original embedding $\sigma: k \hookrightarrow \C$.
	Then the two exact functors 
	\begin{equation*}
		\beta_{\Spec(L),\sigma}^+: \A(\Dcal(L)) \rightarrow \vect_{\Q}, \qquad \beta_{L,\Sigma}^+: \A(\Dcal(L)) \rightarrow \vect_{\Q}
	\end{equation*}
    \begin{equation*}
    	\textup{(resp. $\beta_{\Spec(L),\sigma}^{0,+}: \A(\Dcal^0(L)) \rightarrow \vect_{\Q}, \qquad \beta_{L,\Sigma}^{0,+}: \A(\Dcal^0(L)) \rightarrow \vect_{\Q}$)}
    \end{equation*}
    have the same kernel. In particular, the category $\Mcal(\Spec(L))$ (resp. $\Mcal^0(\Spec(L))$) is canonically equivalent to $\Mcal(L)$ (resp. $\Mcal^0(L)$).
\end{prop}
\begin{proof}
	For sake of simplicity, we only write down the argument for $\Mcal(\Spec(L))$ and $\Mcal(L)$; the argument for $\Mcal^0(\Spec(L))$ and $\Mcal^0(L)$ is a minor variant of the former, and we leave it to the interested reader.
	
    We start by describing the key geometric construction. Fix a $k$-variety $S$ with function field $L$, and write 
	\begin{equation*}
		\Spec(L) = \varprojlim_{U \in \Sm\Op(S)} U
	\end{equation*}
    as the limit in the category of $k$-schemes over all smooth non-empty open subsets of $S$. For every $U \in \Sm\Op(S)$, the canonical morphism of $k$-schemes $\eta_U: \Spec(L) \rightarrow U$ is the inclusion of the generic point; the induced monomorphism of sheaves of $k$-algebras  
    \begin{equation*}
    	\Ocal_U \xrightarrow{\eta_U^*} L \xrightarrow{\Sigma} \C
    \end{equation*}
    defines a complex point $z_{U,\Sigma} \in U^{\sigma}$ which, by construction, does not lie on any proper algebraic $k$-subvariety of $U$. In fact, for every inclusion $V \subset U$ in $\Sm\Op(S)$, the induced inclusion of complex-analytic spaces $V^{\sigma} \subset U^{\sigma}$ identifies $z_{V,\Sigma}$ with $z_{U,\Sigma}$.
    By construction, for every $U \in \Sm\Op(S)$ the diagram 
    \begin{equation*}
        \begin{tikzcd}
    		\DA_{ct}^{\et}(U,\Q) \arrow{rr}{\eta_U^*} \arrow{d}{\Bti_{U,\sigma}^*} && \DA_{ct}^{\et}(L,\Q) \arrow{d}{\Bti_{L,\Sigma}^*} \\
    		D^b_c(U,\Q) \arrow{rr}{z_{U,\Sigma}^*} && D^b(\vect_{\Q})
        \end{tikzcd}
    \end{equation*}
    commutes up to canonical natural isomorphism. In fact, we have a canonical natural isomorphism
    \begin{equation*}
    	z_{U,\Sigma}^* \circ \Bti_{U,\sigma}^* = \Bti_{L,\Sigma}^* \circ \eta_U^*
    \end{equation*} 
    as functors on the entire category of étale motives $\DA^{\et}(U,\Q)$: this extends the identification of complex-analytic spaces
    \begin{equation*}
    	X^{\sigma} \times_{U^{\sigma}} z_{U,\Sigma} = (X \times_U \Spec(L))^{\Sigma},
    \end{equation*}
    functorial with respect to $X \in \Sm/U$ (this fact was observed by Ayoub in the proof of \cite[Prop.~2.20]{Ayo14H2}). 
    From the compatibility of the complex points $z_{U,\Sigma}$ under restriction, we deduce that the diagram
    \begin{equation*}
    	\begin{tikzcd}
    		\twocolim_{U \in \Sm\Op(S)^{op}} \DA_{ct}^{\et}(U,\Q) \arrow{rr}{\sim} \arrow{d}{(\Bti_{U,\sigma}^*)_U} && \DA_{ct}^{\et}(L,\Q) \arrow{d}{\Bti_{L,\Sigma}^*} \\
    		\twocolim_{U \in \Sm\Op(S)^{op}} D^b_c(U,\Q) \arrow{rr}{(z_{U,\Sigma}^*)_U} && D^b(\vect_{\Q})
    	\end{tikzcd}
    \end{equation*}
    commutes up to natural isomorphism as well. 
    Applying the perverse truncation functors, we deduce a commutative diagram of the form
    \begin{equation*}
    	\begin{tikzcd}
    		\twocolim_{U \in \Sm\Op(S)^{op}} \DA_{ct}^{\et}(U,\Q) \arrow{rr}{\sim} \arrow{d}{(\Bti_{U,\sigma}^*)_U} && \DA_{ct}^{\et}(L,\Q) \arrow{d}{\Bti_{L,\Sigma}^*} \\
    		\twocolim_{U \in \Sm\Op(S)^{op}} D^b_c(U,\Q) \arrow{rr}{(z_U^*)_U} \arrow{d}{{^p H^0}} && D^b(\vect_{\Q}) \arrow{d}{H^0} \\
    		\twocolim_{U \in \Sm\Op(S)^{op}} \Perv(U) & \twocolim_{U \in \Sm\Op(S)^{op}} \Loc_p(U) \arrow{l}{\sim} \arrow{r}{(z_{U,\Sigma}^{\dagger})_U} & \vect_{\Q},
    	\end{tikzcd}
    \end{equation*}
    where the lower-left horizontal arrow is the equivalence induced by the inclusions $\Loc_p(U) \subset \Perv(U)$ for $U \in \Sm\Op(S)$ while the lower-right horizontal arrow is defined by the exact shifted inverse image functors
    \begin{equation}\label{formula:z_U^*}
    	z_{U,\Sigma}^{\dagger} := z_{U,\Sigma}^*[-\dim(U)]: \Loc_p(U) \rightarrow \vect_{\Q}.
    \end{equation}
    Passing to the abelian hulls, we obtain a commutative diagram of the form
    \begin{equation}\label{dia:z_U-colim}
    	\begin{tikzcd}
    		\A(\twocolim_{U \in \Sm\Op(S)^{op}} \DA_{ct}^{\et}(U,\Q)) \arrow{d}{\beta_{\Spec(L),\sigma}^+} \arrow{rr}{\sim} && \A(\DA_{ct}^{\et}(L,\Q)) \arrow{d}{\beta_{L,\Sigma}^+} \\
    		\twocolim_{U \in \Sm\Op(S)^{op}} \Loc_p(U) \arrow{rr}{(z_{U,\Sigma}^{\dagger})_U} && \vect_{\Q}.
    	\end{tikzcd}
    \end{equation}
    Proving the thesis amounts to showing that the equivalence in the upper horizontal arrow identifies the kernels of the two vertical exact functors.
     
    Let us first explain how to prove this in the simple case where the field extension $L/k$ is geometrically integral. In this case, each $U \in \Sm\Op(S)$ is geometrically connected, and so the associated complex-analytic space $U^{\sigma}$ is connected as well. This implies that the exact functors \eqref{formula:z_U^*} are all faithful. Therefore the same holds for the lower horizontal arrow in \eqref{dia:z_U-colim}. But then it follows formally that the two vertical exact functors in \eqref{dia:z_U-colim} have the same kernel (for example, see the proof of \cite[Prop.~1.10]{Ter23UAF}).
    
    If the field extension $L/k$ is not assumed to be geometrically integral, the argument is similar but slightly more involved: the issue is that the exact functors \eqref{formula:z_U^*} are no longer necessarily faithful, since the complex-analytic spaces $U^{\sigma}$ are no longer necessarily connected.
    To remedy this issue, the rough idea is to consider multiple complex embeddings of $L$ at once so that te corresponding points $z_{U,\Sigma}$ hit all connected components of $U^{\sigma}$. In order to make this idea work, we need to describe the construction of the points $z_{U,\Sigma}$ in more explicit geometric terms. 
    To this end, fix a transcendence basis $t_1,\dots,t_n$ for $L/k$ such that we can write
    \begin{equation*}
    	L = k(t_1,\dots,t_n)[x]/P(x),
    \end{equation*}  
    where $P(x) \in k(t_1,\dots,t_n)[x]$ is a monic polynomial in $x$ with coefficients in $k[t_1,\dots,t_n]$ defining a smooth subvariety inside the affine space $\A^{n+1}_k$ with coordinates $t_1,\dots,t_n,x$ (the existence of such a basis follows from the strong form of Noether's Normalization Lemma, applied to some chosen affine $U \in \Sm\Op(S)$). 
    In this way, we get an identification
    \begin{equation*}
    	U = \Spec(k[t_1,\dots,t_n,x]/P(x))
    \end{equation*} 
    as closed subvarieties of $\A^{n+1}_k$. Note that, by construction, $P(x)$ is irreducible in $k(t_1,\dots,t_n)[x]$.
    For sake of brevity, set $L_0 := k(t_1,\dots,t_n)$; write $\Sigma_0 := \Sigma|_{L_0} : L_0 \hookrightarrow \C$, and let $\Sigma_1, \dots, \Sigma_{\deg(P)}$ denote the $\deg(P)$ distinct complex embeddings of $L$ extending $\Sigma_0$ (among which figures the original embedding $\Sigma$). 
    Moreover, write $P^{\sigma}(x)$ for the image of $P(x)$ in $\C(t_1,\dots,t_n)[x]$ via $\sigma$, and let
    \begin{equation*}
    	P^{\sigma}(x) = \prod_{i=1}^{r} Q_i(x)
    \end{equation*}
    be its factorization into irreducible polynomials.
    Note that, since the coefficients of $P^{\sigma}(X)$ belong to the unique factorization domain $\C[t_1,\dots,t_n]$, the same holds for the coefficients of each factor $Q_i(x)$. Thus each $Q_i(x)$ defines a connected closed subvariety $U_i$ of $\A^{n+1}_{\C}$, and the complex-analytic space $U^{\sigma}$ breaks into connected components as
    \begin{equation*}
    	U^{\sigma} = \coprod_{i=1}^{r} U_i.
    \end{equation*}
    Note that each complex variety $U_i$ is smooth, since so is $U^{\sigma}$. 
    Moreover, write $P^{\Sigma_0}(x)$ for the image of $P(x)$ in $\C[x]$ via $\Sigma_0$, and consider the factorization 
    \begin{equation*}
    	P^{\Sigma_0}(x) = \prod_{i=1}^{r} \bar{Q}_i(x),
    \end{equation*}
    where $\bar{Q}_i(x)$ denotes the image of $Q_i(x)$ in $\C[x]$ via the $\C$-algebra homomorphism
    \begin{equation*}
    	\C[t_1,\dots,t_n] \rightarrow \C, \quad t_i \mapsto \Sigma_0(t_i).
    \end{equation*}
    The complex embeddings $\Sigma_1,\dots,\Sigma_{\deg(P)}$ extending $\Sigma_0$ are in natural bijection with the complex roots of $P^{\Sigma_0}(x)$: giving such an embedding $\Sigma'$ amounts to choosing a factor $Q_i(x)$ of $P^{\sigma}(x)$ and specifying a complex root of $\bar{Q}_i(x)$; by construction, the corresponding complex point $z_{U,\Sigma'} \in U^{\sigma}$ belongs to the connected component $U_i$. Thus the complex points
    \begin{equation*}
    	z_{U,\Sigma_1}, \dots, z_{U,\Sigma_{\deg(P)}} 
    \end{equation*}
    jointly hit every connected component of $U^{\sigma}$. Similarly, for each dense open subset $V \subset U$, the points $z_{U,\Sigma_i}$ restrict to a family of complex points $z_{V,\Sigma_i} \in V^{\sigma}$ which jointly hit every connected component of $V^{\sigma}$. 
    In order to conclude the proof, consider the commutative diagram
    \begin{equation}\label{dia:z_U-colim(2)}
    	\begin{tikzcd}
    		\A(\twocolim_{U \in \Sm\Op(S)^{op}} \DA_{ct}^{\et}(U,\Q)) \arrow{d}{\beta_{\Spec(L),\sigma}^+} \arrow{rr}{\sim} && \A(\DA_{ct}^{\et}(L,\Q)) \arrow{d}{(\beta_{L,\Sigma_i}^+)_i} \\
    		\twocolim_{U \in \Sm\Op(S)^{op}} \Loc_p(U) \arrow{rr}{(z_{U,\Sigma_i}^{\dagger})_{U,i}} && \prod_{i=1}^{\deg(P)} \vect_{\Q}
    	\end{tikzcd}
    \end{equation}
    obtained by tying together the individual diagrams of shape \eqref{dia:z_U-colim} for the embeddings $\Sigma_1,\dots,\Sigma_{\deg(P)}$. For every dense open subset $V \subset U$, the collection of exact shifted inverse image functors
    \begin{equation*}
    	z_{V,\Sigma_1}^{\dagger}, \dots, z_{V,\Sigma_{\deg(P)}}^{\dagger}: \Loc_p(V) \rightarrow \vect_{\Q}
    \end{equation*}
    is jointly faithful, precisely because the points $z_{V,\Sigma_i}$ hit every connected component of $V^{\sigma}$. Therefore the the lower horizontal arrow in \eqref{dia:z_U-colim(2)} is a faithful exact functor as well. As in the geometrically connected case, this implies formally that the two vertical exact functors in \eqref{dia:z_U-colim(2)} have the same kernel. 
    It now suffices to show that the right-most vertical arrow in the two diagrams \eqref{dia:z_U-colim} and \eqref{dia:z_U-colim(2)} have the same kernel or, equivalently, that the various exact functors
    \begin{equation*}
    	\beta_{L,\Sigma_1}^+, \dots, \beta_{L,\Sigma_{\deg(P)}}^+: \A(\DA_{ct}^{\et}(L,\Q)) \rightarrow \vect_{\Q}
    \end{equation*}
    all have the same kernel. This follows from the fact that the induced exact functors 
    \begin{equation*}
    	\beta_{L,\Sigma_1}^+, \dots, \beta_{L,\Sigma_{\deg(P)}}^+: \A(\DA_{ct}^{\et}(L,\Q)) \rightarrow \vect_{\Q} \xrightarrow{- \otimes_{\Q} L} \vect_L
    \end{equation*}
    are all naturally isomorphic to each other: indeed, the homological functors
    \begin{equation*}
    	\beta_{L,\Sigma_1}, \dots, \beta_{L,\Sigma_{\deg(P)}}: \DA_{ct}^{\et}(L,\Q) \rightarrow \vect_{\Q} \xrightarrow{- \otimes_{\Q} L} \vect_L
    \end{equation*}
    are all isomorphic to the algebraic de Rham realization of Voevodsky motives, via Huber's construction of mixed realizations from \cite{Hub00}. 
\end{proof}

	

Note that, in general, it might be simply impossible to extend the original embedding $\sigma: k \hookrightarrow \C$ to a given function field $L/k$ (for instance if $k = \C$ and $\sigma = \id_{\C}$). 
Luckily, we can always reduce to the case where $L$ (or equivalently, $k$) is finitely generated over $\Q$. 
The next result is a consequence of \Cref{prop:M(L)-indep}, applied not directly to the original field extension $L/k$ but rather to the induced extensions between suitable finitely generated subfields thereof.

\begin{cor}\label{cor:M(L)}
	Let $L/k$ be a finitely generated field extension. Let $\Scal_L$ denote the filtered poset of all subfields of $L$ which are finitely generated over $\Q$; for each $F \in \Scal_L$, choose an embedding $\Sigma_F: F \hookrightarrow \C$ extending the restriction $\sigma|_{F \cap k}: F \cap k \subset k \hookrightarrow \C$.
	
	Then, as $F$ varies in $\Scal_L$, the abelian categories $\Mcal(F)$ (resp. $\Mcal^0(F)$) defined via the chosen embeddings $\Sigma_F$ canonically assemble into an abelian fibered category over $\Scal_L^{op}$, and we have a canonical equivalence
	\begin{equation*}
		\Mcal(L) = \twocolim_{F \in \Scal_L} \Mcal(F)
	\end{equation*}
    \begin{equation*}
    	\textup{(resp. $\Mcal^0(L) = \twocolim_{F \in \Scal_L} \Mcal^0(F)$)}.
    \end{equation*}
\end{cor}
\begin{proof}
	Again, we only give the argument for the categories $\Mcal(L)$ and $\Mcal(F)$; 
	the argument for the categories $\Mcal^0(L)$ and $\Mcal^0(F)$ is an easy variant of the former, and we leave it to the interested reader.
	
	For any fixed $F \in \Scal_L$, and any choice of two complex embeddings $\Sigma_F^{(1)}, \Sigma_F^{(2)}: F \hookrightarrow \C$, the abelian categories $\Mcal(F_1,\Sigma_F^{(1)})$ and $\Mcal(F_2,\Sigma_F^{(2)})$ coincide, as a consequence of Huber's construction of mixed realizations (this was noted by Ivorra--Morel in the proof of \cite[Prop.~6.11]{IM19}).
    Therefore, for every inclusion $F_1 \subset F_2$ in $\Scal_L$, we have a well-defined exact functor
	\begin{equation*}
		\Mcal(F_1) := \Mcal(F_1,\Sigma_{F_1}) = \Mcal(F_1,\Sigma_{F_2}|_{F_1}) \rightarrow \Mcal(\Spec(F_2),\Sigma_{F_2}|_{F_1}) = \Mcal(F_2,\Sigma_{F_2}) =: \Mcal(F_2), 
	\end{equation*} 
    where the last equivalence is given by \Cref{prop:M(L)-indep}.
    It is easy to see that these functors are compatible with composition of inclusions in $\Scal_L$, thereby defining an abelian fibered category as stated. 
    
    The proof of the last assertion is not difficult but a bit lengthy, since it forces us to describe everything in more explicit geometric terms, as for the proof of \Cref{prop:M(L)-indep}. 
    To begin with, fix an affine $k$-variety $S$ with function field $L$, as done at the beginning of \Cref{constr:M(L)} - it is not relevant whether $S$ is smooth or not. 
    Applying the strong form of Noether's Normalization Lemma, we can write
    \begin{equation*}
    	S = \Spec(k[t_1,\dots,t_n][x]/P(x))
    \end{equation*}
    where $t_1,\dots,t_n$ form a transcendence basis for $L/k$ and $P(x) \in k[t_1,\dots,t_n][x]$ is an irreducible monic polynomial in $x$. 
    From now on, we regard $S$ as a closed subvariety of the affine space $\A^{n+1}_k$ with coordinates $t_1,\dots,t_n,x$, and we tacitly identify the function field $L := k(S)$ with $k(t_1,\dots,t_n)[x]/P(x)$.
    
    Let $k_0$ denote the smallest subfield of $k$ such that $P(x) \in k_0[t_1,\dots,t_n][x]$: 
    by construction, it is finitely generated over $\Q$. 
    Consider the filtered poset $\Scal_{k/k_0}$ consisting of all subfields of $k$ which are finitely generated over $\Q$ and contain $k_0$, and let $\Scal_{L/k_0} \subset \Scal_L$ denote the cofinal subposet consisting of fields of the form $k'(t_1,\dots,t_n)/P(x)$ with $k' \in \Scal_{k/k_0}$.
    We have an obvious map of posets $\Scal_{k/k_0} \rightarrow \Scal_{L/k_0}$, and we write the image of a typical element $k' \in \Scal_{k/k_0}$ under it as $F'$.
    
    For each $k' \in \Scal_{k/k_0}$, let $\Aff\Op_{k'}(S) \subset \Aff\Op(S)$ denote the subposet consisting of all affine open subsets of $S$ which are defined over $k'$: we declare that an open subset $U \in \Aff\Op(S)$ belongs to $\Aff\Op_{k'}(S)$ if the closed subvariety $S \setminus U \subset \A^{n+1}_k$ is defined over $k'$, which means that the associated ideal $\mathbb{I}(S \setminus U) \subset k[x_1,\dots,x_n]$ satisfies the equality
    \begin{equation*}
    	\mathbb{I}(S \setminus U) = (\mathbb{I}(S \setminus U) \cap k'[x_1,\dots,x_n]) \otimes_{k'} k.
    \end{equation*} 
    Note that the definition of $\Aff\Op_{k'}(S)$ depends on the chosen presentation of $S$ as a closed subvariety of $\A^{n+1}_k$ - for instance, it is not stable under $k$-automorphisms of $\A^{n+1}_k$.
    Nevertheless, since every closed subvariety of $\A^{n+1}_k$ is defined over some field finitely generated over $\Q$, we have the filtered union
    \begin{equation}\label{eq:AffOp_F}
    	\Aff\Op(S) = \bigcup_{k' \in \Scal_{k'/k_0}} \Aff\Op_{k'}(S).
    \end{equation}
    For each $k' \in \Scal_{k/k_0}$, we let $S_{(k')}$ denote the closed subvariety of $\A^{n+1}_{k'}$ defined by the polynomial $P(x) \in k'[t_1,\dots,t_n][x]$.
    Similarly, for every $U \in \Aff\Op_{k'}(S)$, we let $U_{(k')}$ denote the corresponding affine $k'$-variety: its closed complement $S_{(k')} \setminus U_{(k')}$ is the closed subvariety of $\A^n_{k'}$ with associated ideal $\mathbb{I}(S \setminus U) \cap k'[x_1,\dots,x_n]$. 
    Note that, by construction, every $U \in \Aff\Op(S)$ can be written as a filtered inverse limit with affine transition maps
    \begin{equation*}
    	U = \varprojlim_{k' \in \Scal_{k/k_0}^{op}} U_{(k')}
    \end{equation*}
    in the category of schemes. Therefore, by \cite[Cor.~3.22]{AyoEt}, we have a canonical equivalence
    \begin{equation*}
    	\twocolim_{k' \in \Scal_{k'/k_0}^{op}} \DA_{ct}^{\et}(U_{(k')},\Q) \xrightarrow{\sim} \DA_{ct}^{\et}(U,\Q).
    \end{equation*}
    Note also that, for $k' \in \Scal_{k/k_0}$ and $U \in \Aff\Op_{k'}(S)$, we have a canonical identification of complex-analytic spaces
    \begin{equation*}
    	U_{(k')}^{\sigma|_{k'}} = U^{\sigma},
    \end{equation*}
    and the resulting diagram
    \begin{equation*}
    	\begin{tikzcd}
    		\DA_{ct}^{\et}(U_{(k')},\Q) \arrow{rr} \arrow{d}{\Bti_{U_{(k')},\sigma|_{k'}}^*} && \DA_{ct}^{\et}(U,\Q) \arrow{d}{\Bti_{U,\sigma}^*} \\
    		D^b_c(U_{(k')}^{\sigma|_{k'}},\Q) \arrow{rr} && D^b_c(U^{\sigma},\Q)
    	\end{tikzcd}
    \end{equation*}
    is commutative up to natural isomorphism. In fact, we have a canonical natural isomorphism
    \begin{equation*}
    	\Bti_{U_{(k')},\sigma_{k'}}^* = \Bti_{U,\sigma}^*
    \end{equation*}
    as functors on the entire category of \'{e}tale motives $\DA^{\et}(U_{(k')},\Q)$: this extends the identification of complex-analytic spaces
    \begin{equation*}
    	X^{\sigma_{k'}} = (X \otimes_{k'} k)^{\sigma},
    \end{equation*}
    functorial with respect to $X \in \Sm/{U_{(k')}}$. Passing to the colimit over $\Scal_{k/k_0}^{op}$, and applying \cite[Lemma~5.12]{Ter23UAF}, this yields a canonical equivalence
    \begin{equation}\label{eq:M(UF)}
    	\twocolim_{k' \in \Scal_{k/k_0}} \Mcal(U_{(k')},\sigma|_{k'}) = \Mcal(U,\sigma),
    \end{equation}
    which is clearly compatible with restrictions under inclusions in $\Aff\Op(S)$. Lastly, note that, for every $k' \in \Scal_{k/k_0}$, the obvious inclusion of posets $\Aff\Op_{k'}(S) \subset \Aff\Op(S_{(k')})$ is cofinal.
    Putting all these observation together, we obtain the chain of equivalences
    \begin{align*}
    	\Mcal(L) = & \Mcal(\Spec(L),\sigma) && \textup{by \Cref{prop:M(L)-indep}} \\
    	&= \twocolim_{U \in \Aff\Op(S)^{op}} \Mcal(U,\sigma) && \textup{by definition} \\ 
    	&= \twocolim_{k' \in \Scal_{k/k_0}} 2-\varinjlim_{U \in \Aff\Op_{k'}(S)} \Mcal(U,\sigma) && \textup{by \eqref{eq:AffOp_F}}  \\
    	&= \twocolim_{k' \in \Scal_{k/k_0}} 2-\varinjlim_{U \in \Aff\Op_{k'}(S)} \Mcal(U_{(k')},\sigma|_{k'}) && \textup{by \eqref{eq:M(UF)}} \\
    	&= \twocolim_{k \in \Scal_{k/k_0}} 2-\varinjlim_{V \in \Aff\Op(S_{(k')})} \Mcal(V,\sigma|_{k'}) && \textup{by cofinality} \\
    	&= \twocolim_{k' \in \Scal_{k/k_0}} \Mcal(\Spec(F'),\sigma|_{k'}) && \textup{by definition} \\
    	&= \twocolim_{k' \in \Scal_{k/k_0}} \Mcal(F',\Sigma_{F'}) && \textup{by \Cref{prop:M(L)-indep}} \\
    	&= \twocolim_{F \in \Scal_L} \Mcal(F,\Sigma_F) && \textup{by cofinality}.
    \end{align*}
    This concludes the proof.
\end{proof}

\begin{proof}[Proof of \Cref{thm_Mp0=Mp}]
	We claim that, as $S$ varies, the abelian subcategories $\Mcal^0(S)$ of $\Mcal(S)$ satisfy the conditions in the statement of \Cref{thm-Beilinson-simpler}. This will imply the thesis.
	
	The validity of condition (i) follows from the fact that all the exact functors involved are constructed via the method described in \cite[Prop.~2.5]{Ter23UAF}. Similarly, the validity of condition (ii) follows from the fact that all the natural transformations involved are constructed via the method described in \cite[Prop.~3.4]{Ter23UAF}.
	
	Let us check the validity of condition (iii): we have to show that, for every $k$-variety $S$, the induced faithful exact functor
	\begin{equation}\label{Mp0=Mp-generic}
		\twocolim_{U \in \Sm\Op(S)^{op}} \Mcal^0(U) \hookrightarrow \twocolim_{U \in \Sm\Op(S)^{op}} \Mcal(U)
	\end{equation}
	is an equivalence. 
	Note that $\Sm\Op(S)$-fibered categories $\Mcal$ and $\Mcal^0$ are stacks for the Zariski topology: for $\Mcal$ this result is proved in \cite[Prop.~2.7]{IM19}, while for $\Mcal^0$ the same result follows using condition (i). 
	As a consequence, in order to prove that \eqref{Mp0=Mp-generic} is an equivalence for a general $k$-variety $S$, it suffices to treat each irreducible component of $S$ separately; in other words, without loss of generality we may assume $S$ irreducible.
	
	In this situation, let $L := k(S)$ denote the function field of $S$. Then, by the discussion of \Cref{constr:M(L)}, the functor \eqref{Mp0=Mp-generic} can be identified with the canonical inclusion
	\begin{equation*}
		\Mcal^0(\Spec(L)) \hookrightarrow \Mcal(\Spec(L))
	\end{equation*}
    and, by \Cref{cor:M(L)}, the latter can be identified with the functor
    \begin{equation*}
    	\twocolim_{F \in \Scal_L^{op}} \Mcal^0(F) \hookrightarrow \twocolim_{F \in \Scal_L^{op}} \Mcal(F)
    \end{equation*}
    induced by the inclusions of categories of Nori motives
    \begin{equation}\label{incl:M(F)}
    	\Mcal^0(F) \hookrightarrow \Mcal(F)
    \end{equation}
    over every finitely generated subfield $F \subset L$; here, the category $\Mcal(F)$ is defined in terms of a complex embedding $\Sigma_F: F \hookrightarrow \C$ extending $\sigma|_{F \cap k}$. In order to conclude, it suffices to show that the functors \eqref{incl:M(F)} are all equivalences. 
    So fix a field $F$ finitely generated over $\Q$.
    By \cite[Prop.~3.8]{BVHP20}, in order to prove that \eqref{incl:M(F)} is an equivalence, it suffices to show that the Betti realization
    \begin{equation*}
    	\Bti_{F,\Sigma_F}^*: \DA_{ct}^{\et}(F,\Q) \rightarrow D^b(\vect_{\Q})
    \end{equation*}
    factors through the functor
    \begin{equation*}
    	\iota^0_{F,\Sigma_F}: D^b(\Mcal^0(F)) \rightarrow D^b(\vect_{\Q})
    \end{equation*}
    up to natural isomorphism.  
    The latter fact was established by Nori; see \cite[Prop.~7.1]{ChudGal} or \cite[Thm.~7.4.17]{Harrer} for a detailed proof.
\end{proof}

\section{Construction of the external tensor structure}\label{sect_ETS}

In this section, we start constructing and studying the tensor structure on perverse motives. 
We work in the framework of stable homotopy $2$-functors over the category $\Var_k$ of quasi-projective $k$-varieties, as developed in \cite{Ayo07a,Ayo07b} and \cite{Ayo10}.
Our first main result can be stated as follows:
\begin{thm}\label{thm-boxtimesM}
	The stable homotopy $2$-functor $D^b(\Mcal(\cdot))$ s canonically unitary symmetric monoidal in the sense of \cite[Defn.~2.3.1]{Ayo07a}. Moreover, the morphism of stable homotopy $2$-functors 
	\begin{equation*}
		\iota: D^b(\Mcal(\cdot)) \rightarrow D^b(\Perv(\cdot))
	\end{equation*} 
    is canonically unitary symmetric monoidal in the sense of \cite[Defn.~3.2]{Ayo10}.
\end{thm}
\noindent
The proof of this result occupies most of \Cref{sect_ETS,sect_otimes-dist,sect_unit-Mp}. 
In the present section, after a preliminary discussion about multi-linear functors, we construct the tensor product functors and the monoidality isomorphisms.
In the final part, we construct the associativity and commutativity constraints and check their compatibility.
The construction of the motivic unit constraint, which is more delicate, is the subject of \Cref{sect_unit-Mp}.

\subsection{Categorical preliminaries}

Several constructions in Ivorra--Morel's paper \cite{IM19} are based on a technical recognition principle for functors and natural transformations in the constructible derived categories: given two $k$-varieties $T$ and $S$ and a triangulated functor $F: D^b_c(S,\Q) \rightarrow D^b_c(T,\Q)$ which is $t$-exact for the perverse $t$-structures, one wants to know whether $F$ coincides with the trivial derived functor of its restriction to the perverse hearts
\begin{equation*}
	F: \Perv(S) \rightarrow \Perv(T),
\end{equation*} 
under Beilinson's equivalences. More abstractly, given two abelian categories $\Acal$ and $\Bcal$ and a triangulated functor $F: D^b(\Acal) \rightarrow D^b(\Bcal)$ which is $t$-exact for the obvious $t$-structures, one can ask whether $F$ coincides with the trivial derived functor of its restriction to the hearts
\begin{equation*}
	F: \Acal \rightarrow \Bcal.
\end{equation*} 
This is not always the case, due to the technical defects of triangulated categories. However, by \cite[Thm.~1]{VoloDG}, in order for this to hold it suffices that $F$ admit a dg-enhancement; a similar result holds for triangulated natural transformations between $t$-exact triangulated functors. This is summarized in \cite[Prop.~4.4]{IM19}. 

We are going to apply the same recognition principle to multi-linear functors and natural transformations thereof. Before stating the precise result, let us spell out the relevant notions:
\begin{defn}
	Let $\Dcal_1, \dots, \Dcal_n$ and $\Ecal$ be triangulated categories.
	\begin{enumerate}
		\item A \textit{multi-triangulated functor} from $\Dcal_1 \times \dots \times \Dcal_n$ to $\Ecal$ is a functor
		\begin{equation*}
			F: \Dcal_1 \times \dots \times \Dcal_n \rightarrow \Ecal
		\end{equation*}
	    which is functorially triangulated separately in each variable in such a way that, for every choice of $i,j \in \left\{1, \dots, n\right\}$ with $i < j$, the diagram of functors $\Dcal_1 \times \dots \times \Dcal_n \rightarrow \Ecal$
	    \begin{equation*}
	    	\begin{tikzcd}
	    		F(A_1^{\bullet}, \dots, A_i^{\bullet}[1], \dots, A_j^{\bullet}[1], \dots, A_n^{\bullet}) \arrow{r}{\sim} \isoarrow{d} & F(A_1^{\bullet}, \dots, A_i^{\bullet}, \dots, A_j^{\bullet}[1], \dots, A_n^{\bullet})[1] \isoarrow{d} \\
	    		F(A_1^{\bullet}, \dots, A_i^{\bullet}[1], \dots, A_j^{\bullet}, \dots, A_n^{\bullet})[1] \arrow{r}{\sim} & F(A_1^{\bullet}, \dots, A_i^{\bullet}, \dots, A_j^{\bullet}, \dots, A_n^{\bullet})[2]
	    	\end{tikzcd}
	    \end{equation*}
	    is anti-commutative.
	    \item A \textit{multi-triangulated natural transformation} between $F$ and $G$ is a natural transformation of functors $\Dcal_1 \times \dots \times \Dcal_n \rightarrow \Ecal$
	    \begin{equation*}
	    	\alpha: F \rightarrow G
	    \end{equation*} 
        which is triangulated separately in each variable in the sense that, for every choice of $i \in \left\{1,\dots,n\right\}$, the diagram of functors $\Dcal_1 \times \dots \times \Dcal_n \rightarrow \Ecal$ 
        \begin{equation*}
        	\begin{tikzcd}
        		F(A_1^{\bullet}, \dots, A_i^{\bullet}[1], \dots, A_n^{\bullet}) \arrow{r}{\alpha} \isoarrow{d} & G(A_1^{\bullet}, \dots, A_i^{\bullet}[1], \dots, A_n^{\bullet}) \isoarrow{d} \\
        		F(A_1^{\bullet}, \dots, A_i^{\bullet}, \dots, A_n^{\bullet})[1] \arrow{r}{\alpha} & G(A_1^{\bullet}, \dots, A_i^{\bullet}, \dots, A_n^{\bullet})[1]
        	\end{tikzcd}
        \end{equation*}
        is commutative.
	\end{enumerate}
\end{defn}  

In order to motivate the definition, we quickly review the fundamental example: namely, derived functors of multi-exact functors. For this, we need some ad hoc notation about multi-dimensional complexes:

\begin{nota}
	Let $\Acal$ be an abelian category. For every integer $n \geq 1$, we use the following notation:
	\begin{itemize}
		\item We regard $\Z^n$ as a poset with the order relation defined by
		\begin{equation*}
			(i_1, \dots,i_n) \leq (i'_1, \dots, i'_n) \iff i_r \leq i'_r \;\; \forall r = 1, \dots, n.
		\end{equation*}
		\item We let $\square^b_n(\Acal)$ denote the category of $n$-dimensional bounded complexes with values in $\Acal$: it is the full subcategory of the functor category $\Fun(\Z^n,\Acal)$ consisting of those functors $A^{\square} = (A^{i_1,\dots,i_n})_{i_1,\dots,i_n}$ such that:
		\begin{enumerate}
			\item[(i)] For every $(i_1, \dots, i_n) \in \Z^n$ and every $r = 1, \dots, n$, the composite morphism
			\begin{equation*}
				A^{i_1,\dots, i_r,\dots, i_n} \rightarrow A^{i_1,\dots, i_r + 1,\dots, i_n} \rightarrow A^{i_1,\dots, i_r + 2,\dots, i_n}
			\end{equation*}
			vanishes.
			\item[(ii)] There exists $N \geq 0$ (depending on $A^{\bullet}$) such that $A^{i_1,\dots, i_n} = 0$ whenever $|i_1| + \dots + |i_n| > N$.
		\end{enumerate}
		\item We let
		\begin{equation*}
			\Tot^{\square}: \square^b_n(\Acal) \rightarrow \square^b_1(\Acal) = C^b(\Acal)
		\end{equation*}
		denote the total complex functor with respect to a fixed choice of signs.
	\end{itemize} 
\end{nota}

\begin{lem}\label{lem:multiex-funct}
	Let $\Acal_1, \dots, \Acal_n$ and $\Bcal$ be abelian categories, and let $F: \Acal_1 \times \dots \times \Acal_n \rightarrow \Bcal$ be a multi-exact functor. Then the formula
	\begin{equation*}
		F(A_1^{\bullet}, \dots, A_n^{\bullet}) := \Tot^{\square}(F(A_n^{i_1}, \dots, A_n^{i_n})_{i_1,\dots,i_n})
	\end{equation*}
	canonically defines a multi-triangulated functor
	\begin{equation*}
		F: D^b(\Acal_1) \times \dots \times D^b(\Acal_n) \rightarrow D^b(\Bcal).
	\end{equation*}
\end{lem}
\begin{proof}
	If $n = 2$, the classical computation in the case where $F$ is the tensor product functor in the category of vector spaces over some field carries over to the general framework described in the statement. The same computation also works for $n > 2$. We omit the details. 
\end{proof}

\begin{lem}\label{lem:multiex-nat}
	Let $\Acal_1, \dots, \Acal_n$ and $\Bcal$ be abelian categories, and let $F, G: \Acal_1 \times \dots \times \Acal_n \rightarrow \Bcal$ be two multi-exact functors. Suppose that we are given a natural transformation of functors $\Acal_1 \times \dots \Acal_n \rightarrow \Bcal$
	\begin{equation*}
		\alpha: F(A_1, \dots, A_n) \rightarrow G(A_1, \dots, A_n).
	\end{equation*}
	Then the formula
	\begin{equation*}
		\begin{tikzcd}[font=\small]
			\alpha(A_1^{\bullet}, \dots,  A_n^{\bullet}): & F(A_1^{\bullet}, \dots,  A_n^{\bullet}) \arrow[equal]{d} &&&& G(A_1^{\bullet}, \dots, A_n^{\bullet}) \arrow[equal]{d} \\ 
			& \Tot^{\square}(F(A_1^{i_1}, \dots, A_n^{i_n})_{i_1,\dots,i_n}) \arrow{rrrr}{(\kappa(A_1^{i_1}, \dots, A_n^{i_n})_{i_1,\dots,i_n})} &&&& \Tot^{\square}(G(A_1^{i_1}, \dots, A_n^{i_n})_{i_1,\dots,i_n})
		\end{tikzcd}
	\end{equation*}
	defines a multi-triangulated natural transformation between functors $D^b(\Acal_1) \times \dots \times D^b(\Acal_n) \rightarrow D^b(\Bcal)$
	\begin{equation*}
		\alpha: F(A_1^{\bullet}, \dots, A_n^{\bullet}) \rightarrow G(A_1^{\bullet}, \dots, A_n^{\bullet})
	\end{equation*}
	Moreover, this construction is compatible with identities of functors as well as with composition of natural transformations; in particular, it preserves invertibility of natural transformations.
\end{lem}
\begin{proof}
	The standard computation in the case where $F$ and $G$ are suitable combinations of tensor product functors in the category of vector spaces over some field carries over to the general setting described in the statement. We omit the details.
\end{proof}

We can then state the aforementioned generalization of \cite[Thm.~1]{VoloDG} to the multi-linear setting:
\begin{prop}\label{prop:VoloDG}
	Let $\Acal_1, \dots, \Acal_n$ and $\Bcal$ be abelian categories. The following statements hold:
	\begin{enumerate}
		\item Let $F: D^b(\Acal_1) \times \dots D^b(\Acal_n) \rightarrow D^b(\Bcal)$ be a multi-triangulated functor which is $t$-exact in each variable for the obvious $t$-structures, and suppose that it can be enhanced to a dg-quasifunctor. 
		
		Then $F$ is canonically isomorphic to the functor obtained from the induced multi-exact functor
		\begin{equation*}
			F: \Acal_1 \times \dots \times \Acal_n \rightarrow \Bcal
		\end{equation*} 
	    via the construction of \Cref{lem:multiex-funct}.
		\item Let $F,G: D^b(\Acal_1) \times \dots \times D^b(\Acal_n) \rightarrow D^b(\Bcal)$ be two multi-triangulated functors which are $t$-exact in each variable for the obvious $t$-structures, and suppose that they can be enhanced to dg-quasifunctors; fix one such enhancement. Moreover, let $\alpha: F \rightarrow G$ be a multi-triangulated natural transformation, and suppose that it can be enhanced to a morphism of dg-quasifunctors.
		
		Then $\alpha$ coincides with the canonical natural transformation obtained from the induced natural transformation of functors $\Acal_1 \times \dots \times \Acal_n \rightarrow \Bcal$
		\begin{equation*}
			\alpha: F \rightarrow G
		\end{equation*}
	    via the construction of \Cref{lem:multiex-nat}.
	\end{enumerate}
\end{prop}
\begin{proof}
	For sake of simplicity, we only sketch the argument in the case $n = 2$. 
	\begin{enumerate}
		\item For every fixed object $A_1 \in \Acal_1$, the induced triangulated functor
		\begin{equation*}
			F(A_1,-): D^b(\Acal_2) \rightarrow D^b(\Bcal)
		\end{equation*}
	    it $t$-exact by hypothesis and inherits a dg-enhancement from $F$. Applying \cite[Thm.~1]{VoloDG} to it, we deduce that it coincides with the trivial derived functor of the induced exact functor 
	    \begin{equation*}
	    	F(A_1,-): \Acal_2 \rightarrow \Bcal.
	    \end{equation*}
        Now, arguing by dévissage in the first variable, one obtains the analogous conclusion for $F$.
        \item For every fixed object $A_1 \in \Acal_1$, the induced natural transformation of triangulated functors $\Acal_2 \rightarrow \Bcal$
        \begin{equation*}
        	\alpha(A_1,-): F(A_1,-) \rightarrow G(A_1,-)
        \end{equation*}
        is triangulated by hypothesis and inherits a dg-enhancement from $\alpha$. Applying \cite[Thm.~1]{VoloDG} to it, we deduce that it coincides with the canonical natural transformation induced by the natural transformation between functors $\Acal_2 \rightarrow \Bcal$
        \begin{equation*}
        	\alpha(A_1,-): F(A_1,-) \rightarrow G(A_1,-).
        \end{equation*}
        Again, arguing by dévissage in the first variable, one obtains the analogous conclusion for $\alpha$.
	\end{enumerate}
\end{proof}

\begin{rem}\label{rem:utility-VoloDG}
	In particular, since the construction of \Cref{lem:multiex-nat} is compatible with composition of natural transformations, the result of \Cref{prop:VoloDG}(2) can be used to show the commutativity of natural diagrams of $t$-exact multi-triangulated functors by checking the commutativity of the induced diagram of multi-exact functors on the hearts of the various $t$-structures. 
\end{rem}

In the setting of perverse motives, \Cref{prop:VoloDG}(1) asserts that, given $k$-varieties $S_1, \dots, S_n, T$, every dg-enhanced multi-triangulated functor $F: D^b_c(S_1,\Q) \times \dots \times D^b_c(S_n,\Q) \rightarrow D^b_c(T,\Q)$ which is $t$-exact for the perverse $t$-structures separately in each variable is canonically isomorphic to the multi-derived functor of its restriction
\begin{equation*}
	F: \Perv(S_1) \times \dots \times \Perv(S_n) \rightarrow \Perv(T),
\end{equation*}
under Beilinson's equivalences. \Cref{prop:VoloDG}(2) asserts that the same holds for natural transformations between $t$-exact multi-triangulated functors. 
All the multi-triangulated functors on the constructible derived categories that we need to consider here are combinations of three types of functors: shifted inverse images under smooth morphisms, direct images under closed immersions, and external tensor product functors.


\subsection{External tensor structure}

After the preliminary discussion above, we can start constructing the monoidal structure on perverse motives. As explained in the introduction of the paper, the lack of $t$-exactness of the internal tensor product forces one to work systematically with the external tensor product. The legitimacy of this approach is  justified by our paper \cite{Ter23Fib}, where we show how the theory of monoidal fibered categories can be expressed in terms of the external tensor product: there is a canonical dictionary relating the structure and properties of the internal tensor product to the structure and properties of the external tensor product. In the rest of this section, we adopt the language of \textit{external tensor structures} introduced in \cite{Ter23Fib}, which allows us to express our main constructions concisely.

In fact, all the results below are proved for the abelian categories $\Mcal^0(S)$ defined in \Cref{nota:Mp0}; the corresponding results for the categories $\Mcal(S)$ are deduced by transport of structure via the equivalences of \Cref{thm_Mp0=Mp}. This operation is harmless because the categories $\Mcal^0(S)$ are stable under all the exact functors listed in the statement of \Cref{thm-Beilinson-simpler}, and the structure of stable homotopy $2$-functor of $D^b(\Mcal(\cdot))$ is constructed using only these functors (together with the Verdier duality functors, which also preserve the subcategories $\Mcal^0(S)$).
We prefer to use the notation "$\Mcal^0$" in place of "$\Mcal$" in those constructions and proofs where it is necessary to work with the alternative presentation of perverse motives.
However, as in the case of \Cref{thm-boxtimesM} above, we stick to the notation "$\Mcal$" in the formulation of our main results.

Our first task is to construct the single external tensor product functors, together with the external monoidality isomorphisms witnessing their compatibility with inverse image functors of perverse motives. 
This is based on the lifting results for external tensor structures on abelian fibered categories collected in \cite[\S~6]{Ter23UAF}. Since general inverse image functors are definitely not exact, we apply the method of \cite{Ter23Fact} in order to encode everything in terms of inverse images under smooth morphisms and direct images under closed immersions.
Our first intermediate result towards \Cref{thm-boxtimesM} can be stated as follows:

\begin{prop}\label{prop-extcore-Mp}
	The triangulated $\Var_k$-fibered category $D^b(\Mcal(\cdot))$ carries a canonical triangulated external tensor structure in the sense of \cite[Defn.~2.6]{Ter23Fib}. Moreover, the morphism of $\Var_k$-fibered categories $\iota: D^b(\Mcal(\cdot)) \rightarrow D^b(\Perv(\cdot))$ carries a canonical external tensor structure in the sense of \cite[Defn.~8.5]{Ter23Fib}.
\end{prop}

The theory of triangulated motives provides us with external tensor structures on the $\Var_k$-fibered categories $\DA_{ct}^{\et}(\cdot,\Q)$ and $D^b_c(\cdot,\Q)$, related by an external tensor structure on the Betti realization
\begin{equation*}
	\Bti^*: \DA_{ct}^{\et}(\cdot,\Q) \rightarrow D^b_c(\cdot,\Q).
\end{equation*}
Under the dictionary of \cite{Ter23Fib}, these correspond to the usual monoidal structures constructed as part of the six functor formalism in \cite{Ayo07a,Ayo07b} and \cite{Ayo10}.

\begin{nota}
	From now on, we adopt the following notation:
	\begin{itemize}
		\item Let $f: T \rightarrow S$ be a morphism of $k$-varieties.
		\begin{itemize}
			\item We let $\theta = \theta_f: f^* \circ \Bti_S^* \xrightarrow{\sim} \Bti_T \circ f^*$ denote the $\Bti^*$-transition isomorphism along $f$ constructed in \cite{Ayo10}.
			\item We let $\tilde{\theta} = \tilde{\theta}_f: f^* \circ \iota_S \xrightarrow{\sim} \iota_T \circ f^*$ denote the $\iota$-transition isomorphism along $f$ constructed in \cite[\S~2.1]{IM19}.
		\end{itemize}
		\item Let $z: Z \hookrightarrow S$ be a closed immersion of $k$-varieties.
		\begin{itemize}
			\item We let $\bar{\theta} = \bar{\theta}_z: \Bti_S^* \circ z_* \xrightarrow{\sim} z_* \circ \Bti_Z^*$ denote the natural isomorphism obtained as in \cite[\S~1]{Ter23Fact} from $\theta_z$.
			\item We let $\bar{\tilde{\theta}} = \bar{\tilde{\theta}}_z: \iota_S \circ z_* \xrightarrow{\sim} z_* \circ \iota_Z$ denote the natural isomorphism obtained as in \cite[\S~1]{Ter23Fact} from $\tilde{\theta}_z$.
		\end{itemize}
		\item We use the symbol $(\boxtimes,m;a,c)$ for the symmetric associative triangulated external tensor structures on $\DA_{ct}^{\et}(\cdot,\Q)$ and $D^b_c(\cdot,\Q)$.
		\item We let $\rho$ denote the triangulated external tensor structure on the morphism of $\Var_k$-fibered categories $\Bti^*: \DA_{ct}^{\et}(\cdot,\Q) \rightarrow D^b_c(\cdot,\Q)$.
	\end{itemize}
\end{nota}

In the proof of \Cref{prop-extcore-Mp}, we apply the abstract factorization method described in \cite{Ter23Fact} to fibered categories over $\Var_k$, with respect to the subcategories $\Var_k^{sm}$ and $\Var_k^{cl}$ collecting the smooth morphisms and the closed immersions, respectively (see \cite[Ex.~2.2]{Ter23Fact}). 

In order to simplify the notation in some parts of the proofs, it is convenient to adopt the following convention:
\begin{nota}\label{nota-diamond1}
	Let $\Hbb$ be a triangulated $\Var_k$-fibered category which is localic in the sense of \cite[Defn.~3.4]{Ter23Fact}.
	\begin{itemize}
		\item Let $f: T \rightarrow S$ be a morphism of $k$-varieties such that the number $e := \dim(\Ocal_{T,t}) - \dim(\Ocal_{S,f(t)})$ is independent of $t \in T$.  We define the shifted inverse image functor
		\begin{equation*}
			f^{\dagger} := f^*[e]: \Hbb(S) \rightarrow \Hbb(T).
		\end{equation*} 
		The two cases of interest for us are the following:
		\begin{enumerate}
			\item[(i)] when $f$ is a smooth morphism of constant relative dimension $e$;
			\item[(ii)] when $f$ is a regular closed immersion of constant codimension $e$.
		\end{enumerate}
	    \item More generally, if $f: T \rightarrow S$ is a morphism whose restriction to each connected component of $S$ has the property described in the previous point, we define the shifted inverse image functor $f^{\dagger}$ by extending the previous definition in the obvious way.
	\end{itemize}
	This notation is compatible with composition of morphisms of $k$-varieties as well as with morphisms of $\Var_k$-fibered categories.
\end{nota}

\begin{proof}[Proof of \Cref{prop-extcore-Mp}]
	To begin with, fix two $k$-varieties $S_1$ and $S_2$, and consider the diagram
	\begin{equation*}
		\begin{tikzcd}
			\Dcal^0(S_1) \times \Dcal^0(S_2) \arrow{rr}{-\boxtimes-} \arrow{d}{\beta_{S_1}^0 \times \beta_{S_2}^0} && \Dcal^0(S_1 \times S_2) \arrow{d}{\beta_{S_1 \times S_2}^0} \\
			\Perv(S_1) \times \Perv(S_2) \arrow{rr}{-\boxtimes-} && \Perv(S_1 \times S_2).
		\end{tikzcd}
	\end{equation*}
    The natural isomorphism of functors $\DA_{ct}^{\et}(S_1,\Q) \times \DA_{ct}^{\et}(S_2,\Q) \rightarrow D^b_c(S_1 \times S_2,\Q)$
    \begin{equation*}
	    \Bti_{S_1}^*(A_1) \boxtimes \Bti_{S_2}^*(A_2) \xrightarrow{\sim} \Bti^*_{S_1 \times S_2}(A_1 \boxtimes A_2)
    \end{equation*}
    restricts to a natural isomorphism between the two composite functors  $\Dcal^0(S_1) \times \Dcal^0(S_2) \rightarrow \Perv(S_1 \times S_2)$. Applying \cite[Prop.~3.3]{Ter23UAF}, we obtain a canonical bi-exact functor
    \begin{equation}\label{boxtimes-Mcal_p}
    	- \boxtimes - = - \boxtimes_{S_1,S_2} -: \Mcal^0(S_1) \times \Mcal^0(S_2) \rightarrow \Mcal^0(S_1 \times S_2)
    \end{equation}
    together with a canonical natural isomorphism of functors $\Mcal^0(S_1) \times \Mcal^0(S_2) \rightarrow \Perv(S_1 \times S_2)$
    \begin{equation}\label{rho-Mcal_p}
    	\tilde{\rho} = \tilde{\rho}_{S_1,S_2}: \iota_{S_1}(M_1) \boxtimes \iota_{S_2}(M_2) \xrightarrow{\sim} \iota_{S_1 \times S_2}(M_1 \boxtimes M_2).
    \end{equation}
    By \Cref{lem:multiex-funct} and \Cref{lem:multiex-nat}, these extend to a bi-triangulated functor
	\begin{equation}\label{boxtimes-D^bMp}
		- \boxtimes -: D^b(\Mcal^0(S_1)) \times D^b(\Mcal^0(S_2)) \rightarrow D^b(\Mcal^0(S_1 \times S_2))
	\end{equation}
	and a bi-triangulated natural isomorphism of functors $D^b(\Mcal^0(S_1)) \times D^b(\Mcal^0(S_2)) \rightarrow D^b(\Perv(S_1 \times S_2))$
	\begin{equation}\label{rho-D^bMp}
		\tilde{\rho} = \tilde{\rho}_{S_1,S_2}: \iota_{S_1}(M_1^{\bullet}) \boxtimes \iota_{S_2}(M_2^{\bullet}) \xrightarrow{\sim} \iota_{S_1 \times S_2}(M_1^{\bullet} \boxtimes M_2^{\bullet}).
	\end{equation}
    In order to turn the collection of functors \eqref{boxtimes-D^bMp} into an external tensor structure on $D^b(\Mcal(\cdot))$, we have to construct the external monoidality isomorphisms with respect to inverse image functors in $\Var_k$, compatibly with composition. Once this is done, in order to check that the collection of natural transformations \eqref{rho-D^bMp} defines an external tensor structure on the morphism $\iota: D^b(\Mcal(\cdot)) \rightarrow D^b(\Perv(\cdot))$, we have to show that they are compatible with the external monoidality isomorphisms. If this holds, then the conservativity of $\iota$ implies the validity of the projection formulae on $D^b(\Mcal(\cdot))$. 
    
    As a consequence of the discussion in \cite[\S\S~3, 5]{Ter23Fib}, we see that the sought-after external tensor structures on $D^b(\Mcal(\cdot))$ and on $\iota$ are completely determined by the underlying external tensor cores, as defined in \cite[Defn.~5.2, Defn.~5.6]{Ter23Fib}. To construct these, we apply the results of \cite[\S~6]{Ter23UAF} to the underlying $\Var_k^{sm}$-fibered and $\Var_k^{cl,op}$-fibered categories.
    
    On the one hand, regard the Betti realization as a morphism of $\Var_k^{sm}$-fibered categories with respect to the shifted inverse image functors $p^{\dagger}$ introduced in \Cref{nota-diamond1}. Using the additivity of relative dimensions with respect to direct products of morphisms, we see that the external tensor structures on $\DA_{ct}^{\et}(\cdot,\Q)$, $D^b_c(\cdot,\Q)$ and $\Bti^*$ induce external tensor structures on the $\Var_k^{sm}$-fibered categories considered. By restriction, these in turn induce analogous external tensor structures on the $\Var_k^{sm}$-fibered categories $\Dcal^0(\cdot)$ and $\Perv(\cdot)$ as well as on the morphism of $\Var_k^{sm}$-fibered categories $\beta^0: \Dcal^0(\cdot) \rightarrow \Perv(\cdot)$. Applying \cite[Prop.~6.4]{Ter23UAF}, we obtain external tensor structures on the abelian $\Var_k^{sm}$-fibered category $\Mcal^0(\cdot)$ as well as on the morphism of abelian $\Var_k^{sm}$-fibered categories $\iota: \Mcal^0(\cdot) \rightarrow \Perv(\cdot)$. Finally, applying \Cref{lem:multiex-nat} and shifting the inverse image functors back, the latter define triangulated external tensor structures on the triangulated $\Var_k^{sm}$-fibered category $D^b(\Mcal(\cdot))$ as well as on the triangulated morphism $\iota: D^b(\Mcal(\cdot)) \rightarrow D^b(\Perv(\cdot))$.
    
    On the other hand, regard the Betti realization as a morphism of $\Var_k^{cl,op}$-fibered categories with respect to the direct image functors $z_*$. Using \cite[Lemma~1.9, Lemma~1.13]{Ter23Fact}, we see that the external tensor structures on $\DA_{ct}^{\et}(\cdot,\Q)$, $D^b_c(\cdot,\Q)$ and $\Bti^*$ induce external tensor structures on the $\Var_k^{cl,op}$-fibered categories considered. By restriction, these in turn induce analogous external tensor structures on the $\Var_k^{cl,op}$-fibered categories $\Dcal^0(\cdot)$ and $\Perv(\cdot)$ as well as on the morphism of $\Var_k^{cl,op}$-fibered categories $\beta^0: \Dcal^0(\cdot) \rightarrow \Perv(\cdot)$. Applying \cite[Prop.~6.4]{Ter23UAF}, we obtain external tensor structures on the abelian $\Var_k^{cl,op}$-fibered category $\Mcal^0(\cdot)$ as well as on the morphism of abelian $\Var_k^{cl,op}$-fibered categories $\iota: \Mcal^0(\cdot) \rightarrow \Perv(\cdot)$. Finally, applying \Cref{lem:multiex-nat} and using again \cite[Lemma~1.9, Lemma~1.13]{Ter23Fact}, the latter define triangulated external tensor structures on the triangulated $\Var_k^{cl,op}$-fibered category $D^b(\Mcal(\cdot))$ as well as on the triangulated morphism $\iota: D^b(\Mcal(\cdot)) \rightarrow D^b(\Perv(\cdot))$.
    
    We claim that these two constructions satisfy the compatibility condition of \cite[Defn.~5.2]{Ter23Fact} and thus define an external tensor core on the $\Var_k$-fibered category $D^b(\Mcal(\cdot))$. 
    Since the validity of the conditions of \cite[Defn.~8.5]{Ter23Fact} has been already justified above, this will imply that the same constructions also define an external tensor core on $\iota$. 
    
    Let us then check the two conditions (C'-ETC) and (T-ETC) of \cite[Defn.~5.2]{Ter23Fact} are satisfied:
    \begin{enumerate}
    	\item[(C'-ETC)] Given two Cartesian squares of $k$-varieties
    	\begin{equation*}
    		\begin{tikzcd}
    			P_{i,Z_i} \arrow{r}{z'_i} \arrow{d}{p'_i} & P_i \arrow{d}{p_i} \\
    			Z_i \arrow{r}{z_i} & S_i
    		\end{tikzcd} 
    		\qquad (i = 1,2)
    	\end{equation*}
    	where $p_i$ (and hence $p'_i$) is smooth while $z_i$ (and hence $z'_r$) is a closed immersion, we have to show that the diagram of functors $D^b(\Mcal^0(Z_1)) \times D^b(\Mcal^0(Z_2)) \rightarrow D^b(\Mcal^0(P_1~\times~P_2))$
    	\begin{equation*}
    		\begin{tikzcd}[font=\small]
    			p_1^* z_{1,*} M_1^{\bullet} \boxtimes p_2^* z_{2,*} M_2^{\bullet} \isoarrow{d} \arrow{r}{\tilde{m}} & (p_1 \times p_2)^* (z_{1,*} M_1^{\bullet} \boxtimes z_{2,*} M_2^{\bullet}) \arrow{r}{\bar{\tilde{m}}} & (p_1 \times p_2)^* (z_1 \times z_2)_* (M_1^{\bullet} \boxtimes M_2^{\bullet}) \isoarrow{d} \\
    			z'_{1,*} {p'}_1^* M_1^{\bullet} \boxtimes z'_{2,*} {p'}_2^* M_2^{\bullet} \arrow{r}{\bar{\tilde{m}}} & (z'_1 \times z'_2)_*(p_1^* M_1^{\bullet} \boxtimes p_2^* M_2^{\bullet}) \arrow{r}{\tilde{m}} & (z'_1 \times z'_2)_* (p_1 \times p_2)^* (M_1^{\bullet} \boxtimes M_2^{\bullet})
    		\end{tikzcd}
    	\end{equation*}
    	is commutative. After shifting the inverse image functors under smooth morphisms as in \Cref{nota-diamond1}, all the triangulated functors involved become $t$-exact for the obvious $t$-structures. Since all of them are obviously dg-enhanced, we can apply \Cref{lem:multiex-nat} (see \Cref{rem:utility-VoloDG}), so we reduce to showing that the diagram of exact functors $\Mcal^0(Z_1) \times \Mcal^0(Z_2) \rightarrow \Mcal^0(P_1 \times P_2)$
    	\begin{equation*}
    		\begin{tikzcd}[font=\small]
    			p_1^{\dagger} z_{1,*} M_1 \boxtimes p_2^{\dagger} z_{2,*} M_2 \isoarrow{d} \arrow{r}{\tilde{m}} & (p_1 \times p_2)^{\dagger} (z_{1,*} M_1 \boxtimes z_{2,*} M_2) \arrow{r}{\bar{\tilde{m}}} & (p_1 \times p_2)^{\dagger} (z_1 \times z_2)_* (M_1 \boxtimes M_2) \isoarrow{d} \\
    			z'_{1,*} {p'}_1^{\dagger} M_1 \boxtimes z'_{2,*} {p'}_2^{\dagger} M_2 \arrow{r}{\bar{\tilde{m}}} & (z'_1 \times z'_2)_*(p_1^{\dagger} M_1 \boxtimes p_2^{\dagger} M_2) \arrow{r}{\tilde{m}} & (z'_1 \times z'_2)_* (p_1 \times p_2)^{\dagger} (M_1 \boxtimes M_2)
    		\end{tikzcd}
    	\end{equation*}
    	commutes. Since all the functors and natural transformations in the latter diagram are obtained via the lifting principles for universal abelian factorizations, using \cite[Rem.~4.7(1)(2)]{Ter23UAF} we reduce to checking the commutativity of the analogous diagram of functors on $\DA_{ct}^{\et}(\cdot,\Q)$. After shifting back, the latter follows from the validity of axiom (C'-ETC) on $\DA_{ct}^{\et}(\cdot,\Q)$.
    	\item[(T-ETC)] Given two commutative triangles of $k$-varieties
    	\begin{equation*}
    		\begin{tikzcd}
    			Q_i \arrow{r}{h_i} \arrow{dr}{q_i} & P_i \arrow{d}{p_i} \\
    			& S_i
    		\end{tikzcd}
    		\qquad (i = 1,2)
    	\end{equation*}
    	where $p_i$ and $q_i$ are smooth while $h_i$ is a closed immersion (necessarily regular), we have to show that the diagram of functors $D^b(\Mcal^0(S_1)) \times D^b(\Mcal^0(S_2)) \rightarrow D^b(\Mcal^0(Q_1 \times Q_2))$
    	\begin{equation*}
    		\begin{tikzcd}[font=\small]
    			q_1^* M_1^{\bullet} \boxtimes q_2^* M_2^{\bullet} \arrow{rr}{\tilde{m}} \arrow[equal]{d} && (q_1 \times q_2)^*(M_1^{\bullet} \boxtimes M_2^{\bullet}) \arrow[equal]{d} \\
    			h_1^* p_1^* M_1^{\bullet} \boxtimes h_2^* p_2^* M_2^{\bullet} \arrow{r}{\tilde{m}} & (h_1 \times h_2)^* (p_1^* M_1^{\bullet} \boxtimes p_2^* M_2^{\bullet}) \arrow{r}{\tilde{m}} & (h_1 \times h_2)^* (p_1 \times p_2)^*(M_1^{\bullet} \boxtimes M_2^{\bullet})
    		\end{tikzcd}
    	\end{equation*}
    	is commutative. As before, we may assume that $p_1,p_2$ and $q_1,q_2$ are equidimensional. In this case, after shifting appropriately so that all the functors involved become $t$-exact for the obvious $t$-structures, we are reduced to showing that the diagram of exact functors
    	$\Mcal^0(S_1) \times \Mcal^0(S_2) \rightarrow \Mcal^0(Q_1 \times Q_2)$
    	\begin{equation*}
    		\begin{tikzcd}[font=\small]
    			q_1^{\dagger} M_1 \boxtimes q_2^{\dagger} M_2 \arrow{rr}{\tilde{m}} \arrow[equal]{d} && (q_1 \times q_2)^{\dagger}(M_1 \boxtimes M_2) \arrow[equal]{d} \\
    			h_1^{\dagger} p_1^{\dagger} M_1 \boxtimes h_2^{\dagger} p_2^{\dagger} M_2 \arrow{r}{\tilde{m}} & (h_1 \times h_2)^{\dagger} (p_1^{\dagger} M_1 \boxtimes p_2^{\dagger} M_2) \arrow{r}{\tilde{m}} & (h_1 \times h_2)^{\dagger} (p_1 \times p_2)^{\dagger}(M_1 \boxtimes M_2)
    		\end{tikzcd}
    	\end{equation*}
    	is commutative. Unfortunately, this cannot be deduced directly from general properties of universal abelian factorizations, since the inverse images under closed immersions on perverse motives constructed in \cite[\S~4.1]{IM19} are not defined via the lifting principles of universal abelian factorizations. In any case, in order to prove the above diagram is commutative, it suffices to show that its image under the faithful exact functor $\iota_{Q_1 \times Q_2}$
    	\begin{widepage}
    		\begin{equation*}
    			\begin{tikzcd}[font=\small]
    				\iota_{Q_1 \times Q_2}(q_1^{\dagger} M_1 \boxtimes q_2^{\dagger} M_2) \arrow{rr}{\tilde{m}} \arrow[equal]{d} && \iota_{Q_1 \times Q_2}((q_1 \times q_2)^{\dagger}(M_1 \boxtimes M_2)) \arrow[equal]{d} \\
    				\iota_{Q_1 \times Q_2}(h_1^{\dagger} p_1^{\dagger} M_1 \boxtimes h_2^{\dagger} p_2^{\dagger} M_2) \arrow{r}{\tilde{m}} & \iota_{Q_1 \times Q_2}((h_1 \times h_2)^{\dagger} (p_1^{\dagger} M_1 \boxtimes p_2^{\dagger} M_2)) \arrow{r}{\tilde{m}} & \iota_{Q_1 \times Q_2}((h_1 \times h_2)^{\dagger} (p_1 \times p_2)^{\dagger}(M_1 \boxtimes M_2))
    			\end{tikzcd}
    		\end{equation*}
    	\end{widepage}
    	is commutative, and this can be checked directly. Indeed, by inserting, on the left side of the latter diagram, the diagram
    	\begin{equation*}
    		\begin{tikzcd}[font=\small]
    			\iota_{Q_1 \times Q_2}(q_1^{\dagger} M_1 \boxtimes q_2^{\dagger} M_2) \arrow[equal]{dd} & \iota_{Q_1}(q_1^{\dagger} M_1) \boxtimes \iota_{Q_2}(q_2^{\dagger} M_2) \arrow[equal]{dd} \arrow{l}{\tilde{\rho}} & q_1^{\dagger} \iota_{S_1}(M_1) \boxtimes q_2^{\dagger} \iota_{S_2}(M_2) \arrow[equal]{d} \arrow{l}{\tilde{\theta}} \\
    			&& h_1^{\dagger} p_1^{\dagger} \iota_{S_1}(M_1) \boxtimes h_2^{\dagger} p_2^{\dagger} \iota_{S_2}(M_2)\arrow{d}{\tilde{\theta}} \\
    			\iota_{Q_1 \times Q_2}(h_1^{\dagger} p_1^{\dagger} M_1 \boxtimes h_2^{\dagger} p_2^{\dagger} M_2) & \iota_{Q_1}(h_1^{\dagger} p_1^{\dagger} M_1) \boxtimes \iota_{Q_2}(h_2^{\dagger} p_2^{\dagger} M_2) \arrow{l}{\tilde{\rho}} & h_1^{\dagger} \iota_{P_1}(p_1^{\dagger} M_1) \boxtimes h_2^{\dagger} \iota_{P_2}(p_2^{\dagger} M_2) \arrow{l}{\tilde{\theta}}
    		\end{tikzcd}
    	\end{equation*}
    	and, on its right side, the diagram
    	\begin{widepage}
    		\begin{equation*}
    			\begin{tikzcd}[font=\tiny]
    				(q_1 \times q_2)^{\dagger} (\iota_{S_1}(M_1) \boxtimes \iota_{S_2}(M_2)) \arrow[equal]{dd} \arrow{r}{\tilde{\rho}} & (q_1 \times q_2)^{\dagger} \iota_{S_1 \times S_2}(M_1 \boxtimes M_2) \arrow[equal]{dd} \arrow{r}{\tilde{\theta}} & \iota_{Q_1 \times Q_2}((q_1 \times q_2)^{\dagger} (M_1 \boxtimes M_2)) \arrow[equal]{d} \\
    				&& \iota_{Q_1 \times Q_2}((h_1 \times h_2)^{\dagger} (p_1 \times p_2)^{\dagger} (M_1 \boxtimes M_2)) \\
    				(h_1 \times h_2)^{\dagger} (p_1 \times p_2)^{\dagger} (\iota_{S_1}(M_1) \boxtimes \iota_{S_2}(M_2)) \arrow{r}{\tilde{\rho}} & (h_1 \times h_2)^{\dagger} (p_1 \times p_2)^{\dagger} \iota_{S_1 \times S_2}(M_1 \boxtimes M_2) \arrow{r}{\tilde{\theta}} & (h_1 \times h_2)^{\dagger} \iota_{P_1 \times P_2}((p_1 \times p_2)^{\dagger}(M_1 \boxtimes M_2)) \arrow{u}{\theta}
    			\end{tikzcd}
    		\end{equation*}
    	\end{widepage}
    	both of which are commutative by naturality and by axiom (mor-$\Var_k$-fib) of \cite[Defn.~8.5]{Ter23Fib}, we are reduced to considering the outer part of the diagram
    	\begin{equation*}
    		\begin{tikzcd}[font=\tiny]
    			\iota_{Q_1 \times Q_2}(q_1^{\dagger} M_1 \boxtimes q_2^{\dagger} M_2)  \arrow{rr}{\tilde{m}} && \iota_{Q_1 \times Q_2}((q_1 \times q_2)^{\dagger}(M_1 \boxtimes M_2)) \\
    			\iota_{Q_1}(q_1^{\dagger} M_1) \boxtimes \iota_{Q_2}(q_2^{\dagger} M_2) \arrow{u}{\tilde{\rho}} && (q_1 \times q_2)^{\dagger} \iota_{S_1 \times S_2}(M_1 \boxtimes M_2) \arrow{u}{\tilde{\theta}} \\
    			q_1^{\dagger} \iota_{S_1}(M_1) \boxtimes q_2^{\dagger} \iota_{S_2}(M_2) \arrow{rr}{\tilde{m}} \arrow{u}{\tilde{\theta}} \arrow[equal]{d} && (q_1 \times q_2)^{\dagger}(\iota_{S_1}(M_1) \boxtimes \iota_{S_2}(M_2)) \arrow{u}{\tilde{\rho}} \arrow[equal]{d} \\
    			h_1^{\dagger} p_1^{\dagger} \iota_{S_1}(M_1) \boxtimes h_2^{\dagger} p_2^{\dagger} \iota_{S_2}(M_2) \arrow{r}{\tilde{m}} \arrow{d}{\tilde{\theta}} & (h_1 \times h_2)^{\dagger} (p_1^{\dagger} \iota_{S_1}(M_1) \boxtimes p_2^{\dagger} \iota_{S_2}(M_2)) \arrow{r}{\tilde{m}} \arrow{d}{\tilde{\theta}} & (h_1 \times h_2)^{\dagger} (p_1 \times p_2)^{\dagger} (\iota_{S_1}(M_1) \boxtimes \iota_{S_2}(M_2)) \arrow{d}{\tilde{\rho}} \\
    			h_1^{\dagger} \iota_{P_1}(p_1^{\dagger} M_1) \boxtimes h_2^{\dagger} \iota_{P_2}(p_2^{\dagger} M_2) \arrow{r}{\tilde{m}} \arrow{d}{\tilde{\theta}} & (h_1 \times h_2)^{\dagger} (\iota_{P_1}(p_1^{\dagger} M_1) \boxtimes \iota_{P_2}(p_2^{\dagger} M_2)) \arrow{d}{\tilde{\rho}} & (h_1 \times h_2)^{\dagger} (p_1 \times p_2)^{\dagger} \iota_{S_1 \times S_2}(M_1 \boxtimes M_2) \arrow{d}{\tilde{\theta}} \\
    			\iota_{Q_1}(h_1^{\dagger} p_1^{\dagger} M_1) \boxtimes \iota_{Q_2}(h_2^{\dagger} p_2^{\dagger} M_2) \arrow{d}{\tilde{\rho}} & (h_1 \times h_2)^{\dagger} \iota_{P_1 \times P_2}(p_1^{\dagger} M_1 \boxtimes p_2^{\dagger} M_2) \arrow{d}{\tilde{\theta}} \arrow{r}{\tilde{m}} & (h_1 \times h_2)^{\dagger} \iota_{P_1 \times P_2}((p_1 \times p_2)^{\dagger}(M_1 \boxtimes M_2)) \arrow{d}{\tilde{\theta}} \\
    			\iota_{Q_1 \times Q_2}(h_1^{\dagger} p_1^{\dagger} M_1 \boxtimes h_2^{\dagger} p_2^{\dagger} M_2) \arrow{r}{\tilde{m}} & \iota_{Q_1 \times Q_2}((h_1 \times h_2)^{\dagger}(p_1^{\dagger} M_1 \boxtimes p_2^{\dagger} M_2)) \arrow{r}{\tilde{m}} & \iota_{Q_1 \times Q_2}((h_1 \times h_2)^{\dagger} (p_1 \times p_2)^{\dagger} (M_1 \boxtimes M_2)).
    		\end{tikzcd}
    	\end{equation*}
    	Here, the two four-term pieces are commutative by naturality, the three six-term pieces are commutative by axiom (mor-ETS) of \cite[Defn. 3.6]{Ter23Fib}, and the central five-term piece (after shifting back) is commutative by the validity of axiom (T-ETC) on $D^b(\Perv(\cdot))$. 
    \end{enumerate}
    This proves the claim, thereby concluding the proof.
\end{proof}

\subsection{External associativity and commutativity constraints}

The next step is to equip the external tensor core of \Cref{prop-extcore-Mp} with compatible external associativity and commutativity constraints(see \cite[\S\S~3--5]{Ter23Fib} and \cite[\S~4]{Ter23Fact}). 
For sake of clarity, we construct the two constraints separately, and in a second moment we check their compatibility.

\begin{lem}\label{lem-asso-Mp}
	The external tensor structure on $D^b(\Mcal(\cdot))$ carries a canonical external associativity constraint in the sense of \cite[Defn.~3.2]{Ter23Fib}, with respect to which the morphism of $\Var_k$-fibered categories $\iota: D^b(\Mcal(\cdot)) \rightarrow D^b(\Perv(\cdot))$ is associative in the sense of \cite[Defn.~8.13]{Ter23Fib}.
\end{lem}
\begin{proof}
	To begin with, fix three $k$-varieties $S_1$, $S_2$ and $S_3$, and consider the diagram
	\begin{equation*}
		\begin{tikzcd}[font=\small]
			\Dcal^0(S_1) \times \Dcal^0(S_2) \times \Dcal^0(S_3) \arrow{rr}{(-\boxtimes-) \times \id} \arrow{d}{\beta_{S_1}^0 \times \beta_{S_2}^0 \times \beta^0_{S_3}} && \Dcal^0(S_1 \times S_2) \times \Dcal^0(S_3) \arrow{rr}{-\boxtimes-} \arrow{d}{\beta_{S_1 \times S_2}^0 \times \beta_{S_3}^0} && \Dcal^0(S_1 \times S_2 \times S_3) \arrow{d}{\beta_{S_1 \times S_2 \times S_3}^0} \\
			\Perv(S_1) \times \Perv(S_2) \times \Perv(S_3) \arrow{rr}{(-\boxtimes-) \times \id} && \Perv(S_1 \times S_2) \times \Perv(S_3) \arrow{rr}{-\boxtimes-} && \Perv(S_1 \times S_2 \times S_3) \\
			&& \Downarrow \\
			\Dcal^0(S_1) \times \Dcal^0(S_2) \times \Dcal^0(S_3) \arrow{rr}{\id \times (-\boxtimes-)} \arrow{d}{\beta_{S_1}^0 \times \beta_{S_2}^0 \times \beta_{S_3}^0} && \Dcal^0(S_1) \times \Dcal^0(S_2 \times S_3) \arrow{rr}{-\boxtimes-} \arrow{d}{\beta_{S_1}^0 \times \beta_{S_2 \times S_3}^0} && \Dcal^0(S_1 \times S_2 \times S_3) \arrow{d}{\beta_{S_1 \times S_2 \times S_3}^0} \\
			\Perv(S_1) \times \Perv(S_2) \times \Perv(S_3) \arrow{rr}{\id \times (-\boxtimes-)} && \Perv(S_1) \times \Perv(S_2 \times S_3) \arrow{rr}{-\boxtimes-} && \Perv(S_1 \times S_2 \times S_3).
		\end{tikzcd}
	\end{equation*}
    The associativity isomorphism of functors $\Dcal^0(S_1) \times \Dcal^0(S_2) \times \Dcal^0(S_3) \rightarrow \Dcal^0(S_1 \times S_2 \times S_3)$
    \begin{equation*}
    	(A_1 \boxtimes A_2) \boxtimes A_3 \xrightarrow{\sim} A_1 \boxtimes (A_2 \boxtimes A_3)
    \end{equation*}
    and the associativity isomorphisms of functors $\Perv(S_1) \times \Perv(S_2) \times \Perv(S_3) \rightarrow \Perv(S_1 \times S_2 \times S_3)$
    \begin{equation*}
    	(K_1 \boxtimes K_2) \boxtimes K_3 \xrightarrow{\sim} K_1 \boxtimes (K_2 \boxtimes K_3)
    \end{equation*}
    satisfy the compatibility condition of \cite[Defn.~6.7(1)]{Ter23UAF}, because the same holds for the corresponding natural isomorphisms on $\DA_{ct}^{\et}(\cdot,\Q)$ and on $D^b(\cdot,\Q)$ - this is just a way of rephrasing the associativity of the Betti realization. Applying \cite[Prop.~3.6]{Ter23UAF}, we get a natural isomorphism between multi-exact functors $\Mcal^0(S_1) \times \Mcal^0(S_2) \times \Mcal^0(S_3) \rightarrow \Mcal^0(S_1 \times S_2 \times S_3)$
    \begin{equation}\label{asso:Mp}
    	(M_1 \boxtimes M_2) \boxtimes M_3 \xrightarrow{\sim} M_1 \boxtimes (M_2 \boxtimes M_3),
    \end{equation}
    compatible with the corresponding associativity isomorphism on $\Perv(\cdot)$. By \Cref{lem:multiex-nat}, this extends to a multi-triangulated natural isomorphism between functors $D^b(\Mcal^0(S_1)) \times D^b(\Mcal^0(S_2)) \times D^b(\Mcal^0(S_3)) \rightarrow D^b(\Mcal^0(S_1 \times S_2 \times S_3))$ 
    \begin{equation}\label{asso:DMp}
    	(M_1^{\bullet} \boxtimes M_2^{\bullet}) \boxtimes M_3^{\bullet} \xrightarrow{\sim} M_1^{\bullet} \boxtimes (M_2^{\bullet} \boxtimes M_3^{\bullet}),
    \end{equation}
    again compatible with the corresponding associativity isomorphism on $D^b(\Perv(\cdot))$. 
    
    As $S_1$, $S_2$ and $S_3$ vary, the natural isomorphisms \eqref{asso:DMp} satisfy condition ($a$ETS-1) of \cite[Defn.~3.2]{Ter23Fib}: to see this, it suffices to check that the natural isomorphisms \eqref{asso:Mp} satisfy the same condition, which in turn follows from the validity of axiom ($a$ETS-1) on $\DA_{ct}^{\et}(\cdot,\Q)$.
    
    By \cite[Lemma~5.5(1)]{Ter23Fact}, in order to conclude that the natural isomorphisms \eqref{asso:DMp} define an associativity constraint for the $\Var_k$-fibered category $D^b(\Mcal(\cdot))$, it suffices to show that they define associativity constraints for the underlying $\Var_k^{sm}$-fibered and $\Var_k^{cl,op}$-fibered categories. But this reduces to checking that the natural isomorphisms \eqref{asso:Mp} define associativity constraints for the abelian $\Var_k^{sm}$-fibered and $\Var_k^{cl,op}$-fibered category $\Mcal(\cdot)$, which is indeed the case by \cite[Lemma~6.8(1)]{Ter23UAF}. 
    
    Lastly, the associativity of $\iota$ amounts to the compatibility between the natural isomorphisms \eqref{asso:DMp} and the corresponding natural isomorphisms on $D^b(\Perv(\cdot))$ that we have already mentioned above.
\end{proof}

\begin{lem}\label{lem-comm-Mp}
	The external tensor structure on $D^b(\Mcal(\cdot))$ carries a canonical external commutativity constraint in the sense of \cite[Defn.~4.2]{Ter23Fib}, with respect to which the morphism of $\Var_k$-fibered categories $\iota: D^b(\Mcal(\cdot)) \rightarrow D^b(\Perv(\cdot))$ is symmetric in the sense of \cite[Defn.~8.17]{Ter23Fib}.
\end{lem}
\begin{proof}
	To begin with, fix two $k$-varieties $S_1$ and $S_2$, and consider the diagram
	\begin{equation*}
		\begin{tikzcd}[font=\small]
			\Dcal^0(S_1) \times \Dcal^0(S_2) \arrow{rrrr}{-\boxtimes-} \arrow{d}{\beta_{S_1}^0 \times \beta_{S_2}^0} &&&& \Dcal^0(S_1 \times S_2) \arrow{d}{\beta_{S_1 \times S_2}^0} \\
			\Perv(S_1) \times \Perv(S_2) \arrow{rrrr}{-\boxtimes-} &&&& \Perv(S_1 \times S_2) \\
			&& \Downarrow \\
			\Dcal^0(S_1) \times \Dcal^0(S_2) \arrow{r}{\leftrightarrow} \arrow{d}{\beta_{S_1}^0 \times \beta_{S_2}^0} & \Dcal^0(S_2) \times \Dcal^0(S_1) \arrow{rr}{-\boxtimes-} \arrow{d}{\beta_{S_2}^0 \times \beta_{S_1}^0} && \Dcal^0(S_2 \times S_1) \arrow{r}{\tau^*} \arrow{d}{\beta_{S_2 \times S_1}^0} & \Dcal^0(S_1 \times S_2) \arrow{d}{\beta_{S_1 \times S_2}^0} \\
			\Perv(S_1) \times \Perv(S_2) \arrow{r}{\leftrightarrow} & \Perv(S_2) \times \Perv(S_1) \arrow{rr}{-\boxtimes-} && \Perv(S_2 \times S_1) \arrow{r}{\tau^*} & \Perv(S_1 \times S_2).
		\end{tikzcd}
	\end{equation*}
    The commutativity isomorphisms of functors $\Dcal^0(S_1) \times \Dcal^0(S_2) \rightarrow \Dcal^0(S_2 \times S_1)$
    \begin{equation*}
    	A_1 \boxtimes A_2 \xrightarrow{\sim} \tau^* (A_2 \boxtimes A_1)
    \end{equation*}
    and the commutativity isomorphism of functors $\Perv(S_1) \times \Perv(S_2) \rightarrow\Perv(S_1 \times S_2)$
    \begin{equation*}
    	K_1 \boxtimes K_2 \xrightarrow{\sim} \tau^* (K_2 \boxtimes K_1)
    \end{equation*}
    satisfy the compatibility condition of \cite[Defn.~6.7(2)]{Ter23UAF}, because the same holds for the corresponding natural isomorphisms on $\DA_{ct}^{\et}(\cdot,\Q)$ and on $D^b(\cdot,\Q)$ - this is just a way of rephrasing the symmetry of the Betti realization. Therefore, applying \cite[Prop.~3.6]{Ter23UAF}, we get a natural isomorphism between multi-exact functors $\Mcal^0(S_1) \times \Mcal^0(S_2) \rightarrow \Mcal^0(S_1 \times S_2)$
    \begin{equation}\label{comm:Mp}
    	M_1 \boxtimes M_2 \xrightarrow{\sim} \tau^* (M_2 \boxtimes M_1),
    \end{equation}
    compatible with the corresponding commutativity isomorphism on $\Perv(\cdot)$. By \Cref{lem:multiex-nat}, this extends to a multi-triangulated natural isomorphism between functors $D^b(\Mcal^0(S_1)) \times D^b(\Mcal^0(S_2)) \rightarrow D^b(\Mcal^0(S_1 \times S_2))$ 
    \begin{equation}\label{comm:DMp}
    	M_1^{\bullet} \boxtimes M_2^{\bullet} \xrightarrow{\sim} \tau^* (M_2^{\bullet} \boxtimes M_1^{\bullet}),
    \end{equation}
    again compatible with the corresponding commutativity isomorphism on $D^b(\Perv(\cdot))$. 
    
    As $S_1$ and $S_2$ vary, the natural isomorphisms \eqref{comm:DMp} satisfy condition ($c$ETS-1) of \cite[Defn.~4.2]{Ter23Fib}: to see this it suffices to check that the natural isomorphisms \eqref{comm:Mp} satisfy the same condition, which in turn follows from the validity of axiom ($c$ETS-1) on $\DA_{ct}^{\et}(\cdot,\Q)$.
    
    By \cite[Lemma~5.5(2)]{Ter23Fact}, in order to conclude that the natural isomorphisms \eqref{comm:DMp} define a commutativity constraint for the $\Var_k$-fibered category $D^b(\Mcal(\cdot))$, it suffices to show that they define commutativity constraints for the underlying $\Var_k^{sm}$-fibered and $\Var_k^{cl,op}$-fibered categories. By construction, the latter task reduces to checking that the natural isomorphisms \eqref{comm:Mp} define commutativity constraints for the abelian $\Var_k^{sm}$-fibered and $\Var_k^{cl,op}$-fibered category $\Mcal(\cdot)$, which is indeed the case by \cite[Lemma~6.8(2)]{Ter23UAF}. 
    
    Lastly, the symmetry of $\iota$ amounts to the compatibility between the natural isomorphisms \eqref{comm:DMp} and the corresponding natural isomorphisms on $D^b(\Perv(\cdot))$ that we have already mentioned above.
\end{proof}

\begin{lem}\label{lem-ac-Mp}
	The external associativity constraint of \Cref{lem-asso-Mp} and the external commutativity constraint of \Cref{lem-comm-Mp} are compatible in the sense of \cite[Defn.~6.3]{Ter23Fib}.
\end{lem}
\begin{proof}
	It suffices to show that the external associativity and commutativity constraints on the abelian $\Var_k^{sm}$-fibered and $\Var_k^{cl,op}$-fibered category $\Mcal(\cdot)$, defined by the natural isomorphisms \eqref{asso:Mp} and \eqref{comm:Mp}, are compatible. By \cite[Lemma~6.10]{Ter23UAF}, this follows from the compatibility between the external associativity and commutativity constraints of $\DA_{ct}^{\et}(\cdot,\Q)$.
\end{proof}

\section{Tensoring with distinguished motivic local systems}\label{sect_otimes-dist}

Before completing the construction of the tensor structure on perverse motives, we need to analyze the tensor product functors defined in the previous section more closely:
our aim is to get more information about the motivic internal tensor product.

Recall that, using the dictionary of \cite{Ter23Fib}, the symmetric associative external tensor structure on $D^b(\Mcal(\cdot))$ obtained in \Cref{sect_ETS} can be translated canonically into an internal tensor structure: 
explicitly, for every $k$-variety $S$, the internal tensor product on $D^b(\Mcal(S))$ is defined via the classical formula
\begin{equation}\label{M_1-otimes-M_2}
	M_1^{\bullet} \otimes M_2^{\bullet} := \Delta_S^*(M_1^{\bullet} \boxtimes M_2^{\bullet}),
\end{equation}
where $\Delta_S: S \hookrightarrow S \times S$ denotes the diagonal embedding. 
Hence we can regard $D^b(\Mcal(\cdot))$ as a symmetric monoidal $\Var_k$-fibered category.
The two parts of the promised monoidal structure that it remains to obtain are the unit constraint and the internal homomorphisms. For both of them, we want to use the lifting principles of universal abelian factorizations in relation to the internal tensor product functors.
The problem is that the internal tensor product is not itself defined from these lifting principles. 

The goal of this section is to provide a partial solution to this technical problem:
we give a more direct description of the internal tensor product functor \eqref{M_1-otimes-M_2} in the case where the $k$-variety $S$ is smooth and one of the two complexes involved has a certain distinguished form: 
this is \Cref{prop-distinguished-otimes}, the proof of which occupies most part of the present section. 

\subsection{Distinguished motivic local systems}

In order to formalize our idea, we need to introduce another abelian subcategory $\Mcal^{0,loc}(S)$ of $\Mcal^0(S)$, obtained as the universal abelian factorization of a further additive subquiver $\Dcal^{0,loc}(S)$ of $\Dcal^0(S)$. 
Recall that we let $\Loc(S)$ denote the usual $\Q$-linear abelian category of local systems over $S$.

\begin{nota}\label{nota:Mp0}
	Let $S$ be a smooth $k$-variety.
	\begin{itemize}
		\item We define a full abelian subcategory $\Loc_p(S)$ of $\Perv(S)$ by setting
		\begin{equation*}
			\Loc_p(S) = \left\{K \in \Perv(S) \; | \; K[-\dim(S)] \in \Loc(S) \right\}
		\end{equation*}
		if $S$ is connected, and extending the definition in the obvious way for general smooth $S$. 
		\item We introduce the full additive subcategory
		\begin{equation*}
			\Dcal^{0,loc}(S) := \left\{A \in \Dcal^0(S) \; | \; \Bti_S^*(A) \in \Loc_p(S) \right\} \subset \Dcal^0(S),
		\end{equation*}
		and we let 
		\begin{equation*}
			\beta_S^{0,loc}: \Dcal^{0,loc}(S) \rightarrow \Loc_p(S)
		\end{equation*}
		denote the restriction of $\beta_S^0$ to $\Dcal^{0,loc}(S)$.
		\item We let $\Mcal^{0,loc}(S)$ denote the universal abelian factorization of the additive functor 
		\begin{equation*}
			\beta_S^{0,loc} := \beta_S|_{\Dcal^{0,loc}(S)}: \Dcal^0(S) \rightarrow \Perv(S).
		\end{equation*}
		We implicitly regard it as a (not necessarily full) abelian subcategory of $\Mcal^0(S)$ via the canonical faithful exact functor
		\begin{equation*}
			\Mcal^{0,loc}(S) \hookrightarrow \Mcal^0(S)
		\end{equation*}
		obtained by applying \cite[Prop.~2.5]{Ter23UAF} to the diagram
		\begin{equation*}
			\begin{tikzcd}
				\Dcal^{0,loc}(S) \arrow[hook]{r} \arrow{d}{\beta_S^{0,loc}} & \Dcal^0(S) \arrow{d}{\beta_S^{0}} \\
				\Loc_p(S) \arrow[hook]{r} & \Perv(S).
			\end{tikzcd}
		\end{equation*}
		\item For simplicity, we still write
		\begin{itemize}
			\item $\pi_S$ for the quotient functor $\A(\Dcal^{0,loc}(S)) \rightarrow \Mcal^{0,loc}(S)$, and also for the composite functor $\Dcal^{0,loc}(S) \rightarrow \A(\Dcal^{0,loc}(S)) \rightarrow \Mcal^{0,loc}(S)$,
			\item $\iota_S$ for the induced faithful exact functor $\Mcal^{0,loc}(S) \hookrightarrow \Loc_p(S) \subset \Perv(S)$.
		\end{itemize}
		This convention will not create any ambiguity.
	\end{itemize}
\end{nota}

For every smooth morphism of constant relative dimension $p: P \rightarrow S$ between smooth $k$-varieties, the functor $p^{\dagger}: \Dcal^0(S) \rightarrow \Dcal^0(P)$ introduced in \Cref{nota-diamond1} sends $\Dcal^{0,loc}(S)$ to $\Dcal^{0,loc}(P)$, and therefore the corresponding exact functor $p^{\dagger}: \Mcal^0(S) \rightarrow \Mcal^0(P)$ sends $\Mcal^{0,loc}(S)$ to $\Mcal^{0,loc}(P)$.
For every $k$-variety $S$, we can specialize this observation to the cofiltered category $\Sm\Op(S)$ of its smooth Zariski-dense open subvarieties: as $U$ varies in $\Sm\Op(S)$, the categories $\Mcal^{0,loc}(U)$ canonically assemble into an abelian $\Sm\Op(S)$-fibered subcategory $\Mcal^{0,loc}(\cdot)$ of $\Mcal^0(\cdot)$. 
For applications to \Cref{sect_Homcal}, it is important to note the following fact:

\begin{lem}\label{lem:Mcal^0loc-wcof}
	For every $k$-variety $S$, the $\Sm\Op(S)$-fibered category $\Mcal^{0,loc}$ is cofinal in $\Mcal^0$ in the sense of \Cref{defn:w-local}.
\end{lem}
\begin{proof}
	By \cite[Cor.~5.15]{Ter23UAF}, it suffices to show that the $\Sm\Op(S)$-fibered subcategory $\Dcal^{0,loc}$ of $\Dcal^0$ is cofinal. 
	Since $\Dcal^{0,loc}$ is a full $\Sm\Op(S)$-fibered subcategory of $\Dcal^0$, it suffices to check the cofinality condition on objects. 
	To this end, using the definition of $\Dcal^{0,loc}$, it suffices to show that the inclusion of $\Sm\Op(S)$-fibered categories $\Loc_p(\cdot) \subset \Perv(\cdot)$ is cofinal, which is obvious.
\end{proof}

\begin{rem}
	The existence of a perverse motivic $t$-structure on $\DA_{ct}^{\et}(S,\Q)$ would imply the fullness of the inclusion $\Mcal^{0,loc}(S) \subset \Mcal(S)$. Unfortunately, we have no idea how to show this unconditionally; in any case, \Cref{lem:Mcal^0loc-wcof} suffices for our purposes.
\end{rem}

The main results of the present section are based on the following auxiliary construction: 

\begin{constr}\label{constr:otimes-diamond}
	Fix a smooth $k$-variety $S$. We define a shifted tensor product functor
	\begin{equation}\label{tildeot-DA}
		- \otimes^{\dagger} -: \DA_{ct}^{\et}(S,\Q) \times \DA_{ct}^{\et}(S,\Q) \rightarrow \DA_{ct}^{\et}(S,\Q)
	\end{equation}
	by the formula
	\begin{equation*}
		A_1 \otimes^{\dagger} A_2 := (A_1 \otimes A_2)[-\dim(S)]
	\end{equation*}
	in the case where $S$ is connected, and extending it in the obvious way for general $S$.
	By the same procedure, we define a shifted tensor product functor
	\begin{equation}\label{tildeot-Perv}
		- \otimes^{\dagger} -: D^b(\Perv(S)) \times D^b(\Perv(S)) \rightarrow D^b(\Perv(S)).
	\end{equation}
	Note that the functor \eqref{tildeot-Perv} sends $\Perv(S) \times \Loc_p(S)$ to $\Perv(S)$, and so the functor \eqref{tildeot-DA} sends $\Dcal^0(S) \times \Dcal^{0,loc}(S)$ to $\Dcal^0(S)$. Applying \cite[Prop.~4.3]{Ter23UAF} to the diagram 
	\begin{equation*}
		\begin{tikzcd}
			\Dcal^0(S) \times \Dcal^{0,loc}(S) \arrow{rr}{- \otimes^{\dagger} -} \arrow{d}{\beta_S^0 \times \beta_S^{0,loc}} && \Dcal^0(S) \arrow{d}{\beta_S^0} \\
			\Perv(S) \times \Loc_p(S) \arrow{rr}{- \otimes^{\dagger} -} && \Perv(S),
		\end{tikzcd}
	\end{equation*}
	we obtain a bi-exact functor
	\begin{equation}\label{tildeotimes-Mp}
		- \otimes^{\dagger} -: \Mcal^0(S) \times \Mcal^{0,loc}(S) \rightarrow \Mcal^0(S)
	\end{equation}
	together with a canonical natural isomorphism of functors $\Mcal^0(S) \times \Mcal^{0,loc}(S) \rightarrow \Perv(S)$
	\begin{equation*}
		\tilde{\rho}^{\dagger} = \tilde{\rho}^{\dagger}_S: \iota_S(M) \otimes^{\dagger} \iota_S(L) \xrightarrow{\sim} \iota_S(M \otimes^{\dagger} L).
	\end{equation*}
\end{constr}

A priori it is unclear whether the functor \eqref{tildeotimes-Mp} is related to (the restriction of) the internal tensor product on perverse motives, defined via the formula \eqref{M_1-otimes-M_2}. It is a pleasant surprise that the two functors are indeed related in the way one would expect:
\begin{prop}\label{prop-distinguished-otimes}
	Let $S$ be a connected $k$-variety of dimension $d$. 
	Then there exists a canonical isomorphism of functors $\Mcal^0(S) \times \Mcal^{0,loc}(S) \rightarrow \Mcal^0(S)$
	\begin{equation}\label{iso-otimes-distinguished}
		M \otimes^{\dagger} L \xrightarrow{\sim} (M \otimes L)[-d]
	\end{equation}
	making the diagram of functors $\Mcal^0(S) \times \Mcal^{0,loc}(S) \rightarrow \Perv(S)$
	\begin{equation}\label{rho-otimes-distinguished}
		\begin{tikzcd}
			\iota_S(M \otimes^{\dagger} L) \arrow{r}{\sim} & \iota_S((M \otimes L)[-d]) \arrow[equal]{r} & \iota_S(M \otimes L)[-d] \\
			\iota_S(M) \otimes^{\dagger} \iota_S(L) \arrow{u}{\tilde{\rho}^{\dagger}} \arrow[equal]{rr} && (\iota_S(M) \otimes \iota_S(L))[-d] \arrow{u}{\tilde{\rho}}
		\end{tikzcd}
	\end{equation}
	commute.
\end{prop}

Before discussing the details of the proof, let us explain what is the general strategy and why implementing it correctly is quite subtle.
Consider the diagram
\begin{equation*}
	\begin{tikzcd}
		\Dcal^0(S) \times \Dcal^{0,loc}(S) \arrow{rr}{- \boxtimes -} \arrow{d}{\beta_S^0 \times \beta_S^{0,loc}} && \Dcal^0(S \times S) \arrow{rr}{\Delta_S^*[-d]} \arrow{d}{\beta_{S \times S}^0} && \Dcal(S) \arrow{d}{\Bti_S^*} \\
		\Perv(S) \times \Loc_p(S) \arrow{rr}{- \boxtimes -} && \Perv(S \times S) \arrow{rr}{\Delta_S^*[-d]} && D^b(\Perv(S)),
	\end{tikzcd}
\end{equation*}
where both squares are commutative up to canonical natural isomorphism. By construction, the vertical and horizontal composites recover (the restrictions of) the shifted tensor product functors \eqref{tildeot-DA} and \eqref{tildeot-Perv}, respectively.
In particular, they take values in $\Dcal^0(S)$ and $\Perv(S)$, respectively. 
If the functor
\begin{equation}\label{Delta^*-Perv}
	\Delta_S^*[-d]: D^b(\Perv(S \times S)) \rightarrow D^b(\Perv(S))
\end{equation}  
was $t$-exact, the corresponding functor $\Delta_S^*[-d]: \Dcal(S \times S) \rightarrow \Dcal(S)$ would send $\Dcal^0(S \times S)$ to $\Dcal^0(S)$. Hence one could refine the previous diagram to
\begin{equation*}
	\begin{tikzcd}
		\Dcal^0(S) \times \Dcal^{0,loc}(S) \arrow{rr}{- \boxtimes -} \arrow{d}{\beta_S^0 \times \beta_S^{0,loc}} && \Dcal^0(S \times S) \arrow{rr}{\Delta_S^*[-d]} \arrow{d}{\beta_{S \times S}^0} && \Dcal^0(S) \arrow{d}{\beta^0_S} \\
		\Perv(S) \times \Loc_p(S) \arrow{rr}{- \boxtimes -} && \Perv(S \times S) \arrow{rr}{\Delta_S^*[-d]} && \Perv(S),
	\end{tikzcd}
\end{equation*} 
and one could try to deduce that the functor \eqref{tildeotimes-Mp} is naturally isomorphic to the composite
\begin{equation*}
	\Mcal^0(S) \times \Mcal^{0,loc}(S) \xrightarrow{- \boxtimes -} \Mcal^0(S \times S) \xrightarrow{\Delta_S^*[-d]} \Mcal^0(S) 
\end{equation*}
by using the general principles of \cite{Ter23UAF}: 
the conclusion of \Cref{prop-distinguished-otimes} would follow from the fact that lifting exact functors to universal abelian factorizations is compatible with composition. 

Needless to say, the functor \eqref{Delta^*-Perv} is not $t$-exact, hence the argument sketched here simply does not work. On the other hand, it is well-known that the composite functor 
\begin{equation*}
	\Delta_{S,*} \circ \Delta_S^*[-d]: D^b(\Perv(S \times S)) \rightarrow D^b(\Perv(S \times S))
\end{equation*}
can be modeled as an exact functor $C^b(\Perv(S \times S)) \rightarrow C^b(\Perv(S \times S))$ by means of a suitable \v{C}ech complex construction. For instance, for $S = \A^1_k$ one can use the total complex associated to the exact functor $C^b(\Perv(\A^2_k)) \rightarrow C^{[0,1]}(C^b(\Perv(\A^2_k)))$
\begin{equation}\label{cech:A1}
	K^{\bullet} \rightsquigarrow [j_! j^* K^{\bullet} \rightarrow K^{\bullet}],
\end{equation}  
where $j$ denotes the affine open immersion of the complement of the diagonal in $\A^2_k$. The key advantage is that the complex in \eqref{cech:A1} lifts canonically to the level of perverse motives.
In general, even if the model for $\Delta_{S,*} \circ \Delta_S^*[-d]$ depends on the choice of a finite affine covering of $S \times S \setminus \Delta_S(S)$, the derived functor is essentially independent of such a choice. This idea is exploited in the proof of \cite[Prop.~4.2]{IM19} to construct the inverse image functor $\Delta_S^*$ - or, more generally, inverse image functors along closed immersions - in the setting of perverse motives. We are going to prove \Cref{prop-distinguished-otimes} by a similar method.

\subsection{Some \v{C}ech complex computations}

For the reader's convenience, we start by recalling the \v{C}ech complex construction just mentioned in the setting of perverse sheaves and perverse motives.

\begin{constr}\label{constr:cech-perv}
	Let $s: S \hookrightarrow Y$ be a closed immersion of $k$-varieties, and let $u: U \hookrightarrow Y$ denote the complementary open immersion.
	Suppose that we are given a finite affine open covering $\Ucal = \left\{U_i \subset U \right\}_{i = 1}^{d}$ of $U$. For every subset $I \subset \left\{1,\dots,d\right\}$, set
	\begin{equation*}
		U_I := \bigcap_{i \in I} U_i
	\end{equation*} 
	and let $j_I: U_I \hookrightarrow Y$ denote the corresponding open immersion (with the agreement that $U_{\emptyset} = Y$ and $j_{\emptyset} = \id_Y$). Note that each $j_I$ is affine, because $Y$ is separated over $k$ by definition. For every choice of subsets $I, I' \subset \left\{1,\dots,d\right\}$ satisfying $I' \subset I$, there exists a canonical natural transformation of functors $\Perv(Y) \rightarrow \Perv(Y)$
	\begin{equation*}
		\alpha^{\Ucal}_{I,I'}: j_{I,!} j_I^* K \rightarrow j_{I',!} j_{I'}^* K
	\end{equation*}
	induced by the co-unit along the inclusion of $U_I$ into $U_{I'}$. For each $r = 0,\dots,d$, set
	\begin{equation*}
		U^{(r)}_{\Ucal} := \coprod_{\substack{I \subset \left\{1,\dots,d \right\} \\ \# I = r}} U_I
	\end{equation*}
	and let $J_r: U^{(r)} \rightarrow Y$ denote the canonical morphism. We introduce the \v{C}ech complex functor
	\begin{equation*}
		\check{C}^{\bullet}_{Y,\Ucal}: \Perv(Y) \rightarrow C^{[-d,0]}(\Perv(Y)), \quad K \rightsquigarrow [J_{d,!} J_d^* K \to \dots \to J_{1,!} J_1^* K \to K],
	\end{equation*}
	where, for each $r = 1,\dots,d$, the arrow
	\begin{equation*}
		J_{r,!} J_r^* K = \bigoplus_{\substack{I \subset \left\{1,\dots,d\right\} \\ \# I = r}} j_{I,!} j_I^* K \rightarrow \bigoplus_{\substack{I' \subset \left\{1,\dots,d\right\} \\ \# I' = r-1}} j_{I',!} j_{I'}^* K = J_{r-1,!} J_{r-1}^* K
	\end{equation*}  
	is induced by the usual alternating sum of the arrows $\alpha^{\Ucal}_{I,I'}$. We extend the \v{C}ech complex functor to an exact functor
	\begin{equation*}
		\check{C}_{Y,\Ucal}: C^b(\Perv(Y)) \xrightarrow{\check{C}_{Y,\Ucal}^{\bullet}} C^{[-d,0]}(C^b(\Perv(Y))) \xrightarrow{\Tot^{\bullet}} C^b(\Perv(Y))
	\end{equation*}
	and we consider its trivial derived functor
	\begin{equation}\label{cech-D^bPcal}
		\check{C}_{Y,\Ucal}: D^b(\Perv(Y)) \rightarrow D^b(\Perv(Y)).
	\end{equation}
	The argument used in the proof of \cite[Prop.~4.2]{IM19} implies that the functor \eqref{cech-D^bPcal} is canonically left adjoint to the inclusion $D^b_{S}(\Perv(Y)) \hookrightarrow D^b(\Perv(Y))$, hence canonically isomorphic to (the co-restriction of) the functor $s_* s^*: D^b(\Perv(Y)) \rightarrow D^b(\Perv(Y))$. 
	
	The construction carries over to the setting of perverse motives: one can define a \v{C}ech complex functor
	\begin{equation*}
		\check{C}^{\bullet}_{Y,\Ucal}: \Mcal^0(Y) \rightarrow C^{[-d,0]}(\Mcal^0(Y)), \quad M \rightsquigarrow [J_{d,!} J_d^* M \to \dots \to J_{1,!} J_1^* M \to M]
	\end{equation*}
	and use it to obtain a triangulated functor
	\begin{equation*}
		\check{C}_{Y,\Ucal}: D^b(\Mcal^0(Y)) \rightarrow D^b(\Mcal^0(Y))
	\end{equation*}
	which is canonically left adjoint to the inclusion $D^b_{S}(\Mcal^0(Y)) \hookrightarrow D^b(\Mcal^0(Y))$, hence canonically isomorphic to (the co-restriction of) the functor $s_* s^*: D^b(\Mcal^0(Y)) \rightarrow D^b(\Mcal^0(Y))$.
	Note that we have a canonical natural isomorphism between exact functors $C^b(\Mcal^0(Y)) \rightarrow C^b(\Perv(Y))$
	\begin{equation*}
		\check{C}_{Y,\Ucal} (\iota_Y(M^{\bullet})) \xrightarrow{\sim} \iota_Y(\check{C}_{Y,\Ucal}(M^{\bullet}))
	\end{equation*}
	defined as the composite
	\begin{equation*}
		\begin{tikzcd}
			\check{C}_{Y,\Ucal} (\iota_Y(M^{\bullet})) \arrow[equal]{d} && \iota_Y(\check{C}_{Y,\Ucal}(M^{\bullet})) \arrow[equal]{d} \\
			\Tot^{\bullet} \check{C}^{\bullet}_{Y,\Ucal}(\iota_Y(M^{\bullet})) \arrow{r}{\sim} & \Tot^{\bullet} \iota_Y(\check{C}^{\bullet}_{Y,\Ucal}(M^{\bullet})) \arrow[equal]{r} & \iota_Y(\Tot^{\bullet} \check{C}^{\bullet}_{Y,\Ucal}(M^{\bullet}))
		\end{tikzcd}
	\end{equation*}
	where the left-most horizontal arrow is induced by the family of natural isomorphisms of functors $C^b(\Mcal^0(Y)) \rightarrow C^b(\Perv(Y))$
	\begin{equation*}
		\check{C}^r_{Y,\Ucal}(\iota_Y(M^{\bullet})) := J_{r,!} J_r^* \iota_Y(M^{\bullet}) \xrightarrow{\sim} J_{r,!} \iota_{U_{\Ucal}^{(r)}}(J_r^* M^{\bullet}) \xrightarrow{\sim} \iota_Y (J_{r,!} J_r^* M^{\bullet}) =: \iota_Y(\check{C}^r_{Y,\Ucal}(M^{\bullet}))
	\end{equation*}
	as $r = 0,\dots,d$ varies. In the following, we consider the induced natural isomorphism of triangulated functors $D^b(\Mcal^0(Y)) \rightarrow D^b(\Perv(Y))$
	\begin{equation*}
		\check{C}_{Y,\Ucal} (\iota_Y(M^{\bullet})) \xrightarrow{\sim} \iota_Y(\check{C}_{Y,\Ucal}(M^{\bullet})).
	\end{equation*}
\end{constr}

We need to provide an analogous \v{C}ech complex construction on the level of triangulated étale motives. This is possible since the stable homotopy $2$-functor $\DA_{ct}^{\et}(\cdot,\Q)$ canonically underlies a stable homotopy algebraic derivator in the sense of \cite[\S~2.4.2]{Ayo07a}.

\begin{constr}\label{constr:cech-DA}
	Keep the notation and assumptions of \Cref{constr:cech-perv}. 
	For every choice of subsets $I, I' \subset \left\{1,\dots,d\right\}$ with $I' \subset I$, consider the natural transformation of functors $\DA_{ct}^{\et}(Y,\Q) \rightarrow \DA_{ct}^{\et}(Y,\Q)$
	\begin{equation*}
		\alpha^{\Ucal}_{I,I'}: j_{I,!} j_I^* A \rightarrow j_{I',!} j_{I'}^* A
	\end{equation*}
	induced by the co-unit along the inclusion of $U_I$ into $U_{I'}$. By construction, the diagram of functors $\DA_{ct}^{\et}(Y,\Q) \rightarrow D^b_c(Y,\Q)$
	\begin{equation*}
		\begin{tikzcd}
			j_{I,!} j_I^* \Bti_Y^*(A) \arrow{r}{\alpha^{\Ucal}_{I,I'}} \isoarrow{d} & j_{I',!} j_{I'}^* \Bti_Y^*(A) \isoarrow{d} \\
			j_{I,!} \Bti_{U_I}^* (j_I^* A) \isoarrow{d} &  j_{I',!} \Bti_{U_{I'}}^* (j_{I'}^* A) \isoarrow{d} \\
			\Bti_Y^* (j_{I,!} j_I^* A) \arrow{r}{\alpha^{\Ucal}_{I,I'}} & \Bti_Y^* (j_{I',!} j_{I'}^* A)
		\end{tikzcd}
	\end{equation*}
	is commutative under Beilinson's equivalence $D^b(\Perv(Y)) \xrightarrow{\sim} D^b_c(Y,\Q)$. 
	Let $[-1,0]$ denote the partially ordered set $\left\{-1 < 0\right\}$; for every integer $n \geq 0$ we regard the product category $[-1,0]^n$ as a small diagram (with the convention that $[-1,0]^0$ is just the singleton category). We are particularly interested in the diagram $[-1,0]^d$: in this case, for every subset $I \subset \left\{0,\dots,d\right\}$, we let $\underline{I}$ denote the object $(\epsilon_1,\dots,\epsilon_d) \in [-1,0]^d$ whose $i$-th coordinate $\epsilon_i$ equals $-1$ if $i \in I$ and $0$ otherwise. The assignment $I \rightsquigarrow \underline{I}$ defines an anti-isomorphism between the powerset of $\left\{1,\dots,d\right\}$ (regarded as a poset with the usual inclusion ordering) and $[-1,0]^d$. 
	Given two subsets $I, I' \subset \left\{1,\dots,d\right\}$ satisfying $I' \subset I$, we let $e_{I,I'}: \underline{I} \rightarrow \underline{I}'$ denote the corresponding morphism in $[-1,0]^d$.
	Similarly to \cite[Lem.~3.4, Lem.~3.6]{IM19}, we can construct a triangulated functor
	\begin{equation*}
		\check{C}^{\square}_{Y,\Ucal}: \DA_{ct}^{\et}(Y,\Q) \rightarrow \DA_{ct}^{\et}((Y,[-1,0]^{d}),\Q)
	\end{equation*}
	providing a coherent lifting of the $[-1,0]^d$-indexed diagram of functors $\DA_{ct}^{\et}(Y,\Q) \rightarrow \DA_{ct}^{\et}(Y,\Q)$
	\begin{equation*}
		\underline{I}(A) := j_{I,!} j_I^* A, \qquad e_{I,I'}(A) := \alpha_{I,I'}^{\Ucal}(A): j_{I,!} j_I^* A \rightarrow j_{I',!} j_{I'}^* A.
	\end{equation*}
	Concretely, one can consider the two obvious morphisms of diagram of $k$-varieties
	\begin{equation*}
		\begin{tikzcd}
			(\underline{I} \mapsto U_I)_{\underline{I} \in [-1,0]^d} \arrow{r}{\xi} \arrow{d}{\nu} & (\underline{I} \mapsto Y)_{\underline{I} \in [-1,0]^d} \\
			Y
		\end{tikzcd}
	\end{equation*} 
	and define the functor $\check{C}^{\square}_{Y,\Ucal}$ by the formula
	\begin{equation*}
		\check{C}^{\square}_{Y,\Ucal} := \xi_! \circ \nu^*.
	\end{equation*}
	Now, for every integer $n \geq 1$, let
	\begin{equation*}
		\Cof_n: \DA_{ct}^{\et}((Y,[-1,0]^n),\Q) \rightarrow \DA_{ct}^{\et}((Y,[-1,0]^{n-1}),\Q)
	\end{equation*}
	denote the co-fiber with respect to the last coordinate of $[-1,0]^n = [-1,0]^{n-1} \times [-1,0]$. We let 
	\begin{equation*}
		\Tot^{\square} := \Cof_1 \circ \dots \circ \Cof_d: \DA_{ct}^{\et}((Y,[-1,0]^d),\Q) \rightarrow \DA_{ct}^{\et}(Y,\Q)
	\end{equation*} 
	denote the associated derived total complex functor, and we set
	\begin{equation*}
		\check{C}_{Y,\Ucal} := \Tot^{\square} \circ \check{C}^{\square}_{Y,\Ucal}: \DA_{ct}^{\et}(Y,\Q) \rightarrow \DA_{ct}^{\et}(Y,\Q).
	\end{equation*}
	Similarly, we can define a functor
	\begin{equation*}
		\sigma_{\leq (-1,\dots,-1)} \check{C}^{\square}_{Y,\Ucal}: \DA_{ct}^{\et}(Y,\Q) \rightarrow \DA_{ct}^{\et}((Y,[-1,0]^{d}),\Q)
	\end{equation*}
	providing a coherent lifting of the $[-1,0]^d$-indexed diagram of functors $\DA_{ct}^{\et}(Y,\Q) \rightarrow \DA_{ct}^{\et}(Y,\Q)$
	\begin{equation*}
		\underline{I}(A) := 
		\begin{cases}
			J_{d,!} J_d^* A, & I = \left\{1,\dots,d\right\} \\
			0, & I \neq \left\{1,\dots,d\right\}.
		\end{cases} 
	\end{equation*}
	Note that there is an obvious natural transformation of functors $\DA_{ct}^{\et}(Y,\Q) \rightarrow \DA_{ct}^{\et}((Y,[-1,0]^d),\Q)$
	\begin{equation}\label{nat:cech-sigmacech}
		\check{C}^{\square}_{Y,\Ucal}(A) \rightarrow \sigma_{\leq (-1,\dots,-1)} \check{C}^{\square}_{Y,\Ucal}(A).
	\end{equation}
	One can repeat the entire construction replacing $\DA_{ct}^{\et}(\cdot,\Q)$ with any stable homotopy algebraic derivator in the sense of \cite[\S~2.4.2]{Ayo07a}. As a consequence of \cite[Thm.~1.8, Lem.~1.10]{Ayo10}, the stable homotopy $2$-functor $D^b_c(\cdot,\Q)$ canonically underlies a stable homotopy algebraic derivator, and the Betti realization
	\begin{equation*}
		\Bti^*: \DA_{ct}^{\et}(\cdot,\Q) \rightarrow D^b_c(\cdot,\Q)
	\end{equation*}
	can be canonically extended over diagrams of $k$-varieties. Thus we can define a functor
	\begin{equation*}
		\check{C}^{\square}_{Y,\Ucal}: D^b_c(Y,\Q) \rightarrow D^b_c((Y,[-1,0]^{d}),\Q)
	\end{equation*}
	providing a coherent lifting of the $[-1,0]^d$-indexed diagram of functors $D^b_c(Y,\Q) \rightarrow D^b_c(Y,\Q)$
	\begin{equation*}
		\underline{I}(K^{\bullet}) := j_{I,!} j_I^* K^{\bullet}, \qquad e_{I,I'}(A) := \alpha_{I,I'}^{\Ucal}(K^{\bullet}): j_{I,!} j_I^* K^{\bullet} \rightarrow j_{I',!} j_{I'}^* K^{\bullet},
	\end{equation*}
	as well as a functor
	\begin{equation*}
		\sigma_{\leq (-1,\dots,-1)} \check{C}^{\square}_{Y,\Ucal}: D^b_c(Y,\Q) \rightarrow D^b_c((Y,[-1,0]^{d}),\Q)
	\end{equation*}
	providing a coherent lifting of the $[-1,0]^d$-indexed diagram of functors $D^b_c(Y,\Q) \rightarrow D^b_c(Y,\Q)$
	\begin{equation*}
		\underline{I}(K^{\bullet}) :=
		\begin{cases}
			J_{d,!} J_d^* K^{\bullet}, & I = \left\{1,\dots,d\right\} \\
			0, & I \neq \left\{1,\dots,d\right\}.
		\end{cases} 
	\end{equation*}
	As above, there is a canonical natural transformation of functors $D^b_c(Y,\Q) \rightarrow D^b_c((Y,[-1,0]^d),\Q)$
	\begin{equation*}
		\check{C}^{\square}_{Y,\Ucal}(K^{\bullet}) \rightarrow \sigma_{\leq (-1,\dots,-1)} \check{C}^{\square}_{Y,\Ucal}(K^{\bullet}).
	\end{equation*}
	Moreover, there are canonical natural isomorphisms of functors $\DA_{ct}^{\et}(Y,\Q) \rightarrow D^b_c((Y,[-1,0]^d),\Q)$
	\begin{equation*}
		\check{C}^{\square}_{Y,\Ucal} \Bti_Y^*(A) \xrightarrow{\sim} \Bti_{(Y,[-1,0]^d)}^*(\check{C}^{\square}_{Y,\Ucal}(A))
	\end{equation*}
	and
	\begin{equation*}
		\sigma_{\leq (-1,\dots,-1)} \check{C}^{\square}_{Y,\Ucal} \Bti_Y^*(A) \xrightarrow{\sim} \Bti_{(Y,[-1,0]^d)}^*(\sigma_{\leq (-1,\dots,-1)} \check{C}^{\square}_{Y,\Ucal}(A))
	\end{equation*}
	rendering the diagram of functors $\DA_{ct}^{\et}(Y,\Q) \rightarrow D^b_c((Y,[-1,0]^d),\Q)$
	\begin{equation*}
		\begin{tikzcd}
			\check{C}^{\square}_{Y,\Ucal}(\Bti_Y^*(A)) \arrow{r}{\sim} \arrow{d} & \Bti_{(Y,[-1,0]^d)}^* (\check{C}^{\square}_{Y,\Ucal}(A)) \arrow{d} \\
			\sigma_{\leq (-1,\dots,-1)} \check{C}^{\square}_{Y,\Ucal}(\Bti_Y^*(A)) \arrow{r}{\sim} & \Bti_{(Y,[-1,0]^d)}^* (\sigma_{\leq (-1,\dots,-1)} \check{C}^{\square}_{Y,\Ucal}(A))
		\end{tikzcd}
	\end{equation*}
	commutative. We can also define a total complex functor 
	\begin{equation*}
		\Tot^{\square} := \Cof_1 \circ \dots \circ \Cof_d: D^b_c((Y,[-1,0]^d),\Q) \rightarrow D^b_c(Y,\Q).
	\end{equation*} 
	By construction, the composite
	\begin{equation*}
		\check{C}_{Y,\Ucal} := \Tot^{\square} \circ \check{C}^{\square}_{Y,\Ucal}: D^b_c(Y,\Q) \rightarrow D^b_c(Y,\Q)
	\end{equation*}
	recovers the functor \eqref{cech-D^bPcal} under Beilinson's equivalence $D^b(\Perv(Y)) \xrightarrow{\sim} D^b_c(Y,\Q)$. Therefore we obtain a canonical natural isomorphism of functors $\DA_{ct}^{\et}(Y,\Q) \rightarrow D^b_c(Y,\Q)$
	\begin{equation*}
		\check{C}_{Y,\Ucal} \Bti_Y(A) \xrightarrow{\sim} \Bti_Y^* (\check{C}_{Y,\Ucal}(A))
	\end{equation*}
	by taking the composite
	\begin{equation*}
		\begin{tikzcd}
			\check{C}_{Y,\Ucal} \Bti_Y(A) \arrow[equal]{d} && \Bti_Y^* (\check{C}_{Y,\Ucal}(A)). \arrow[equal]{d} \\
			\Tot^{\square} \check{C}^{\square}_{Y,\Ucal} (\Bti_Y^*(A)) \arrow{r}{\sim} & \Tot^{\square} \Bti_{Y,[-1,0]^d}^* (\check{C}^{\square}_{Y,\Ucal}(A)) \arrow[equal]{r} & \Bti_Y^*(\Tot^{\square} \check{C}^{\square}_{Y,\Ucal}(A)) 
		\end{tikzcd}
	\end{equation*}
	By construction, the diagram of functors $\DA_{ct}^{\et}(Y,\Q) \rightarrow D^b_c(Y,\Q)$
	\begin{equation}\label{dia:cech-s_*s^*-Bti}
		\begin{tikzcd}
			\check{C}_{Y,\Ucal} \Bti_Y(A) \arrow{rr}{\sim} \arrow[equal]{d} && \Bti_Y^* (\check{C}_{Y,\Ucal}(A)) \arrow[equal]{d} \\
			s_* s^* \Bti_Y^*(A) \arrow{r}{\theta} & s_* \Bti_S^* (s^* A) & \Bti_Y^*(s_* s^* A) \arrow{l}{\bar{\theta}}
		\end{tikzcd}
	\end{equation}
	commutes.
\end{constr}

We need to compare the total \v{C}ech complex functor on the level of perverse motives with the analogous functor on triangulated motives. This turns out to be quite subtle, since the two functors are constructed in different ways. In order to proceed, it is convenient to introduce yet a further subdiagram of $\Dcal^0(S)$: 

\begin{nota}
	Given a closed immersion $s: S \hookrightarrow Y$ between smooth $k$-varieties, we introduce the full additive subcategory
	\begin{equation*}
		\Dcal^{0,s}(Y) := \left\{B \in \Dcal^0(Y) \; | \; s^{\dagger} B \in \Dcal^0(S) \right\} \subset \Dcal^0(Y).
	\end{equation*}
\end{nota}

We can state the promised result about \v{C}ech complex functors as follows:
\begin{lem}\label{lem-cech-dist}
	Let $s: S \hookrightarrow Y$ be a closed immersion of constant codimension $d$ between smooth $k$-varieties. 
	Suppose that the open complement $U$ of $S$ in $Y$ admits a finite affine open covering $\Ucal$ with exactly $d$ elements. Then there exist a canonical natural isomorphism of functors $\Dcal^{0,s}(Y) \rightarrow \Mcal^0(Y)$
	\begin{equation}\label{cech-Dcal-Mp}
		\pi_Y(\check{C}_{Y,\Ucal}[-d](B)) \xrightarrow{\sim} \check{C}_{Y,\Ucal}[-d] (\pi_Y(B))
	\end{equation}
	making the diagram of functors $\Dcal^{0,s}(Y) \rightarrow \Perv(Y)$
	\begin{equation}\label{dia:cech-lem}
		\begin{tikzcd}
			\Bti_Y^* (\check{C}_{Y,\Ucal}[-d] B) \arrow[equal]{d} & \check{C}_{Y,\Ucal}[-d] \Bti_Y^*(B) \arrow[equal]{r} \arrow{l}{\sim} & \check{C}_{Y,\Ucal}[-d] \iota_Y \pi_Y(B) \isoarrow{d} \\
			\iota_Y \pi_Y (\check{C}_{Y,\Ucal}[-d] B) \arrow{rr}{\sim} && \iota_Y (\check{C}_{Y,\Ucal}[-d] \pi_Y(B))
		\end{tikzcd}
	\end{equation}
	commute. 
\end{lem}
\begin{proof}
	First of all, let us check that the composite functor
	\begin{equation*}
		\check{C}_{Y,\Ucal}[-d] \circ \pi_Y: \Dcal^{0,s}(Y) \rightarrow \Mcal^0(Y) \rightarrow D^b(\Mcal^0(Y))
	\end{equation*} 
	really takes values in $\Mcal^0(Y)$ - or, equivalently, that its composition with the conservative $t$-exact functor $\iota_Y: D^b(\Mcal^0(Y)) \rightarrow D^b(\Perv(Y))$ takes values in $\Perv(Y)$. To this end, for every object $B \in \Dcal^{0,s}(Y)$, we compute
	\begin{align*}
		\iota_Y (\check{C}_{Y,\Ucal}[-d](\pi_Y(B))) &= \check{C}_{Y,\Ucal}[-d] (\iota_Y \pi_Y(B)) \\
		&= \check{C}_{Y,\Ucal}[-d](\beta_Y(B)) \\ 
		&= s_* s^*[-d] \beta_Y(B) \\
		&= s_* s^*[-d] \Bti_Y^*(B) && \textup{as $B \in \Dcal^0(Y)$}\\
		&= s_* \Bti_S^*(s^*[-d] B) \\
		&= s_* \beta_S(s^*[-d] B)  && \textup{as $B \in \Dcal^{0,s}(Y)$}.
	\end{align*}
	The latter object belongs indeed to $\Perv(Y)$, in view of the definition of $\beta_S$ and by the $t$-exactness of $s_*: D^b_c(S,\Q) \rightarrow D^b_c(Y,\Q)$.
	
	In order to define the sought-after natural isomorphism \eqref{cech-Dcal-Mp}, we start by constructing a natural transformation between exact functors $\Dcal^0(Y) \rightarrow C^{[-d,0]}(\Mcal^0(Y))$
	\begin{equation*}
		\pi_Y(\check{C}_{Y,\Ucal}[-d](B))[d] \rightarrow \check{C}^{\bullet}_{Y,\Ucal}(\pi_Y(B)).
	\end{equation*}
	With the notation of \Cref{constr:cech-perv} and \Cref{constr:cech-DA}, this amounts to giving a natural transformation of functors $\Dcal^0(Y) \rightarrow \Mcal^0(Y)$
	\begin{equation}\label{nat:cech-tr-complex-d}
		\pi_Y(\check{C}_{Y,\Ucal}[-d](B)) \rightarrow J_{d,!} J_d^* \pi_Y(B)
	\end{equation}
	such that the composite
	\begin{equation}\label{nat:cech-tr-composite-d}
		\pi_Y(\check{C}_{Y,\Ucal}[-d](B)) \rightarrow J_{d,!} J_d^* \pi_Y(B) \rightarrow J_{d-1,!} J_{d-1}^* \pi_Y(B) = \bigoplus_{\substack{I \subset \left\{1,\dots,d\right\} \\ \# I = d-1}} j_{I,!} j_I^* \pi_Y(B)
	\end{equation}
	vanishes. We define \eqref{nat:cech-tr-complex-d} by taking the composite
	\begin{equation*}
		\begin{tikzcd}
			\pi_Y(\check{C}_{Y,\Ucal}[-d](B)) \arrow[equal]{d} && J_{d,!} J_d^* \pi_Y(B), \arrow[equal]{d} \\
			\pi_Y(\Tot^{\square}[-d] \check{C}^{\square}_{Y,\Ucal}(B)) \arrow{r} & \pi_Y(\Tot^{\square}[-d] \sigma_{\leq (-1,\dots,-1)} \check{C}^{\square}_{Y,\Ucal}(B)) \arrow[equal]{r} & \pi_Y(J_{d,!} J_d^* B)
		\end{tikzcd}
	\end{equation*}
	where the first horizontal arrow is induced by \eqref{nat:cech-sigmacech}.
	The vanishing of \eqref{nat:cech-tr-composite-d} follows from the existence of a commutative diagram of the form
	\begin{equation*}
		\begin{tikzcd}
			\pi_Y(\check{C}_{Y,\Ucal}[-d] B) \arrow{r} \arrow{dr} & \pi_Y(J_{d,!} J_d^* B) \arrow{r} \arrow[equal]{d} & \bigoplus_{\substack{I \subset \left\{1,\dots,d\right\} \\ \# I = d-1}} \pi_Y(j_{I,!} j_I^* B) \arrow[equal]{d} \\
			& J_{d,!} J_d^* \pi_Y(B) \arrow{r} & \bigoplus_{\substack{I \subset \left\{1,\dots,d\right\} \\ \# I = d-1}} j_{I,!} j_I^* \pi_Y(B),
		\end{tikzcd}
	\end{equation*}
	where the upper horizontal composite vanishes by construction.
	We deduce a natural transformation of functors $\Dcal^0(Y) \rightarrow D^{[0,d]}(\Mcal^0(Y))$
	\begin{equation}\label{cech-D^0Y}
		\begin{tikzcd}
			\pi_Y(\check{C}_{Y,\Ucal}[-d] B) \arrow[equal]{d} & \check{C}_{Y,\Ucal}[-d](\pi_Y(B)). \arrow[equal]{d} \\ 
			\Tot^{\bullet} \pi_Y(\check{C}_{Y,\Ucal}[-d](B)) \arrow{r} & \Tot^{\bullet} \check{C}^{\bullet}_{Y,\Ucal}[-d](\pi_Y(B))
		\end{tikzcd}
	\end{equation} 
	Similarly, we get a natural transformation of functors $\Dcal^0(Y) \rightarrow D^{[-d,0]}(\Perv(Y))$
	\begin{equation}\label{cech-D^0PervY}
		\beta_Y(\check{C}_{Y,\Ucal}[-d] (B)) \rightarrow \check{C}_{Y,\Ucal}[-d] (\beta_Y(B))
	\end{equation}
	making the diagram of functors $\Dcal^0(Y) \rightarrow D^{[0,d]}(\Perv(Y))$
	\begin{equation}\label{dia:cech-iota-pi}
		\begin{tikzcd}[font=\small]
			\beta_Y(\check{C}_{Y,\Ucal}[-d](B)) \arrow[equal]{dd} \arrow{r} & \check{C}_{Y,\Ucal}[-d] (\beta_Y(B)) \arrow[equal]{d} \\
			& \check{C}_{Y,\Ucal}[-d] (\iota_Y \pi_Y(B)) \isoarrow{d} \\
			\iota_Y \pi_Y(\check{C}_{Y,\Ucal}[-d](B)) \arrow{r} & \iota_Y(\check{C}_{Y,\Ucal}[-d](\pi_Y(B))) 
		\end{tikzcd}
	\end{equation}
	commutative. We claim that the value of \eqref{cech-D^0Y} on a given object $B \in \Dcal^0(Y)$ is an isomorphism if and only if $B \in \Dcal^{0,s}(Y)$. Using the commutativity of \eqref{dia:cech-iota-pi} and the conservativity of $\iota_Y: D^{[0,d]}(\Mcal^0(Y)) \rightarrow D^{[0,d]}(\Perv(Y))$, we see that this is equivalent to showing that the value of \eqref{cech-D^0PervY} on a given object $B \in \Dcal^0(Y)$ is an isomorphism if and only if $B \in \Dcal^{0,s}(Y)$. To this end, taking the image under ${^p H^0}$ of the commutative diagram \eqref{dia:cech-s_*s^*-Bti}, the arrow \eqref{cech-D^0PervY} can be identified with the canonical arrow in $\Perv(Y)$
	\begin{equation}\label{arrow-cech-s_*s^*}
		{^p H^0} \Bti_Y^*(s_* s^{\dagger} B) = {^p H^0} s_* \Bti_S^*(s^{\dagger} B) = s_* {^p H^0} \Bti_S^* (s^{\dagger} B) \rightarrow s_* \Bti_S^* (s^{\dagger} B).
	\end{equation}
	Here, the latter arrow in the composite can be also written as 
	\begin{equation*}
		s_* {^p H^0} s^{\dagger} \Bti_Y^*(B) \rightarrow s_* s^{\dagger} \Bti_Y^*(B)
	\end{equation*}
	and is induced by the left-$t$-exactness of the functor $s^{\dagger} = s^*[-d]: D^b_c(Y,\Q) \rightarrow D^b_c(S,\Q)$ (which follows precisely from the fact that $U$ admits an affine open covering with $d$ elements). Using the conservativity of $s_*: D^b_c(S,\Q) \rightarrow D^b_c(Y,\Q)$, we see that the arrow \eqref{arrow-cech-s_*s^*} is invertible if and only if $ \Bti_S^*(s^{\dagger} B) \in \Perv(S)$. 
	
	We define the natural isomorphism \eqref{cech-Dcal-Mp} by taking the restriction of \eqref{cech-D^0Y} to $\Dcal^{0,s}(Y)$. By construction, the diagram \eqref{dia:cech-lem} coincides with the outer part of the diagram
	\begin{equation*}
		\begin{tikzcd}[font=\small]
			\Bti_Y^* (\check{C}_{Y,\Ucal}[-d] B) \arrow[equal]{dr} \arrow[equal]{dd} &&& \check{C}_{Y,\Ucal}[-d] \Bti_Y^*(B) \arrow[equal]{d} \arrow{lll}{\sim} \\
			& \beta_Y (\check{C}_{Y,\Ucal}[-d] B) \arrow{r}{\sim} \arrow[equal]{dl} & \check{C}_{Y,\Ucal}[-d] \beta_Y(B) \arrow[equal]{ur} \arrow[equal]{r} & \check{C}_{Y,\Ucal}[-d] \iota_Y \pi_Y(B) \isoarrow{d} \\
			\iota_Y \pi_Y (\check{C}_{Y,\Ucal}[-d] B) \arrow{rrr}{\sim} &&& \iota_Y (\check{C}_{Y,\Ucal}[-d] \pi_Y(B)),
		\end{tikzcd}
	\end{equation*}
	where the two triangles are commutative by definition and the bottom piece is already known to be commutative. Hence, in order to conclude the proof, it suffices to show that the top piece is commutative as well. Unwinding the various definitions, the commutativity of the upper piece follows from the commutativity of the outer part of the diagram 
	\begin{widepage}
		\begin{equation*}
			\begin{tikzcd}[font=\tiny]
				\Bti_Y^*(\Tot^{\square}[-d] \check{C}^{\square}_{Y,\Ucal}(B)) \arrow[equal]{r} \arrow{d} & \Tot^{\square}[-d] \Bti_{(Y,[-1,0]^d)}^* (\check{C}^{\square}_{Y,\Ucal}(B)) \arrow{d} & \Tot^{\square}[-d] \check{C}^{\square}_{Y,\Ucal}(\Bti_Y^*(B)) \arrow{l}{\sim} \arrow{d} \\
				\Bti_Y^* (\Tot^{\square}[-d] \sigma_{\leq (-1,\dots,-1)} \check{C}^{\square}_{Y,\Ucal}(B)) \arrow[equal]{d} \arrow[equal]{r} & \Tot^{\square}[-d] \Bti_{(Y,[-1,0]^d)}^* (\sigma_{\leq (-1,\dots,-1)} \check{C}^{\square}_{Y,\Ucal}(B)) & \Tot^{\square}[-d] \sigma_{\leq (-1,\dots,-1)} \check{C}^{\square}_{Y,\Ucal}(\Bti_Y^*(B)) \arrow[equal]{d} \arrow{l}{\sim} \\
				\Bti_Y^* (J_{d,!} J_d^* B) \arrow[equal]{d} && J_{d,!} J_d^* \Bti_Y^*(B) \arrow{ll}{\sim} \arrow[equal]{d} \\
				\beta_Y (J_{d,!} J_d^* B) && J_{d,!} J_d^* \beta_Y(B), \arrow{ll}{\sim}
			\end{tikzcd}
		\end{equation*}	
	\end{widepage}
	which in turn follows from the commutativity of the four inner pieces, which holds by construction. This concludes the proof.
\end{proof}

\begin{cor}\label{cor:s^diamond-pi}
	Let $s: S \hookrightarrow Y$ be a closed immersion of constant codimension $d$ between smooth $k$-varieties. 
	Then there exists a natural isomorphism of functors $\Dcal^{0,s}(Y) \rightarrow \Mcal^0(S)$
	\begin{equation}\label{formula:zeta_S}
		\zeta_s: \pi_S(s^{\dagger} B) \xrightarrow{\sim} s^{\dagger} \pi_Y(B)
	\end{equation} 
	making the diagram of functors $\Dcal^{0,s}(Y) \rightarrow \Perv(S)$
	\begin{equation}\label{dia:zeta_S}
		\begin{tikzcd}
			\Bti_S^* (s^{\dagger} B) \arrow[equal]{dd} & s^{\dagger} \Bti_Y^*(B) \arrow{l}{\theta} \arrow[equal]{d} \\
			& s^{\dagger} \iota_Y \pi_Y(B) \arrow{d}{\tilde{\theta}} \\
			\iota_S \pi_S (s^{\dagger} B) \arrow{r}{\zeta_s} & \iota_S(s^{\dagger} \pi_V(B)) 
		\end{tikzcd}
	\end{equation}
	commute.
\end{cor}
\begin{proof}
	Note that, for every closed immersion $s: S \hookrightarrow Y$ as in the statement, the faithfulness and conservativity of $\iota_S: \Mcal^0(S) \hookrightarrow \Perv(S)$ implies that a natural transformation as in \eqref{formula:zeta_S} rendering the diagram \eqref{dia:zeta_S} commutative is uniquely determined and necessarily invertible, provided it exists. 
	
	To show the existence, we start with a preliminary observation. Let $s: S \hookrightarrow Y$ be a closed immersion as in the statement; let $v: V \hookrightarrow Y$ be an open immersion, and form the Cartesian square
	\begin{equation*}
		\begin{tikzcd}
			U \arrow{r}{s_U} \arrow{d}{u} & V \arrow{d}{v} \\
			S \arrow{r}{s} & Y.
		\end{tikzcd}
	\end{equation*}
	Note that the $k$-variety $U$ is smooth by construction.
	Suppose for a moment that the thesis is known to hold both for $s$ and for $s_U$. We claim that, in this situation, the diagram of functors $\Dcal^{0,s}(Y) \rightarrow \Mcal^0(U)$
	\begin{equation*}
		\begin{tikzcd}[font=\small]
			u^* \pi_S(s^{\dagger} B) \arrow{r}{\zeta_s} & u^* s^{\dagger} \pi_Y(B) \\
			\pi_U(u^* s^{\dagger} B) \arrow[equal]{u} \arrow[equal]{d} & s^{\dagger} v^* \pi_Y(B) \arrow[equal]{u} \arrow[equal]{d} \\
			\pi_U(s_U^{\dagger} v^* B) \arrow{r}{\zeta_{s_U}} & s_U^{\dagger} \pi_V(v^* B)
		\end{tikzcd}
	\end{equation*}
	commutes. To check this, it suffices to consider its image under the faithful functor $\iota_U: \Mcal^0(U) \hookrightarrow \Perv(U)$
	\begin{equation*}
		\begin{tikzcd}[font=\small]
			\iota_U (u^* \pi_S(s^{\dagger} B)) \arrow{r}{\zeta_s} & \iota_U (u^* s^{\dagger} \pi_Y(B)) \\
			\iota_U \pi_U(u^* s^{\dagger} B) \arrow[equal]{u} \arrow[equal]{d} & \iota_U (s_U^{\dagger} v^* \pi_Y(B)) \arrow[equal]{u} \arrow[equal]{d} \\
			\iota_U \pi_U(s_U^{\dagger} v^* B) \arrow{r}{\zeta_{s_U}} & \iota_U (s_U^{\dagger} \pi_V(v^* B)),
		\end{tikzcd}
	\end{equation*}
    which we can decompose as
	\begin{equation*}
		\begin{tikzcd}[font=\small]
			\iota_U (u^* \pi_S(s^{\dagger} B)) \arrow{rrrr}{\zeta_s} &&&& \iota_U (u^* s^{\dagger} \pi_Y(B)) \\
			& u^* \iota_S \pi_S (s^{\dagger} B) \arrow{ul}{\tilde{\theta}} \arrow{rr}{\zeta_s} && u^* \iota_S(s^{\dagger} \pi_Y(B)) \arrow{ur}{\tilde{\theta}} \\
			& u^* \Bti_S^* (s^{\dagger} B) \arrow[equal]{u} \arrow{d}{\theta} & u^* s^{\dagger} \Bti_Y^*(B) \arrow[equal]{d} \arrow{l}{\theta} \arrow[equal]{r} & u^* s^{\dagger} \iota_Y \pi_Y(B) \arrow{u}{\tilde{\theta}} \\
			\iota_U \pi_U (u^* s^{\dagger} B) \arrow[equal]{uuu} \arrow[equal]{ddd} \arrow[equal]{r} & \Bti_U^* (u^* s^{\dagger} B) \arrow[equal]{d} & s_U^{\dagger} v^* \Bti_Y^*(B) \arrow{d}{\theta} \arrow[equal]{r} & s_U^{\dagger} v^* \iota_V \pi_V(B) \arrow[equal]{u} \arrow{d}{\tilde{\theta}} & \iota_U(s_U^{\dagger} v^* \pi_Y(B)) \arrow[equal]{uuu} \arrow[equal]{ddd} \\
			& \Bti_U^* (s_U^{\dagger} v^* B) & s_U^{\dagger} \Bti_V^*(v^* B) \arrow{l}{\theta} \arrow[equal]{dr} & s_U^{\dagger} \iota_V(v^* \pi_Y(B)) \arrow[equal]{d} \arrow{ur}{\tilde{\theta}} \\
			&&& s_U^{\dagger} \iota_V \pi_V(v^* B) \arrow{dr}{\tilde{\theta}} \\
			\iota_U \pi_U(s_U^{\dagger} v^* B) \arrow{rrrr}{\zeta_{s_U}} \arrow[equal]{uur} &&&& \iota_U (s_U^{\dagger} \pi_V(v^* B)).
		\end{tikzcd}
	\end{equation*}
    Here, the two five-term pieces involving $\zeta_s$ and $\zeta_{s_U}$ are commutative by assumption, the six-term piece in the upper-right part as well as the two five-term pieces in the upper-left part and in the central part are commutative by \cite[Prop.~2.5(2)]{Ter23UAF}, the central six-term piece is commutative by construction, and the remaining three pieces are commutative by naturality.
	
	As a consequence of this observation we deduce that, in order to show the existence of the natural isomorphism $\zeta_s$ along a given closed immersion $s: S \hookrightarrow Y$, it suffices to show that, for some finite open covering $\left\{v_i: V_i \hookrightarrow Y\right\}_{i \in I}$, the isomorphisms $\zeta_{s_i}$ along all restrictions $s_i: s^{-1}(V_i) \hookrightarrow V_i$ exist: this follows from the fact that both $\Mcal^0(\cdot)$ and $\Perv(\cdot)$ are stacks for the Zariski topology. Therefore, since both $Y$ and $S$ are smooth by assumption (hence, in particular, the closed immersion $s$ is regular of codimension $d$), we may reduce to the case where $Y$ is affine and connected and the open complement $U$ of $S$ in $Y$ admits a finite affine open covering $\Ucal$ with exactly $d$ elements: for example, for each closed point of $S$ one can take the distinguished open subsets associated to the elements of some regular sequence of length $d$ defining $S$ inside $Y$ locally around the chosen point. 
	
	In this situation, fix one such open covering $\Ucal$; we define the sought-after natural isomorphism $\zeta_s$ by declaring that its image under the fully faithful functor $s_*: \Mcal^0(S) \hookrightarrow \Mcal^0(Y)$ be the composite
	\begin{equation*}
		\begin{tikzcd}
			s_* \pi_S(s^{\dagger} B) \arrow[equal]{d} && s_* s^{\dagger} \pi_Y(B), \arrow[equal]{d} \\
			\pi_Y(s_* s^{\dagger} B) \arrow[equal]{r} & \pi_Y(\check{C}_{Y,\Ucal}[-d](B)) \arrow{r}{\sim} & \check{C}_{Y,\Ucal}[-d](\pi_Y(B)) 
		\end{tikzcd}
	\end{equation*}
	where the bottom-left horizontal arrow is given by \Cref{lem-cech-dist}. The commutativity of \eqref{dia:zeta_S}is equivalent to the commutativity of its image under the fully faithful functor $s_*: \Mcal^0(S) \hookrightarrow \Mcal^0(Y)$
	\begin{equation*}
		\begin{tikzcd}
			s_* \Bti_S^* (s^{\dagger} B) \arrow[equal]{dd} & s_* s^{\dagger} \Bti_Y^*(B) \arrow{l}{\theta} \arrow[equal]{d} \\
			& s_* s^{\dagger} \iota_Y \pi_Y(B) \arrow{d}{\tilde{\theta}} \\
			s_* \iota_S \pi_S (s^{\dagger} B) \arrow{r}{\zeta_s} & s_* \iota_S (s^{\dagger} \pi_Y(B)),
		\end{tikzcd}
	\end{equation*} 
    which in turn is equivalent to the commutativity of the outer rectangle of the diagram
	\begin{equation*}
		\begin{tikzcd}[font=\small]
			s_* \Bti_S^* (s^{\dagger} B) \arrow[equal]{dd} & s_* s^{\dagger} \Bti_Y^*(B) \arrow{l}{\theta} \arrow[equal]{d} \\
			& s_* s^{\dagger} \iota_Y \pi_Y(B) \arrow{d}{\tilde{\theta}} \\
			s_* \iota_S \pi_S (s^{\dagger} B) \arrow{r}{\zeta_s} & s_* \iota_S (s^{\dagger} \pi_Y(B)) \\
			\iota_Y(s_* \pi_S (s^{\dagger} B)) \arrow{r}{\zeta_s} \arrow{u}{\bar{\tilde{\theta}}} & \iota_Y (s_* s^{\dagger} \pi_Y(B)) \arrow{u}{\bar{\tilde{\theta}}}.
		\end{tikzcd}
	\end{equation*}
	Expanding the definition of the bottom horizontal arrow in the latter diagram, we obtain the outer part of the diagram
	\begin{equation*}
		\begin{tikzcd}[font=\small]
			& s_* \Bti_S^* (s^{\dagger} B) \arrow[equal]{dl} & s_* s^{\dagger} \Bti_Y^*(B) \arrow{l}{\theta} \arrow[equal]{dr} \arrow[equal]{d} \\
			s_* \iota_S \pi_S (s^{\dagger} B) && \check{C}_{Y,\Ucal}[-d] \Bti_Y^*(B) \arrow[equal]{d} \arrow{ddl}{\sim} & s_* s^{\dagger} \iota_Y (\pi_Y(B)) \arrow{d}{\tilde{\theta}} \\
			\iota_Y(s_* \pi_S(s^{\dagger} B)) \arrow{u}{\bar{\tilde{\theta}}} \arrow[equal]{d} & \Bti_Y^* (s_* s^{\dagger} B) \arrow[equal]{dl} \arrow[equal]{d} \arrow{uu}{\bar{\theta}} & \check{C}_{Y,\Ucal}[-d] \iota_Y \pi_Y(B) \isoarrow{dd} \arrow[equal]{ur} & s_* \iota_S(s^{\dagger} \pi_Y(B)) \\
			\iota_Y \pi_Y (s_* s^{\dagger} B) \arrow[equal]{dr} & \Bti_Y^* (\check{C}_{Y,\Ucal}[-d] B) \arrow[equal]{d} && \iota_Y(s_* s^{\dagger} \pi_Y(B)) \arrow{u}{\bar{\tilde{\theta}}} \\
			& \iota_Y \pi_Y (\check{C}_{Y,\Ucal}[-d] B) \arrow{r}{\sim} & \iota_Y(\check{C}_{Y,\Ucal}[-d] \pi_Y(B)). \arrow[equal]{ur}
		\end{tikzcd}
	\end{equation*}
    Here, the top-left piece is commutative by \cite[Prop.~2.5(2)]{Ter23UAF}, the top-central piece and the bottom-right piece are commutative by construction, the bottom-central piece is commutative by \Cref{lem-cech-dist}, while the other two pieces are commutative by naturality.
    This concludes the proof.
\end{proof}

\subsection{Proof of the main result and applications}

We can finally compare the shifted tensor product of \Cref{constr:otimes-diamond} with the one obtained from the external tensor product, as announced at the beginning of this section:

\begin{proof}[Proof of \Cref{prop-distinguished-otimes}]
	We want to apply the abstract criterion of \cite[Cor.~4.5]{Ter23UAF}. To this end, we have to exhibit the existence of a natural isomorphism of functors $\Dcal^0(S) \times \Dcal^{0,loc}(S) \rightarrow \Mcal^0(S)$
	\begin{equation}\label{iso:prop-dist-otimes}
		\chi_S: (\pi_S(A) \otimes \pi_S(L))[-d] \xrightarrow{\sim} \pi_S((A \otimes L)[-d])
	\end{equation} 
	rendering the diagram of functors $\Dcal^0(S) \times \Dcal^{0,loc}(S) \rightarrow \Perv(S)$
	\begin{equation}\label{dia:prop-dist-otimes}
		\begin{tikzcd}[font=\small]
			(\iota_S \pi_S(A) \otimes \iota_S \pi_S(L))[-d] \arrow{dd}{\tilde{\rho}} \arrow[equal]{r} & (\Bti_S^*(A) \otimes \Bti_S^*(L))[-d] \arrow{r}{\rho} & \Bti_S^*(A \otimes L)[-d] \\
			&& \Bti_S^*((A \otimes L)[-d]) \arrow[equal]{u} \arrow[equal]{d} \\
			\iota_S(\pi_S(A) \otimes \pi_S(L))[-d] \arrow{rr}{\chi_S} && \iota_S \pi_S((A \otimes L)[-d])
		\end{tikzcd}
	\end{equation}
	commutative. We define the natural isomorphism \eqref{iso:prop-dist-otimes} as the composite
	\begin{equation*}
		\begin{tikzcd}
			(\pi_S(A) \otimes \pi_S(L))[-d] \arrow[equal]{d} && \pi((A \otimes L)[-d]), \arrow[equal]{d} \\
			\Delta_S^{\dagger}(\pi_S(A) \boxtimes \pi_S(L)) \arrow[equal]{r} & \Delta_S^{\dagger} \pi_{S \times S}(A \boxtimes L) & \pi_S(\Delta_S^{\dagger}(A \boxtimes L)) \arrow{l}{\zeta_{\Delta_S}}
		\end{tikzcd}
	\end{equation*}
	where the isomorphism
	\begin{equation}\label{zeta:AL}
		\zeta_{\Delta_S}: \pi_S(\Delta_S^{\dagger}(A \boxtimes L)) \rightarrow \Delta_S^{\dagger} \pi_{S \times S}(A \boxtimes L)
	\end{equation}  
	is provided by specializing the natural transformation \eqref{formula:zeta_S} to the object $A \boxtimes L$. The fact that \eqref{zeta:AL} is actually invertible follows from \Cref{cor:s^diamond-pi}, since the object $A \boxtimes L$ belongs to $\Dcal^{0,\Delta_S}(S \times S)$ by construction: this is just a rephrasing of the discussion in \Cref{constr:otimes-diamond}. 
	In order to check the commutativity of \eqref{dia:prop-dist-otimes}, it suffices to express all the arrows in terms of the external tensor product: in this way, we obtain the outer part of the diagram
	\begin{equation*}
		\begin{tikzcd}[font=\small]
			\Delta_S^*[-d] (\iota_S \pi_S(A) \boxtimes \iota_S \pi_S(L)) \arrow[equal]{r} \arrow{d}{\tilde{\rho}} & \Delta_S^*[-d] (\Bti_S^*(A) \boxtimes \Bti_S^*(L)) \arrow{r}{\rho} & \Delta_S^*[-d] \Bti_{S \times S}^* (A \boxtimes L) \arrow{d}{\tilde{\theta}} \arrow[equal]{dl} \\
			\Delta_S^*[-d] \iota_{S \times S}(\pi_S(A) \boxtimes \pi_S(L)) \arrow{d}{\tilde{\theta}} \arrow[equal]{r} & \Delta_S^*[-d] \iota_{S \times S} \pi_{S \times S}(A \boxtimes L) \arrow{d}{\tilde{\theta}} & \Bti_S^*(\Delta_S^*(A \boxtimes L)[-d]) \arrow[equal]{d} \\
			\iota_S(\Delta_S^*(\pi_S(A) \boxtimes \pi_S(L)))[-d] \arrow[equal]{d} \arrow[equal]{r} & \iota_S(\Delta_S^* \pi_{S \times S}(A \boxtimes L))[-d] \arrow[equal]{d} & \Bti_S^* (\Delta_S^*[-d](A \boxtimes L)) \arrow[equal]{d} \\
			\iota_S(\Delta_S^*[-d] (\pi_S(A) \boxtimes \pi_S(L))) \arrow[equal]{r} & \iota_S(\Delta_S^*[-d] \pi_{S \times S}(A \boxtimes L)) & \iota_S \pi_S (\Delta_S^*[-d](A \boxtimes L)), \arrow{l}{\zeta_{\Delta_S}}
		\end{tikzcd}
	\end{equation*}
	where the right-most piece is commutative by \Cref{cor:s^diamond-pi}, the top piece is commutative by \cite[Prop.~4.3(2)]{Ter23UAF}, while the other pieces are commutative by naturality. This concludes the proof.
\end{proof}

\begin{rem}\label{rem:sym-prop-dist-otimes}
	Of course, the result of \Cref{prop-distinguished-otimes} admits a symmetric version where the roles of the two variables are exchanged: there exists a canonical isomorphism of functors $\Mcal^{0,loc}(S) \times \Mcal^0(S) \rightarrow \Mcal^0(S)$
	\begin{equation*}
		L \otimes^{\dagger} M \xrightarrow{\sim} (L \otimes M)[-d]
	\end{equation*}
	such that the diagram of functors $\Mcal^{0,loc}(S) \times \Mcal^0(S) \rightarrow \Perv(S)$
	\begin{equation*}
		\begin{tikzcd}
			\iota_S(L \otimes^{\dagger} M) \arrow{r}{\sim} & \iota_S((L \otimes M)[-d]) \arrow[equal]{r} & \iota_S(L \otimes M)[-d] \\
			\iota_S(L) \otimes^{\dagger} \iota_S(M) \arrow{u}{\tilde{\rho}^{\dagger}} \arrow[equal]{rr} && (\iota_S(L) \otimes \iota_S(M))[-d] \arrow{u}{\rho}
		\end{tikzcd}
	\end{equation*}
	is commutative.
\end{rem}

Recall that both stable homotopy $2$-functors $\DA_{ct}^{\et}(\cdot,\Q)$ and $D^b_c(\cdot,\Q)$ are endowed with Verdier duality functors defined, over a given $k$-variety $S$, by the formula 
\begin{equation*}
	\Dbb_S(A) := \Homcal_S(A, a_S^! \unit_{\Spec(k)}),
\end{equation*} 
where $a_S: S \rightarrow \Spec(k)$ denotes the structural projection. Since the Betti realization functors $\Bti_S^*: \DA_{ct}^{\et}(S,\Q) \rightarrow D^b(S,\Q)$ commute with the six operations on the two sides, they also commute with Verdier duality in a canonical way. The fact that $\Dbb_S$ defines a duality means that, in both cases, the canonical arrow
\begin{equation*}
	A \rightarrow \Dbb_S(\Dbb_S(A))
\end{equation*}
is invertible. Using this, one can deduce a canonical identification
\begin{equation}\label{formula:homcal-Verdier}
	\Homcal_S(A,B) \xrightarrow{\sim} \Dbb_S(A \otimes \Dbb_S(B))
\end{equation}
in both cases; again, these are compatible with one another under the Betti realization.
Since the equivalence 
\begin{equation*}
	\Dbb_S: D^b_c(S,\Q) \xrightarrow{\sim} D^b_c(S,\Q)^{op}
\end{equation*}
is $t$-exact for the perverse $t$-structures, it restricts to an equivalence
\begin{equation*}
	\Dbb_S: \Perv(S) \xrightarrow{\sim} \Perv(S)^{op}
\end{equation*}
and so, by \cite[\S~2.8]{IM19}, it canonically lifts to an equivalence
\begin{equation*}
	\Dbb_S: \Mcal(S) \xrightarrow{\sim} \Mcal(S)^{op},
\end{equation*}
which one might still call Verdier duality. Note that these functors are obtained for free without even the need to know that internal homomorphisms of perverse motives actually exist! In fact, in the setting of perverse motives one might try to use the formula \eqref{formula:homcal-Verdier} as a way to define internal homomorphisms. In this way one would obtain, over every $k$-variety $S$, a bi-triangulated functor
\begin{equation}\label{formula:homcal-VerdierM}
	\Homcal_S(-,-): D^b(\Mcal(S))^{op} \times D^b(\Mcal(S)) \rightarrow D^b(\Mcal(S))
\end{equation} 
endowed with a canonical natural isomorphism of functors $D^b(\Mcal(S))^{op} \times D^b(\Mcal(S)) \rightarrow D^b(\Perv(S))$ 
\begin{equation*}
	\iota_S \Homcal_S(M_1^{\bullet},M_2^{\bullet}) \xrightarrow{\sim} \Homcal_S(\iota_S(M_1^{\bullet}),\iota_S(M_2^{\bullet})).
\end{equation*}
However, with this approach it would remain unclear whether the functor \eqref{formula:homcal-VerdierM} defines a right adjoint to the internal tensor product: the issue is that, for a general object $M^{\bullet} \in D^b(\Mcal(S))$, none of the two functors
\begin{equation*}
	- \otimes M^{\bullet}, \; \Homcal_S(M^{\bullet},-): D^b(\Mcal(S)) \rightarrow D^b(\Mcal(S))
\end{equation*}
is obtained directly from the lifting principles of universal abelian factorizations. However, the result of \Cref{prop-distinguished-otimes} puts us in the desired situation at least for some choices of $M^{\bullet}$. This is what we want to prove now. 

Recall that, for any additive category $\Dcal$, we let $\Rbf(\Dcal)$ denote the category of finitely presented abelian presheaves on $\Dcal$ (i.e. cokernels of representable presheaves), and we let $\A(\Dcal)$ denote Freyd's abelian hull of $\Dcal$. 
By construction, every object of $\A(\Dcal)$ can be written as the kernel of an arrow in $\Rdf(\Dcal)$. 
Moreover, for every additive functor $\beta: \Dcal \rightarrow \Acal$ into an abelian category $\Acal$, the corresponding universal abelian factorization $\A(\beta)$ is defined as a Serre quotient of $\A(\Dcal)$, and so the image of $\Rbf(\Dcal)$ in $\A(\beta)$ generates the whole of $\A(\beta)$ under kernels. 

\begin{cor}\label{cor:adj-Mloc}
	Let $S$ be a smooth, connected $k$-variety of dimension $d$. Let $L \in \Mcal^{0,loc}(S)$, and suppose that $L$ lies in the image of $\mathbf{R}(\Dcal^{0,loc}(S))$ under the quotient functor $\pi_S: \A(\Dcal^{0,loc}(S)) \rightarrow \Mcal^{0,loc}(S)$. Then there is a canonical adjunction
	\begin{equation*}
		- \otimes L[-d] : \Mcal^0(S) \rightleftarrows \Mcal^0(S): \Dbb_S(L[-d] \otimes \Dbb_S(-))
	\end{equation*}
	compatible with the corresponding adjunction on the underlying categories of perverse sheaves (in the sense of \cite[Defn.~3.6]{Ter23UAF}).
\end{cor}
\begin{proof}
	As a consequence of \Cref{prop-distinguished-otimes} and \Cref{rem:sym-prop-dist-otimes}, it suffices to construct, for every $L$ as in the statement, an adjunction
	\begin{equation}\label{adj:L-otimesdiamond}
		- \otimes^{\dagger} L: \Mcal^0(S) \leftrightarrows \Mcal^0(S): \Dbb_S(L \otimes^{\dagger} \Dbb_S(-))
	\end{equation}
	compatible with the corresponding adjunction on the underlying categories of perverse sheaves. Using the faithfulness of $\iota_S: \Mcal^0(S) \hookrightarrow \Perv(S)$, we see that such an adjunction (or, more precisely, its associated unit and co-unit transformations) is uniquely determined, provided it exists.
 
	To construct the adjoint pair in the statement, we divide the argument into two steps of increasing generality.
	Suppose first that $L = \pi_S(A)$ for some $A \in \Dcal^{0,loc}(S)$. In this case, we obtain the sought-after adjunction by applying \cite[Cor.~3.9]{Ter23UAF} to the diagrams
	\begin{equation*}
		\begin{tikzcd}[font=\small]
			\Dcal^0(S) \arrow{rr}{- \otimes^{\dagger} A} \arrow{d}{\beta^0_S} && \Dcal^0(S) \arrow{d}{\beta^0_S} \\
			\Perv(S) \arrow{rr}{- \otimes^{\dagger} \beta_S^0(A)} && \Perv(S)
		\end{tikzcd}
	\end{equation*}
	and
	\begin{equation*}
		\begin{tikzcd}[font=\small]
			\Dcal^0(S) \arrow{rr}{\Dbb_S} \arrow{d}{\beta^0_S} && \Dcal^0(S) \arrow{rr}{A \otimes^{\dagger} -} \arrow{d}{\beta^0_S} && \Dcal^0(S) \arrow{rr}{\Dbb_S} \arrow{d}{\beta^0_S} && \Dcal^0(S) \arrow{d}{\beta^0_S} \\
			\Perv(S) \arrow{rr}{\Dbb_S} && \Perv(S) \arrow{rr}{\beta_S^0(A) \otimes^{\dagger} -} && \Perv(S) \arrow{rr}{\Dbb_S} && \Perv(S).
		\end{tikzcd}
	\end{equation*}
	Note that, as a consequence of \cite[Rem.~3.5(2)]{Ter23UAF}, the natural bijections
	\begin{equation*}
		\Hom_{\Mcal^0(S)}(M \otimes^{\dagger} \pi_S(A), N) = \Hom_{\Mcal^0(S)}(M, \Dbb_S(\pi_S(A) \otimes^{\dagger} \Dbb_S(N)))
	\end{equation*}
	are functorial with respect to $A \in \Dcal^{0,loc}(S)$.
	
	In general, write $L = \pi_S(C)$ for some $C \in \mathbf{R}(\Dcal^{0,loc}(S))$, and choose a presentation of the form $C = \coker\left\{\pi_S(A) \rightarrow \pi_S(B)\right\}$ for some morphism $A \rightarrow B$ in $\Dcal^{0,loc}(S)$. Then we obtain a natural bijection
	\begin{equation*}
		\Hom_{\Mcal^0(S)}(M \otimes^{\dagger} L, N) = \Hom_{\Mcal^0(S)}(M, \Dbb_S(L \otimes^{\dagger} \Dbb_S(N)))
	\end{equation*}
	by combining the two chains of bijections
	\begin{equation*}
		\begin{aligned}
			&\;\quad \Hom_{\Mcal^0(S)}(M \otimes^{\dagger} \pi_S(C), N) \\
			&= \Hom_{\Mcal^0(S)}(M \otimes^{\dagger} \coker\left\{\pi_S(A) \rightarrow \pi_S(B) \right\},N) \\
			&= \Hom_{\Mcal^0(S)}(\coker\left\{M \otimes^{\dagger} \pi_S(A) \rightarrow M \otimes^{\dagger} \pi_S(B) \right\},N) \\
			&= \ker\left\{\Hom_{\Mcal^0(S)}(M \otimes^{\dagger} \pi_S(B),N) \rightarrow \Hom_{\Mcal^0(S)}(M \otimes^{\dagger} \pi_S(A),N) \right\} 
		\end{aligned}
	\end{equation*}
	and
	\begin{equation*}
		\begin{aligned}
			&\;\quad \Hom_{\Mcal^0(S)}(M, \Dbb_S(\pi_S(C) \otimes^{\dagger} \Dbb_S(N))) \\
			&= \Hom_{\Mcal^0(S)}(M,\Dbb_S(\coker\left\{\pi_S(A) \rightarrow \pi_S(B) \right\} \otimes^{\dagger} \Dbb_S(N))) \\
			&= \Hom_{\Mcal^0(S)}(M,\Dbb_S(\coker\left\{\pi_S(A) \otimes^{\dagger} \Dbb_S(N) \rightarrow \pi_S(B) \otimes^{\dagger} \Dbb_S(N) \right\})) \\
			&= \Hom_{\Mcal^0(S)}(M,\ker\left\{\Dbb_S(\pi_S(B) \otimes^{\dagger} \Dbb_S(N)) \rightarrow \Dbb_S(\pi_S(A) \otimes^{\dagger} \Dbb_S(N)) \right\}) \\
			&= \ker\left\{\Hom_{\Mcal^0(S)}(M,\Dbb_S(\pi_S(B) \otimes^{\dagger}\Dbb_S(N))) \rightarrow \Hom_{\Mcal^0(S)}(M,\Dbb_S(\pi_S(A) \otimes^{\dagger} \Dbb_S(N)))\right\},
		\end{aligned}
	\end{equation*}
	both of which use the adjunctions obtained in the special case above. It is clear from the construction that the adjunction obtained in this way is compatible with the corresponding adjunction on perverse sheaves.
\end{proof}

In the final part of \Cref{sect_Homcal}, we explain how to prove the same result in full generality:
the precise statement is \Cref{prop:MVerdier}, which is based on the construction of internal homomorphisms to come.
In turn, the crucial ingredient in this construction is precisely the partial result of \Cref{cor:adj-Mloc}.

\section{Construction of the unit constraint}\label{sect_unit-Mp}

In \Cref{prop-extcore-Mp}, we have constructed a monoidal structure on perverse motives endowed with compatible associativity and commutativity constraints. 
In order to complete the proof of \Cref{thm-boxtimesM}, we have to construct a unit constraint compatible with the rest of the structure:
this is achieved in the present section, building on the auxiliary results of \Cref{sect_otimes-dist}. 
Using the extension techniques developed in \cite{Ter23Emb}, we can complete our construction by working with smooth varieties only: this allows us to avoid annoying technical issues related to unit objects over singular varieties. 
Over smooth varieties, we can use the shifted inverse image functors and shifted tensor products, defined in \Cref{nota-diamond1} and \Cref{constr:otimes-diamond}, respectively.

As usual, we use the notation "$\Mcal^0$" in place of "$\Mcal$" in those passages that depend closely on the construction of the external tensor product, but not in the statement of the main result:
\begin{prop}\label{thm:unit-constr-Mp}
	The internal tensor structure on $D^b(\Mcal(\cdot))$ constructed in  \Cref{prop-extcore-Mp} carries a canonical internal unit constraint in the sense of \cite[Defn.~5.1]{Ter23Fib}, with respect to which the morphism of $\Var_k$-fibered categories $\iota: D^b(\Mcal(\cdot)) \rightarrow D^b(\Perv(\cdot))$ is canonically unitary in the sense of \cite[Defn.~8.20]{Ter23Fib}.
\end{prop}

\begin{nota}
	We adopt the following notation:
	\begin{itemize}
		\item We use the symbol $\unit = (\unit_S)_{S \in \Var_k}$ to denote the canonical unit sections of the monoidal $\Var_k$-fibered categories $\DA_{ct}^{\et}(\cdot,\Q)$ and $D^b_c(\cdot,\Q)$.
		\item We let $w: \Bti^*(\unit) \xrightarrow{\sim} \unit$ denote the canonical isomorphism of unit sections defined in \cite{Ayo10}.
		\item We use the symbol $u$ to denote the internal unit constraints of $\DA_{ct}^{\et}(\cdot,\Q)$ and $D^b_c(\cdot,\Q)$.
	\end{itemize}
\end{nota}

\subsection{Construction of the unit section}

In the first place, we need to define the unit section of the $\Var_k$-fibered category $D^b(\Mcal(\cdot))$. 

\begin{prop}\label{prop-unitMcal}
	The $\Var_k$-fibered category $D^b(\Mcal(\cdot))$ possesses a canonical unit section $\unit^{\Mcal} = (\unit^{\Mcal}_S)_{S \in \Var_k}$, and the morphism of $\Var_k$-fibered categories $\iota: D^b(\Mcal(\cdot)) \rightarrow D^b(\Perv(\cdot))$ induces a canonical isomorphism of unit sections
	\begin{equation*}
		\tilde{w}: \iota(\unit^{\Mcal}) \xrightarrow{\sim} \unit.
	\end{equation*}
\end{prop}

The basic technical issue is that, over a singular $k$-variety $S$, the unit sheaf $\unit_S \in D^b_c(S,\Q)$ is not a perverse sheaf (up to shiftings); in fact, a priori it is not even obvious that the unit sheaf belongs to the essential image of the triangulated functor $\iota: D^b(\Mcal(S)) \rightarrow D^b(\Perv(S))$. 
To overcome this issue, we use the embedding method of \cite[\S~1]{Ter23Emb} for fibered categories over $\Var_k$, with respect to the full subcategory $\Sm_k$ of smooth varieties. 

The most difficult part in the proof of \Cref{prop-unitMcal} is to construct the isomorphisms witnessing the compatibility of the motivic unit objects with inverse images. Using the factorization method of \cite{Ter23Fact}, we easily reduce to the cases of smooth morphisms and closed immersions, which we treat separately. While the case of smooth morphisms is straightforward, the case of closed immersions is quite tricky.
Using an induction argument, we reduce to the case of closed immersions of codimension $1$, which can be treated directly via localization triangles; it is plausible that the \v{C}ech complex constructions described in \Cref{sect_otimes-dist} could lead to an alternative approach, but we have not tried to follow this direction. 

Our preliminary results can be summarized as follows:
\begin{lem}\label{lem:sect-perv}
	Regard $D^b(\Mcal(\cdot))$ and $D^b(\cdot,\Q)$ as triangulated $\Sm_k$-fibered categories with inverse image functors $f^{\dagger}$ as in \Cref{nota-diamond1}.
	
	Let $K$ be a section of the $\Sm_k$-fibered category $D^b(\cdot,\Q)$ such that, for every smooth, connected $k$-variety $S$, the object $K_S \in D^b(S,\Q)$ lies in $\Perv(S)$. 
	Suppose that we are given, for every smooth $k$-variety $S$, an object $K^{\Mcal}_S \in \Mcal(S)$ together with an isomorphism in $\Perv(S)$
	\begin{equation}\label{w_S:lem}
		w_S: \iota_S(K_S^{\Mcal}) \xrightarrow{\sim} K_S.
	\end{equation}
    Then the following statements hold:
    \begin{enumerate}
    	\item Fix a morphism $f: T \rightarrow S$ between smooth $k$-varieties. Then there is at most one morphism in $D^b(\Mcal(T))$
    	\begin{equation}\label{KM*}
    		K^{\Mcal,\dagger}_f: f^{\dagger} K^{\Mcal}_S \rightarrow K^{\Mcal}_T
    	\end{equation}
    	making the diagram in $\Perv(T)$
    	\begin{equation*}
    		\begin{tikzcd}
    			f^{\dagger} \iota_S(K^{\Mcal}_S) \arrow{r}{\tilde{\theta}_f} \arrow{d}{\tilde{w}_S} & \iota_T(f^{\dagger} K^{\Mcal}_S) \arrow{r}{K_f^{\Mcal,\dagger}} & \iota_T(K^{\Mcal}_T) \arrow{d}{\tilde{w}_T} \\
    			f^{\dagger} K_S \arrow{rr}{K_f^{\dagger}} && K_T
    		\end{tikzcd}
    	\end{equation*}
    	commute; if this is the case, it is automatically an isomorphism.
    	\item Suppose that, for every morphism $f: T \rightarrow S$ between smooth $k$-varieties, the isomorphism \eqref{KM*} exists. 
    	Then the isomorphisms \eqref{KM*} turn the collection of objects $\left\{K^{\Mcal}_S\right\}_{S \in \Sm_k}$ into a section $K$ of the $\Sm_k$-fibered category $D^b(\Mcal(\cdot))$, and the morphisms \eqref{w_S:lem} define an isomorphism between sections of $D^b(\Perv(\cdot))$
    	\begin{equation}\label{w:lem}
    		w: \iota(K^{\Mcal}) \xrightarrow{\sim} K.
    	\end{equation}
        \item Suppose that the isomorphism \eqref{KM*} is known to exist whenever $f$ is either a smooth morphism or a closed immersion of codimension $1$ between smooth connected $k$-varieties. Then the natural isomorphism \eqref{KM*} exists for every morphism $f: T \rightarrow S$ between smooth $k$-varieties.
    \end{enumerate}
\end{lem}
\begin{proof}
	\begin{enumerate}
		\item The uniqueness of $K^{\Mcal,\dagger}_f$ follows from the fact that $\iota_T: \Mcal(T) \rightarrow \Perv(T)$ is faithful, while its invertibility follows from the fact that $\iota_T$ is also conservative (being exact).
		\item We only need to check that the isomorphisms \eqref{KM*} are compatible with composition: given two composable morphisms $f: T \rightarrow S$ and $g: S \rightarrow V$ of smooth $k$-varieties, we have to show that the diagram in $\Mcal(T)$
	    \begin{equation*}
	    	\begin{tikzcd}
	    		f^{\dagger} g^{\dagger} K^{\Mcal}_V \arrow[equal]{d} \arrow{r}{K_g^{\Mcal,\dagger}} & f^{\dagger} K^{\Mcal}_S \arrow{d}{K_f^{\Mcal,\dagger}} \\
	    		(gf)^{\dagger} K^{\Mcal}_V \arrow{r}{K_{gf}^{\Mcal,\dagger}} & K^{\Mcal}_T
	    	\end{tikzcd}
	    \end{equation*}
        is commutative or, equivalently, that its image under the faithful functor $\iota_T: \Mcal(T) \rightarrow \Perv(T)$
	    \begin{equation*}
	    	\begin{tikzcd}[font=\small]
	    		\iota_T(f^{\dagger} g^{\dagger} K^{\Mcal}_V) \arrow[equal]{d} \arrow{r}{K_g^{\Mcal,\dagger}} & \iota_T(f^{\dagger} K^{\Mcal}_S) \arrow{d}{K_f^{\Mcal,\dagger}} \\
	    		\iota_T((gf)^{\dagger} K^{\Mcal}_V) \arrow{r}{K_{gf}^{\Mcal,\dagger}} & \iota_T(K^{\Mcal}_T).
	    	\end{tikzcd}
	    \end{equation*}
        is commutative. But the latter coincides with the outer rectangle of the diagram
	    \begin{equation*}
	    	\begin{tikzcd}[font=\small]
	    		\iota_T(f^{\dagger} g^{\dagger} K_V^{\Mcal}) \arrow{rrrr}{K_g^{\Mcal,\dagger}} \arrow[equal]{dddd} &&&& \iota_T(f^{\dagger} K_S^{\Mcal}) \arrow{dddd}{K_f^{\Mcal,\dagger}} \\
	    		& f^{\dagger} \iota_S(g^{\dagger} K_V^{\Mcal}) \arrow{ul}{\tilde{\theta}_f} \arrow{rr}{K_g^{\Mcal,\dagger}} && f^{\dagger} \iota_S(K^{\Mcal}_S) \arrow{ur}{\tilde{\theta}_f} \arrow{d}{w_S} \\
	    		& f^{\dagger} g^{\dagger} \iota_V(K^{\Mcal}_V) \arrow{u}{\tilde{\theta}_g} \arrow{r}{\tilde{w}_V} \arrow[equal]{d} & f^{\dagger} g^{\dagger} K_V \arrow{r}{K_g^{\dagger}} \arrow[equal]{d} & f^{\dagger} K_S \arrow{d}{K_f^{\dagger}} \\
	    		& (gf)^{\dagger} \iota_V(K^{\Mcal}_V) \arrow{r}{\tilde{w}_V} \arrow{dl}{\tilde{\theta}_{gf}} & (gf)^{\dagger} K_V \arrow{r}{K_{gf}^{\dagger}} & K_T \\
	    		\iota_T((gf)^{\dagger} K^{\Mcal}_V) \arrow{rrrr}{K_{gf}^{\Mcal,\dagger}} &&&& \iota_Q(K^{\Mcal}_T), \arrow{ul}{w_T}
	    	\end{tikzcd}
	    \end{equation*}
	    where the three five-term pieces are commutative by construction while all the other inner pieces are commutative by naturality.
		\item As one can easily deduce from the proof of the previous point, given two composable morphisms $f: T \rightarrow S$ and $g: S \rightarrow V$ such that the isomorphisms $K^{\Mcal,\dagger}_f$ and $K^{\Mcal,\dagger}_g$ are already known to exist, the natural isomorphism $K^{\Mcal,\dagger}_{gf}$ exists as well and equals the composite
		\begin{equation*}
			(gf)^{\dagger} K^{\Mcal}_V = f^{\dagger} g^{\dagger} K^{\Mcal}_V \xrightarrow{K^{\Mcal,\dagger}_g} f^{\dagger} K^{\Mcal}_S \xrightarrow{K^{\Mcal,\dagger}_f} K^{\Mcal}_T.		
		\end{equation*}
	    Hence, it suffices to know that the isomorphism $K^{\Mcal,\dagger}_f$ exists whenever $f$ is a smooth morphism or a closed immersion. Even better, once one knows that $K^{\Mcal,\dagger}_p$ exists for every smooth morphism $p$ (in particular, for every open immersion), it suffices to show that the isomorphism $K^{\Mcal,\dagger}_z$ exists for every closed immersion $z$ between smooth, connected $k$-varieties.
	    
	    We are left to show that, in order to know the existence of $K^{\Mcal,\dagger}_z$ for every closed immersion $z: Z \hookrightarrow S$ between smooth connected $k$-varieties, it suffices to know that it exists whenever $z$ is a closed immersion of pure codimension $1$. Starting from a closed immersion $z: Z \hookrightarrow S$, it is always possible to find a finite open covering $\left\{S_i \subset S \right\}_{i \in I}$ such that each closed immersion $z_i: z^{-1}(S_i) \hookrightarrow S_i$ can be written as a composite of closed immersions of pure codimension $1$. In this situation, for each index $i \in I$ form the Cartesian square 
	    \begin{equation*}
	    	\begin{tikzcd}
	    		Z_i \arrow{r}{z_i} \arrow{d}{u_{Z,i}} & S_i \arrow{d}{u_i} \\
	    		Z \arrow{r}{z} & S.
	    	\end{tikzcd}
	    \end{equation*}
        Arguing as in the proof of the previous point, one easily sees that, if the four isomorphism $K^{\Mcal,\dagger}$, $K^{\Mcal,\dagger}$, $K^{\Mcal,\dagger}$, $K^{\Mcal,\dagger}$ all exist, then they make the diagram in $\Mcal(Z_i)$
        \begin{equation*}
        	\begin{tikzcd}
        		z_i^{\dagger} u_i^{\dagger} K_S \arrow{rr}{K^{\Mcal,\dagger}_{u_i}} \arrow[equal]{dd} && z_i^{\dagger} K_{S_i} \arrow{d}{K^{\Mcal,\dagger}_{z_i}} \\
        		&& K_{Z_i} \\
        		u_{Z,i}^{\dagger} z^{\dagger} K_S \arrow{rr}{K^{\Mcal,\dagger}_z} && u_{Z,i}^{\dagger} K_Z \arrow{u}{K^{\Mcal,\dagger}_{u_{Z,i}}}
        	\end{tikzcd}
        \end{equation*}
        commute. Conversely, since the $\Sm\Op(S)$-fibered categories $\Mcal(\cdot)$ and $\Perv(\cdot)$ are stacks for the Zariski topology, the existence of the isomorphisms $K^{\Mcal,\dagger}_{z_i}$, $K^{\Mcal,\dagger}_{u_i}$ and $K^{\Mcal,\dagger}_{u_{Z,i}}$ for all $i \in I$ implies the existence of $K^{\Mcal,\dagger}_z$. This proves the claim and concludes the proof.
	\end{enumerate}
\end{proof}

\begin{proof}[Proof of \Cref{prop-unitMcal}]
	By \cite[Prop.~1.7, Lemma~2.7]{Ter23Emb}, it suffices to construct the motivic unit section $\unit^{\Mcal}$ and the isomorphism $w: \iota(\unit^{\Mcal}) \xrightarrow{\sim} \unit$ over $\Sm_k$. To this end, we apply the criterion of \Cref{lem:sect-perv} taking $K$ to be the shifted unit section of the $\Sm_k$-fibered category $D^b(\Perv(\cdot))$: for every connected $S \in \Sm_k$ we set $K_S := \unit_S[\dim(S)]$, and we extend the definition to general $S \in \Sm_k$ in the obvious way. In a similar manner, we define the shifted motivic unit objects: if $S$ is connected we set $K^{\Mcal}_S := \pi_S(\unit_S[\dim(S)])$, and we extend the definition to general $S$ in the obvious way. Note that we have a canonical isomorphism in $\Perv(S)$
	\begin{equation}\label{w_S}
		\tilde{w}_S: \iota_S(K^{\Mcal}_S) \xrightarrow{\sim} K_S,
	\end{equation}
	which is defined as the composite
	\begin{equation*}
		\iota_S(K^{\Mcal}_S) := \iota_S \pi_S(\unit_S[\dim(S)]) = \Bti_S^*(\unit_S[\dim(S)]) = \unit_S[\dim(S)] =: K_S
	\end{equation*}
    in the case when $S$ is connected and by the obvious extension in the general case. Finally, we define the unshifted motivic unit objects by setting $\unit^{\Mcal}_S := K^{\Mcal}_S[-\dim(S)]$ when $S$ is connected, and extending the definition to general $S$ in the obvious way. As above, we have a canonical isomorphism in $D^b(\Perv(S))$
    \begin{equation}\label{w_S:unit}
    	\iota_S(\unit^{\Mcal}_S) \xrightarrow{\sim} \unit_S.
    \end{equation}
    Clearly, turning the collection $(\unit^{\Mcal}_S)_{S \in \Sm_k}$ into a section of $D^b(\Mcal(\cdot))$ with respect to the usual inverse image functors $f^*$ in such a way that the isomorphisms \eqref{w_S:unit} define an isomorphism of sections is the same as turning the collection $(K^{\Mcal}_S)_{S \in \Sm_k}$ into a section of $D^b(\Mcal(\cdot))$ with respect to the shifted inverse image functors $f^{\dagger}$ in such a way that the isomorphisms \eqref{w_S} define an isomorphism of sections.
    
    By \Cref{lem:sect-perv}(3), in order to achieve the latter task, it suffices to construct the isomorphisms $K^{\Mcal,\dagger}_f$ in the cases where $f$ is either a smooth morphism or a closed immersion of codimension $1$. 
    
    In the case of a smooth morphism $p: P \rightarrow S$, we define the isomorphism in $\Mcal(P)$
    \begin{equation*}
    	K^{\Mcal,\dagger}_p: p^{\dagger} K^{\Mcal}_S \xrightarrow{\sim} K^{\Mcal}_P
    \end{equation*}  
    by taking the composite
    \begin{equation*}
    	p^{\dagger} K^{\Mcal}_S := p^{\dagger} \pi_S(\unit_S[\dim(S)]) = \pi_P(p^{\dagger} \unit_S[\dim(S)]) \xrightarrow{\unit_p^{\dagger}} \pi_P(\unit_P[\dim(P)]) =: K^{\Mcal}_P.
    \end{equation*}
    The associated diagram in $\Perv(P)$
    \begin{equation*}
    	\begin{tikzcd}[font=\small]
    		p^{\dagger} \iota_S(K^{\Mcal}_S) \arrow{r}{\tilde{\theta}_p} \arrow{d}{\tilde{w}_S} & \iota_P(p^{\dagger} K^{\Mcal}_S) \arrow{r}{K_p^{\Mcal,\dagger}} & \iota_P(K^{\Mcal}_P) \arrow{d}{\tilde{w}_P} \\
    		p^{\dagger} K_S \arrow{rr}{K_p^{\dagger}} && K_P,
    	\end{tikzcd}
    \end{equation*}
    which can be written in the more explicit form
    \begin{equation*}
    	\begin{tikzcd}[font=\small]
    		p^{\dagger} \Bti_S^*(\unit_S[\dim(S)]) \arrow{r}{\theta_p} \arrow{d}{w_S} & \Bti_P^*(p^{\dagger} \unit_S[\dim(S)]) \arrow{r}{\unit_p^{\dagger}} & \Bti_P^*(\unit_P[\dim(P)]) \arrow{d}{w_P} \\
    		p^{\dagger} \unit_S[\dim(S)] \arrow{rr}{\unit_p^{\dagger}} && \unit_P[\dim(P)],
    	\end{tikzcd}
    \end{equation*}
    is commutative since $w: \Bti^*(\unit) \xrightarrow{\sim} \unit$ is a morphism of sections.
    
    In the case of a closed immersion $z: Z \hookrightarrow S$ of pure codimension $1$ between smooth $k$-varieties, in order to define the isomorphism in $\Mcal(Z)$
    \begin{equation*}
    	\unit_z^{\Mcal,*}: z^{\dagger} K_S^{\Mcal} \xrightarrow{\sim} K^{\Mcal}_Z,
    \end{equation*}
    we describe the morphism in $D^b(\Mcal(S))$
    \begin{equation}\label{unit_z-adj}
    	K^{\Mcal}_S \rightarrow z_* K^{\Mcal}_Z[1]
    \end{equation}
    corresponding to it under adjunction. To this end, let $u: U \hookrightarrow S$ denote the open immersion complementary to $z$. Consider the distinguished triangle in $\DA_{ct}^{\et}(S,\Q)$
    \begin{equation}\label{loc-rot-DA}
    	z_* \unit_Z[\dim(Z)] \rightarrow u_! u^* \unit_S[\dim(S)] \rightarrow \unit_S[\dim(S)] \xrightarrow{+}
    \end{equation}
    obtained from the usual localization triangle
    \begin{equation*}
    	u_! u^* \unit_S[\dim(S)] \rightarrow \unit_S[\dim(S)] \rightarrow z_* z^* \unit_S[\dim(S)] \xrightarrow{+}
    \end{equation*}
    after inserting the isomorphism $\unit_z^*: z^* \unit_S \xrightarrow{\sim} \unit_Z$ and rotating backwards. Note that, by construction, all the terms in \eqref{loc-rot-DA} belong to the subcategory $\Dcal^0(S)$ of \Cref{nota:Mp0}. Hence, applying the functor $\pi_S: \Dcal^0(S) \rightarrow \Mcal(S)$ and using the canonical isomorphisms $\beta_S \circ z_* = z_* \circ \beta_Z$ and $\beta_S \circ u_! u^* = u_! \circ \beta_U \circ u^* = u_! u^* \circ \beta_S$, we obtain a short exact sequence in $\Mcal(S)$
    \begin{equation*}
    	0 \rightarrow z_* K^{\Mcal}_Z \rightarrow u_! u^* K^{\Mcal}_S \rightarrow K^{\Mcal}_S \rightarrow 0.
    \end{equation*}
    We define \eqref{unit_z-adj} to be the second arrow of the distinguished triangle in $D^b(\Mcal^0(S))$
    \begin{equation*}
    	u_! u^* K^{\Mcal}_S \rightarrow K^{\Mcal}_S \rightarrow z_* K^{\Mcal}_Z[1] \xrightarrow{+}
    \end{equation*}
    obtained by taking the triangle associated to this short exact sequence and rotating forwards.
    By construction, showing the commutativity of the associated diagram in $\Perv(Z)$
    \begin{equation*}
    	\begin{tikzcd}[font=\small]
    		z^{\dagger} \iota_S(K^{\Mcal}_S) \arrow{d}{\tilde{w}_S} \arrow{r}{\tilde{\theta}_z} & \iota_Z(z^{\dagger} K^{\Mcal}_S) \arrow{r}{K^{\Mcal,\dagger}_z} & \iota_Z(K^{\Mcal}_Z) \arrow{d}{\tilde{w}_Z} \\
    		z^{\dagger} K_S \arrow{rr}{K_z^{\dagger}} && K_Z
    	\end{tikzcd}
    \end{equation*}
    is equivalent to showing the commutativity of the diagram
    \begin{equation*}
    	\begin{tikzcd}[font=\small]
    		\iota_S(K^{\Mcal}_S) \arrow{dd}{\tilde{w}_S} \arrow{r} & \iota_S(z_* K^{\Mcal}_Z)[1] \arrow{d}{\bar{\tilde{\theta}}_z} \\
    		& z_* \iota_Z(K^{\Mcal}_Z)[1] \arrow{d}{\tilde{w}_Z} \\
    		K_S \arrow{r} & z_* K_Z[1].
    	\end{tikzcd}
    \end{equation*}
    Since the latter can be written equivalently in the form 
    \begin{equation*}
    	\begin{tikzcd}[font=\small]
    		\Bti_S^*(\unit_S[\dim(S)]) \arrow{d}{\tilde{w}_S} \arrow{r} & z_* \Bti_Z(\unit_Z[\dim(Z)])[1] \arrow{d}{w_z} \\
    		\unit_S[\dim(S)] \arrow{r} & (z_* \unit_Z[\dim(Z)])[1],
    	\end{tikzcd}
    \end{equation*}
    its commutativity is in the end equivalent to that of the diagram
    \begin{equation*}
    	\begin{tikzcd}[font=\small]
    		z^* \Bti_S^*(\unit_S) \arrow{d}{w_S} \arrow{r}{\theta_z} & \Bti_Z^* (z^* \unit_S) \arrow{r}{\unit_z^*} & \Bti^*_Z(\unit_Z) \arrow{d}{w_z} \\
    		z^* \unit_S \arrow{rr}{\unit_z^*} && \unit_Z,
    	\end{tikzcd}
    \end{equation*}
    which indeed holds since $w: \Bti^*(\unit) \xrightarrow{\sim} \unit$ is a morphism of sections.
    This concludes the proof.
\end{proof}

\subsection{Construction of the unit constraint}

To complete the proof of \Cref{thm:unit-constr-Mp}, we have to show that the $\Var_k$-fibered category $D^b(\Mcal(\cdot))$ admits an internal unit constraint with respect to the motivic unit section constructed in \Cref{prop-unitMcal}.
As done in the previous subsection, we apply \cite[\S\S~3, 4]{Ter23Emb} to reduce the construction to smooth varieties, so that we can apply \Cref{prop-distinguished-otimes}.

In order to check that the motivic unit objects act compatibly with inverse image functors, we use the method of \cite{Ter23Fact} to reduce to the cases of smooth morphisms and closed immersions, which we treat separately. 
Again, the case of smooth morphisms is straightforward, while the case of closed immersions requires some preliminary work:
we need to rephrase the compatibility condition for inverse images under closed immersions in terms of the corresponding direct images.

For sake of brevity, we formulate our preliminary results only for the right unit action; of course, the same conclusions apply to the left unit action as well:
\begin{lem}\label{lem:red-ITS}
	Regard $D^b(\Mcal(\cdot))$ as a $\Sm_k$-fibered category, and consider the section $\unit^{\Mcal} = (\unit^{\Mcal}_S)_{S \in \Sm_k}$ constructed in \Cref{prop-unitMcal}. 
	Suppose that we are given, for every smooth $k$-variety $S$, a natural isomorphism of functors $D^b(\Mcal(S)) \rightarrow D^b(\Mcal(S))$
	\begin{equation*}
		u_r = u_{r,S}: M^{\bullet} \otimes \unit^{\Mcal}_S \xrightarrow{\sim} M^{\bullet}.
	\end{equation*}
    Then the following statements hold:
    \begin{enumerate}
    	\item For every morphism $f: T \rightarrow S$ between smooth $k$-varieties, let $D_f$ denote the diagram of functors $D^b(\Mcal(S)) \rightarrow D^b(\Mcal(T))$
    	\begin{equation*}
    		\begin{tikzcd}
    			f^* M^{\bullet} \otimes f^* \unit^{\Mcal}_S \arrow{r}{m} \arrow{d}{\unit_f^*} & f^* (M^{\bullet} \otimes \unit^{\Mcal}_S) \arrow{d}{u_r} \\
    			f^* M^{\bullet} \otimes \unit^{\Mcal}_T \arrow{r}{u_r} & f^* M^{\bullet}.
    		\end{tikzcd}
    	\end{equation*}
    	Then, in order for $D_f$ to be commutative for every morphism $f$ of smooth $k$-varieties, it suffices that it be commutative for every smooth morphism and for every closed immersion.
    	\item Let $z: Z \hookrightarrow S$ be a closed immersion between smooth $k$-varieties, and let $D'_z$ denote the diagram of functors $D^b(\Mcal(Z)) \rightarrow D^b(\Mcal(S))$
    	\begin{equation*}
    		\begin{tikzcd}
    			z_* N^{\bullet} \otimes z_* z^* \unit^{\Mcal}_S \arrow{d}{\unit_z^*} && z_* N^{\bullet} \otimes \unit^{\Mcal}_S \arrow{ll}{\eta} \arrow{d}{u_r} \\
    			z_* N^{\bullet} \otimes z_* \unit^{\Mcal}_Z \arrow{r}{\bar{m}} & z_* (N^{\bullet} \otimes \unit^{\Mcal}_Z) \arrow{r}{u_r} & z_* N^{\bullet}.
    		\end{tikzcd}
    	\end{equation*}
        Then $D_z$ is commutative if and only if $D'_z$ is commutative.
    \end{enumerate}
\end{lem}
\begin{proof}
	Both assertions follow by a standard application of the factorization method of \cite{Ter23Fact}. Since unit constraints are not discussed in \cite{Ter23Fact}, we include the argument for the reader's convenience:
	\begin{enumerate}
		\item  Given a morphism $f: T \rightarrow S$ between smooth $k$-varieties, choose a factorization of the form
		\begin{equation*}
			f: T \xrightarrow{t} P \xrightarrow{p} S
		\end{equation*} 
	    with $t$ a closed immersion and $p$ a smooth morphism; note that the $k$-variety $P$ is smooth, since so is $S$. By construction, the diagram $D_f$ coincides with the outer part of the diagram
		\begin{equation*}
			\begin{tikzcd}[font=\tiny]
				f^* M^{\bullet} \otimes f^* \unit^{\Mcal}_S \arrow{rrrr}{m} \arrow{dddd}{\unit^{\Mcal,*}} \arrow[equal]{dr} &&&& f^* (M^{\bullet} \otimes \unit^{\Mcal}_S) \arrow[equal]{dl} \arrow{dddd}{u_r} \\
				& t^* p^* M^{\bullet} \otimes t^* p^* \unit^{\Mcal}_S \arrow{r}{m} \arrow{d}{\unit^{\Mcal,*}} & t^* (p^* M^{\bullet} \otimes p^* \unit^{\Mcal}_S) \arrow{r}{m} \arrow{d}{\unit^{\Mcal,*}} & t^* p^* (M^{\bullet} \otimes \unit^{\Mcal}_S) \arrow{dd}{u_r} \\
				& t^* p^* M^{\bullet} \otimes t^* \unit_P \arrow{r}{m} \arrow{d}{\unit^{\Mcal,*}} & t^* (p^* M^{\bullet} \otimes \unit_P) \arrow{dr}{u_r} \\
				& t^* p^* M^{\bullet} \otimes \unit_T \arrow{rr}{u_r} && t^* p^* M^{\bullet} \\
				f^* M^{\bullet} \otimes \unit^{\Mcal}_T \arrow[equal]{ur} \arrow{rrrr}{m} &&&& f^* M^{\bullet} \arrow[equal]{ul}
			\end{tikzcd}
		\end{equation*}
	    where the two central trapezoids are essentially the diagrams $D_p$ and $D_z$. Since all the remaining pieces of the diagram are already known to be commutative, the thesis follows.
		\item Let $D''_z$ denote the diagram of functors $D^b(\Mcal(Z)) \rightarrow D^b(\Mcal(Z))$
		\begin{equation*}
			\begin{tikzcd}
				z^* z_* N^{\bullet} \otimes z^* \unit^{\Mcal}_S \arrow{r}{m} \arrow{d}{\unit_z^{\Mcal,*}} & z^* (z_* N^{\bullet} \otimes \unit^{\Mcal}_S) \arrow{d}{u_r} \\
				z^* z_* N^{\bullet} \otimes \unit^{\Mcal}_Z \arrow{r}{u_r} & z^* z_* N^{\bullet}.
			\end{tikzcd}
		\end{equation*}
	    We claim that both the commutativity of $D_z$ and that of $D'_z$ are equivalent to the commutativity of $D''_z$; this will imply the thesis.
	    
	    It is clear that the commutativity of $D_z$ implies that of $D''_z$. Conversely, $D_z$ can be written as the outer part of the diagram
	    \begin{equation*}
	    	\begin{tikzcd}[font=\small]
	    		z^* M^{\bullet} \otimes z^* \unit^{\Mcal}_S \arrow{rrr}{m} \arrow{dr}{\eta} \arrow{ddd}{\unit^*} &&& z^* (M^{\bullet} \otimes \unit^{\Mcal}_S) \arrow{ddd}{u_r} \arrow{dl}{\eta} \\
	    		& z^* z_* z^* M^{\bullet} \otimes z^* \unit^{\Mcal}_S \arrow{r}{m} \arrow{d}{\unit^*} & z^* (z_* z^* M^{\bullet} \otimes \unit^{\Mcal}_S) \arrow{d}{u_r} \\
	    		& z^* z_* z^* M^{\bullet} \otimes \unit^{\Mcal}_Z  \arrow{r}{u_r} & z^* z_* z^* M^{\bullet} \\
	    		z^* M^{\bullet} \otimes \unit^{\Mcal}_Z \arrow{rrr}{u_r} \arrow{ur}{\eta} &&& z^* M^{\bullet} \arrow{ul}{\eta}
	    	\end{tikzcd}
	    \end{equation*}
	    where the central square is essentially $D''_z$ and the lateral pieces are already known to be commutative. This shows that $D_z$ is commutative if and only if so is $D''_z$.
	    
	    The diagram $D'_z$ can be written as the outer part of the diagram
	    \begin{equation*}
	    	\begin{tikzcd}[font=\tiny]
	    		z_* N^{\bullet} \otimes z_* z^* \unit^{\Mcal}_S \arrow{dr}{\eta} \arrow{dd}{\unit^{\Mcal,*}} &&&& z_* N^{\bullet} \otimes \unit^{\Mcal}_S \arrow{llll}{\eta} \arrow{dd}{u_r} \arrow{dll}{\eta} \\
	    		& z_* z^* (z_* N^{\bullet} \otimes z_* z^* \unit^{\Mcal}_S) \arrow{dd}{\unit^{\Mcal,*}} & z_* z^* (z_* N^{\bullet} \otimes \unit^{\Mcal}_S) \arrow{l}{\eta} \arrow{dr}{u_r} \\
	    		z_* N^{\bullet} \otimes z_* \unit^{\Mcal}_Z \arrow{dr}{\eta} \arrow{dd}{\eta} && z_*(z^* z_* N^{\bullet} \otimes z^* \unit^{\bullet}_S) \arrow{u}{m} \arrow{dr}{\unit^{\Mcal,*}} & z_* z^* z_* N^{\bullet} \arrow{r}{\epsilon} & z_* B \\
	    		& z_* z^* (z_* N^{\bullet} \otimes z_* \unit^{\Mcal}_Z) & z_* (z^* z_* N^{\bullet} \otimes z^* z_* z^* \unit^{\Mcal}_S) \arrow{d}{\unit^{\Mcal,*}} \arrow{u}{\epsilon} \arrow{uul}{m} & z_* (z^* z_* N^{\bullet} \otimes \unit^{\bullet}_Z) \arrow{u}{u_r} \arrow{dr}{\epsilon} \\
	    		z_* z^* (z_* N^{\bullet} \otimes z_* \unit^{\bullet}_Z) && z_*(z^* z_* N^{\bullet} \otimes z^* z_* \unit^{\bullet}_Z) \arrow{ll}{m} \arrow{rr}{\epsilon} \arrow{ul}{m} \arrow{ur}{\epsilon} && z_* (N^{\bullet} \otimes \unit^{\bullet}_Z) \arrow{uu}{u_r}
	    	\end{tikzcd}
	    \end{equation*}
	    where all arrows are invertible by construction and all pieces except possibly the square
	    \begin{equation*}
	    	\begin{tikzcd}[font=\small]
	    		z_*(z^* z_* N^{\bullet} \otimes z^* \unit^{\Mcal}_S) \arrow{r}{m} \arrow{d}{\unit_z^{\Mcal,*}} & z_* z^* (z_* N^{\bullet} \otimes \unit^{\Mcal}_S) \arrow{d}{u_r} \\
	    		z_*(z^* z_* N^{\bullet} \otimes \unit^{\Mcal}_Z) \arrow{r}{u_r} & z_* z^* z_* N^{\bullet}
	    	\end{tikzcd}
	    \end{equation*} 
	    are commutative by naturality. Hence $D'_z$ is commutative if and only if the above square is commutative.
	    Since the functor $z_*: D^b(\Mcal(Z)) \rightarrow D^b(\Mcal(S))$ is fully faithful, this is the case if and only if $D''_z$ is commutative.
	    This proves the claim and concludes the proof.
    \end{enumerate}
\end{proof}

In order to be able to apply \Cref{lem:red-ITS}(2), we need to use the following variant of \Cref{prop-distinguished-otimes}:

\begin{prop}\label{rem:dist-otimes-gener-closed}
	Let $z: Z \hookrightarrow S$ be a closed immersion between smooth $k$-varieties. Then the bi-exact functor
	\begin{equation*}
		z_*(-) \otimes z_*(-[-\dim(Z)]) = (z_*(-) \otimes z_*(-))[-\dim(Z)]: \Mcal^0(Z) \times \Mcal^{0,loc}(Z) \rightarrow \Mcal^0(S)
	\end{equation*}
	is canonically isomorphic to the functor obtained by applying \cite[Prop.~4.3]{Ter23UAF} to the diagram
	\begin{equation*}
		\begin{tikzcd}[font=\small]
			\Dcal^0(Z) \times \Dcal^{0,loc}(Z) \arrow{rrrr}{(z_*(-) \otimes z_*(-))[-\dim(Z)]} \arrow{d}{\beta^0_Z \times \beta^{0,loc}_Z} &&&& \Dcal^0(S) \arrow{d}{\beta^0_S} \\
			\Perv(Z) \times \Perv(Z) \arrow{rrrr}{(z_*(-) \otimes z_*(-))[-\dim(Z)]} &&&& \Perv(S).
		\end{tikzcd}
	\end{equation*}
\end{prop}
\begin{proof}
	As the reader can check, the argument used for the proof of \Cref{prop-distinguished-otimes} goes through, mutatis mutandis; we leave the details to the interested reader.
\end{proof}

\begin{proof}[Proof of \Cref{thm:unit-constr-Mp}]
	By \cite[Prop.~3.9, Lemma~4.3(3)]{Ter23Emb}, it suffices to construct the motivic unit constraint and check the unitarity of $\iota: D^b(\Mcal(\cdot)) \rightarrow D^b(\cdot,\Q)$ (via the isomorphism $w: \iota(\unit^{\Mcal}) \xrightarrow{\sim} \unit$ constructed in \Cref{prop-unitMcal}) over $\Sm_k$. To this end, we apply the criterion of \Cref{lem:red-ITS}. 
	
	To begin with, fix a smooth $k$-variety $S$, and consider the two diagrams
	\begin{equation*}
		\begin{tikzcd}[font=\small]
			\Dcal^0(S) \arrow{rr}{- \otimes^{\dagger} \unit_S[\dim(S)]} \arrow{d}{\beta_S^0} && \Dcal^0(S) \arrow{d}{\beta_S^0} \\
			\Perv(S) \arrow{rr}{- \otimes^{\dagger} \unit_S[\dim(S)]} && \Perv(S) \\
			& \Downarrow \\
			\Dcal^0(S) \arrow{rr}{\id} \arrow{d}{\beta_S^0} && \Dcal^0(S) \arrow{d}{\beta_S^0} \\
			\Perv(S) \arrow{rr}{\id} && \Perv(S)
		\end{tikzcd}
		\qquad \qquad \qquad
		\begin{tikzcd}[font=\small]
			\Dcal^0(S) \arrow{rr}{\unit_S[\dim(S)] \otimes^{\dagger} -} \arrow{d}{\beta_S^0} && \Dcal^0(S) \arrow{d}{\beta_S^0} \\
			\Perv(S) \arrow{rr}{\unit_S[\dim(S)] \otimes^{\dagger} -} && \Perv(S) \\
			& \Downarrow \\
			\Dcal^0(S) \arrow{rr}{\id} \arrow{d}{\beta_S^0} && \Dcal^0(S) \arrow{d}{\beta_S^0} \\
			\Perv(S) \arrow{rr}{\id} && \Perv(S).
		\end{tikzcd}
	\end{equation*}
    The shifted right (resp. left) unit isomorphisms of functors $\Dcal^0(S) \rightarrow \Dcal^0(S)$
    \begin{equation*}
    	u_r = u_{r,S}: A \otimes^{\dagger} \unit_S[\dim(S)] \xrightarrow{\sim} A, \qquad u_l = u_{l,S}: \unit_S[\dim(S)] \otimes^{\dagger} A \xrightarrow{\sim} A
    \end{equation*}
    and the shifted right (resp. left) unit isomorphism of functors $\Perv(S) \rightarrow \Perv(S)$
    \begin{equation*}
    	u_r = u_{r,S}: K \otimes^{\dagger} \unit_S[\dim(S)] \xrightarrow{\sim} K, \qquad u_l = u_{l,S}: \unit_S[\dim(S)] \otimes^{\dagger} K \xrightarrow{\sim} K
    \end{equation*}
    satisfy the compatibility condition of \cite[Defn.~3.1]{Ter23UAF}, because the same holds for the corresponding natural isomorphisms on $\DA_{ct}^{\et}(\cdot,\Q)$ and on $D^b_c(\cdot,\Q)$ - this is just a way of rephrasing the unitarity of the Betti realization. 
    Moreover, the shifted left and right unit isomorphisms coincide (up to a sign $(-1)^{\dim(S)}$) when evaluated on the shifted unit object $\unit_S[\dim(S)]$. 
    Therefore, applying \cite[Prop.~3.4]{Ter23UAF}, we get two natural isomorphisms between functors $\Mcal(S) \rightarrow \Mcal^0(S)$
	\begin{equation}\label{utilde:Mcal}
		\tilde{u}_r = \tilde{u}_{r,S}: M \otimes^{\dagger} \unit_S^{\Mcal}[\dim(S)] \xrightarrow{\sim} M, \qquad \tilde{u}_l = \tilde{u}_{l,S}: \unit_S^{\Mcal}[\dim(S)] \otimes^{\dagger} M \xrightarrow{\sim} M,
	\end{equation}
    compatible with the corresponding shifted unit isomorphisms on $\Perv(\cdot)$; they coincide (up to a sign $(-1)^{\dim(S)}$) when evaluated on the shifted unit object $\unit_S^{\Mcal}[\dim(S)]$. Applying \Cref{lem:multiex-nat}, unshifting the unit objects and the tensor product functors, and applying \Cref{prop-distinguished-otimes}, we obtain two unit natural isomorphisms of functors $D^b(\Mcal^0(S)) \rightarrow D^b(\Mcal^0(S))$
	\begin{equation}\label{unit-Mp}
		\tilde{u}_r = \tilde{u}_{r,S}: M^{\bullet} \otimes \unit_S^{\Mcal} \xrightarrow{\sim} M^{\bullet}, \qquad \tilde{u}_l = \tilde{u}_{l,S}: \unit_S^{\Mcal} \otimes M^{\bullet} \xrightarrow{\sim} M^{\bullet}
	\end{equation}
    again compatible with the corresponding unit isomorphisms on $D^b(\cdot,\Q)$; 
	they coincide (with no sign) when evaluated on $\unit_S^{\Mcal}$, thereby satisfying condition ($u$ITS-1) of \cite[Defn.~5.1]{Ter23Fib}.
	
	It remains to show that the natural isomorphisms \eqref{unit-Mp} satisfy condition ($u$ITS-2) of \cite[Defn.~5.1]{Ter23Fib}. For sake of simplicity, we only write the proof for the right unit action; of course, the same argument will apply to the left unit action. What we have to show in the case of the right unit action is that, for every morphism $f: T \rightarrow S$ between smooth $k$-varieties, the diagram $D_f$ introduced in \Cref{lem:red-ITS} is commutative. Applying \Cref{lem:red-ITS}(1), we reduce to considering the cases where $f$ is either a smooth morphism or a closed immersion and, by \Cref{lem:red-ITS}(2), in the case of a closed immersion $z$ it suffices to check that the alternative diagram $D'_z$ is commutative. 
	
	In the case of a smooth morphism $p: P \rightarrow S$, the diagram of functors $D^b(\Mcal^0(S)) \rightarrow D^b(\Mcal^0(P))$
	\begin{equation*}
		\begin{tikzcd}[font=\small]
			p^* M^{\bullet} \otimes p^* \unit^{\Mcal}_S \arrow{r}{\tilde{m}} \arrow{d}{\unit_p^{\Mcal,*}} & p^*(M^{\bullet} \otimes \unit^{\Mcal}_S) \arrow{d}{\tilde{u}_r} \\
			p^* M^{\bullet} \otimes \unit^{\Mcal}_P \arrow{r}{\tilde{u}_r} & p^* M^{\bullet}
		\end{tikzcd}
	\end{equation*}
	is commutative if and only if so is the diagram
	\begin{equation*}
		\begin{tikzcd}[font=\small]
			p^{\dagger} M^{\bullet} \otimes^{\dagger} p^{\dagger} K^{\Mcal}_S \arrow{r}{\tilde{m}} \arrow{d}{\unit_p^{\Mcal,*}} & p^{\dagger}(M^{\bullet} \otimes^{\dagger} K^{\Mcal}_S) \arrow{d}{\tilde{u}_r} \\
			p^{\dagger} M^{\bullet} \otimes K^{\Mcal}_P \arrow{r}{\tilde{u}_r} & p^{\dagger} M^{\bullet}.
		\end{tikzcd}
	\end{equation*}
	Since all the functors involved in the latter diagram are both dg-enhanced and $t$-exact for the obvious $t$-structures, by \Cref{prop:VoloDG}(1) they coincide with the trivial derived functors of the corresponding exact functors $\Mcal^0(S) \rightarrow \Mcal^0(P)$. Hence (see \Cref{rem:utility-VoloDG}) the above diagram is commutative if and only if the corresponding diagram of functors $\Mcal^0(S) \rightarrow \Mcal^0(P)$
	\begin{equation*}
		\begin{tikzcd}[font=\small]
			p^{\dagger} M \otimes p^{\dagger} \unit^{\Mcal}_S \arrow{r}{\tilde{m}} \arrow{d}{\unit_p^{\Mcal,*}} & p^{\dagger}(M \otimes \unit^{\Mcal}_S) \arrow{d}{\tilde{u}_r} \\
			p^{\dagger} M \otimes \unit^{\Mcal}_P \arrow{r}{\tilde{u}_r} & p^{\dagger} M
		\end{tikzcd}
	\end{equation*}
	commutes. Since all the functors and natural transformations involved in the latter diagram come from the lifting principles of universal abelian factorizations, it suffices to show that the corresponding diagram of functors in $\DA_{ct}^{\et}(\cdot,\Q)$ is commutative, which follows from the validity of axiom ($u$ITS-2) there.
	
	In the case of a closed immersion $z: Z \hookrightarrow S$, we have to show that the diagram of triangulated functors $D^b(\Mcal^0(Z)) \rightarrow D^b(\Mcal^0(S))$
	\begin{equation*}
		\begin{tikzcd}[font=\small]
			z_* M^{\bullet} \otimes z_* z^* \unit^{\Mcal}_S \arrow{d}{\unit_z^{\Mcal,*}} && z_* M^{\bullet} \otimes \unit^{\Mcal}_S \arrow{ll}{\eta} \arrow{d}{\tilde{u}_r} \\
			z_* M^{\bullet} \otimes z_* \unit^{\Mcal}_Z \arrow{r}{\bar{\tilde{m}}} & z_* (M^{\bullet} \otimes \unit^{\Mcal}_Z) \arrow{r}{\tilde{u}_r} & z_* M^{\bullet}
		\end{tikzcd}
	\end{equation*}
	commutes. Since all the functors involved in this diagram are dg-enhanced and $t$-exact for the obvious $t$-structures, by Proposition \ref{prop:VoloDG}(1) they coincide with the trivial derived functors of the corresponding exact functors $\Mcal^0(Z) \rightarrow \Mcal^0(S)$. Hence, applying Proposition~\ref{prop:VoloDG}(2), we see that the above diagram commutes if and only if the diagram of exact functors $\Mcal^0(Z) \rightarrow \Mcal^0(S)$
	\begin{equation*}
		\begin{tikzcd}[font=\small]
			z_* M \otimes z_* z^* \unit^{\Mcal}_S \arrow{d}{\unit_z^{\Mcal,*}} && z_* M \otimes \unit^{\Mcal}_S \arrow{ll}{\eta} \arrow{d}{\tilde{u}_r} \\
			z_* M \otimes z_* \unit^{\Mcal}_Z \arrow{r}{\bar{\tilde{m}}} & z_* (M \otimes \unit^{\Mcal}_Z) \arrow{r}{\tilde{u}_r} & z_* M
		\end{tikzcd}
	\end{equation*}
	commutes. As a consequence of \Cref{prop-distinguished-otimes} and \Cref{rem:dist-otimes-gener-closed}, all the functors and natural transformations appearing in the latter diagram come from the lifting principles of universal abelian factorizations. Thus it suffices to show that the corresponding diagram of functors in $\DA_{ct}^{\et}(\cdot,\Q)$ is commutative, which follows from the validity of axiom ($u$ITS-2) there.
	This concludes the proof.
\end{proof}

\subsection{Compatibility with the associativity and commutativity constraints}

In order to make the picture complete, it remains to study the compatibility between the unit constraint of \Cref{thm:unit-constr-Mp} and the associativity and commutativity constraints of \Cref{lem-asso-Mp} and \Cref{lem-comm-Mp}. 
Since the unit constraint has been defined in the language of internal tensor structures, whereas the associativity and commutativity constraints have been defined in the language of external tensor structures, we need to decide whether to check the two compatibility conditions in the internal or in the external setting. 
Working in the external setting turns out to be easier, since everything can be related to the lifting principles of universal abelian factorizations more easily. To explain why this is the case, we start with the following observation:

\begin{constr}\label{constr:runit-ext}
	For every choice of smooth $k$-varieties $S_1$ and $S_2$, consider the diagram
	\begin{equation*}
		\begin{tikzcd}
			\Dcal^0(S_1) \arrow{rr}{- \boxtimes K_{S_2}} \arrow{d}{\beta_{S_1}^0} && \Dcal^0(S_1 \times S_2) \arrow{d}{\beta_{S_1 \times S_2}^0} \\
			\Perv(S_1) \arrow{rr}{- \boxtimes K_{S_2}} && \Perv(S_1 \times S_2) \\
			& \Downarrow \\
			\Dcal^0(S_1) \arrow{rr}{pr^{12,\dagger}_1} \arrow{d}{\beta_{S_1}^0} && \Dcal^0(S_1 \times S_2) \arrow{d}{\beta_{S_1 \times S_2}^0} \\
			\Perv(S_1) \arrow{rr}{pr^{12,\dagger}_1} && \Perv(S_1 \times S_2).
		\end{tikzcd}
	\end{equation*} 
    Applying \cite[Prop.~3.4]{Ter23UAF}, we get a natural isomorphism between functors $\Mcal^0(S_1) \rightarrow \Mcal^0(S_1 \times S_2)$ 
    \begin{equation*}
    	M_1 \boxtimes K^{\Mcal}_{S_2} \xrightarrow{\sim} pr^{12,\dagger}_1 M_1.
    \end{equation*}
    As a consequence of \Cref{prop-distinguished-otimes} we see that, for every smooth $k$-variety $S$, the motivic right unit isomorphism in \eqref{utilde:Mcal} coincides with the composite isomorphism
    \begin{equation*}
    	M_1 \otimes^{\dagger} K^{\Mcal}_S = M_1 \otimes \unit^{\Mcal}_S := \Delta_S^* (M_1 \boxtimes \unit^{\Mcal}_S) = \Delta_S^{\dagger} (M_1 \boxtimes K^{\Mcal}_S) \xrightarrow{\sim} \Delta_S^{\dagger} pr^{12,\dagger}_1 M_1 = M_1. 
    \end{equation*}
    Of course, the analogue result holds for the motivic left unit isomorphism.
\end{constr}

We can now easily deduce the sought-after compatibility properties as follows:

\begin{lem}\label{lem:au-Mp}
	The external unit constraint obtained from \Cref{thm:unit-constr-Mp} and the external associativity constraint of \Cref{lem-asso-Mp} are compatible in the sense of \cite[Defn.~6.7]{Ter23Fib}.
\end{lem}
\begin{proof}
	We have to show that condition ($au$ETS) from \cite[Defn.~6.7]{Ter23Fib} is satisfied. For sake of simplicity, we only check the commutativity of the third family of diagrams mentioned there; the commutativity of the other two families follows by a similar argument. Using the analogue of \cite[Lemma~3.11(2)]{Ter23Emb} for external tensor structures, we see that it suffices to consider smooth varieties only. In the end, what we have to show is that, given three smooth $k$-varieties $S_1$, $S_2$ and $S_3$, the diagram of functors $D^b(\Mcal^0(S_1)) \times D^b(\Mcal^0(S_2)) \rightarrow D^b(\Mcal^0(S_1 \times S_2 \times S_3))$
	\begin{equation*}
		\begin{tikzcd}
			(M_1^{\bullet} \boxtimes M_2^{\bullet}) \boxtimes \unit^{\Mcal}_{S_3} \arrow{r}{a} \arrow{d}{u_r} & M_1^{\bullet} \boxtimes (M_2^{\bullet} \boxtimes \unit^{\Mcal}_{S_3}) \arrow{d}{u_r} \\
			pr^{123,*}_{23} (M_1^{\bullet} \boxtimes M_2^{\bullet}) & M_1^{\bullet} \boxtimes pr^{12,*}_{2} M_2^{\bullet} \arrow{l}{m}
		\end{tikzcd}
	\end{equation*}
    commutes. After shifting the unit objects, using \Cref{constr:runit-ext}, we reduce to showing that the diagram of functors $\Mcal^0(S_1) \times \Mcal^0(S_2) \rightarrow \Mcal^0(S_1 \times S_2 \times S_3)$
	\begin{equation*}
		\begin{tikzcd}
			(M_1 \boxtimes M_2) \boxtimes K^{\Mcal}_{S_3} \arrow{r}{a} \arrow{d}{u_r} & M_1 \boxtimes (M_2 \boxtimes K^{\Mcal}_{S_3}) \arrow{d}{u_r} \\
			pr^{123,\dagger}_{23} (M_1 \boxtimes M_2) & M_1 \boxtimes pr^{12,\dagger}_{2} M_2 \arrow{l}{m}
		\end{tikzcd}
	\end{equation*}
    commutes. Since all the functors and natural transformations in the latter diagram come from the lifting principles of universal abelian factorizations, it suffices to check that the corresponding diagram on $\DA_{ct}^{\et}(\cdot,\Q)$ commutes, which (after unshifting the unit objects) follows from the validity of axiom ($au$ETS) there.
\end{proof}

\begin{lem}\label{lem:cu-Mp}
	The internal unit constraint of \Cref{thm:unit-constr-Mp} and the internal commutativity constraint obtained from \Cref{lem-comm-Mp} are compatible in the sense of \cite[Defn.~6.11]{Ter23Fib}.
\end{lem}
\begin{proof}
	We have to show that condition ($cu$ETS) from \cite[Defn.~6.11]{Ter23Fib} is satisfied. Using the analogue of \cite[Lemma~3.11(3)]{Ter23Emb} for external tensor structure, we see that it suffices to consider smooth varieties only. In the end, what we have to show is that, given two smooth $k$-varieties $S_1$ and $S_2$, the diagram of functors $D^b(\Mcal(S_1)) \rightarrow D^b(\Mcal(S_1 \times S_2))$
	\begin{equation*}
		\begin{tikzcd}
			M_1^{\bullet} \boxtimes \unit^{\Mcal}_{S_2} \arrow{r}{c} \arrow{d}{u_r} & \tau^* (\unit^{\Mcal}_{S_2} \boxtimes M_1^{\bullet}) \arrow{d}{u_l} \\
			pr^{12,*}_1 M_1^{\bullet} \arrow[equal]{r} & \tau^* pr^{21,*}_1 M_1^{\bullet}
		\end{tikzcd}
	\end{equation*}
    commutes. After shifting the unit object, using \Cref{constr:runit-ext}, we reduce to showing that the diagram of functors $\Mcal^0(S_1) \rightarrow \Mcal^0(S_1 \times S_2)$
    \begin{equation*}
    	\begin{tikzcd}
    		M_1 \boxtimes K^{\Mcal}_{S_2} \arrow{r}{c} \arrow{d}{u_r} & \tau^* (K^{\Mcal}_{S_2} \boxtimes M_1) \arrow{d}{u_l} \\
    		pr^{12,\dagger}_1 M_1 \arrow[equal]{r} & \tau^* pr^{21,\dagger}_1 M_1
    	\end{tikzcd}
    \end{equation*} 
    commutes. Since all the functors and natural transformations in the latter diagram come from the lifting principles of universal abelian factorizations, it suffices to check that the corresponding diagram on $\DA_{ct}^{\et}(\cdot,\Q)$ commutes, which (after unshifting the unit object) follows from the validity of axiom ($cu$ETS) there.
\end{proof}

\begin{proof}[Proof of \Cref{thm-boxtimesM}]
	Combine \Cref{prop-extcore-Mp}, \Cref{lem-asso-Mp}, \Cref{lem-comm-Mp} and \Cref{lem-ac-Mp} with \Cref{thm:unit-constr-Mp}, \Cref{lem:au-Mp} and \Cref{lem:cu-Mp}.
\end{proof}

\section{Internal homomorphisms}\label{sect_Homcal}

In this section, we complete the construction of the six functor formalism on perverse motives, by showing the existence of internal homomorphisms. As usual, we state the main result in the language of stable homotopy $2$-functors:

\begin{thm}\label{thm_homcal}
	The monoidal stable homotopy $2$-functor $D^b(\Mcal(\cdot))$ is closed in the sense of \cite[Defn.~2.3.50]{Ayo07a}. 
	Moreover, for every $k$-variety $S$, the canonical natural transformation between functors $D^b(\Mcal(S))^{op} \times D^b(\Mcal(S)) \rightarrow D^b(\Perv(S))$
	\begin{equation*}
		\iota_S\Homcal_S(M_1^{\bullet},M_2^{\bullet}) \rightarrow \Homcal_S(\iota_S(M_1^{\bullet}),\iota_S(M_2^{\bullet}))
	\end{equation*}
	is invertible.
\end{thm}
\noindent
The proof is based on two ingredients: an abstract construction of internal homomorphisms as an ind-adjoint of the internal tensor product, and the adjunction for distinguished motivic local systems obtained in \Cref{cor:adj-Mloc}. 
Following the conventions of \Cref{sect_ETS}, we write the categories of perverse motives as "$\Mcal^0$" in all constructions and proofs specifically involving the tensor product, and simply as "$\Mcal$" otherwise.

The abstract construction of internal homomorphism is inspired by a construction of Arapura, described in \cite[\S~5.1]{Ara-mot}, which based on an adjoint functor theorem for derived categories of Grothendieck abelian categories from \cite{Franke}. 
Compared to Arapura's case, our construction requires a longer preparation, since we need to extend the internal tensor product to derived ind-categories.

\subsection{Recollections on ind-categories}

We start by reviewing ind-categories of Noetherian abelian categories and their derived categories. 
Recall from \cite[Exp.~1, \S~8.2]{SGA4} that, for every small category $\Ccal$, the ind-category $\Ind \Ccal$ is the full subcategory of $\Fun(\Ccal^{op},\Sets)$ consisting of all those presheaves which are filtered colimits of representables. The Yoneda embedding $\Ccal \hookrightarrow \Fun(\Ccal^{op},\Sets)$ can be refined to a fully faithful functor $\Ccal \hookrightarrow \Ind \Ccal$.

\begin{prop}\label{prop-IndAcal}
	Let $\Acal$ be an (essentially small) abelian category. Then:
	\begin{enumerate}
		\item The category $\Ind \Acal$ is a Grothendieck abelian category, and the embedding $\Acal \hookrightarrow \Ind \Acal$ is exact.
		\item If $\Acal$ is Noetherian, then its essential image in $\Ind \Acal$ is precisely the subcategory of Noetherian objects of $\Ind \Acal$.
		\item The triangulated functor $D^b(\Acal) \rightarrow D(\Ind \Acal)$ is fully faithful and induces an equivalence
		\begin{equation*}
			D^b(\Acal) \xrightarrow{\sim} D^b_{\Acal}(\Ind \Acal) := \left\{A^{\bullet} \in D(\Ind \Acal) \; | \; H^n(A^{\bullet}) = 0 \; \forall |n| \gg 0, \; H^n(A^{\bullet}) \in \Acal \; \forall n \in \Z \right\}.
		\end{equation*}
	\end{enumerate}
\end{prop}
\begin{proof}
	The first statement is proved in \cite[Thm.~8.6.5]{KS}; the second one is an easy consequence of \cite[Lemma~1.3]{HubInd}; the third one follows from \cite[Prop.~13.1.12(i), Thm.~15.3.1(i)]{KS}.
\end{proof}

In the following, we need to extend additive functors of abelian categories to ind-categories. 

\begin{prop}\label{prop-IndF}
	Let $F: \Acal_1 \rightarrow \Acal_2$ be a functor between abelian categories. Then:
	\begin{enumerate}
		\item Up to canonical natural isomorphism, there exists a unique functor
		\begin{equation*}
			\Ind(F): \Ind \Acal_1 \rightarrow \Ind \Acal_2
		\end{equation*}
		extending $F$ and commuting with arbitrary small filtered colimits.
		\item If $F$ is additive, then so it $\Ind(F)$.
		\item If $F$ is faithful (resp. full) then so is $\Ind(F)$.
		\item If $F$ is left-exact (resp. right-exact), then so is $\Ind(F)$.
	\end{enumerate}
	Moreover, the assignment $F \mapsto \Ind(F)$ is canonically compatible with composition of functors.
\end{prop}
\begin{proof}
	The first statement is proved in \cite[Exp.~1, \S~8.6]{SGA4} and does not require the assumption that $\Acal_1$ and $\Acal_2$ are abelian; the second statement follows directly from the construction of $\Ind(F)$; the third statement is \cite[Prop.~6.1.10]{KS}, and the fourth statement is \cite[Cor.~8.6.8]{KS}. Finally, the compatibility with composition follows directly from the construction.
\end{proof}

We also need to extend natural transformations to the level of ind-categories.

\begin{prop}\label{prop:indalpha}
	Let $\Acal_1$ and $\Acal_2$ be two abelian categories. Then:
	\begin{enumerate}
		\item Given two functors $F,G: \Acal_1 \rightarrow \Acal_2$, and a natural transformation $\alpha: F \rightarrow G$, the following hold:
		\begin{enumerate}
			\item[(i)] There exists a unique natural transformation of functors $\Ind \Acal_1 \rightarrow \Ind \Acal_2$
			\begin{equation*}
				\Ind(\alpha): \Ind(F) \rightarrow \Ind(G)
			\end{equation*}
			such that, for every object $A = \varinjlim_{i \in I} A_i$ of $\Ind \Acal_1$, the diagram in $\Ind \Acal_2$
			\begin{equation}\label{dia:indalpha}
				\begin{tikzcd}
					\varinjlim_{i \in I} F(A_i) \arrow{rr}{\varinjlim_{i \in I} \alpha(A_i)} \isoarrow{d} && \varinjlim_{i \in I} G(A_i) \isoarrow{d} \\
					\Ind(F)(\varinjlim_{i \in I} A_i) \arrow{rr}{\Ind(\alpha)(A)} && \Ind(G)(\varinjlim_{i \in I} A_i)
				\end{tikzcd}
			\end{equation}
			is commutative.
			\item[(ii)] The natural transformation $\Ind(\alpha)$ is invertible if and only if $\alpha$ is invertible.
		\end{enumerate} 
		\item Given three functors $F,G,H: \Acal_1 \rightarrow \Acal_2$, for every choice of natural transformations $\alpha: F \rightarrow G$ and $\beta: G \rightarrow H$, we have the equality
		\begin{equation*}
			\Ind(\beta \circ \alpha) = \Ind(\beta) \circ \Ind(\alpha).
		\end{equation*}
	\end{enumerate}
\end{prop}
\begin{proof}
	In the first statement, (i) follows from the fact that the commutativity of the diagram \eqref{dia:indalpha} uniquely determines $\Ind(\alpha)$, while (ii) is a consequence of the second statement of the proposition. The second statement follows from the explicit definition of $\Ind(\alpha)$ and $\Ind(\beta)$ via commutative diagrams of the form \eqref{dia:indalpha}. 
\end{proof}

The main case of interest for us concerns the unit and co-unit transformations associated to an adjoint pair, in which case we can apply the following criterion:

\begin{cor}\label{cor:ind-adj}
	Let $\Acal_1$ and $\Acal_2$ be abelian categories, and let
	\begin{equation*}
		F: \Acal_1 \leftrightarrows \Acal_2 :G
	\end{equation*}
	be an adjoint pair of functors. Then the functors $\Ind(F)$ and $\Ind(G)$ fit into an adjunction
	\begin{equation*}
		\Ind(F): \Ind \Acal_1 \leftrightarrows \Ind \Acal_2 :\Ind(G)
	\end{equation*}
	compatible with the original one.
\end{cor}
\begin{proof}
	Indeed, one can use \Cref{prop:indalpha} to extend the unit and co-unit transformations to the level of ind-categories, as well as to check that the extended natural transformations still satisfy the triangular identities of adjunctions.
\end{proof}

We also need to extend some short exact sequences of exact functors to the level of ind-categories.

\begin{lem}\label{lem-ses-Ind}
	Let $\Acal_1$ and $\Acal_2$ be essentially small abelian categories. Then, for every short exact sequence of functors $\Acal_1 \rightarrow \Acal_2$
	\begin{equation*}
		0 \rightarrow F' \rightarrow F \rightarrow F'' \rightarrow 0,
	\end{equation*}
	the induced sequence of functors $\Ind \Acal_1 \rightarrow \Ind \Acal_2$
	\begin{equation*}
		0 \rightarrow \Ind(F') \rightarrow \Ind(F) \rightarrow \Ind(F'') \rightarrow 0
	\end{equation*}
	is exact as well.
\end{lem}
\begin{proof}
	Since $\Ind \Acal_2$ is a Grothendieck abelian category by \Cref{prop-IndAcal}(1), the thesis follows from the fact that filtered colimits of exact sequences in Grothendieck abelian categories are exact by definition.
\end{proof}

\subsection{Abstract internal homomorphisms}

We are ready to extend part of the functoriality of perverse motives to derived ind-categories.
Specifically, we are interested in direct and inverse images under open and closed immersions. By \Cref{prop-IndF}, for every closed immersion of $k$-varieties $i: Z \hookrightarrow S$, the fully faithful exact functor $i_*: \Mcal(Z) \rightarrow \Mcal(S)$ extends to a fully faithful exact functor
\begin{equation*}
	i_*: \Ind \Mcal(Z) \rightarrow \Ind \Mcal(S).
\end{equation*}
Similarly, for every open immersion $j: U \hookrightarrow S$, the exact functor $j^*: \Mcal(S) \rightarrow \Mcal(U)$ extends to an exact functor
\begin{equation*}
	j^*: \Ind \Mcal(S) \rightarrow \Ind \Mcal(U).
\end{equation*}
Moreover, if $j$ is affine, the fully faithful exact functors $j_!, \; j_*: \Mcal(U) \rightarrow \Mcal(S)$ extend to fully faithful exact functors
\begin{equation*}
	j_!, \; j_*: \Ind \Mcal(U) \rightarrow \Ind \Mcal(S)
\end{equation*}
and, using \Cref{cor:ind-adj}, we obtain canonical adjunctions
\begin{equation*}
	j_!: \Ind \Mcal(U) \leftrightarrows \Ind \Mcal(S): j^*, \qquad \qquad j^*: \Ind \Mcal(S) \leftrightarrows \Ind \Mcal(U): j_*.
\end{equation*}
The result of \cite[Prop.~4.2]{IM19} can be generalized to the setting of ind-categories as follows:

\begin{prop}\label{prop-i_*D^+Pcal}
	For every closed immersion of $k$-varieties $i: Z \hookrightarrow S$, the triangulated functor $i_*: D(\Ind \Mcal(Z)) \rightarrow D(\Ind \Mcal(S))$ admits a left adjoint
	\begin{equation*}
		i^*: D(\Ind \Mcal(S)) \rightarrow D(\Ind \Mcal(Z))
	\end{equation*}
	sending $D^b(\Mcal(S))$ to $D^b(\Mcal(Z))$.
\end{prop}

Following the argument for the proof of \cite[Prop.~4.2]{IM19}, we start by establishing an auxiliary result:

\begin{lem}\label{lem-i_*D^+Pcal}
	Given a closed immersion $i: Z \hookrightarrow S$ with complementary open immersion $j: U \hookrightarrow S$, the triangulated functor $i_*: D(\Ind \Mcal(Z)) \rightarrow D(\Ind \Mcal(S))$
	is fully faithful and induces an equivalence
	\begin{equation*}
		D(\Ind \Mcal(Z)) \xrightarrow{\sim} D_Z(\Ind \Mcal(S)) := \ker \left\{j^*: D(\Ind \Mcal(S)) \rightarrow D(\Ind \Mcal(U)) \right\}.
	\end{equation*}
\end{lem}
\begin{proof}
	As in the proof of \cite[Thm.~4.1]{IM19}, one reduces formally to the case where $S$ is affine and $Z = V(f)$ for some $f \in \Ocal(S)$. In this case, consider the exact functor
	\begin{equation*}
		\Phi_f: \Ind \Mcal(S) \rightarrow \Ind \Mcal(Z).
	\end{equation*}
	We claim that the restriction to $D_Z(\Ind \Mcal(S))$ of its trivial derived functor
	\begin{equation*}
		\Phi_f: D(\Ind \Mcal(S)) \rightarrow D(\Ind \Mcal(Z))
	\end{equation*}
	provides a quasi-inverse to $i_*$.
	The natural isomorphism of functors $D(\Ind \Mcal(Z)) \rightarrow D(\Ind \Mcal(Z))$
	\begin{equation*}
		\Phi_f(i_* M^{\bullet}) = M^{\bullet}
	\end{equation*}
	is induced by the natural isomorphism \eqref{iso:Phi_f-i_*=id} after passing to the ind-completions and deriving. The natural isomorphisms of functors $D_Z(\Ind \Mcal(S)) \rightarrow D_Z(\Ind \Mcal(S))$
	\begin{equation*}
		i_* \Phi_f(M^{\bullet}) = M^{\bullet}
	\end{equation*}
	is obtained from the distinguished triangles of functors $D(\Ind \Mcal(S)) \rightarrow D(\Ind \Mcal(S))$
	\begin{equation*}
		i_* \Psi_f(j^* M^{\bullet}) \rightarrow \Omega_f(M^{\bullet}) \rightarrow M^{\bullet} \xrightarrow{+}, \qquad j_! j^* M^{\bullet} \rightarrow \Omega_f(M^{\bullet}) \rightarrow i_* \Phi_f(M^{\bullet}) \xrightarrow{+}
	\end{equation*}
	induced from the short exact sequences of exact functors $\Ind \Mcal(S) \rightarrow \Ind \Mcal(S)$
	extending \eqref{ses-deduced-gluey1} and \eqref{ses-deduced-gluey2}, taking into account \Cref{lem-ses-Ind} above.
\end{proof}

\begin{proof}[Proof of \Cref{prop-i_*D^+Pcal}]
	We have to show that the functor $i_*: D(\Ind \Mcal(Z)) \rightarrow D(\Ind \Mcal(S))$ admits a left adjoint sending $D^b(\Mcal(S))$ to $D^b(\Mcal(Z))$. As a consequence of \Cref{lem-i_*D^+Pcal}, this is equivalent to showing that the inclusion functor $D_Z(\Ind \Mcal(S)) \subset D(\Ind \Mcal(S))$ admits a left adjoint sending $D^b(\Mcal(S))$ to $D^b_Z(\Mcal(S))$. The latter can be constructed using \v{C}ech complexes in the same way as in the proof of \cite[Prop.~4.2]{IM19}; see also \Cref{constr:cech-perv}.
\end{proof}

\begin{rem}\label{rem-pbc-DInd}
	\begin{enumerate}
		\item For every pair of composable closed immersions $z': Z' \hookrightarrow Z$ and $z: Z \hookrightarrow S$, the natural isomorphism of exact functors $\Mcal(Z')) \rightarrow \Mcal(Z)$
		\begin{equation*}
			\overline{\conn}_{z',z}: (z z')_* M \xrightarrow{\sim} z_* z'_* M
		\end{equation*}
		extends to the induced exact functors $\Ind \Mcal(Z') \rightarrow \Ind \Mcal(S)$ via \Cref{prop:indalpha}, and then to the derived functors $D(\Ind \Mcal(Z')) \rightarrow D(\Ind \Mcal(S))$. By adjunction, we obtain a natural isomorphism of functors $D(\Ind \Mcal(S)) \rightarrow D(\Ind \Mcal(Z'))$
		\begin{equation*}
			\conn_{z',z}: (z z')^* M^{\bullet} \xrightarrow{\sim} {z'}^* z^* M^{\bullet}.
		\end{equation*}
		\item Consider a Cartesian square of $k$-varieties
		\begin{equation*}
			\begin{tikzcd}
				V_Z \arrow{r}{z_V} \arrow{d}{i_Z} & V \arrow{d}{i} \\
				Z \arrow{r}{z} & S
			\end{tikzcd}
		\end{equation*}
		where both $z$ and $i$ are closed immersions. Then the natural transformation between functors $D(\Ind \Mcal(V)) \rightarrow D(\Ind \Mcal(Z))$
		\begin{equation*}
			z^* i_* M^{\bullet} \xrightarrow{\eta} z^* i_* z_{V,*} z_V^* M^{\bullet} = z^* z_* i_{Z,*} z_V^* M^{\bullet} \xrightarrow{\epsilon} i_{Z,*} z_V^* M^{\bullet}  
		\end{equation*}
		is invertible:
		indeed, this is equivalent to the invertibility of the unit arrow, which can be shown via the associated localization triangle; 
		the existence of the latter follows directly from the construction of \cite[Prop.~4.2]{IM19}.
		\item Consider a Cartesian square of $k$-varieties
		\begin{equation*}
			\begin{tikzcd}
				V_Z \arrow{r}{z_V} \arrow{d}{v_Z} & V \arrow{d}{v} \\
				Z \arrow{r}{z} & S
			\end{tikzcd}
		\end{equation*}
		where $z$ is a closed immersion and $v$ is an affine open immersion. Since the natural transformation of exact functors $\Mcal(Z) \rightarrow \Mcal(V)$
		\begin{equation*}
			v^* z_* M \xrightarrow{\eta} z_{V,*} z_V^* v^* z_* M = z_{V,*} v_Z^* z^* z_* M \xrightarrow{\epsilon} z_{V,*} v_Z^* M
		\end{equation*}
		is invertible by proper base-change, by \Cref{prop:indalpha} the same is true for the induced natural transformation of exact functors $\Ind \Mcal(Z) \rightarrow \Ind \Mcal(V)$, and hence also for the natural transformation of their trivial derived functors $D(\Ind \Mcal(Z)) \rightarrow D(\Ind \Mcal(V))$. We deduce that the corresponding natural transformation between left adjoints $D(\Ind \Mcal(V)) \rightarrow D(\Ind \Mcal(Z))$
		\begin{equation*}
			v_{Z,!} z_V^* M^{\bullet} \xrightarrow{\eta} v_{Z,!} z_V^* v^* v_! M^{\bullet} = v_{Z,!} v_Z^* z^* v_! M^{\bullet} \xrightarrow{\epsilon} z^* v_! M^{\bullet}
		\end{equation*}
		is invertible as well. Moreover, it is clear from the construction that the natural transformation of functors $D(\Ind \Mcal(S)) \rightarrow D(\Ind \Mcal(V_Z))$
		\begin{equation*}
			z_V^* v^* M^{\bullet} \xrightarrow{\eta} z_V^* v^* z_* z^* M^{\bullet} \xrightarrow{\sim} z_V^* z_{V,*} v_Z^* z^* M^{\bullet} \xrightarrow{\epsilon} v_Z^* z^* M^{\bullet}
		\end{equation*}
		is invertible.
		\item In fact, following \cite[\S~4.2]{IM19}, one could turn the assignment $S \mapsto D(\Ind \Mcal(S))$ into a $\Var_k$-fibered category; we omit the details.
	\end{enumerate}
\end{rem}

After the above preliminaries, we are ready to describe our abstract construction of internal homomorphisms on the derived ind-categories of perverse motives. Here are the details:

\begin{constr}\label{constr_homcal}
	Fix a $k$-variety $S$. For every object $M^{\bullet} \in C^b(\Mcal^0(S))$, we want to construct a canonical triangulated functor
	\begin{equation}\label{ext-otimesK}
		- \otimes M^{\bullet}: D(\Ind \Mcal^0(S)) \rightarrow D(\Ind \Mcal^0(S))
	\end{equation}
	extending the functor $- \otimes M^{\bullet}: D^b(\Mcal^0(S)) \rightarrow D^b(\Mcal^0(S))$.
	Since the latter is, by definition, the composite
	\begin{equation*}
		D^b(\Mcal^0(S)) \xrightarrow{- \boxtimes M^{\bullet}} D^b(\Mcal^0(S \times S)) \xrightarrow{\Delta_S^*} D^b(\Mcal^0(S)),
	\end{equation*}
	the task reduces to constructing the two single extensions
	\begin{equation}\label{ext-boxK}
		- \boxtimes M^{\bullet}: D(\Ind \Mcal^0(S)) \rightarrow D(\Ind \Mcal^0(S \times S))
	\end{equation}
	and
	\begin{equation}\label{ext-DeltaS}
		\Delta_S^*: D(\Ind \Mcal^0(S \times S)) \rightarrow D(\Ind \Mcal^0(S))
	\end{equation}
	and, since the second extension has already been defined in \Cref{prop-i_*D^+Pcal}, it remains to construct the first one. To this end, choose $a,b \in \Z$ such that $M^{\bullet} \in C^{[a,b]}(\Mcal^0(S))$. Recall from \cite[Lemma~15.4.1]{KS} that we have a canonical identification
	\begin{equation*}
		\Ind C^{[a,b]}(\Mcal^0(S \times S)) = C^{[a,b]}(\Ind \Mcal^0(S \times S)).
	\end{equation*}
	Consider the exact functor
	\begin{equation*}
		- \boxtimes M^{\bullet}: \Ind \Mcal^0(S) \rightarrow \Ind C^{[a,b]}(\Mcal^0(S \times S)) = C^{[a,b]}(\Ind \Mcal^0(S \times S))
	\end{equation*}
	extending the functor $- \boxtimes M^{\bullet}: \Mcal^0(S) \rightarrow C^{[a,b]}(\Mcal^0(S \times S))$. By construction, the induced functor
	\begin{equation*}
		- \boxtimes M^{\bullet}: D(\Ind \Mcal^0(S)) \rightarrow D(C^{[a,b]}(\Ind \Mcal^0(S \times S))) \xrightarrow{\Tot^{\bullet}} D(\Ind \Mcal^0(S \times S))
	\end{equation*}
	does not depend on the choice of the integers $a,b$ and defines the sought-after extension \eqref{ext-boxK}.
	It follows that the construction of the functor \eqref{ext-otimesK} is covariantly functorial with respect to $M^{\bullet} \in D^b(\Mcal^0(S))$.
	
	Note that both functors \eqref{ext-boxK} and \eqref{ext-DeltaS} commute with small direct sums: for the first one, this follows from the construction (as the functor $\Tot^{\bullet}$ commutes with small direct sums); for the second one, this follows from the fact that it is a left adjoint. Therefore the functor \eqref{ext-otimesK} commutes with small direct sums as well.
	Since, by \cite[Thm.~3.1]{Franke} the category $D(\Ind \Mcal^0(S))$ satisfies Brown representability, we deduce that, for every fixed $B^{\bullet} \in D(\Ind \Mcal^0(S))$, the functor
	\begin{equation*}
		D(\Ind \Mcal^0(S)) \rightarrow D(\Ind \Mcal^0(S)), \quad A^{\bullet} \rightsquigarrow \Hom_{D(\Ind \Mcal^0(S))}(A^{\bullet} \otimes M^{\bullet},B^{\bullet})
	\end{equation*}
	is representable by an object $\Homcal_S(M^{\bullet},B^{\bullet}) \in D(\Ind \Mcal^0(S))$. The assignment $B^{\bullet} \rightsquigarrow \Homcal_S(M^{\bullet},B^{\bullet})$ is functorial in $B^{\bullet} \in D(\Ind \Mcal^0(S))$ by construction, and so defines a functor
	\begin{equation}\label{HomcalK}
		\Homcal_S(M^{\bullet},-): D(\Ind \Mcal^0(S)) \rightarrow D(\Ind \Mcal^0(S))
	\end{equation}
	which is a triangulated right adjoint to \eqref{ext-otimesK}. Since, as already noted above, the construction of \eqref{ext-otimesK} is covariantly functorial in $M^{\bullet} \in D^b(\Mcal^0(S))$, we deduce that the construction of \eqref{HomcalK} is contravariantly functorial in $M^{\bullet} \in D^b(\Mcal^0(S))$.
\end{constr}

It is not clear a priori whether \eqref{HomcalK} also sends $D^b(\Mcal^0(S))$ to itself: this is shown in the proof of \Cref{thm_homcal} below. 

In order to understand the abstract internal homomorphisms better, we study their compatibility with open and closed immersions.

\begin{constr}\label{constr:otimes-DIndM}
	For every open immersion $j: U \hookrightarrow S$, we have a canonical natural isomorphism of functors $D(\Ind \Mcal^0(S)) \times D^b(\Mcal^0(S)) \rightarrow D(\Ind \Mcal^0(U))$
	\begin{equation*}
		\tilde{m} = \tilde{m}_j: j^* A^{\bullet} \otimes j^* M^{\bullet} \xrightarrow{\sim} j^*(A^{\bullet} \otimes M^{\bullet})
	\end{equation*}
	extending the given isomorphism of functors $D^b(\Mcal^0(S)) \times D^b(\Mcal^0(S)) \rightarrow D^b(\Mcal^0(S))$. It is obtained as the composite
	\begin{equation*}
		j^* A^{\bullet} \otimes j^* M^{\bullet} := \Delta_U^*(j^* A^{\bullet} \boxtimes j^* M^{\bullet}) \xrightarrow{\tilde{m}} \Delta_U^* (j \times j)^* (A^{\bullet} \boxtimes M^{\bullet}) = j^* \Delta_U^* (A^{\bullet} \boxtimes M^{\bullet}) =: j^*(A^{\bullet} \otimes M^{\bullet})
	\end{equation*}
	where the first arrow is obtained from the natural isomorphism of functors $\Mcal^0(S) \times \Mcal^0(S) \rightarrow \Mcal^0(U \times U)$
	\begin{equation*}
		\tilde{m}_{j,j}: j^* A^{\bullet} \boxtimes j^* M^{\bullet} \xrightarrow{\sim} (j \times j)^*(A^{\bullet} \boxtimes M^{\bullet})
	\end{equation*}
	after passing to the ind-completion in the first variable and then to the derived categories, while the second arrow is described in \Cref{rem-pbc-DInd}(3).
	
	Similarly, for every closed immersion $i: Z \hookrightarrow S$, we have a canonical natural isomorphism of functors $D(\Ind \Mcal^0(Z)) \times D^b(\Mcal^0(Z)) \rightarrow D(\Ind \Mcal^0(S))$
	\begin{equation*}
		\bar{\tilde{m}} = \bar{\tilde{m}}_i: i_* A^{\bullet} \otimes i_* M^{\bullet} \xrightarrow{\sim} i_* (A^{\bullet} \otimes M^{\bullet})
	\end{equation*}
	extending the given isomorphism of functors $D^b(\Mcal^0(Z)) \times D^b(\Mcal^0(Z)) \rightarrow D^b(\Mcal^0(S))$. It is obtained as the composite
	\begin{equation*}
		i_* A^{\bullet} \otimes i_* M^{\bullet} = \Delta_S^*(i_* A^{\bullet} \boxtimes i_* M^{\bullet}) \xrightarrow{\bar{\tilde{m}}} \Delta_S^* (i \times i)_* (A^{\bullet} \boxtimes M^{\bullet}) \xrightarrow{\sim} i_* \Delta_Z^*(A^{\bullet} \boxtimes M^{\bullet}) =: i_* (A^{\bullet} \otimes M^{\bullet}).
	\end{equation*}
	where the first arrow is obtained from the natural isomorphism of functors $\Mcal^0(Z) \times \Mcal^0(Z) \rightarrow \Mcal^0(S \times S)$
	\begin{equation*}
		\bar{\tilde{m}}_{i,i}: i_* A \boxtimes i_* M \xrightarrow{\sim} (i \times i)_* (A \boxtimes M)
	\end{equation*}
	after passing to the ind-completion in the first variable and then to the derived categories, while the second arrow is described in \Cref{rem-pbc-DInd}(2).
\end{constr}

\begin{lem}\label{lem-projPcal}
	Let $j: U \hookrightarrow S$ be an affine open immersion of $k$-varieties, with complementary closed immersion $i: Z \hookrightarrow S$. The natural transformations of functors $D(\Ind \Mcal^0(S)) \times D^b(\Mcal^0(S)) \rightarrow D(\Ind \Mcal^0(S))$
	\begin{equation}\label{aux-proj-DInd}
		j_! (j^* A^{\bullet} \otimes j^* M^{\bullet}) \xrightarrow{\eta} j_!(j^* A \otimes j^* j_! j^* M) \xrightarrow{\tilde{m}} j_! j^*(A \otimes j_! j^* M) \xrightarrow{\epsilon} A^{\bullet} \otimes j_! j^* M^{\bullet}
	\end{equation}
	and
	\begin{equation}\label{aux-sup-DInd}
		A^{\bullet} \otimes i_* i^* M^{\bullet} \xrightarrow{\eta} i_* i^* A^{\bullet} \otimes i_* i^* M^{\bullet} \xrightarrow{\bar{\tilde{m}}} i_* (i^* A^{\bullet} \otimes i^* M^{\bullet})
	\end{equation}
	are invertible.
\end{lem}
\begin{proof}
	Showing that the natural transformation \eqref{aux-proj-DInd} is invertible is equivalent to showing the invertibility of the co-unit transformation
	\begin{equation*}
		\epsilon: j_! j^*(A^{\bullet} \otimes j_! j^* M^{\bullet}) \rightarrow A^{\bullet} \otimes j_! j^* M^{\bullet}.
	\end{equation*}
	The latter can be written as the composite
	\begin{equation*}
		\begin{tikzcd}[font=\small]
			j_! j^*(A^{\bullet} \otimes j_! j^* M^{\bullet}) \arrow[equal]{d} & A^{\bullet} \otimes j_! j^* M^{\bullet} \arrow[equal]{d} \\
			j_! j^* \Delta_S^*(A^{\bullet} \boxtimes j_! j_* M^{\bullet}) \arrow[equal]{d} & \Delta_S^* (A^{\bullet} \boxtimes j_! j^* M^{\bullet}) \\
			j_! \Delta_U^* (j \times j)^* (A^{\bullet} \boxtimes j_! j^* M^{\bullet}) \arrow{r}{\sim} & \Delta_S^* (j \times j)_! (j \times j)^* (A^{\bullet} \boxtimes j_! j^* M^{\bullet}) \arrow{u}{\epsilon}
		\end{tikzcd}
	\end{equation*}
	where the central arrow is invertible by \Cref{rem-pbc-DInd}(3). Thus it suffices to show the invertibility of the co-unit
	\begin{equation*}
		\epsilon: (j \times j)_! (j \times j)^* (A^{\bullet} \boxtimes j_! j^* M^{\bullet}) \rightarrow A^{\bullet} \boxtimes j_! j^* M^{\bullet},
	\end{equation*}
	which follows from the invertibility of the analogous natural transformation between functors $\Mcal^0(S) \times \Mcal^0(S) \rightarrow \Mcal^0(S \times S)$.
	The same argument allows us to conclude that the co-unit transformation of functors $D(\Ind \Mcal^0(S)) \times D^b(\Mcal^0(S)) \rightarrow D(\Ind \Mcal^0(S))$
	\begin{equation*}
		\epsilon: j_! j^* A^{\bullet} \otimes j_! j^* M^{\bullet} \rightarrow j_! j^* A^{\bullet} \otimes M^{\bullet}
	\end{equation*}
	is invertible. By the localization triangle in the second variable, this implies the vanishing of the functor $D(\Ind \Mcal^0(S)) \times D^b(\Mcal^0(S)) \rightarrow D(\Ind \Mcal^0(S))$
	\begin{equation*}
		j_! j^* A^{\bullet} \otimes i_* i^* M^{\bullet} = 0.
	\end{equation*}
	By the localization triangle in the first variable, this in turn implies that the unit transformation of functors $D(\Ind \Mcal^0(S)) \times D^b(\Mcal^0(S)) \rightarrow D(\Ind \Mcal^0(S))$
	\begin{equation*}
		\eta: A^{\bullet} \otimes i_* i^* M^{\bullet} \rightarrow i_* i^* A^{\bullet} \otimes i_* i^* M^{\bullet}
	\end{equation*}
	is invertible, which is equivalent to the invertibility of the natural transformation \eqref{aux-sup-DInd}.
\end{proof}

\begin{cor}\label{cor-projPcal}
	Keep the notation and assumptions of \Cref{lem-projPcal}. Then there exist canonical natural isomorphisms of functors $D^b(\Mcal^0(S))^{op} \times D(\Ind \Mcal^0(S)) \rightarrow D(\Ind \Mcal^0(S))$
	\begin{equation*}
		\Homcal_S(j_! j^* M^{\bullet}, B^{\bullet}) \xrightarrow{\sim} j_* \Homcal_U (j^* M^{\bullet}, j^* B^{\bullet}), \quad \Homcal_S(i_* i^* M^{\bullet},B^{\bullet}) \xrightarrow{\sim} i_* \Homcal_Z(i^* M^{\bullet}, i^! B^{\bullet}).
	\end{equation*}
\end{cor}
\begin{proof}
	For every fixed $M^{\bullet} \in D^b(\Mcal^0(S))$, the natural isomorphisms of functors $D(\Ind \Mcal^0(S)) \rightarrow D(\Ind \Mcal^0(S))$ in the statement are obtained by adjunction from those of \Cref{lem-projPcal}. Since the latter natural isomorphisms are functorial in $M^{\bullet} \in D^b(\Mcal^0(S))$, so are those in the statement.
\end{proof}

\subsection{Proof of the main result}

The strategy for the proof of \Cref{thm_homcal} is based on two abstract results that we now describe. For sake of clarity we formulate them in the setting of perverse motives and perverse sheaves, but they might be useful in more general settings as well. 

Firstly, we need a criterion ensuring that the abstract internal homomorphisms just constructed respect the categories $D^b(\Mcal(S))$.

\begin{prop}\label{prop-crit-homcal}
	Suppose that, for every $k$-variety $S$ and every $M \in \Mcal^0(S)$, there is a dense affine open immersion $j: U \hookrightarrow S$ such that the functor
	\begin{equation*}
		\Homcal_U(j^* M, -): D(\Ind \Mcal^0(U)) \rightarrow D(\Ind \Mcal^0(U))
	\end{equation*}
	sends $D^b(\Mcal^0(U))$ to itself. Then, for every $k$-variety $S$, the functor
	\begin{equation}\label{homcal-propstat}
		\Homcal_S(-,-): D^b(\Mcal^0(S))^{op} \times D(\Ind \Mcal^0(S)) \rightarrow D(\Ind \Mcal^0(S))
	\end{equation}
	sends $D^b(\Mcal^0(S))^{op} \times D^b(\Mcal^0(S))$ to $D^b(\Mcal^0(S))$.
\end{prop}
\begin{proof}
	We have to show that, for every $k$-variety $S$ and every $M^{\bullet} \in D^b(\Mcal^0(S))$, the functor
	\begin{equation*}
		\Homcal_S(M^{\bullet},-): D(\Ind \Mcal^0(S)) \rightarrow D(\Ind \Mcal^0(S))
	\end{equation*}
	sends $D^b(\Mcal^0(S))$ to itself. We argue by Noetherian induction on $S$. If $\dim(S) = 0$, the result holds by hypothesis if $M^{\bullet}$ is concentrated in degree $0$; the general case follows by dévissage, since $D^b(\Mcal^0(S))$ is a triangulated subcategory of $D(\Ind \Mcal^0(S))$ and the functor \eqref{homcal-propstat} is triangulated in the first variable.
	
	For the inductive step, suppose that $\dim(S) > 0$ and that the result is known to hold for all proper closed subvarieties of $S$. As in the previous paragraph, we reduce immediately to the case where the complex $M^{\bullet}$ is concentrated in degree $0$, so that we may identify it with an object $M \in \Mcal^0(S)$. In this case, choose a dense affine open immersion $j: U \hookrightarrow S$ satisfying the hypothesis in the statement; let $i: Z \hookrightarrow S$ denote the complementary closed immersion. For every $B^{\bullet} \in D^b(\Mcal^0(S))$, applying the functor $\Homcal_S(-,B^{\bullet})$ to the localization triangle
	\begin{equation*}
		j_! j^* M \rightarrow M \rightarrow i_* i^* M \xrightarrow{+}
	\end{equation*}
	we obtain the distinguished triangle in $D(\Ind \Mcal^0(S))$
	\begin{equation*}
		\Homcal_S(i_* i^* M,B^{\bullet}) \rightarrow \Homcal_S(M,B^{\bullet}) \rightarrow \Homcal_S(j_! j^* M,B^{\bullet}) \xrightarrow{+}.
	\end{equation*}
	Since $D^b(\Mcal^0(S))$ is a triangulated subcategory of $D(\Ind \Mcal^0(S))$, it suffices to show that the two lateral terms in the latter triangle belong to $D^b(\Mcal^0(S))$ in order to conclude the same for the term in the middle. By \Cref{cor-projPcal} we have canonical isomorphisms
	\begin{equation*}
		\Homcal_S(i_* i^* M,B^{\bullet}) = i_* \Homcal_Z(i^* M, i^! B^{\bullet}), \qquad \Homcal_S(j_! j^* M, B^{\bullet}) = j_* \Homcal_U (j^* M, j^* B^{\bullet}).
	\end{equation*}
	Since we know that $i^* M, i^! B^{\bullet} \in D^b(\Mcal^0(S))$, applying the inductive hypothesis we deduce that $\Homcal_Z(i^* M,i^! B^{\bullet}) \in D^b(\Mcal^0(Z))$, and therefore that $i_* \Homcal_Z(i^* M,i^! B^{\bullet}) \in D^b(\Mcal^0(S))$. Similarly, since we know that $j^* B^{\bullet} \in D^b(\Mcal^0(U))$, applying the hypothesis on $j^* M$ we deduce that $\Homcal_U(j^* M,j^* B^{\bullet}) \in D^b(\Mcal^0(U))$. This concludes the proof.
\end{proof}

Secondly, we need a criterion ensuring that the monoidal morphism $\iota: D^b(\Mcal^0(\cdot)) \rightarrow D^b(\Perv(\cdot))$ is compatible with internal homomorphisms.

\begin{prop}\label{prop-crit-homcal-R}
	Suppose that the hypotheses (and hence the conclusions) of \Cref{prop-crit-homcal} are satisfied. Suppose in addition that, for every $k$-variety $S$ and every $M \in \Mcal^0(S)$, there is a dense affine open immersion $j: U \hookrightarrow S$ such that the natural transformation of functors $D^b(\Mcal^0(U)) \rightarrow D^b(\Perv(U))$
	\begin{equation*}
		\iota_U \Homcal_U(j^* M,B^{\bullet}) \rightarrow \Homcal_S(\iota_U(j^* M),\iota_U(B^{\bullet}))
	\end{equation*}
	is invertible. Then, for every $k$-variety $S$, the natural transformation of functors $D^b(\Mcal^0(S)) \rightarrow D^b(\Perv(S))$
	\begin{equation*}
		\iota_S \Homcal_S(M^{\bullet},B^{\bullet}) \rightarrow \Homcal_S(\iota_S(M^{\bullet}),\iota_S(B^{\bullet}))
	\end{equation*}
	is invertible.
\end{prop}

As a preparation for the proof of \Cref{prop-crit-homcal-R}, we need to describe the relations among the various natural isomorphisms witnessing the compatibility of internal homomorphisms with open and closed immersions. For simplicity, we state these auxiliary results in the setting of (closed) monoidal stable homotopy $2$-functors of \cite[\S~2.3]{Ayo07a}.

\begin{lem}\label{lem-comm-Rproj}
	Let $R: \Hbb_1 \rightarrow \Hbb_2$ be a monoidal morphism of monoidal stable homotopy $2$-functors. Then the following assertions hold:
	\begin{enumerate}
		\item For every closed immersion $i: Z \hookrightarrow S$, and every $M \in \Hbb_1(Z)$, the diagram of functors $\Hbb_1(S) \rightarrow \Hbb_2(S)$
		\begin{equation*}
			\begin{tikzcd}
				R_S(i_*(i^* A \otimes M)) \isoarrow{dd} & i_* R_Z(i^* A \otimes M) \arrow{l}{\bar{\theta}^{-1}} & i_*(R_Z(i^* A) \otimes R_Z(M)) \arrow{l}{\rho} \\
				&& i_*(i^* R_S(A) \otimes R_Z(M)) \arrow{u}{\theta} \isoarrow{d} \\
				R_S(A \otimes i_* M) & R_S(A) \otimes R_S(i_* M) \arrow{l}{\rho}  & R_S(A) \otimes i_* R_Z(M) \arrow{l}{\bar{\theta}^{-1}}
			\end{tikzcd}
		\end{equation*}
		is commutative.
		\item For every open immersion $j: U \hookrightarrow S$, and every $M \in \Hbb_1(U)$, the diagram of functors $\Hbb_1(S) \rightarrow \Hbb_2(S)$
		\begin{equation*}
			\begin{tikzcd}
				R_S(j_!(j^* A \otimes M)) \isoarrow{dd} & j_! R_U(j^* A \otimes M) \arrow{l}{\bar{\theta}} & j_!(R_U(j^* A) \otimes R_U(M)) \arrow{l}{\rho} \\
				&& j_!(j^* R_S(A) \otimes R_U(M)) \arrow{u}{\theta} \isoarrow{d} \\
				R_S(A \otimes j_! M) & R_S(A) \otimes R_S(j_! M) \arrow{l}{\rho} & R_S(A) \otimes j_! R_U(M) \arrow{l}{\bar{\theta}}
			\end{tikzcd}
		\end{equation*}
		is commutative.
	\end{enumerate}
\end{lem}
\begin{proof}
	We only prove the second assertion; the first one can be proved in a similar way.
	Developing the definitions of the two vertical isomorphisms in the second diagram, we obtain the more explicit diagram
	\begin{equation*}
		\begin{tikzcd}[font=\small]
			R_S(j_!(j^* A \otimes M)) \arrow{d}{\eta} & j_! R_U(j^* A \otimes M) \arrow{l}{\bar{\theta}} & j_!(R_U(j^* A) \otimes R_U(M)) \arrow{l}{\rho} \\
			R_S(j_!(j^* A \otimes j^* j_! M)) \arrow{dd}{m} && j_!(j^* R_S(A) \otimes \iota_U(M)) \arrow{u}{\theta} \arrow{d}{\eta} \\
			&& j_!(j^* R_S(A) \otimes j^* j_! R_U(M)) \arrow{d}{m} \\
			R_S(j_! j^*(A \otimes j_! M)) \arrow{d}{\epsilon} && j_! j^*(R_S(A) \otimes j_! R_U(M)) \arrow{d}{\epsilon} \\
			R_S(A \otimes j_! M) & R_S(A) \otimes R_S(j_! M) \arrow{l}{\rho} & R_S(A) \otimes j_! R_U(M). \arrow{l}{\bar{\theta}}
		\end{tikzcd}
	\end{equation*}
	By inserting, on the top side of the latter, the diagram
	\begin{equation*}
		\begin{tikzcd}[font=\tiny]
			R_S(j_!(j^* A \otimes M)) \arrow{d}{\eta} & j_! R_U(j^* A \otimes M) \arrow{l}{\theta} \arrow{d}{\eta} & j_!(R_U(j^* A) \otimes R_U(M)) \arrow{l}{m} \arrow{d}{\eta} \\
			R_S(j_!(j^* A \otimes j^* j_! M)) & j_! R_U(j^* A \otimes j^* j_! M) \arrow{l}{\theta} & j_!(R_U(j^* A) \otimes R_U(j^* j_! M)) \arrow{l}{m} & j_!(j^* R_S(A) \otimes R_U(M)) \arrow{ul}{\theta} \arrow{d}{\eta} \\
			&&& R_S(A) \otimes R_U(j^* j_! M)) \arrow{ul}{\theta}
		\end{tikzcd}
	\end{equation*}
	and, on its bottom side, the diagram
	\begin{equation*}
		\begin{tikzcd}[font=\tiny]
			j_! R_U(j^*(A \otimes j_! M)) \arrow{d}{\bar{\theta}} \\
			R_S(j_! j^*(A \otimes j_! M)) \arrow{dr}{\epsilon} & j_! j^* (R_S(A \otimes j_! M)) \arrow{d}{\epsilon} \arrow{ul}{\theta} & j_! j^* (R_S(A) \otimes R_S(j_! M)) \arrow{l}{m} \arrow{d}{\epsilon} & j_! j^*(R_S(A) \otimes j_! R_U(M)) \arrow{d}{\epsilon} \arrow{l}{\bar{\theta}} \\
			& R_S(A \otimes j_! M) & R_S(A) \otimes R_S(j_! M) \arrow{l}{m} & j_! (j^* R_S(A) \otimes j_! R_U(M))\arrow{l}{\bar{\theta}}
		\end{tikzcd}
	\end{equation*}
	we are reduced to considering the outer part of the diagram
	\begin{equation*}
		\begin{tikzcd}[font=\tiny]
			R_S(j_!(j^* A \otimes j^* j_! M)) \arrow{dd}{m} & j_! R_U(j^* A \otimes j^* j_! M) \arrow{l}{\theta} \arrow{dd}{m} & j_!(R_U(j^* A) \otimes R_U(j^* j_! M)) \arrow{l}{m} \\
			&& j_!(j^* R_S(A) \otimes R_U(j^* j_! M))) \arrow{u}{\theta} \arrow{r}{\eta} & j_!(j^* R_S(A) \otimes R_U(M)) \arrow{d}{\eta} \\
			R_S(j_! j^*(A \otimes j_! M)) & j_! R_U(j^*(A \otimes j_! M)) \arrow{l}{\theta} & j_!(j^* R_S(A) \otimes j^* R_S(j_! M)) \arrow{d}{m} \arrow{u}{\theta} & j_!(j^* R_S(A) \otimes j^* j_! R_U(M)) \arrow{d}{m} \arrow{l}{\theta} \\
			& j_! j^* (R_S(A \otimes j_! M)) \arrow{u}{\theta} & j_! j^* (R_S(A) \otimes R_S(j_! M)) \arrow{l}{m}  & j_! j^*(R_S(A) \otimes j_! R_U(M)) \arrow{l}{\theta}
		\end{tikzcd}
	\end{equation*}
	where the two lateral pieces are commutative by naturality while the central rectangle is commutative by definition.
\end{proof}

\begin{cor}\label{cor-comm-Rproj}
	Let $R: \Hbb_1 \rightarrow \Hbb_2$ be a monoidal morphism of monoidal stable homotopy $2$-functors. Then the following assertions hold:
	\begin{enumerate}
		\item For every closed immersion $i: Z \hookrightarrow S$, the diagram of functors $\Hbb_1(S) \rightarrow \Hbb_2(S)$
		\begin{equation*}
			\begin{tikzcd}
				R_S i_* \Homcal_Z(M, i^! B) \arrow{r}{\bar{\theta}} & i_* R_Z \Homcal_Z(M, i^! B) \arrow{r}{\rho} & i_* \Homcal_Z(R_Z(M), R_Z(i^! B)) \arrow{d}{\theta^!} \\
				&& i_* \Homcal_Z(R_Z(M), i^! R_S(B)) \\
				R_S \Homcal_S(i_* M, B) \isoarrow{uu} \arrow{r}{\rho} & \Homcal_S(R_S(i_* M), R_S(B)) \arrow{r}{\bar{\theta}} & \Homcal_S(i_* R_Z(M), R_S(B)) \isoarrow{u}
			\end{tikzcd}
		\end{equation*}
		is commutative.
		\item For every open immersion $j: U \hookrightarrow S$, and every $K \in \Hbb_1(U)$ the diagram of functors $\Hbb_1(S) \rightarrow \Hbb_2(S)$
		\begin{equation*}
			\begin{tikzcd}
				R_S j_* \Homcal_U(M, j^! B) \arrow{r}{\bar{\theta}} & j_* R_U \Homcal_U(M, j^! B) \arrow{r}{\rho} & j_* \Homcal_U(R_U(M), R_U(j^! B)) \arrow{d}{\theta^!} \\
				&& j_* \Homcal_U(R_U(M), j^! R_S(B)) \\
				R_S \Homcal_S(j_! M, B) \isoarrow{uu} \arrow{r}{\rho} & \Homcal_S(R_S(j_! M), R_S(B)) \arrow{r}{\bar{\theta}} & \Homcal_S(j_! R_U(M), R_S(B)) \isoarrow{u}
			\end{tikzcd}
		\end{equation*}
		is commutative.
	\end{enumerate}
\end{cor}
\begin{proof}
	By adapting the argument in the proof of \cite[Prop.~2.3.57]{Ayo07a}, one sees that the two statements are formal consequences of the corresponding statements of \Cref{lem-comm-Rproj}. We omit the details.
\end{proof}

\begin{proof}[Proof of \Cref{prop-crit-homcal-R}]
	We have to show that, for every $k$-variety $S$ and every $M^{\bullet} \in D^b(\Mcal^0(S))$, the natural transformation of functors $D^b(\Mcal^0(S)) \rightarrow D^b(\Perv(S))$
	\begin{equation*}
		\iota_S \Homcal_S(M^{\bullet},B^{\bullet}) \rightarrow \Homcal_S(\iota_S(M^{\bullet}),\iota_S(B^{\bullet}))
	\end{equation*}
	is invertible. We argue by Noetherian induction on $S$. If $\dim(S) = 0$, the result holds by hypothesis if $M^{\bullet}$ is concentrated in degree $0$; the general case follows immediately by dévissage.
	
	For the inductive step, assume that $\dim(S) > 0$ and that the result is known to hold for all proper closed subvarieties of $S$. As in the previous paragraph, we may reduce immediately to the case where the complex $M^{\bullet}$ is concentrated in degree $0$, so that we may identify it with an object $M \in \Mcal^0(S)$. In this case, choose a Zariski-dense affine open immersion $j: U \hookrightarrow S$ satisfying the hypothesis in the statement; let $i: Z \hookrightarrow S$ denote the complementary closed immersion. For every $B^{\bullet} \in D^b(\Mcal^0(S))$, we obtain the morphism of distinguished triangles in $D^b(\Mcal^0(S))$
	\begin{equation}\label{dia_tr_homcal}
		\begin{tikzcd}
			\iota_S \Homcal_S(i_* i^* M,B^{\bullet}) \arrow{r} \arrow{d} & \iota_S\Homcal_S(M,B^{\bullet}) \arrow{r} \arrow{d} & \iota_S \Homcal_S(j_! j^* M,B^{\bullet}) \arrow{r}{+1} \arrow{d} & {}\\
			\Homcal_S(\iota_S(i_* i^* M), \iota_S(B^{\bullet})) \arrow{r} & \Homcal_S(\iota_S(M), \iota_S(B^{\bullet})) \arrow{r} & \Homcal_S^{\Mcal}(\iota_S(j_! j^* M), \iota_S(B^{\bullet})) \arrow{r}{+1} & {}.
		\end{tikzcd}
	\end{equation}
	To show that the central vertical arrow is an isomorphism, it suffices to show that so are the two lateral vertical arrows. But the two arrows in question are the two marked arrows in the diagrams in $D^b_c(S,\Q)$
	\begin{equation*}
		\begin{tikzcd}[font=\small]
			\iota_S i_* \Homcal_Z(i^* M, i^! B^{\bullet}) \arrow{r}{\sim} & i_* \iota_Z \Homcal_Z(i^* M, i^! B^{\bullet}) \arrow{r}{\sim} & i_* \Homcal_Z(\iota_Z(i^* M), \iota_Z(i^! B^{\bullet})) \isoarrow{d} \\
			&& i_* \Homcal_Z(\iota_Z(i^* M), i^! \iota_S(B^{\bullet})) \\
			\iota_S \Homcal_S(i_* i^* M, B^{\bullet}) \isoarrow{uu} \arrow{r}{\star} & \Homcal_S(\iota_S(i_* i^* M), \iota_S(B^{\bullet})) \arrow{r}{\sim} & \Homcal_S(i_* \iota_Z(i^* M), \iota_S(B^{\bullet})) \isoarrow{u}
		\end{tikzcd}
	\end{equation*}
	and
	\begin{equation*}
		\begin{tikzcd}[font=\small]
			\iota_S j_* \Homcal_U(j^* M, j^* B^{\bullet}) \arrow{r}{\sim} & j_* \iota_U \Homcal_U(j^* M, j^* B^{\bullet}) \arrow{r}{\sim} & j_* \Homcal_U(\iota_U(j^* M), \iota_U(j^* B^{\bullet})) \isoarrow{d} \\
			&& j_* \Homcal_U(\iota_U(j^* M), j^* \iota_S(B^{\bullet})) \\
			\iota_S \Homcal_S(j_! j^* M, B^{\bullet}) \isoarrow{uu} \arrow{r}{\star} & \Homcal_S(\iota_S(j_! j^* M), \iota_S(B^{\bullet})) \arrow{r}{\sim} & \Homcal_S(j_! \iota_U(j^* M), \iota_S(B^{\bullet})), \isoarrow{u}
		\end{tikzcd}
	\end{equation*}
	respectively. Now, taking into account \Cref{rem-pbc-DInd}(4), we see that the result of \Cref{lem-comm-Rproj} applies to our setting as well, and thus the same holds for \Cref{cor-comm-Rproj}. Therefore the two natural diagrams above are commutative. Moreover, in each of them, the vertical arrow in the lower-right corner is invertible by \Cref{cor-projPcal}, while the horizontal arrow in the upper-right corner is invertible by assumption. We deduce that the marked arrow is invertible as well.
	This concludes the proof.
\end{proof}

\begin{proof}[Proof of \Cref{thm_homcal}]
	We want to apply the criteria of \Cref{prop-crit-homcal} and \Cref{prop-crit-homcal-R}. To this end, fix an object $M \in \Mcal^0(S)$. Choose a smooth Zariski-dense affine open immersion $j: U \hookrightarrow S$ such that $j^* M \in \Mcal^{0,loc}(U)$: this is possible since the $\Sm\Op(S)$-fibered subcategory $\Mcal^{0,loc}$ of $\Mcal^0$ is weakly cofinal by \Cref{lem:Mcal^0loc-wcof}.
	Up to writing $U$ as the disjoint union of its connected components and treating each component separately, we may assume that $U$ is connected of dimension $d \geq 0$. We divide the proof into two main steps: first we check that the hypotheses are satisfied when $j^* M$ lies in the image of $\mathbf{R}(\Dcal^{0,loc}(U))$, then we explain how to deduce the general case. 
	
	Suppose first that $j^* M$ lies in the image of $\mathbf{R}(\Dcal^{0,loc}(U))$ under the quotient functor $\pi_U: \A(\Dcal^{0,loc}(U)) \rightarrow \Mcal^{0,loc}(U)$. 
	In this case, by \Cref{cor:adj-Mloc}, we have a canonical adjunction of exact functors
	\begin{equation}\label{adj:F-G-proofHomcal}
		F: \Mcal^0(U) \rightleftarrows \Mcal^0(U): G,
	\end{equation}
	where $F := - \otimes j^* M[-d]$ and $G := \Dbb_U(j^* M[-d] \otimes \Dbb_U(-))$.
	In order to deduce the thesis, the main point is to make sure that the induced exact functor $\Ind(F): \Ind \Mcal^0(U) \rightarrow \Ind \Mcal^0(U)$ is naturally isomorphic to (the co-restriction of) the functor
	\begin{equation*}
		- \otimes j^* M[-d]: \Ind \Mcal^0(U) \rightarrow D(\Ind \Mcal^0(U))
	\end{equation*} 
	obtained from \Cref{constr_homcal}. 
	In order to establish this claim, we only need to analyze carefully the various constructions involved. Since the functor $\Delta_{U,*}: \Ind \Mcal^0(U) \rightarrow \Ind \Mcal^0(U \times U)$ is exact and fully faithful by \Cref{prop-IndF}, this is equivalent to saying that the composite functor
	\begin{equation*}
		\Ind \Mcal^0(U) \xrightarrow{\Ind(F)} \Ind \Mcal^0(U) \xrightarrow{\Delta_{U,*}} \Ind \Mcal^0(U \times U)
	\end{equation*} 
	is naturally isomorphic to (the co-restriction of) the composite functor
	\begin{equation*}
		\Ind \Mcal^0(U) \xrightarrow{- \otimes j^* M[-d]} D(\Ind \Mcal^0(U)) \xrightarrow{\Delta_{U,*}} D(\Ind \Mcal^0(U \times U)).
	\end{equation*}
	By construction, the latter can be written as the composite
	\begin{equation*}
		\Ind \Mcal^0(U) \xrightarrow{- \boxtimes j^* M[-d]} \Ind \Mcal^0(U \times U)[-d] \xrightarrow{\Delta_{U,*} \Delta_U^*} D(\Ind \Mcal^0(U \times U)).
	\end{equation*}
	Moreover, following the idea of \Cref{constr:cech-perv}, the triangulated functor $\Delta_{U,*} \Delta_U^*: \Ind \Mcal^0(U \times U)[-d] \rightarrow D(\Ind \Mcal^0(U \times U))$ admits a model of the form
	\begin{equation*}
		\begin{aligned}
			\Ind \Mcal^0(U \times U)[-d] \xrightarrow{\check{C}_{U \times U,\Ucal}^{\bullet}} C^{[-n,d-n]}(\Ind \Mcal^0(U \times U)) &\rightarrow D^{[-n,d-n]}(\Ind \Mcal^0(U \times U)) \\ 
			&\subset D(\Ind \Mcal^0(U \times U)),
		\end{aligned}
	\end{equation*}
	where $\Ucal$ denotes some fixed affine open covering of $U \times U \setminus \Delta_U(U)$ with $n$ elements. Now, for every Grothendieck abelian category $\Acal$, the additive functors $H^i: C^b(\Acal) \rightarrow \Acal$ commute with filtered colimits. Since the category $\Ind \Mcal^0(U \times U)$ is Grothendieck abelian by \Cref{prop-IndAcal}(1), we deduce that the composite functors
	\begin{equation*}
		\Ind \Mcal^0(U) \xrightarrow{- \boxtimes j^* M[-d]} \Ind \Mcal^0(U \times U)[-d]  \xrightarrow{\check{C}_{U \times U,\Ucal}^{\bullet}} C^{[-n,d-n]}(\Ind \Mcal^0(U \times U)) \xrightarrow{H^i} \Ind \Mcal^0(U \times U)
	\end{equation*}
	vanish for every $i \neq 0$, and that the composite functor
	\begin{equation*}
		\Ind \Mcal^0(U) \xrightarrow{- \boxtimes j^* M[-d]} \Ind \Mcal^0(U \times U)[-d]  \xrightarrow{\check{C}_{U \times U,\Ucal}^{\bullet}} C^{[-n,d-n]}(\Ind \Mcal^0(U \times U)) \xrightarrow{H^0} \Ind \Mcal^0(U \times U)
	\end{equation*} 
	commutes with filtered colimits. Hence the two functors considered in the claim both satisfy the conditions of \Cref{prop-IndF}(1). By the essential uniqueness in the statement of \Cref{prop-IndF}(1), we conclude that they are naturally isomorphic as claimed.
	Now, consider the adjunction
	\begin{equation*}
		\Ind(F): D(\Ind \Mcal^0(U)) \rightleftarrows D(\Ind \Mcal^0(U)): \Ind(G)
	\end{equation*}
	obtained from \eqref{adj:F-G-proofHomcal} by applying \Cref{cor:ind-adj} and then passing to the derived categories. As a consequence of the claim just proved, the triangulated functor $\Ind(F): D(\Ind \Mcal^0(U)) \rightarrow D(\Ind \Mcal^0(U))$ is naturally isomorphic to the functor 
	\begin{equation*}
		- \otimes j^* M[-d]: D(\Ind \Mcal^0(U)) \rightarrow D(\Ind \Mcal^0(U))
	\end{equation*}
	obtained from \Cref{constr_homcal}: indeed, since the latter triangulated functor is dg-enhanced by construction and $t$-exact for the obvious $t$-structures by the claim, by \Cref{prop:VoloDG}(1) it is naturally isomorphic to the trivial derived functor of the induced exact functor between the hearts.
	Taking the right adjoints, we deduce that the functor $\Ind(G): D(\Ind \Mcal^0(U)) \rightarrow D(\Ind \Mcal^0(U))$ is naturally isomorphic to the triangulated functor $\Homcal_U(j^* M[-d],-)$ obtained from \Cref{constr_homcal}. But $\Ind(G)$ sends $D^b(\Mcal^0(U))$ to itself by construction, and the induced natural transformation of functors $D^b(\Mcal^0(U)) \rightarrow D^b(\Perv(U))$
	\begin{equation*}
		\iota_U G(B^{\bullet}) \rightarrow \Homcal_U(\iota_U(j^* M)[-d], \iota_U(B^{\bullet}))
	\end{equation*}
	is invertible as a consequence of Corollary \ref{cor:adj-Mloc}. We deduce that the triangulated functor
	\begin{equation*}
		\Homcal_U(j^* M,-) = \Homcal_U(j^* M[-d],-)[-d]: D(\Ind \Mcal^0(U)) \rightarrow D(\Ind \Mcal^0(U))
	\end{equation*}  
	sends $D^b(\Mcal^0(U))$ to itself, and that the natural transformation of functors $D^b(\Mcal^0(U)) \rightarrow D^b(\Perv(U))$
	\begin{equation*}
		\iota_U \Homcal_U(j^* M,B^{\bullet}) \rightarrow \Homcal_U(\iota_U(j^* M),\iota_U(B^{\bullet}))
	\end{equation*}
	is invertible. This concludes the proof in this particular case.
	
	For a general choice of $M \in \Mcal^0(S)$, choose a co-presentation of the form
	\begin{equation*}
		j^* M = \ker\left\{\pi_U(C_1) \rightarrow \pi_U(C_2)\right\}
	\end{equation*} 
	for some morphism $C_1 \rightarrow C_2$ in $\mathbf{R}(\Dcal^{0,loc}(U))$. 
	Write $N_1 := \pi_U(C_1)$, $N_2 := \pi_U(C_2)$ for convenience, and set $N_3 := \coker\left\{\pi_U(C_1) \rightarrow \pi_U(C_2)\right\}$. Since the image of $\mathbf{R}(\Dcal^{0,loc}(U))$ in $\A(\Dcal^{0,loc}(U))$ is closed under cokernels (as every cokernel of presheaves of finite presentation is itself of finite presentation), the object $N_3$ belongs to the image of $\mathbf{R}(\Dcal^{0,loc}(U))$. By construction, the object $j^* M$ is canonically isomorphic to the complex $[N_1 \rightarrow N_2 \rightarrow N_3]$ in $D^{[0,3]}(\Mcal^0(U))$. Moreover, by the previous step of the proof, each functor
	\begin{equation*}
		\Homcal_U(N_i,-): D(\Ind \Mcal^0(U)) \rightarrow D(\Ind \Mcal^0(U)) \qquad (i = 1,2,3)
	\end{equation*}
	sends $D^b(\Mcal^0(U))$ to itself, and each natural transformation of functors $D^b(\Mcal^0(U)) \rightarrow D^b(\Perv(U))$ 
	\begin{equation*}
		\iota_U \Homcal_S(N_i,B^{\bullet}) \rightarrow \Homcal_U(\iota_U(N_i),\iota_U(B^{\bullet})) \qquad (i = 1,2,3)
	\end{equation*}
	is invertible. Since both of these properties are triangulated with respect to the first variable, we conclude that they hold for $M$ as well. This concludes the proof.
\end{proof}

\begin{rem}
	As explained in \cite[\S~6]{IM19}, there exists a theory of weights for perverse motives compatible with the analogous notion of weight for perverse sheaves of geometric origin. By construction, the operations $f^*$, $f_*$, $f_!$ and $f^!$ constructed in \cite{IM19} on the level of perverse motives enjoy the same weight-exactness properties as the corresponding operations on perverse sheaves as described in \cite[\S~5.1.14]{BBD82}. 
	
	The same observation applies to the tensor product and internal homomorphisms functors constructed here: namely, given a $k$-variety $S$, and letting $D^b_{\leq w}(\Mcal(S))$ (resp. $D^b_{\geq w}(\Mcal(S))$) denote the full subcategory of $D^b(\Mcal(S))$ collecting the complexes of weights at most $w$ (resp. at least $w$), the internal tensor product functor restricts to
	\begin{equation*}
		- \otimes -: D^b_{\leq w_1}(\Mcal(S)) \times D^b_{\leq w_2}(\Mcal(S)) \rightarrow D^b_{\leq w_1 + w_2}(\Mcal(S)),
	\end{equation*}
    while the internal homomorphisms functor restricts to
    \begin{equation*}
    	\Homcal_S(-,-): D^b_{\leq w_1}(\Mcal(S)) \times D^b_{\geq w_2}(\Mcal(S)) \rightarrow D^b_{\geq -w_1 + w_2}(\Mcal(S)).
    \end{equation*}
\end{rem}

\subsection{Motivic Verdier duality}

We close this section with some useful remarks about Verdier duality for perverse motives.
Recall that, for every $k$-variety $S$, the motivic Verdier duality functor
\begin{equation*}
	\Dbb_S: \Mcal(S) \rightarrow \Mcal(S)^{op}
\end{equation*} 
is defined in \cite[\S~2.8]{IM19} as an application of the general lifting principles of universal abelian factorizations. 
The functor $\Dbb_S$ canonically defines an involutive anti-equivalence of $\Mcal(S)$: 
there is a natural isomorphism between functors $\Mcal(S) \rightarrow \Mcal(S)$
\begin{equation*}
	M \xrightarrow{\sim} \Dbb_S(\Dbb_S(M)),
\end{equation*} 
induced by the usual bi-duality isomorphisms on $\DA_{ct}^{\et}(S,\Q)$ and on $\Perv(S)$. 

Typically, Verdier duality in a six functor formalism is related to internal homomorphisms in a precise way, as explained in \cite{Ayo07a}.
In the setting of perverse motives, a partial result in this direction is given by \Cref{cor:adj-Mloc}.
In order to say more, we have to use the Nori realization of Voevodsky motivic sheaves, obtained by Tubach in \cite{Tub23} and, independently, by Jacobsen and the author in \cite{JacTer25}:
\begin{thm}[{\cite[Thm.~4.5]{Tub23}}, {\cite[Thm.~6.3]{JacTer25}}]\label{thm:Nri-real}
	There exists a canonical monoidal morphism between closed monoidal stable homotopy $2$-functors
	\begin{equation}\label{Nri^*:sh2f}
		\Nri^*: \DA_{ct}^{\et}(\cdot,\Q) \rightarrow D^b(\Mcal(\cdot)),
	\end{equation}
	with the property that the composite morphism
	\begin{equation*}
		\DA_{ct}^{\et}(\cdot,\Q) \xrightarrow{\Nri^*} D^b(\Mcal(\cdot)) \xrightarrow{\iota} D^b(\Perv(\cdot)) = D^b_c(\cdot,\Q)
	\end{equation*}
	recovers the Betti realization $\Bti^*: \DA_{ct}^{\et}(\cdot,\Q) \rightarrow D^b_c(\cdot,\Q)$ up to canonical $2$-isomorphism.
\end{thm}
\noindent
We do not need to discuss the details of the two proof;
let us just mention that both of them rely heavily on our construction of the monoidal structure on perverse motives, up to \Cref{thm_homcal} (the argument of \cite{JacTer25} also uses \Cref{thm_Mloc} below, which does not depend on the results of the present subsection).
In both approaches, the realization from Voevodsky motivic sheaves is ultimately an application of Drew--Gallauer's universality theorem \cite[Thm.~7.14]{DrewGal}.

Note that, since the Nori realization \eqref{Nri^*:sh2f} is compatible with the Betti realization, the composite functor
\begin{equation*}
	\DA_{ct}(S,\Q) \xrightarrow{\Nri_S^*} D^b(\Mcal(S)) \xrightarrow{H^0} \Mcal(S)
\end{equation*}
is naturally isomorphic to the structural homological functor $\pi_S \colon \DA_{ct}(S,\Q) \to \Mcal(S)$.
Here is the promised application to motivic Verdier duality:
\begin{prop}\label{prop:MVerdier}
	Fix a $k$-variety $S$, with structural morphism $a_S: S \rightarrow \Spec(k)$. 
	Then:
	\begin{enumerate}
		\item There is a canonical natural isomorphism of functors $D^b(\Mcal(S)) \rightarrow D^b(\Mcal(S))^{op}$
		\begin{equation*}
			\Dbb_S(M^{\bullet}) = \Homcal_S(M^{\bullet}, a_S^! \unit_k).
		\end{equation*}
	    \item There is a canonical natural isomorphism of functors $D^b(\Mcal(S))^{op} \times D^b(\Mcal(S)) \rightarrow D^b(\Mcal(S))$
	    \begin{equation*}
	    	\Homcal_S(M^{\bullet},N^{\bullet}) \xrightarrow{\sim} \Dbb_S(M^{\bullet} \otimes \Dbb_S(N^{\bullet})).
	    \end{equation*}
	\end{enumerate} 
\end{prop}
\begin{proof}
	\begin{enumerate}
		\item Using the conservativity of $\iota_S: D^b(\Mcal(S)) \rightarrow D^b(\Perv(S)) = D^b_c(S,\Q)$, one easily sees that the natural transformation of functors $\DA_{ct}^{\et}(S,\Q)^{op} \times \DA_{ct}^{\et}(S,\Q) \rightarrow \DA_{ct}^{\et}(S,\Q)$
		\begin{equation*}
			\Nri_S^* \Homcal_S(A,B) \rightarrow \Homcal_S(\Nri_S^* (A), \Nri_S^* (B))
		\end{equation*}
	    induced by the monoidality of $\Nri_S^*$ is necessarily invertible, since the same holds true for the Betti realization. 
	    Hence we get a canonical natural isomorphism of functors $\DA_{ct}^{\et}(S,\Q)^{op} \rightarrow D^b(\Mcal(S))$
	    \begin{equation*}
	    	\Nri_S^* (\Homcal_S(A, a_S^! \unit_k)) \xrightarrow{\sim} \Homcal_S(\Nri_S^* (A), a_S^! \unit_k).
	    \end{equation*}
	    Since Verdier duality on $\DA^{\et}_{ct}(S,\Q)$ is defined by the formula
	    \begin{equation*}
	    	\Dbb_S(A) := \Homcal_S(A,a_S^! \unit_{\Spec(k)}),
	    \end{equation*}
	    the above natural isomorphism satisfies the hypotheses of the recognition criterion \cite[Cor.~2.9]{Ter23UAF}.
	    The thesis follows.
		\item We have a canonical natural transformation of functors $D^b(\Mcal(S))^{op} \times D^b(\Mcal(S)) \rightarrow D^b(\Mcal(S))$
		\begin{equation}\label{arrow:bidual}
			\Homcal_S(M^{\bullet},N^{\bullet}) \rightarrow \Homcal_S(M^{\bullet} \otimes \Homcal_S(N^{\bullet},a_S^! \unit_k),a_S^! \unit_k)
		\end{equation}
		whose value on the pair $(M^{\bullet},N^{\bullet})$ is the morphism in $D^b(\Mcal(S))$ corresponding under adjunction to the composite
		\begin{equation*}
			\begin{tikzcd}
				M^{\bullet} \otimes \Homcal_S(N^{\bullet},a_S^! \unit_k) \otimes \Homcal_S(M^{\bullet},N^{\bullet}) \arrow[equal]{d} \\
				\Homcal_S(N^{\bullet},a_S^! \unit_k) \otimes M^{\bullet} \otimes \Homcal_S(M^{\bullet},N^{\bullet}) \arrow[equal]{d} \\
				\Homcal_S(N^{\bullet},a_S^! \unit_k) \otimes \Homcal_S(M^{\bullet},N^{\bullet}) \otimes M^{\bullet} \arrow{r}{ev} & \Homcal_S(N^{\bullet},a_S^! \unit_k) \otimes N^{\bullet} \arrow{r}{ev} & a_S^! \unit_k.
			\end{tikzcd}
		\end{equation*}
		Since the analogous natural transformation on $D^b(\cdot,\Q)$ is known to be invertible, it follows from \Cref{thm_homcal} that \eqref{arrow:bidual} is invertible as well. Combining this with the previous point, we obtain the result.
	\end{enumerate}
\end{proof}

\begin{rem}
	Clearly, the natural isomorphism of \Cref{prop:MVerdier}(1) extends the one of \Cref{cor:adj-Mloc}.
\end{rem}

\section{Applications and complements}\label{sect:appli-compl}

This final section is devoted to important consequences of our main construction.
In the first part, we define Tannakian categories of motivic local systems over smooth, geometrically connected varieties. 
We also show that the derived categories of perverse motives satisfy the axiomatic of \textit{constructible systems} from Gallauer's paper \cite{GalSupp}, and we deduce a topological reconstruction theorem for perverse motives.
In the final part, using the Nori realization of Voevodsky motivic sheaves, we define a canonical orientation on the derived categories of perverse motives.
This yields a construction of Chern class of vector bundles in terms of the six operations, which leads to a motivic version of the relative Hard Lefschetz Theorem.

\subsection{Tannakian categories of motivic local systems}

We refer to \cite{DM82} for the general theory of Tannakian categories.
As explained in \cite[\S~II.9]{HMS17}, Nori's $\Q$-linear abelian category $\Mcal(k)$ is neutral Tannakian over $\Q$, with canonical fibre functor provided by the forgetful functor $\iota_k: \Mcal(k) \hookrightarrow \vect_{\Q}$. 
The Tannaka dual of $\Mcal(k)$ is a pro-affine $\Q$-group scheme $\Gcal_{mot}(k)$, called the \textit{motivic Galois group} of $k$. 

Over more general bases, a candidate category of motivic local systems was proposed by Arapura in \cite[\S~6]{Ara-mot}, within his own theory of motivic sheaves;
see also \cite[\S~2]{Ara-rev}. 
While internal homomorphisms are not available in full generality in Arapura's setting, one can still show that these categories of motivic local systems are Tannakian by using Beilinson's Basic Lemma and Poincaré--Verider duality. 
A technical defect of Arapura's construction is that his categories of motivic local systems are not visibly full inside the corresponding categories of motivic sheaves. 

Here we propose an alternative approach which exploits the six functor formalism on perverse motives in its full strength. 
Let $S$ be a smooth, geometrically connected $k$-variety of dimension $d$. Recall the subcategory of shifted local systems $\Loc_p(S) \subset \Perv(S)$ introduced in \Cref{nota-diamond1}.
Endowed with the shifted tensor product
\begin{equation*}
	- \otimes^{\dagger} -: \Loc_p(S) \times \Loc_p(S) \rightarrow \Loc_p(S), \quad (\Fcal,\Gcal) \rightsquigarrow (\Fcal[-d] \otimes \Gcal[-d])[d]
\end{equation*}
it becomes a neutral Tannakian category over $\Q$ with unit object $K_S := \unit_S[d]$. 
Indeed, the shifted tensor product is designed in such a way that the diagram
\begin{equation*}
	\begin{tikzcd}
		\Loc_p(S) \times \Loc_p(S) \arrow{rr}{- \otimes^{\dagger} -} \arrow{d}{[-d] \times [-d]} && \Loc_p(S) \arrow{d}{[-d]} \\
		\Loc(S) \times \Loc(S) \arrow{rr}{- \otimes -} && \Loc(S)
	\end{tikzcd}
\end{equation*}
commutes. 
This also means that $\Loc_p(S)$ can be endowed with compatible associativity, commutativity, and unit constraints, obtained by transport of structure from those of the neutral Tannakian category $\Loc(S)$: 
as the interested reader can easily check, the validity of the various coherence and mutual compatibility conditions for the constraints on $\Loc_p(S)$ is just a formal consequence of the definitions.
Every object $\Fcal \in \Loc_p(S)$ has a strong dual $\Fcal^{\lor} \in \Loc_p(S)$, defined by the formula
\begin{equation*}
	\Fcal^{\lor} := \Homcal_S(\Fcal[-d],\unit_S)[d].
\end{equation*}
Again, this is designed in such a way that the diagram
\begin{equation*}
	\begin{tikzcd}
		\Loc_p(S)^{op} \arrow{rr}{(-)^{\lor}} \arrow{d}{[-d]} && \Loc_p(S) \arrow{d}{[-d]} \\
		\Loc(S)^{op} \arrow{rr}{(-)^{\lor}} && \Loc(S)
	\end{tikzcd}
\end{equation*}
commutes.  
The category $\Loc_p(S)$ is neutral Tannakian over $\Q$:
the fibre functor at a point $s \in S^{\sigma}$ is the shifted inverse image functor $s^{\dagger} \colon \Loc_p(S) \to \vect_{\Q}$.

 
\begin{defn}\label{defn:MLoc_p}
	Let $S$ be a smooth, geometrically connected $k$-variety. 
	We introduce the full abelian subcategory
	\begin{equation*}
		\Mcal \Loc_p(S) := \left\{M \in \Mcal(S) \; | \; \iota_S(M) \in \Loc_p(S) \right\} \subset \Mcal(S).
	\end{equation*}
	We call \textit{motivic (shifted) local systems} over $S$ the objects of $\Mcal\Loc_p(S)$.
\end{defn}

\begin{rem}\label{rem:MLoc_p-subquots-ext}
	The abelian subcategory $\Loc_p(S)$ of $\Perv(S)$ is stable under subquotients and extensions:
	stability under subquotients follows from the description of simple perverse sheaves (in particular, see \cite[Lem.~4.3.3]{BBD82}), while stability under extensions follows from the fact that ordinary local systems are stable under extensions inside ordinary sheaves.
	Hence $\Mcal\Loc_p(S)$ enjoys the same stability properties inside $\Mcal(S)$.
\end{rem}

We can prove straightaway the main structural result about motivic local systems:
\begin{thm}\label{thm_Mloc}
	Let $S$ be a smooth, geometrically connected $k$-variety of dimension $d$. 
	The abelian category $\Mcal\Loc_p(S)$, endowed with the shifted tensor product
	\begin{equation}\label{shift-otimes}
		- \otimes^{\dagger} -: \Mcal\Loc_p(S) \times \Mcal\Loc_p(S) \rightarrow \Mcal\Loc_p(S), \quad (M,N) \rightsquigarrow (M[-d] \otimes N[-d])[d],
	\end{equation}
	is neutral Tannakian over $\Q$.
\end{thm}
\begin{proof}
	Arguing as for $\Loc_p(S)$, one can define compatible associativity, commutativity, and unit constraints on $\Mcal\Loc_p(S)$ with respect to the shifted tensor product \eqref{shift-otimes}, starting from the homologous constraints for the usual tensor product on $D^b(\Mcal\Loc_p(S))$. 
	In this way the faithful exact functor 
	\begin{equation}\label{iota_S:Loc}
		\iota_S \colon \Mcal \Loc_p(S) \hookrightarrow \Loc_p(S)
	\end{equation} 
    becomes a tensor functor. 
    Composing any fibre functor of $\Loc_p(S)$ with the forgetful functor \eqref{iota_S:Loc}, we obtain a fibre functor for $\Mcal\Loc_p(S)$.
	The motivic shifted unit object $K^{\Mcal}_S := \unit_S^{\Mcal}[d] \in \Mcal \Loc_p(S)$ satisfies
	\begin{equation*}
		\Hom_{\Mcal \Loc_p(S)}(K_S^{\Mcal}, K_S^{\Mcal}) = \Q,
	\end{equation*} 
	because the same holds for the shifted unit object of $\Loc_p(S)$ and the functor \eqref{iota_S:Loc} is faithful and conservative. Thus, in order to conclude that $\Mcal \Loc_p(S)$ is neutral Tannakian over $\Q$, we only have to prove the existence of strong duals inside it.
	Given an object $M \in \Mcal(S)$, it is natural to define its dual $M^{\lor}$ by the formula
	\begin{equation*}
		M^{\lor} := \Homcal_S(M[-d], \unit_S^{\Mcal})[d].
	\end{equation*}
	As a consequence of \Cref{thm_homcal}, this defines a functor $\Mcal \Loc_p(S) \rightarrow \Mcal \Loc_p(S)^{op}$ endowed with a canonical isomorphisms of functors $\Mcal \Loc_p(S) \rightarrow \Loc_p(S)^{op}$
	\begin{equation*}
		\iota_S(M^{\lor}) \xrightarrow{\sim} (\iota_S M)^{\lor}.
	\end{equation*} 
	In order to conclude that the construction $M \rightsquigarrow M^{\lor}$ defines the strong dual in $\Mcal \Loc_p(S)$, we have to check the following two conditions:
	\begin{enumerate}
		\item[(i)] The canonical natural transformation of functors $\Mcal \Loc_p(S)^{op} \times \Mcal \Loc_p(S) \rightarrow \Mcal \Loc_p(S)$
		\begin{equation*}
			M^{\lor}[-d] \otimes N[-d] := \Homcal_S(M[-d],\unit^{\Mcal}_S) \otimes N[-d] \rightarrow \Homcal_S(M[-d],N[-d])
		\end{equation*}
	    corresponding under adjunction to
	    \begin{equation*}
	    	\begin{tikzcd}
	    		\Homcal_S(M[-d],\unit_S^{\Mcal}) \arrow{r}{N[-d] \otimes -} & \Homcal_S(N[-d] \otimes M[-d],N[-d]) \arrow[equal]{d}  \\
	    		& \Homcal_S(N[-d], \Homcal_S(M[-d],N[-d])).
	    	\end{tikzcd}
	    \end{equation*}
        is invertible.
		\item[(ii)] The canonical natural transformation of functors $\Mcal \Loc_p(S) \rightarrow \Mcal \Loc_p(S)$
		\begin{equation*}
			M \rightarrow \Homcal_S(\Homcal_S(M[-d],\unit^{\Mcal}_S)[d][-d],\unit^{\Mcal}_S)[d] =: (M^{\lor})^{\lor}
		\end{equation*}
	    defined as the composite
	    \begin{equation*}
	    	M = M[-d][d] \rightarrow \Homcal_S(\Homcal_S(M[-d],\unit^{\Mcal}_S),\unit^{\Mcal}_S)[d] = \Homcal_S(\Homcal_S(M[-d],\unit^{\Mcal}_S)[d][-d],\unit^{\Mcal}_S)[d]
	    \end{equation*}
	    is invertible.
	\end{enumerate} 
    As usual, the validity of each of these two conditions follows from the validity of the corresponding condition on $\Loc_p(S)$ and the conservativity of \eqref{iota_S:Loc}. 
    This concludes the proof.
\end{proof}

\begin{rem}
	The definition of the shifted tensor product in \eqref{shift-otimes} does not coincide with the one given in \Cref{sect_otimes-dist}. Of course, the two functors in question are naturally isomorphic via any of the two composites in the diagram of functors $\Mcal \Loc_p(S) \times \Mcal \Loc_p(S) \rightarrow \Mcal \Loc_p(S)[-d]$
	\begin{equation*}
		\begin{tikzcd}
			M[-d] \otimes N[-d] \arrow{r}{\sim} \isoarrow{d} & (M \otimes N[-d])[-d] \isoarrow{d} \\
			(M[-d] \otimes N)[-d] \arrow{r}{\sim} & (M \otimes N)[-2d].
		\end{tikzcd}
	\end{equation*}
	However, the latter diagram is $(-1)^d$-commutative. Hence, if $d$ is odd, the two versions of the shifted tensor product are not canonically isomorphic. The definition of \Cref{sect_otimes-dist} is adapted to the setting of \Cref{prop-distinguished-otimes}, but it is not well-suited to defining the shifted associativity, commutativity and unit constraints.
\end{rem}

\begin{defn}
	Let $S$ be a smooth, geometrically connected $k$-variety.
	For every point $s \in S^{\sigma}$, we let $\Gcal_{mot}(S,s)$ denote the Tannaka dual of $\Mcal\Loc_p(S)$ with respect to the fibre functor
	\begin{equation*}
		\Mcal\Loc_p(S) \xrightarrow{\iota_S} \Loc_p(S) \xrightarrow{s^{\dagger}} \vect_{\Q},
	\end{equation*}
	and we call it the \textit{motivic Galois group} of $S$ with base-point $s$.
\end{defn}

By construction, the motivic Galois group $\Gcal_{mot}(S,s)$ is a pro-affine $\Q$-group scheme.
The most natural choice for the base-point $s$ is a closed point in $S(\bar{k})$.
In this case, the difference between $\Gcal_{mot}(S,s)$ and Nori's motivic Galois group $\Gcal_{mot}(k)$ can be measured precisely in terms of the theory of local systems of geometric origin on the complex-analytic space $S^{\sigma}$:
the fundamental case, when $s$ is a $k$-rational point, is established by Jacobsen in \cite[Thm.~7.7]{JacMalcev}, and the general case is deduced in \cite[Thm.~2.19]{JacTer25}.

\subsection{Reconstruction theorem}

Gallauer's paper \cite{GalSupp} describes abstract reconstruction results in the framework of tensor-triangulated geometry. 
The general question is how to recover topological information about a $k$-variety $S$ from the structure of a tensor-triangulated category attached to it. 
This is particularly interesting when such a category is the value at $S$ of a monoidal $\Var_k$-fibered category:
in this case, Gallauer's approach allows one to address the topological reconstruction question for all $k$-varieties at once. 

In \cite[Defn.~2.3]{GalSupp}, a \textit{constructible system} over $\Var_k$ is defined as a monoidal $\Var_k$-fibered category $\Hbb$ satisfying the following two conditions:
\begin{enumerate}
	\item[(i)] Given an open immersion of $k$-varieties $j \colon U \hookrightarrow S$, with complementary closed immersion $i \colon Z \hookrightarrow S$, the functors
	\begin{equation*}
		\Hbb(U) \xleftarrow{j^*} \Hbb(S) \xrightarrow{i^*} \Hbb(Z)
	\end{equation*} 
	define a \textit{recollement} in the sense of \cite[\S~1.4.3]{BBD82}.
	\item[(ii)] For every smooth $k$-variety $S$, there is a monoidal triangulated full subcategory $\Hbb^{ls}(S) \subset \Hbb(S)$ such that:
	\begin{itemize}
		\item[(a)] given a $k$-variety $S$, for every $A \in \Hbb(S)$ there exists a dense open immersion $j \colon U \hookrightarrow S$ with $U$ smooth such that $j^* A \in \Hbb_{ls}(U)$;
		\item[(b)] for every (locally closed) immersion $i \colon Z \hookrightarrow S$ between smooth $k$-varieties, the functor $i^* \colon \Hbb(S) \to \Hbb(Z)$ takes $\Hbb^{ls}(S)$ to $\Hbb_{ls}(Z)$;
		\item[(c)] if $S$ is smooth and connected, for every schematic point $s \in S(k)$ with closure $Z$, the natural functor
		\begin{equation*}
			\Hbb_{ls}(S) \to \Hbb(Z) \to \twocolim_{V \in \Op(Z)^{op}} \Hbb(V)
		\end{equation*}  
		is conservative.
	\end{itemize}
\end{enumerate}
In \cite[Defn.~2.13]{GalSupp}, a constructible system $\Hbb$ over $\Var_k$ is called \textit{generically simple} if, for every $k$-variety $S$, the monoidal triangulated category
\begin{equation*}
	\twocolim_{U \in \Sm\Op(S)^{op}} \Hbb(U)
\end{equation*}
is simple, meaning that it contains no non-zero proper thick tensor-ideal.

Let $\Hbb$ be a generically simple constructible system over $\Var_k$.
By \cite[Thm.~5.21]{GalSupp}, the Zariski topological space underlying any given $k$-variety $S$ can be canonically recovered as the smashing spectrum of the tensor-triangulated category $\Hbb(S)$, provided the following \textit{Lefschetz-type hypothesis} is satisfied:
For every connected $k$-curve $C$, we have $\Hom_{\Hbb(C)}(\unit_C,\unit_C[-1]) = 0$ and, for every dense open immersion $j \colon U \hookrightarrow C$, the map
\begin{equation*}
	\Hom_{\Hbb(C)}(\unit_C,\unit_C) \to \Hom_{\Hbb(U)}(\unit_U,\unit_U)
\end{equation*}  
is injective.
As explained in \cite[\S~3]{GalSupp}, many six functor formalisms of interests give rise to generically constructible systems satisfying the Lefschetz-type hypothesis:
for example, by \cite[Prop.~3.1]{GalSupp}, this is the case for the $\Var_k$-fibered category $D^b_c(\cdot,\Q)$, at least when $k$ is algebraically closed;
in general the generic simplicity fails, but a weaker reconstruction result holds (see \cite[Thm.~4.8]{GalSupp}).
Other known examples of constructible systems include $\ell$-adic sheaves and mixed Hodge modules.

Conjecturally, the $\Var_k$-fibered category $\DA_{ct}(\cdot,\Q)$ should define a constructible system as well (see \cite[\S~3.5]{GalSupp}).
In the setting of perverse motives, we get an unconditional result:
\begin{prop}\label{prop:constr-syst}
	The $\Var_k$-fibered category $D^b(\Mcal(\cdot))$ is a constructible system, which is moreover generically simple if $k$ is algebraically closed.
\end{prop}
\begin{proof}
	The fact that, for every open immersion $j \colon U \hookrightarrow S$ with complementary closed immersion $i \colon Z \hookrightarrow S$, the functors
	\begin{equation*}
		D^b(\Mcal(U)) \xleftarrow{j^*} D^b(\Mcal(S)) \xrightarrow{i^*} D^b(\Mcal(Z))
	\end{equation*}
	define a recollement follows from the basic properties of stable homotopy $2$-functors (see \cite[\S~1.4.4]{Ayo07a}).
	For every smooth $k$-variety $S$, we define the subcategory of lisse objects in $D^b(\Mcal(S))$ to be
	\begin{equation*}
		D^b_{ls}(\Mcal(S)) := \left\{M^{\bullet} \in D^b(\Mcal(S)) \; | \; H^i(M^{\bullet}) \in \Mcal \Loc_p(S) \;\; \forall i \in \Z \right\} \subset D^b(\Mcal(S)).
	\end{equation*}
	By construction, a complex $M^{\bullet} \in D^b(\Mcal(S))$ belongs to $D^b_{ls}(\Mcal(S))$ if and only if the ordinary cohomology objects of the underlying complex $\iota_S(M^{\bullet}) \in D^b(\Perv(S)) = D^b_c(S,\Q)$ are local systems: 
	in other words, $D^b_{ls}(\Mcal(S))$ is the inverse image of the subcategory 
	\begin{equation*}
		D^b_{ls}(S,\Q) := \left\{K^{\bullet} \in D^b_c(S,\Q) \; | \; H^i(K^{\bullet}) \in \Loc(S) \;\; \forall i \in \Z \right\}
	\end{equation*}
	under $\iota_S: D^b(\Mcal(S)) \rightarrow D^b(\Perv(S)) = D^b_c(S,\Q)$.
	Since $D^b_{ls}(S,\Q)$ is a tensor-triangulated subcategory of $D^b_c(S,\Q)$, we deduce that $D^b_{ls}(\Mcal(S))$ is a tensor-triangulated subcategory of $D^b(\Mcal(S))$. 
	As observed in \cite{GalSupp}, the subcategories $D^b_{ls}(S,\Q)$ satisfy the three conditions in \cite[Defn.~2.3]{GalSupp}.
	It is easy to see that each of these conditions implies the analogous condition for the subcategories $D^b_{ls}(\Mcal(S))$.
	Thus we have a constructible system.
	
	Now assume that $k$ is algebraically closed, and let us check that this constructible system is generically simple:
    we have to show that, for every smooth connected $k$-variety $S$, the tensor-triangulated category 
	\begin{equation*}
		\twocolim_{U \in \Sm\Op(S)^{op}} D^b(\Mcal(U))
	\end{equation*} 
	is simple. 
	To check this claim, it is convenient to replace the tensor product by the shifted tensor product, which is adapted to shifted local systems;
	this does not affect the conclusion.
	We have canonical equivalences of tensor-triangulated categories
	\begin{align*}
		\twocolim_{U \in \Sm\Op(S)^{op}} D^b(\Mcal(U))
		&= D^b(\twocolim_{U \in \Sm\Op(S)^{op}} \Mcal(U)) &&\textup{by \cite[Lemma~2.6]{GalNote}} \\
		&= D^b(\twocolim_{U \in \Sm\Op(S)^{op}} \Mcal \Loc_p(U)) && \textup{by cofinality}.
	\end{align*}
	Since $k$ is assumed to be algebraically closed, each $U \in \Sm\Op(S)$ is geometrically connected.
	Hence, by \Cref{thm_Mloc}, the abelian tensor categories $\Mcal \Loc_p(U)$ are neutral Tannakian over $\Q$, and the restriction functors along inclusions in $\Sm\Op(S)$ form a filtered system of exact tensor functors. 
	The simplicity of the tensor-triangulated category 
	\begin{equation*}
		D^b(\twocolim_{U \in \Sm\Op(S)^{op}} \Mcal \Loc_p(U))
	\end{equation*} 
	follows from \cite[Prop.~2.18]{GalSupp}. 
	This concludes the proof.
\end{proof}

This leads us to the following motivic reconstruction result:
\begin{thm}\label{thm:reconstr}
	Suppose that $k$ is algebraically closed.
	Then the underlying topological space of a $k$-variety $S$ is canonically homeomorphic to the smashing spectrum of the tensor-triangulated category $D^b(\Mcal(S))$.
\end{thm}
\begin{proof}
	It suffices to show that the tensor-triangulated category $D^b(\Mcal(S))$ satisfies the hypotheses of \cite[Thm.~5.21]{GalSupp}. 
	In view of \Cref{prop:constr-syst}, it remains to check that the the constructible system $D^b(\Mcal(\cdot))$ satisfies the Lefschetz-type condition on curves.
	
	In order to deal with possibly singular curves, we need to use the constructible $t$-structure on $D^b(\Mcal(S))$, defined at the end of \cite[\S~5.3]{IM19}:
	an object $M^{\bullet} \in D^b(\Mcal(S))$ belongs to its heart $\Mcal_{ct}(S)$ if and only if the object $\iota_S(M^{\bullet}) \in D^b(\Perv(S)) = D^b_c(S,\Q)$ belongs to the usual constructible heart $\Ct(S)$.
    In particular, for every $k$-variety $S$, the motivic unit object $\unit_S^{\Mcal}$ of \Cref{prop-unitMcal} belongs to $\Mcal_{ct}(S)$. The conservative functor $\iota_S: D^b(\Mcal(S)) \rightarrow D^b(\Perv(S))$ restricts to a faithful exact functor $\iota_S: \Mcal_{ct}(S) \hookrightarrow \Ct(S)$.
	
	Hence, for an irreducible $k$-curve $C$, the vanishing of $\Hom_{D^b(\Mcal(C))}(\unit_C,\unit_C[-1])$ follows directly from the semi-orthogonality axiom of $t$-structures \cite[Defn.~1.3.1(i)]{BBD82}, while the injectivity of $\Hom_{D^b(\Mcal(C))}(\unit_C,\unit_C) \to \Hom_{D^b(\Mcal(U))}(\unit_U,\unit_U)$ follows from the commutative diagram
	\begin{equation*}
		\begin{tikzcd}
			\Hom_{\Mcal_{ct}(C)}(\unit^{\Mcal}_C, \unit^{\Mcal}_C) \arrow[hook]{d} \arrow{r} & \Hom_{\Mcal_{ct}(U)}(\unit^{\Mcal}_U, \unit^{\Mcal}_U) \arrow[hook]{d} \\
			\Hom_{\Ct(C)}(\unit_C,\unit_C) \arrow[hook]{r} & \Hom_{\Ct(U)}(\unit_U,\unit_U),
		\end{tikzcd}
	\end{equation*}
	where the two vertical maps are injective by faithfulness and
	the lower horizontal arrow is known to be injective.
	This concludes the proof.
\end{proof}

\begin{rem}
	As remarked in \cite[Cor.~5.2]{GalSupp}, in certain cases the scheme $S$ is completely determined by the knowledge of its underlying topological space. In these cases, it follows that the scheme $S$ is determined by the tensor-triangulated category $D^b(\Mcal(S))$.
\end{rem}

\subsection{Motivic relative Hard Lefschetz}

The existence of the Nori realization, as stated in \Cref{thm:Nri-real}, leads to a well-behaved theory of Chern classes in the setting of perverse motives.
This is based on the abstract notion of \textit{orientation} on stable homotopy $2$-functors, as introduced in~\cite[\S~2.4.c]{CisDeg}, and we start with a reminder about this.

Let $\Hbb$ be a stable homotopy $2$-functor in the sense of \cite{Ayo07a,Ayo07b}.
To every $k$-variety $S$ and every vector bundle $E$ over $S$ one associates the \textit{Thom equivalence}
\begin{equation*}
	\Thom_{E/S}\colon \Hbb(S) \xrightarrow{s_*} \Hbb(E) \xrightarrow{p_{\#}} \Hbb(S),
\end{equation*}
where $p\colon E \rightarrow S$ denotes the structural projection and $s\colon S \hookrightarrow E$ denotes the zero-section.
The fact that $\Thom_{E/S}$ defines an autoequivalence of $\Hbb(S)$ is the content of the stability axiom of stable homotopy $2$-functors (see~\cite[Defn.~1.4.1]{Ayo07a}).
Thom equivalences are important in the construction of the exceptional functors $f_!$ and $f^!$, as described in~\cite[\S~1.5]{Ayo07a}.
Let us point out the following remarkable properties:
\begin{enumerate}
	\item[(i)] Fix a $k$-variety $S$. 
	For $i = 1,2$, let $E_i$ be vector bundle over $S$, with projection $p_i\colon E_i \rightarrow S$ and zero-section $s_i\colon S \hookrightarrow E_i$; 
	moreover, let $g\colon E_1 \xrightarrow{\sim} E_2$ be an isomorphism of vector bundles over $S$, and consider the commutative diagram
	\begin{equation*}
		\begin{tikzcd}
			& S \arrow{dl}{s_1} \arrow{dr}{s_2} \\
			E_1 \arrow{rr}{g} \arrow{dr}{p_1} && E_2 \arrow{dl}{p_2} \\
			& S.
		\end{tikzcd}
	\end{equation*}
	Then there is a canonical natural isomorphism between functors $\Hbb(S) \rightarrow \Hbb(S)$
	\begin{equation*}
		\psi_g\colon \Thom_{E_1/S}(A) := p_{1,\#} s_{1,*} A = p_{2,\#} g_{\#} s_{1,*} A \xrightarrow{\sim} p_{2,\#} g_* s_{1,*} A = p_{2,\#} s_{2,*} A =: \Thom_{E_2/S}(A),
	\end{equation*}
	where the central arrow is induced by the natural isomorphism $g_{\#} \xrightarrow{\sim} g_*$.
	\item[(ii)] Fix a morphism of $k$-varieties $f\colon T \rightarrow S$. 
	Let $E$ be vector bundle over $S$, with projection $p\colon E \rightarrow S$ and zero-section $s\colon S \hookrightarrow E$; let $E_T := E \times_S T$ be the induced vector bundle over $T$, with projection $p_T\colon E_T \rightarrow T$ and zero-section $t\colon T \hookrightarrow E_T$, and consider the diagram with Cartesian squares
	\begin{equation*}
		\begin{tikzcd}
			T \arrow{d}{t} \arrow{r}{f} & S \arrow{d}{s} \\
			E_T \arrow{r}{f_E} \arrow{d}{p_T} & E \arrow{d}{p} \\
			T \arrow{r}{f} & S.
		\end{tikzcd}
	\end{equation*}
	Then there is a canonical natural isomorphism between functors $\Hbb(S) \rightarrow \Hbb(T)$
	\begin{equation*}
		\xi_f\colon f^* \Thom_{E/S}(A) := f^* p_{T,\#} s_* A \xleftarrow{\sim} p_{T,\#} f_E^* s_* \xrightarrow{\sim} p_{T,\#} t_* f^* A =: \Thom_{E_T/T}(f^* A),
	\end{equation*}
	where the two exchange transformations are both invertible by smooth and proper base-change.
	\item[(iii)] Fix a $k$-variety $S$. 
	Let
	\begin{equation*}
		0 \rightarrow E' \xrightarrow{i} E \xrightarrow{q} E'' \rightarrow 0
	\end{equation*}
	be a short exact sequence of vector bundle over $S$, with projections
	\begin{equation*}
		p' \colon E' \rightarrow S, \quad p \colon E \rightarrow S, \quad p'' \colon E'' \rightarrow S
	\end{equation*}
	and zero-sections
	\begin{equation*}
		s' \colon S \hookrightarrow E', \quad s \colon S \hookrightarrow E, \quad s'' \colon S \hookrightarrow E'',
	\end{equation*}
	and consider the commutative diagram with Cartesian square
	\begin{equation*}
		\begin{tikzcd}
			S \arrow{d}{s'} \arrow{dr}{s} \\
			E' \arrow{r}{i} \arrow{d}{p'} & E \arrow{d}{q} \arrow{dr}{p} \\
			S \arrow{r}{s''} & E'' \arrow{r}{p''} & S.
		\end{tikzcd}
	\end{equation*}
	Then there is a canonical natural isomorphism between functors $\Hbb(S) \rightarrow \Hbb(S)$
	\begin{equation*}
		\Thom_{E/S}(A) := p_{\#} s_* A = p''_{\#} q_{\#} i_* s'_* A \xrightarrow{\sim} p''_{\#} s''_* p'_{\#} s'_{\#} A =: \Thom_{E''/S}(\Thom_{E'/S}(A)),
	\end{equation*}
	where the central arrow is invertible by smooth base-change. If moreover $\Hbb$ is monoidal in the sense of~\cite[\S~2.3]{Ayo07a}, there is another natural isomorphism between functors $\Hbb(S) \rightarrow \Hbb(S)$
	\begin{equation*}
		\begin{aligned}
			\Thom_{E''/S}(B) &:= p''_{\#} s''_* B = p''_{\#} s''_*(\unit_S \otimes B) = p''_{\#} s''_* (\unit_S \otimes {s''}^* {p''}^* B) \xrightarrow{\sim} \\
			&\xrightarrow{\sim} p''_{\#}(s''_* \unit_S \otimes {p''}^* B) \xrightarrow{\sim} p''_{\#} s''_* \unit_S \otimes p''_{\#} {p''}^* B \xrightarrow{\epsilon} p''_{\#} s''_* \unit_S \otimes B =: \Thom_{E''/S}(\unit_S) \otimes B,
		\end{aligned}
	\end{equation*}
	where the central isomorphism witnesses the projection formula. Applying this to $B = \Thom_{E'/S}(A)$, and combining it with the previous isomorphism, we obtain a natural isomorphism of functors $\Hbb(S) \rightarrow \Hbb(S)$
	\begin{equation*}
		\phi_{i,q}\colon \Thom_{E/S}(A) \xrightarrow{\sim} \Thom_{E''/S}(\unit_S) \otimes \Thom_{E'/S}(A).
	\end{equation*}
\end{enumerate}
Recall from~\cite[Defn.~2.4.38]{CisDeg} that an \textit{orientation} $\tfr$ on $\Hbb$ is the datum of
\begin{itemize}
	\item for every $k$-variety $S$ and every rank $r$ vector bundle $E$ over $S$, an isomorphism in $\Hbb(S)$
	\begin{equation*}
		\tfr_{E/S}\colon \Thom_{E/S}(\unit_S) \xrightarrow{\sim} \unit_S(r)[2r],
	\end{equation*}
	called the \textit{Thom isomorphism},
\end{itemize}
satisfying the following conditions:
\begin{enumerate}
	\item[(a)] For every $k$-variety $S$, and every isomorphism of rank $r$ vector bundles $g\colon E_1 \xrightarrow{\sim} E_2$ over $S$, the diagram in $\Hbb(S)$
	\begin{equation*}
		\begin{tikzcd}
			\Thom_{E_1/S}(\unit_S) \arrow{rr}{\psi_g} \arrow{dr}{\tfr_{E_1/S}} && \Thom_{E_2/S}(\unit_S) \arrow{dl}{\tfr_{E_2/S}} \\
			& \unit_S(r)[2r]
		\end{tikzcd}
	\end{equation*}
	commutes.
	\item[(b)] For every morphism of $k$-varieties $f\colon T \rightarrow S$, and every rank $r$ vector bundle $E$ over $S$, the diagram in $\Hbb(T)$
	\begin{equation*}
		\begin{tikzcd}
			f^* \Thom_{E/S}(\unit_S) \arrow{rr}{\xi_f} \arrow{d}{\tfr_{E/S}} && \Thom_{E_T/T}(\unit_T) \arrow{d}{\tfr_{E_T/T}} \\
			f^* \unit_S(r)[2r]  \arrow{rr}{\sim} && \unit_T(r)[2r]
		\end{tikzcd}
	\end{equation*}
	commutes.
	\item[(c)] For every $k$-variety $S$, and every short exact sequence of vector bundles over $S$
	\begin{equation*}
		0 \rightarrow E' \xrightarrow{i} E \xrightarrow{q} E'' \rightarrow 0,
	\end{equation*}
	of respective ranks $r'$, $r$, $r''$, the diagram in $\Hbb(S)$
	\begin{equation*}
		\begin{tikzcd}
			\Thom_{E/S}(\unit_S) \arrow{d}{\tfr_{E/S}} \arrow{rr}{\phi_{i,q}} && \Thom_{E''/S}(\unit_S) \otimes \Thom_{E'/S}(\unit_S) \arrow{d}{\tfr_{E''/S} \otimes \tfr_{E'/S}}  \\
			\unit_S[2r](r) \arrow{rr}{\sim} && \unit_S[2r''](r'') \otimes \unit_S[2 r'](r')
		\end{tikzcd}
	\end{equation*}
	commutes.
\end{enumerate}
The structure of an orientation allows one to construct Chern classes of vector bundles in terms of the six operations on $\Hbb$.
Let us focus on the first Chern class of a line bundle, which is the only case needed in our applications.
Given a $k$-variety $S$, and a line bundle $L$ over $S$, with projection $p \colon L \to S$ and zero-section $s \colon S \hookrightarrow L$, one defines the \textit{first Chern class} to be the morphism in $\Hbb(S)$
\begin{equation*}
	c_1(L) \colon \unit_S = s^* s_* \unit_S \xrightarrow{\eta} s^* p^* p_{\#} s_* \unit_S = p_{\#} s_* \unit_S =: \Thom_{L/S}(\unit_S) \xrightarrow{\tfr_{L/S}} \unit_S(1)[2],
\end{equation*}
where the third passage is induced by the natural isomorphism $s^* p^* = \id_{\Hbb(S)}$.

In general, the existence of orientations is not a mere formal by-product of the six functor formalism.
In fact, not every stable homotopy $2$-functor can be endowed with an orientation:
for instance, this is not possible for Morel--Voevodsky's stable homotopy theory $\mathsf{SH}(\cdot,\Q)$.
The existence of an orientation on $\mathsf{SH}(\cdot,\Q)$ would imply the invertibility of the natural transformation of functors $\mathsf{SH}(S,\Q) \rightarrow \mathsf{SH}(S,\Q)$
\begin{equation}\label{eq:Hopf-fibr}
	\eta \colon A \rightarrow H_B \otimes A
\end{equation}
induced by the algebraic Hopf fibration.
In fact, the full subcategory of $\mathsf{SH}(S,\Q)$ on which \eqref{eq:Hopf-fibr} is invertible is precisely the category of \textit{Beilinson motives}, introduced in~\cite[\S~14.2]{CisDeg}. 
As a consequence of~\cite[Thm.~14.3.4]{CisDeg}, the latter is equivalent to Ayoub's category $\DA^{\et}(S,\Q)$, and this identification is compatible with the six operations.
The theory of Beilinson motives admits a canonical orientation: 
it is constructed by the method of~\cite{Deg08} starting from the canonical identification
\begin{equation*}
	\mathsf{Pic}(S) = \Hom_{\DA^{\et}(S,\Q)}(\unit_S, \unit_S(1)[2]).
\end{equation*}
Since the Thom isomorphisms only involve constructible objects, this orientation is completely described within the smaller stable homotopy $2$-functor $\DA^{\et}_{ct}(\cdot,\Q)$.

Note that, given a monoidal morphism of stable homotopy $2$-functors $R \colon \Hbb_1 \to \Hbb_2$, any orientation $\tfr^{(1)}$ for $\Hbb_1$ induces an orientation $\tfr^{(2)}$ for $\Hbb_2$:
to every $k$-variety $S$, and every rank $r$ vector bundle $E$ over $S$, one defines the Thom isomorphism $\tfr^{(2)}_{E/S} \colon \Thom_{E/S}(\unit_S) \xrightarrow{\sim} \unit_S(r)[2r]$ as the composite
\begin{equation}
	\tfr^{(2)}_{E/S} \colon \Thom_{E/S}(\unit_S) = \Thom_{E/S}(R_S(\unit_S)) = R_S(\Thom_{E/S}(\unit_S)) \xrightarrow{\tfr^{(1)}_{E/S}} R_S(\unit_S) = \unit_S,
\end{equation}
where the various equalities witness the compatibility of $R$ with Thom equivalences.
The validity of conditions (a), (b), (c) for $\tfr^{(2)}$ follows formally from the validity of the same conditions for $\tfr^{(1)}$;
we leave the details to the interested reader (see \cite[Ex.~2.4.40]{CisDeg} for a related discussion).
In particular, the canonical orientation of $\DA^{\et}(\cdot,\Q)$ induces an orientation on every stable homotopy $2$-functor to which it maps.
Applying this to the Betti realization
\begin{equation*}
	\Bti^* \colon \DA^{\et}_{ct}(\cdot,\Q) \to D^b(\cdot,\Q),
\end{equation*}
one recovers the classical theory of Chern classes of topological vector bundles:
for instance, the first Chern class of line bundles over $S$ is described by the map
\begin{equation*}
	\mathsf{Pic}(S) = \Hom_{\DA^{\et}_{ct}(S,\Q)}(\unit_S, \unit_S(1)[2]) \xrightarrow{\Bti_S^*} \Hom_{D^b(S,\Q)}(\unit_S,\unit_S(1)[2]).
\end{equation*}
This leads to the following application to perverse motives:
\begin{prop}\label{prop:orient_Nori}
	The stable homotopy $2$-functor $D^b(\Mcal(\cdot))$ is endowed with a canonical orientation $\tfr^{\Mcal}$ such that, for every $k$-variety $S$, and every rank $r$ vector bundle $E$ over $S$, the diagram in $D^b(S,\Q)$
	\begin{equation}\label{dia:Thom-iota}
		\begin{tikzcd}
			\iota_S(\Thom_{E/S}(\unit^{\Mcal}_S)) \arrow{r}{\tfr^{\Mcal}_{E/S}} \arrow[equal]{d} & \iota_S(\unit_S^{\Mcal}(r)[2r]) \arrow[equal]{d} \\
			\Thom_{E/S}(\unit_S) \arrow{r}{\tfr_{E/S}} & \unit_S(r)[2r]
		\end{tikzcd}
	\end{equation}
	commutes.
\end{prop}
\begin{proof}
	By the above discussion, the Nori realization $\Nri^* \colon \DA^{\et}_{ct}(\cdot,\Q) \to D^b(\Mcal(\cdot))$ induces an orientation on $D^b(\Mcal(\cdot))$.
	The compatibility of the latter with the orientation of $D^b(\cdot,\Q)$ follows formally from the $2$-isomorphism $\Bti^* = \iota \circ \Nri^*$ of \Cref{thm:Nri-real}. 
\end{proof}

\begin{cor}\label{cor:chern}
	For every $k$-variety $S$, and every line bundle $L$ over $S$, there exists a canonical morphism in $D^b(\Mcal(S))$
	\begin{equation*}
		c_1^{\Mcal}(L) \colon \unit^{\Mcal}_S \rightarrow \unit^{\Mcal}_S(1)[2]
	\end{equation*}
	rendering the diagram in $D^b(S,\Q)$
	\begin{equation}\label{dia:c_1-Nri}
		\begin{tikzcd}
			\iota_S(\unit^{\Mcal}_S) \arrow{rr}{c_1^{\Mcal}(L)} \arrow[equal]{d} && \iota_S(\unit_S(1)[2]) \arrow[equal]{d} \\
			\unit_S \arrow{rr}{c_1(L)} && \unit_S(1)[2]
		\end{tikzcd}
	\end{equation}
	commutative.
\end{cor}
\begin{proof}
	In view of the way the first Chern class is defined starting from an orientation, the commutativity of \eqref{dia:c_1-Nri} follows formally from the commutativity of \eqref{dia:Thom-iota}.
\end{proof}

We can finally state the main consequence of this discussion:
\begin{thm}[Motivic relative Hard Lefschetz]\label{thm:Lefs}
	Let $f \colon X \rightarrow S$ be a projective morphism of $k$-varieties, and fix a relatively ample line bundle $L$ over $X$. 
	Then, for every $r \geq 0$, the natural transformation between functors $\Mcal(X) \rightarrow \Mcal(S)$
	\begin{equation}\label{eq:mot_rHL}
		{^p H^{-r}} (f_* M) = {^p H^{-r}}(f_* M \otimes \unit_S^{\Mcal}) \xrightarrow{c_1^{\Mcal}(L)^r} {^p H^{-r}}(f_* M \otimes \unit_S^{\Mcal}(r)[2r]) = {^p H^r} (f_* M)(r)
	\end{equation}
	is invertible when evaluated on semi-simple objects.
\end{thm}

The proof is based on the theory of weights for perverse motives, as developed in~\cite[\S~6]{IM19}.
For every $k$-variety $S$, objects of $\Mcal(S)$ are endowed with a strictly functorial weight filtration lifting the one on the underlying perverse sheaves of geometric origin, defined in \cite{BBD82};
when $S = \Spec(k)$, this was proved by Arapura in \cite[Thm.~6.3.5, Thm.~6.3.6]{Ara-mot} (see \cite[Thm.~10.2.5]{HMS17} for an alternative proof).
Note that the abelian category $\Mcal(S)$ is Artinian and Noetherian, since the same holds for $\Perv(S)$.
Thus every object $M \in \Mcal(S)$ can be obtained in a finite number of steps by successive extensions of simple objects.
Every simple object is necessarily pure of some weight;
conversely, for every $w \in \Z$, the full abelian subcategory $\Mcal(S)_w \subset \Mcal(S)$ spanned by the objects which are pure of weight $w$ is semi-simple by~\cite[Thm.~6.24]{IM19}.
It follows that a perverse motive is semi-simple if and only if it is the direct sum of its weight-graded pieces.
Moreover, there is a natural notion of \textit{strict support} for perverse motives, introduced in \cite[Defn.~6.21]{IM19};
it is compatible with the classical notion for perverse sheaves.
By \cite[Prop.~6.23]{IM19}, every simple perverse motive $M \in \Mcal(S)$ has strict support on a unique irreducible closed subvariety $Z$ of $S$, and in particular $M \in \Mcal_Z(S)$.

For our applications, we need to relate semi-simplicity of a perverse motive to that of the underlying perverse sheaf of geometric origin:
\begin{lem}\label{lem:MLoc_p-semisimple}
	Let $S$ be a $k$-variety.
	If an object $M \in \Mcal(S)$ is semi-simple, then the object $\iota_S(M) \in \Perv(S)$ is semi-simple as well.
\end{lem}
\begin{proof}
	Without loss of generality, we may assume that $M$ is simple.
	In this case, $M$ has strict support on an irreducible closed subvariety $Z$ of $S$.
	By definition, the perverse sheaf $\iota_S(M)$ has strict support in $Z$ as well.
	Up to replacing $S$ by $Z$ via the equivalences $\Mcal(Z) \xrightarrow{\sim} \Mcal_Z(S)$ and $\Perv(Z) \xrightarrow{\sim} \Perv_Z(S)$, we may assume that $Z = S$.
	
	In this case, fix a dense open immersion $j \colon U \hookrightarrow S$ with $U$ smooth such that the object $j^* M \in \Mcal(U)$ belongs to $\Mcal\Loc_p(U)$.
	In order to conclude, we want to prove the chain of implications
	\begin{equation*}
		\textup{$M$ simple} \implies \textup{$j^* M$ simple} \implies \textup{$\iota_U(j^* M) = j^* \iota_S(M)$ semi-simple} \implies \textup{$\iota_S(M)$ semi-simple}.
	\end{equation*}
	To begin with, note that $M$ is pure (being simple), say of weight $w$.
	By definition of weights, $j^* M$ is then pure of weight $w$ as well.
	Since $M$ has strict support in $S$, it is isomorphic to $j_{!*} j^* M$ by definition.
	But the intermediate extension functor $j_{!*} : \Mcal(U) \to \Mcal(S)$ is faithful (since it has a right inverse $j^*$) and becomes exact when restricted to $\Mcal(S)_w$ (by \cite[Prop.~6.20]{IM19}). 
	Therefore, any non-zero proper subobject $N \subset j^* M$ would yield a non-zero proper subobject $j_{!*} N \subset M$.
	We deduce that $j^* M$ is simple in $\Mcal(U)$, hence also in $\Mcal\Loc_p(U)$.
	
	Since $\Loc_p(U)$ is stable under subquotients inside $\Perv(U)$ (see \Cref{rem:MLoc_p-subquots-ext}), the object $\iota_U(j^* M)$ is semi-simple in $\Perv(U)$ if and only if it is semi-simple in $\Loc_p(U)$.
	The fact that $\iota_U(j^* M)$ is semi-simple as a (shifted) local system was proved by Jacobsen in \cite[Lem.~7.6]{JacMalcev}:
	the key point is that it underlies a pure variation of Hodge structure, to which one can apply Deligne's Semi-simplicity Theorem.
	This argument is based on the Tubach's results for mixed Hodge modules, in particular \cite[Prop.~4.9]{Tub23}.
	
	Lastly, since $\iota_S(M)$ has strict support in $S$, it is isomorphic to $j_{!*} j^* \iota_S(M)$ by definition.
	Writing the direct sum decomposition
	\begin{equation*}
		j^* \iota_S(M) = \bigoplus_{i=1}^{n} L_i
	\end{equation*}
	into simple shifted local systems $L_i$, we get the decomposition in $\Perv(S)$
	\begin{equation*}
		\iota_S(M) = \bigoplus_{i=1}^{n} j_{!*} L_i.
	\end{equation*}
	Since each $j_{!*} L_i$ is simple in $\Perv(S)$ by \cite[Thm.~4.3.1(ii)]{BBD82}, we see that $\iota_S(M)$ is semi-simple.
\end{proof}

\begin{proof}[Proof of \Cref{thm:Lefs}]
	It suffices to check that the image of \eqref{eq:mot_rHL} under the conservative functor $\iota_S \colon D^b(\Mcal(S)) \to D^b_c(S,\Q)$ is invertible.
	Using \Cref{cor:chern}, this is equivalent to checking that the morphism in $\Perv(S)$
	\begin{equation*}
		{^p H^{-r}}(f_* \iota_X(M)) = {^p H^{-r}}(f_* \iota_X(M) \otimes \unit_S) \xrightarrow{c_1(L)^r} {^p H^r}(f_* \iota_X(M) \otimes \unit_S(r)[2r]) = {^p H^r}(f_* \iota_X(M))(r)
	\end{equation*}
	is invertible as soon as $M$ is semi-simple.
	Since $\iota_X(M)$ is then semi-simple by \Cref{lem:MLoc_p-semisimple}, this follows from the classical relative Hard Lefschetz Theorem~\cite[Thm.~6.2.10]{BBD82}.
\end{proof}

In the case where $S = \Spec(k)$, \Cref{thm:Lefs} yields a motivic version of the Hard Lefschetz Theorem for intersection cohomology.
Recall that the intersection cohomology groups of a projective $k$-variety $X$ of pure dimension $d \geq 0$ are defined as
\begin{equation*}
	IH^n(X;\Q) := H^{n-d}(a_{X,*} j_{!*} K_U),
\end{equation*}
where $a_X \colon X \to \Spec(k)$ denotes the structural projection, $j \colon U \hookrightarrow X$ is a dense open immersion with $U$ smooth affine, and $K_U := \unit_U[d]$ denotes the shifted unit object.
By \cite[Thm.~4.3.1(ii)]{BBD82}, the perverse sheaf $K_U$ is semi-simple by \cite[Lem.~4.3.3]{BBD82}.
Hence the Hard Lefschetz Theorem gives an isomorphism
\begin{equation*}
	IH^{d-r}(X;\Q) := H^{-r}(a_{X,*} j_{!*} K_U) \xrightarrow{\sim} H^r(a_{X,*} j_{!*} K_U)(r) =: IH^{d+r}(X;\Q)(r),
\end{equation*}
and \Cref{thm:Lefs} asserts that this lifts to an isomorphism of the corresponding Nori motives.

These Nori motives are pure by \cite[Cor.~6.29]{IM19}, hence semi-simple by \cite[Thm.~6.24]{IM19}.
Using the theory of motivic Chern classes, we get further information: 
\begin{prop}\label{prop:primitive}
	Let $X$ be a projective $k$-variety of pure dimension $d \geq 0$, and let $L$ be an ample line bundle over $X$.
	Then, for every every $0 \leq r \leq d$, the primitive intersection cohomology group
	\begin{equation*}
		P^{d-r}(X;\Q) := \ker\left\{c_1(L)^{r+1} \colon IH^{d-r}(X;\Q) \to IH^{d+r+2}(X;\Q)(r + 1)\right\}
	\end{equation*} 
	canonically underlies a Nori motive over $k$.
\end{prop}
\begin{proof}
	Since the functor $\iota_k \colon \Mcal(k) \to \vect_{\Q}$ is exact, this follows immediately from \Cref{cor:chern}.
\end{proof}

\begin{cor}\label{cor:Lefs-dec}
	Let $X$ be a projective $k$-variety of pure dimension $d \geq 0$, and let $L$ be an ample line bundle over $X$.
	Then, for every $n \geq 0$, the Lefschetz decomposition in intersection cohomology
	\begin{equation*}
		IH^n(X;\Q) = \bigoplus_{r\geq \max(n-d,0)} c_1(L)^r (P^{n-2r}(X;\Q))(-r)
	\end{equation*}
	underlies a decomposition of Nori motives over $k$.
\end{cor}
\begin{proof}
	The classical argument deducing the Lefschetz decomposition from the Hard Lefschetz Theorem (for example, see \cite[Rem.~1.31]{PetSte}) carries over to the motivic setting. 
\end{proof}

If $X$ is a smooth projective $k$-variety, intersection cohomology groups recover usual cohomology groups. 
To the author's knowledge, the result of \Cref{cor:Lefs-dec} is new even in this case.

\end{document}